\pdfoutput=1

\documentclass[10pt]{article}

\usepackage{url}
\usepackage{mathtools}
\usepackage{amssymb}
\usepackage{amsthm}
\usepackage{latexsym}
\usepackage[shortlabels]{enumitem}
\usepackage{dsfont}
\usepackage{appendix}
\usepackage{color} 
\usepackage[utf8]{inputenc}
\usepackage[T1]{fontenc}
\usepackage{geometry}
\usepackage{todonotes}
\usepackage{lmodern}
\usepackage{anyfontsize}
\usepackage{stmaryrd}
\usepackage{comment}
\usepackage{bm}
\usepackage{mathrsfs}
\usepackage{mdframed}
\usepackage{natbib}
\usepackage[english]{babel}
\usepackage{cases}
\usepackage{braket}

\usepackage{hyperref}
\hypersetup{colorlinks, linkcolor={red}, citecolor={green}, urlcolor={blue}}
\numberwithin{equation}{section}

\newtheorem{theorem}{Theorem}[section]
\newtheorem{assumption}[theorem]{Assumption}
\newtheorem{standing_assumption}[theorem]{Standing assumption}
\newtheorem{corollary}[theorem]{Corollary}
\newtheorem{example}[theorem]{Example}

\newtheorem{lemma}[theorem]{Lemma}
\newtheorem{proposition}[theorem]{Proposition}

\newtheorem{definition}[theorem]{Definition}
\newtheorem{remark}[theorem]{Remark}

\usepackage[capitalize,nameinlink,noabbrev]{cleveref}

\crefname{assumption}{assumption}{assumptions}
\Crefname{assumption}{Assumption}{Assumptions}
\crefname{assumptio}{assumption}{assumptions}
\Crefname{assumptio}{Assumption}{Assumptions}
\crefname{assumptionalt}{assumption}{assumptions}
\Crefname{assumptionalt}{Assumption}{Assumptions}
\crefname{lemma}{lemma}{lemmas}
\Crefname{lemma}{Lemma}{Lemmas}
\crefname{proposition}{proposition}{propositions}
\Crefname{proposition}{Proposition}{Propositions}
\crefname{definition}{definition}{definitions}
\Crefname{definition}{Definition}{Definitions}
\crefname{remark}{remark}{remarks}
\Crefname{remark}{Remark}{Remarks}
\crefname{appendix}{appendix}{appendices}
\Crefname{appendix}{Appendix}{Appendices}
\crefname{corollary}{corollary}{corollaries}
\Crefname{corollary}{Corollary}{Corollaries}
\crefname{example}{example}{examples}
\Crefname{example}{Example}{Examples}

\AddToHook{env/assumption/begin}{\crefalias{theorem}{assumption}}
\AddToHook{env/lemma/begin}{\crefalias{theorem}{lemma}}
\AddToHook{env/proposition/begin}{\crefalias{theorem}{proposition}}
\AddToHook{env/definition/begin}{\crefalias{theorem}{definition}}
\AddToHook{env/remark/begin}{\crefalias{theorem}{remark}}
\AddToHook{env/corollary/begin}{\crefalias{theorem}{corollary}}
\AddToHook{env/example/begin}{\crefalias{theorem}{example}}

\mdfsetup{middlelinecolor=blue,middlelinewidth=2pt,linewidth=0pt,backgroundcolor=blue!15,,roundcorner=10pt}

\allowdisplaybreaks

\setcitestyle{numbers,open={[},close={]}}

\definecolor{red}{rgb}{0.7,0.15,0.15}
\definecolor{green}{rgb}{0,0.5,0}
\definecolor{blue}{rgb}{0,0,0.7}

\newcommand\cA{\mathcal A}
\newcommand\cB{\mathcal B}

\newcommand\cE{\mathcal E}
\newcommand\cF{\mathcal F}
\newcommand\cG{\mathcal G}
\newcommand\cH{\mathcal H}

\newcommand\cL{\mathcal L}
\newcommand\cM{\mathcal M}
\newcommand\cN{\mathcal N}
\newcommand\cO{\mathcal O}
\newcommand\cP{\mathcal P}

\newcommand\cS{\mathcal S}
\newcommand\cT{\mathcal T}
\newcommand\cU{\mathcal U}
\newcommand\cV{\mathcal V}
\newcommand\cW{\mathcal W}

\newcommand\cY{\mathcal Y}
\newcommand\cZ{\mathcal Z}

\newcommand\sN{\mathscr N}

\newcommand\sT{\mathscr T}
\newcommand\sU{\mathscr U}

\newcommand\sY{\mathscr Y}
\newcommand\sZ{\mathscr Z}

\newcommand\sfB{\mathsf B}
\newcommand\sfC{\mathsf C}

\newcommand\frM{\mathfrak M}

\newcommand\fH{\mathfrak H}
\newcommand\fP{\mathfrak{P}}

\newcommand\ff{\mathfrak f}
\newcommand\fg{\mathfrak g}

\newcommand{\smallertext}[1]{\text{\fontsize{6}{6}\selectfont$#1$}}
\newcommand{\smalltext}[1]{\text{\fontsize{4}{4}\selectfont$#1$}}
\newcommand{\tinytext}[1]{\text{\fontsize{3}{3}\selectfont$#1$}}

\newcommand{\vertiii}[1]{{\left\vert\kern-0.25ex\left\vert\kern-0.25ex\left\vert #1 \right\vert\kern-0.25ex\right\vert\kern-0.25ex\right\vert}}

\def \D{\mathbb{D}}
\def \E{\mathbb{E}}
\def \F{\mathbb{F}}
\def \G{\mathbb{G}}
\def \H{\mathbb{H}}

\def \L{\mathbb{L}}

\def \N{\mathbb{N}}
\def \P{\mathbb{P}}
\def \Q{\mathbb{Q}}
\def \R{\mathbb{R}}
\def \S{\mathbb{S}}

\def \Y{\mathbb{Y}}
\def \Z{\mathbb{Z}}

\newcommand{\bcdot}{\boldsymbol{\cdot}}

\newcommand{\1}{\mathbf{1}}

\def\d{\mathrm{d}}

\DeclareMathOperator*{\esssup}{ess\,sup}

\newcommand\sgn{\text{sgn}}

\begin{document}
	
	\title{Aggregation of value processes for semi-martingale BSDEs with jumps}

	\author{Dylan {\sc Possama\"{i}}\footnote{ETH Z\"{u}rich, Department of Mathematics, Switzerland, dylan.possamai@math.ethz.ch. This author gratefully acknowledges partial support by the SNF project MINT 205121-219818.}\and Marco {\sc Rodrigues}\footnote{TU Berlin and Weierstrass Institute, Germany, marco.rodrigues@tu-berlin.de and marco.rodrigues@wias-berlin.de. This author gratefully acknowledges support from the Swiss National Science Foundation project MINT 205121-219818 and the Deutsche Forschungsgemeinschaft CRC/TRR 388 Project ID 516748464.} \and Alexandros {\sc Saplaouras}\footnote{University of the Aegean, Department of Statistics and Actuarial--Financial Mathematics, Greece, alsapl@aegean.gr.  This author gratefully acknowledges the financial support from the Hellenic foundation for research and innovation grant 235 `Stability and numerics for BSDEs under model uncertainty and applications' (2nd call for H.F.R.I. research projects to support post-doctoral researchers).}}
	
	\date{September 20, 2026}
	
	\maketitle
		
	\begin{abstract}
		We construct a measurable aggregator for the value processes associated with stochastic control problems, where the criterion to optimise is given by solutions to semi-martingale backward stochastic differential equations (BSDEs) with jumps. The results can be applied to control problems where the triplet of semi-martingale characteristics is controlled in a possibly non-dominated case or where uncertainty about the characteristics is present in the optimisation. The construction also provides a time-consistent system of fully nonlinear conditional expectations on the Skorokhod space. We also construct an appropriate path-regularisation of the value function and prove a corresponding dynamic programming principle. The generality we seek allows for the treatment of controlled diffusions, pure-jump processes, and discrete-time processes in a unified setting.
	\end{abstract}


\section{Introduction}
	
Many stochastic control problems in weak formulation, and many optimisation problems under model uncertainty, can be written on the canonical space of c\`adl\`ag paths, with canonical process $X$ and canonical filtration $\F = (\cF_{t})_{t \in [0,\infty)}$, as optimisation problems over a family $\fP_0$ of probability measures, each of which is a possible law of $X$. When the criterion is itself defined through a backward stochastic differential equation (BSDE), as is the case for recursive utilities or when the controls act on the drift of $X$ and on the compensator of its jump measure, the value of the problem at time $t$, seen under a given $\P \in \fP_0$, takes the form
\begin{equation*}
V^\P_t = \underset{\bar{\P} \in \fP_\smalltext{0}(\cF_{\smalltext{t}\tinytext{+}},\P)}{{\esssup}^\P} \cY^{\bar{\P}}_t(T,\xi),\;\textnormal{$\P$--a.s.}, \; \textnormal{$\P \in \fP_0$.}
\end{equation*}
Here $\fP_0(\cF_{t\smallertext{+}},\P) \coloneqq \{\overline{\P} \in \fP_0 : \textnormal{$\overline{\P} = \P$ on $\cF_{t\smallertext{+}}$}\}$, and $\cY^{\bar{\P}}(T,\xi)$ denotes the first component of the solution $(\cY,\cZ,\cU,\cN)$ to the $\overline{\P}$--BSDE
\begin{align}\label{eq_BSDEs_intro}
	\cY_t &= \xi + \int_t^{T} f^{\bar{\P}}_r\big(\cY_r,\cY_{r\smallertext{-}},\cZ_r,\cU_r(\cdot)\big)\d C_r - \int_t^{T} \cZ_r \d X^{c,\bar{\P}}_r  - \int_t^{T}\int_{\R^\smalltext{d}} \cU_r(x)\tilde\mu^{X,\bar{\P}}(\d r, \d x) - \int_t^{T}\d \cN_r, \; t \in [0,\infty],
\end{align}
where $C$ is a fixed non-decreasing process, $X^{c,\bar{\P}}$ denotes the $\overline{\P}$--continuous local martingale part of $X$, $\tilde{\mu}^{X,\bar{\P}}$ denotes the $\overline{\P}$--compensated jump measure of $X$, and $\cN$ is a $\overline{\P}$--martingale orthogonal to $X^{c,\bar{\P}}$ and $\tilde\mu^{X,\bar{\P}}$. We refer to \citeauthor*{nutz2012quasi} \cite{nutz2012quasi}, \citeauthor*{cvitanic2018dynamic} \cite{cvitanic2018dynamic}, and \citeauthor*{soner2013dual} \cite{soner2013dual} for problems of this type.

\medskip
Each $V^\P$ is only defined up to a $\P$--null set. When the measures in $\fP_0$ are mutually singular, which is the typical situation as soon as the volatility of $X$ is uncertain or controlled, there is no reference measure under which the processes $(V^\P)_{\P\in\fP_\smalltext{0}}$ can be compared, let alone glued together. The question we address is whether there is a single process $\widehat{\cY}^\smallertext{+}$, defined on the canonical space without reference to any particular measure, which coincides with $V^\P$ under every $\P\in\fP_0$, and whether this process has the regularity in time and the dynamic programming properties one expects from a value process. Under Lipschitz-type conditions on the generator, we answer both questions positively when $X$ is a semi-martingale with jumps whose characteristics are absolutely continuous with respect to $C$. This setting contains controlled diffusions with jumps, pure-jump processes, and, by letting $C$ be piecewise constant, discrete-time processes.

\medskip
To see where equations of the form \eqref{eq_BSDEs_intro} come from, and why mutually singular measures cannot be avoided, consider a control problem in weak formulation. We regard $X$ as the state process and suppose that its $\P$--semi-martingale characteristics are of the form $(\mathsf{b}_t\d C_t, \mathsf{c}_t\d C_t, \mathsf{K}_t(\d x)\d C_t)$ for a fixed predictable process $C$ with right-continuous and non-decreasing paths. The controller wants to
\begin{equation}\label{eq::control_problem}
	\textnormal{maximise} \; \E^{\P^\smalltext{\alpha}}\bigg[\frac{\xi}{D_T} + \int_0^T \frac{1}{D_{r\smallertext{-}}}g_r(X_{\cdot \land r},\alpha_r)\d C_r\bigg], \; \textnormal{over $\alpha \in \cA$.}
\end{equation}
Here $D \coloneqq \cE(\int_0^\cdot d_r \d C_r)$ is a discount factor, $\xi$ is the terminal reward, $g$ is the running reward, $\cA$ is the set of controls, and $\P^\alpha$ is the law of $X$ induced by the control $\alpha$. These laws are usually obtained from $\P$ through stochastic exponentials, so that $\P^\alpha$ is absolutely continuous with respect to $\P$, and the control only changes the drift of $X$ and the compensator of its jump measure; see \cite[Theorem III.3.24]{jacod2003limit}. Writing the conditional expected reward under $\P^\alpha$ as the solution to a linear BSDE under $\P$, one finds that the expectation in \eqref{eq::control_problem} equals $\E^\P[Y^\alpha_0]$, where $Y^\alpha$ is the first component of the solution $(Y^\alpha,Z^\alpha,U^\alpha,N^\alpha)$ to
\[
	Y^\alpha_t = \xi + \int_t^T f^\alpha_s\big(Y^\alpha_s,Y^\alpha_{s\smallertext{-}},Z^\alpha_s,U^\alpha_s(\cdot)\big) \d C_s - \int_t^T Z^\alpha_s\d X^{c,\P}_s - \int_t^T\int_{\R^\smalltext{d}} U^\alpha_s(x) \tilde{\mu}^{X,\P}(\d s,\d x)- \int_t^T \d N^\alpha_s,\; t\in[0,T],
\]
and where, at least at an intuitive level, the generator is affine in $(y,z,u(\cdot))$ and given by
\begin{equation*}
	f^\alpha_t\big(\omega,y,\mathrm{y},z,u(\cdot)\big) = g_t(X_{\cdot\land t}(\omega),\alpha_t(\omega)) - \frac{d_t(\omega)}{1+d_t(\omega)\Delta C_t(\omega)}y +  z^\top \mathsf{c}_t(\omega) \beta^\alpha_t(\omega) + \int_{\R^\smalltext{d}} u(x) \big( \gamma^\alpha_t(\omega,x) - 1\big) \mathsf{K}_{\omega,t}(\d x).
\end{equation*}
The processes $\beta^\alpha$ and $\gamma^\alpha$ describe how the control changes the drift of $X$ and the compensator of its jump measure. If the generator $f \coloneqq \sup_{\alpha \in \cA} f^\alpha$ is well defined, and if some control $\alpha^\star$ attains this supremum along the solution $(Y,Z,U,N)$ of the BSDE with generator $f$, that is,
\[
	f_s\big(Y_s,Y_{s\smallertext{-}},Z_s,U_s(\cdot)\big) = f^{\alpha^\smalltext{\star}}_s\big(Y_s,Y_{s\smallertext{-}},Z_s,U_s(\cdot)\big), \; \textnormal{$\P\otimes\d C_s$--a.e.},
\]
then the comparison principle for BSDEs shows that $\alpha^\star$ is optimal for \eqref{eq::control_problem}. Two features of the general problem are already visible in this simple case. First, $C$ carries the time scale of the problem, and the characteristics of $X$ must be absolutely continuous with respect to $C$ for BSDE methods to apply; allowing $C$ to jump is what brings discrete-time problems into the same framework. Second, as soon as the control acts on the second characteristic of $X$, or changes the compensator of its jump measure in a way that is not absolutely continuous, the laws $\P^\alpha$ become mutually singular. The argument above, which takes place under a single measure, then breaks down, the natural value process is of the form $V^\P$, and one is led back to the aggregation question.

\medskip
When the generator vanishes, $\cY^{\bar{\P}}_t(T,\xi) = \E^{\bar{\P}}[\xi|\cF_{t\smallertext{+}}]$, and the question becomes that of constructing a conditional sublinear expectation. This question has a long history. \citeauthor*{denis2006theoretical} \cite{denis2006theoretical} formulated the pricing of contingent claims under volatility uncertainty in terms of capacities, while \citeauthor*{peng2007g} \cite{peng2007g,peng2008multi,peng2019nonlinear} built the $G$-expectation from the solutions of a fully nonlinear heat equation; \citeauthor*{denis2011function} \cite{denis2011function} later showed that the $G$-expectation is the supremum of expectations over a family of mutually singular laws. \citeauthor*{soner2011quasi} \cite{soner2011quasi} proved that families of processes indexed by such laws can be aggregated under a separability condition on the family, and used this to obtain a martingale representation under the $G$-expectation in \cite{soner2011martingale}. Separability was then replaced by measurable selection arguments. \citeauthor*{nutz2012superhedging} \cite{nutz2012superhedging} obtained a c\`adl\`ag version of the corresponding nonlinear martingale by path regularisation and related it to superhedging, \citeauthor*{nutz2013random} \cite{nutz2013random} allowed the volatility bounds to depend on the path, \citeauthor*{nutz2013constructing} \cite{nutz2013constructing} gave conditions on the family of measures, formulated with analytic sets, under which conditional sublinear expectations of Borel-measurable random variables are time-consistent, and \citeauthor*{elkaroui2013capacities} \cite{elkaroui2013capacities,elkaroui2013capacities2} developed a general framework for measurable selection and dynamic programming. Related constructions in general spaces and for conditional nonlinear expectations were given by \citeauthor*{cohen2012quasi} \cite{cohen2012quasi} and \citeauthor*{bartl2020conditional} \cite{bartl2020conditional}, and the superhedging of measurable claims under volatility uncertainty was studied by \citeauthor*{neufeld2013superreplication} \cite{neufeld2013superreplication} and \citeauthor*{possamai2013robust} \cite{possamai2013robust}.

\medskip
Jumps raise new difficulties already at this level. The approach through partial differential equations extends to the $G$--L\'evy processes of \citeauthor*{hu2021g} \cite{hu2021g}, but only for jumps of finite variation, and the corresponding martingale representation was investigated by \citeauthor*{paczka2014gmartingale} \cite{paczka2014gmartingale,paczka2014ito}. A probabilistic approach was proposed by \citeauthor*{neufeld2014measurability} \cite{neufeld2014measurability}, who showed that the semi-martingale characteristics of the canonical process can be chosen measurably with respect to its law, and used this in \cite{neufeld2016nonlinear} to construct nonlinear L\'evy processes. We also refer to \citeauthor*{nutz2015robust} \cite{nutz2015robust} for robust superhedging with jumps, and to \citeauthor*{criens2023nonlinear} \cite{criens2023nonlinear,criens2023robust,criens2025nonlinear} for nonlinear continuous semi-martingales, robust utility maximisation in that setting, and nonlinear semi-martingales and Markov processes with jumps. All these works concern sublinear expectations, that is, the case of a vanishing generator.

\medskip
With a non-zero generator, conditional expectations are replaced by solutions to BSDEs. Under a single measure, this leads to nonlinear conditional expectations whose theory is close to that of classical martingales. We refer to \citeauthor*{elkaroui1997backward} \cite{elkaroui1997backward} for their early use in finance. \citeauthor*{peng1999monotonic} \cite{peng1999monotonic} proved a nonlinear Doob--Meyer decomposition for the associated super-martingales, and \citeauthor*{chen2000general} \cite{chen2000general,chen2001continuous} obtained down-crossing inequalities and c\`adl\`ag modifications. These results were extended to BSDEs with jumps by \citeauthor*{royer2006backward} \cite{royer2006backward} and \citeauthor*{lin2003nonlinear} \cite{lin2003nonlinear}, to uniformly continuous generators by \citeauthor*{shi2014nonlinear} \cite{shi2014nonlinear}, to systems of nonlinear super-martingales by \citeauthor*{bouchard2016general} \cite{bouchard2016general}, and to optimal stopping with irregular rewards by \citeauthor*{grigorova2020optimal} \cite{grigorova2020optimal}. Criteria given by BSDEs also appear in the stochastic differential utility of \citeauthor*{duffie1992stochastic} \cite{duffie1992stochastic}, and their optimisation over controls leads to the generalised dynamic programming principle of \citeauthor*{peng1992generalized} \cite{peng1992generalized}; see also \citeauthor*{buckdahn2010probabilistic} \cite{buckdahn2010probabilistic} for coupled systems of Hamilton--Jacobi--Bellman equations, and \citeauthor*{hu2017dynamic} \cite{hu2017dynamic} for the same questions under the $G$-expectation.

\medskip
For non-dominated families of measures and a non-zero generator, the first general results came from the theory of second-order BSDEs. Following \citeauthor*{cheridito2007second} \cite{cheridito2007second}, \citeauthor*{soner2012wellposedness} \cite{soner2012wellposedness,soner2013dual} established the well-posedness of second-order BSDEs, whose first component is a value process of this type, under strong regularity conditions on the terminal condition and the generator, and \citeauthor*{possamai2018stochastic} \cite{possamai2018stochastic} removed these conditions by measurable selection, for continuous semi-martingales with absolutely continuous characteristics. This last work is the closest to ours. For processes with jumps, \citeauthor*{kazi2015second} \cite{kazi2015second,kazi2015second2} followed the approach of \cite{soner2012wellposedness} in a Brownian--Poisson setting, and \citeauthor*{denis2024second} \cite{denis2024second} adapted the arguments of \cite{possamai2018stochastic} to it.

\medskip
The BSDEs \eqref{eq_BSDEs_intro} are driven by the continuous martingale part of $X$ and by its compensated jump measure, with an integrator $C$ that may have jumps. BSDEs in general spaces were studied by \citeauthor*{cohen2012existence} \cite{cohen2012existence}, BSDEs driven by marked point processes, with applications to the control of their compensators, by \citeauthor*{confortola2013backward} \cite{confortola2013backward} and \citeauthor*{confortola2014backward} \cite{confortola2014backward}, and BSDEs driven by general c\`adl\`ag martingales and integer-valued random measures, with an arbitrary non-decreasing integrator, by \citeauthor*{papapantoleon2016existence} \cite{papapantoleon2016existence,papapantoleon2021stability} and \citeauthor*{possamai2024reflections} \cite{possamai2024reflections}. We rely on the well-posedness results of \cite{possamai2024reflections} throughout. When $C$ jumps, the generator may depend on both $\cY_r$ and $\cY_{r\smallertext{-}}$, and well-posedness requires the jumps of $C$, weighted by the Lipschitz constants of the generator, to be small enough.

\medskip
We now describe our results. The starting point is the pathwise value function
\[
	\widehat\cY_s(T,\xi)(\omega) \coloneqq \sup_{\P \in \fP(s,\omega)} \E^\P\big[\cY^{s,\omega,\P}_0((T-s\land T)^{s,\omega},\xi^{s,\omega})\big],\; (\omega,s)\in\Omega\times[0,\infty),
\]
where $\fP(s,\omega)$ is the set of laws that remain admissible once the path $\omega$ has been observed up to time $s$, and $\cY^{s,\omega,\P}$ solves the BSDE shifted by $(s,\omega)$. In \Cref{thm::measurability2}, we show that $\widehat\cY_s(T,\xi)$ is upper semi-analytic, and therefore universally measurable, and that under every $\P\in\fP_0$ it coincides with the essential supremum of the conditional expectations of $\cY^{\bar\P}_s(T,\xi)$ over the measures $\overline\P\in\fP_0$ that agree with $\P$ on $\cF_s$. In particular, $\widehat\cY$ defines a time-consistent family of fully nonlinear conditional expectations on the Skorokhod space. The difficulty lies in the measurability of $(\omega,\P)\longmapsto\E^\P[\cY^{s,\omega,\P}_0((T-s\land T)^{s,\omega},\xi^{s,\omega})]$. Solutions to BSDEs are obtained by Picard iteration under each measure separately, and the dependence on $\P$ has to be followed through every step. The results of \cite{neufeld2014measurability} provide characteristics that are measurable in $\P$, but the iteration also involves the integrands $\cZ$ and $\cU$ of martingale representations, which must be chosen measurably in $\P$ as well. For $\cU$, even the space in which the integrand lives depends on $\P$, since it is defined through the compensator of the jump measure of $X$ under $\P$. We prove that such measurable choices exist. Since the measures in $\fP_0$ are not dominated, we cannot work with completed filtrations, and the results on stochastic calculus without the usual conditions that we need are collected in \Cref{app::raw_filtration_stochastic_calculus}.

\medskip
Our second main result, \Cref{thm::down-crossing}, concerns regularity in time. Even for a vanishing generator, the pathwise value function need not admit a c\`adl\`ag modification; see \cite[Example 4.6]{nutz2012superhedging}. Under an integrability condition on the data, and under conditions on the generator that are close to those ensuring a comparison principle, we show that, outside a $\fP_0$--polar set, its right limits along the dyadic numbers exist and define a c\`adl\`ag process $\widehat\cY^\smallertext{+}(T,\xi)$, adapted to the right-continuous version of the universally completed canonical filtration augmented by the $\fP_0$--polar sets. Under every $\P\in\fP_0$, this process coincides with the essential supremum of $\cY^{\bar\P}_t(T,\xi)$ over the measures $\overline\P\in\fP_0$ that agree with $\P$ up to time $t$, and it satisfies a square-integrability estimate that is uniform over $\fP_0$. If these conditions on the generator hold after every shift by $(s,\omega)$, the process $\widehat\cY^\smallertext{+}(T,\xi)$ is moreover a nonlinear super-martingale under every $\P\in\fP_0$ with respect to the BSDE with generator $f^\P$, between any two stopping times; this is the dynamic programming principle in our setting. The existence of right limits rests on crossing inequalities for nonlinear super-martingales, which do not follow from the existing literature in our generality. A result of this kind was used in \cite[Lemma 3.2]{possamai2018stochastic} for continuous processes, through the argument in the proof of \cite[Lemma A.1]{bouchard2016general}. In that argument, Doob's down-crossing inequality is applied to a process that need not be a super-martingale. Our proof does not have this problem, and the statements of \cite{possamai2018stochastic} remain valid; see \Cref{rem::gap_regularisation}. \Cref{cor::optimisation} then relates the two value functions: $\widehat\cY_0(T,\xi)$ is the supremum over $\P\in\fP_0$ of $\E^\P[\widehat\cY^\smallertext{+}_0(T,\xi)]$, so that optimal laws for the original problem can be found by working with the c\`adl\`ag process $\widehat\cY^\smallertext{+}(T,\xi)$ first, and then optimising conditionally on the information at time $0$.

\medskip
Along the way, we prove in \Cref{sec_stability} stability and comparison results for BSDEs in the generality of \cite{possamai2024reflections}. The stability estimates do not rely on It\^o's formula, and the comparison principle holds under conditions that are weaker than those of \cite[Assumption 7.1]{possamai2024reflections}. We believe these results are of interest beyond the present paper.

\medskip
The process $\widehat\cY^\smallertext{+}(T,\xi)$ is the first component of the solution to a second-order BSDE with jumps. Its semi-martingale decomposition under every $\P\in\fP_0$, the construction of a single integrand for the continuous martingale part of $X$, and the well-posedness of the corresponding second-order BSDE are the subject of our companion paper \cite{possamai2025mind}. There, we also explain why, in contrast to the continuous case, the integrands of the compensated jump measures cannot in general be chosen independently of the measure. The construction in the present paper does not rely on second-order BSDEs, and can be used on its own, for instance to define nonlinear conditional expectations generated by BSDEs with jumps, or to study robust control problems in which the characteristics of the state process are controlled.

\medskip
The paper is organised as follows. \Cref{sec::preliminaries} introduces the canonical space, semi-martingale laws and their conditioning, and the data of our BSDEs. \Cref{sec::main_results} contains our main results, whose proofs are given in \Cref{sec::proofs_main_results}. \Cref{app::raw_filtration_stochastic_calculus} collects the results on stochastic calculus without the usual conditions used throughout. \Cref{sec::proofs_preliminaries} contains the proofs of the results of \Cref{sec::preliminaries}, \Cref{sec::lemmas_main_results} those of the auxiliary lemmata of \Cref{sec::proofs_main_results}, and \Cref{sec_stability} the stability and comparison results for BSDEs.

\medskip
{\small\textbf{\small Notation and terminology:} we fix a positive integer $d$. Let $\N$, $\Q$ and $\R$ denote the nonnegative integers, rational numbers, and real numbers, respectively. We write $\Q_\smallertext{+} \coloneqq \Q \cap [0,\infty)$. For $n \in \N$, we write $\D^n_\smallertext{+} = \{k2^{-n}: k\in \N\}$, and then denote by $\D_\smallertext{+} = \cup_{n\in \N} \D^n_\smallertext{+}$ the collection of nonnegative dyadic numbers, and by $\D$ the collection of all real dyadic numbers. We write $\S^d_\smallertext{+}$ for the set of positive semi-definite (that is, nonnegative), symmetric and real $d \times d$ matrices. Points in $\R^d$ are understood to be column vectors.
		
\smallskip
For a matrix $A$, we denote its transpose by $A^\top$, its Moore–Penrose (pseudo-)inverse by $A^\oplus$, and, in the case that $A$ is a square matrix, we denote its trace by $\textnormal{Tr}[A]$. For a set $\Omega$ and $A \subseteq \Omega$, we denote by $\mathbf{1}_A$ its indicator function defined on $\Omega$. We will abuse notation and also denote by $\1_{\{P\}}$ the Iverson bracket of a mathematical statement $P$, that is, $\1_{\{P\}}$ equals $1$ if and only if $P$ is true; otherwise, it takes the value $0$. For a measurable space $(\Omega,\cF)$, we denote the Dirac measure at $x \in \Omega$ by $\boldsymbol{\delta}_x$. For $(a, b) \in [-\infty,\infty]^2$, we write $a \lor b \coloneqq \max\{a,b\}$, $a \land b \coloneqq \min\{a,b\}$, $a^+ \coloneqq a \lor 0$ and $a^- \coloneqq (-a)^+$. For two measurable spaces $(\Omega,\cF)$ and $(\Omega^\prime,\cF^\prime)$, we denote by $\cF \otimes \cF^\prime$ the product $\sigma$-algebra of $\cF$ and $\cF^\prime$ on the product space $\Omega \times \Omega^\prime$. 
		
\smallskip
For a c\`adl\`ag function $X$ defined on some interval $I\subseteq [0,\infty]$ with $0 \in I$, we let $\Delta X_t \coloneqq X_t - X_{t\smallertext{-}}$ if $t \in I \setminus\{0\}$ and $\Delta X_0 \coloneqq 0$. For $t \in (0,\infty]$, a limit of the form $s \uparrow\uparrow t$ (resp. $s\uparrow t$) means that $s \longrightarrow t$ along $s < t$ (resp. along $s\leq t$). We define $s \downarrow\downarrow t$ (resp. $s\downarrow t$) analogously. We use throughout the conventions $\sup\varnothing = - \infty$ and $0 / 0 \coloneqq 0$. 
		
\smallskip
For a probability measure $\P$ on a measurable space $(\Omega,\cF)$, and a subset $A \subseteq \Omega$, we refer to $A$ as an $(\cF,\P)$--null set, if there exists $B \in \cF$ such that $A \subseteq B$ and $\P[B] = 0$. We always use the convention $\infty - \infty \coloneqq -\infty$. In particular, if $\xi : \Omega \longrightarrow [-\infty,\infty]$ is $\cF$-measurable, then $\E^\P[\xi] \coloneqq \E^\P[\xi\lor 0] - \E^\P[(-\xi) \lor 0] = -\infty$ in case $\E^\P[(-\xi) \lor 0] = \infty$. For two measures $\mu$ and $\nu$ defined on the same underlying measurable space, we write $\mu \ll \nu$ if $\mu$ is absolutely continuous with respect to $\nu$. For a family of measures $\mathfrak{M}$ on the same underlying measurable space $(\Omega,\cF)$, a property $P$ holds $\mathfrak{M}$--quasi-surely (abbreviated as $\mathfrak{M}$--q.s.) if the subset of $\Omega$ on which $P$ does not hold is an  $(\cF,\P)$--null set for each measure in $\mathfrak{M}$. Then $\xi : \Omega \longrightarrow [-\infty,\infty]$ is $\mathfrak{M}$--essentially bounded if there exists $\mathfrak{C} \in (0,\infty)$ such that $|\xi| \leq \mathfrak{C}$ holds $\mathfrak{M}$--quasi-surely.
		
\smallskip
For two measurable spaces $(\Omega,\cF)$ and $(\Omega^\prime,\cF^\prime)$, a kernel $K$ on $(\Omega^\prime,\cF^\prime)$ given $(\Omega,\cF)$ is a map $K : \Omega \times \cF^\prime \longrightarrow [0,\infty]$ such that $\omega \longmapsto K_\omega(A)$ is $\cF$-measurable for every $A \in \cF^\prime$ and $K_\omega(\cdot)$ is a measure on $(\Omega^\prime,\cF^\prime)$ for each $\omega \in \Omega$; if $K_\omega(\Omega^\prime) = 1$ for each $\omega \in \Omega$, then $K$ is referred to as a stochastic kernel.
		
\smallskip
Lastly, an integral over an interval $I \subseteq [0,\infty]$ never includes the points $0$ or $\infty$ in the domain of integration; that is, $\int_I$ is the same as $\int_{I\setminus\{0,\infty\}}$. Moreover, $\int_a^b$ is always understood as $\int_{(a,b]}$, and $\int_{a\smallertext{-}}^b$ as $\int_{[a,b]}$.}

\section{Preliminaries}\label{sec::preliminaries}

Due to the presence of possibly non-dominated families of probability measures,
we work without imposing the usual conditions of stochastic calculus. The
stochastic-calculus facts needed in this raw-filtration setting are collected in
\Cref{app::raw_filtration_stochastic_calculus}. This section is devoted to the
canonical space of c\`adl\`ag paths, where we fix notation and discuss
conditioning, concatenation, and shifted semimartingale laws. A key property in
this work is the measurability, on the canonical space, of semimartingale
characteristics with respect to the underlying probability measure, established
by \citeauthor*{neufeld2014measurability}
\cite{neufeld2014measurability}. We adopt their conventions throughout. Proofs of the results in this preliminary section are collected in \Cref{sec::proofs_preliminaries}.

\subsection{Setup and semi-martingale laws on the canonical space}\label{sec::semimartingale_laws}

We denote by $\Omega$ the collection of all c\`adl\`ag paths $\omega : [0,\infty) \longrightarrow \R^d$ with $\omega_0 = 0$, by $\F = (\cF_t)_{t \in [0,\infty)}$ the canonical (raw) filtration generated by the canonical process $X = (X_t)_{t \in [0,\infty)}$, and we set $\cF \coloneqq \cF_{\infty\smallertext{+}} \coloneqq \cF_\infty \coloneqq \cF_{\infty\smallertext{-}} \coloneqq \sigma(\cup_{t\in[0,\infty)}\cF_t)$. We endow $\Omega$ with the Skorokhod topology, under which $\Omega$ is Polish and $\cB(\Omega) = \cF$ (see \cite[Theorem VI.1.14]{jacod2003limit}). We denote by $\F_\smallertext{+}$ and $\F_\smallertext{-}$ the right-continuous and left-continuous modifications of $\F$, respectively, and use the conventions $\cF_{0\smallertext{-}}=\cF_{(0\smallertext{+})\smallertext{-}} = \{\varnothing,\Omega\}$. The predictable $\sigma$-algebra $\cP(\F)$ on $\Omega\times[0,\infty)$, generated by all real-valued, $\F_\smallertext{-}$-adapted processes that are left-continuous on $(0,\infty)$, coincides with the predictable $\sigma$-algebra of $\F_\smallertext{+}$. For a probability measure $\P$ on $(\Omega,\cF)$, let $\sN^\P$ be the collection of all subsets of $(\cF,\P)$--null sets. We denote by $\F^\P = (\cF^\P_t)_{t\in[0,\infty)}$ and $\F^\P_\smallertext{+} = (\cF^\P_{t\smallertext{+}})_{t\in[0,\infty)}$ the $\P$-augmentations of $\F$ and $\F_\smallertext{+}$, respectively, that is, $\cF^\P_{t} \coloneqq \sigma(\cF_t,\sN^\P)$ and $\cF^\P_{t\smallertext{+}} \coloneqq \sigma(\cF_{t\smallertext{+}},\sN^\P)$ with the convention $\cF^{\P}_{0\smallertext{-}} = \cF^{\P}_{(0\smallertext{+})\smallertext{-}} = \sigma(\sN^\P)$. 
Then $\cP(\F^\P) = \cP(\F^\P_{\smallertext{+}})$, and we write $\cP \coloneqq \cP(\F)$, $\cP^\P \coloneqq \cP(\F^\P_\smallertext{+})$, $\widetilde{\cP} \coloneqq \cP(\F)\otimes\cB(\R^d)$ and $\widetilde{\cP}^\P \coloneqq \cP(\F^\P_\smallertext{+})\otimes\cB(\R^d)$; the latter two are the corresponding predictable $\sigma$-algebras on $\widetilde{\Omega}\coloneqq \Omega \times [0,\infty)\times \R^d$. 
Instead of stating that a process is $\cP$-measurable (resp. $\cP^\P$-measurable), we also use the terminology $\F$-predictable (resp. $\F^\P_\smallertext{+}$-predictable). If a property holds outside an $(\cF,\P)$--null set, we will say that it holds outside a $\P$--null set, or that it holds $\P$--a.s., or for $\P$--a.e. $\omega \in \Omega$.

\medskip
We denote by $\fP(\Omega)$ the collection of all probability measures on $(\Omega,\cF)$ and endow it with the topology of convergence in distribution, that is, the coarsest topology for which $\P \longmapsto \E^\P[f]$ is continuous for every bounded and continuous function $f : \Omega \longrightarrow \R$. Since $\Omega$ is Polish, so is $\fP(\Omega)$ (see \citeauthor*{aliprantis2006infinite} \cite[Theorem 15.15]{aliprantis2006infinite} or \citeauthor*{bertsekas1978stochastic} \cite[Proposition 7.23]{bertsekas1978stochastic}), and the Borel $\sigma$-algebra $\cB(\fP(\Omega))$ is generated by the mappings $\fP(\Omega) \ni \P \longmapsto \P[A] \in [0,1]$, $A \in \cB(\Omega)$ (see \cite[Proposition 7.25]{bertsekas1978stochastic}). Moreover, the subset
	\begin{equation}\label{eq_semimartingale_laws}
		\fP_\textnormal{sem} \coloneqq \{\P\in\fP(\Omega) : \textnormal{$X$ is an $(\F,\P)$-semi-martingale}\} \subseteq \fP(\Omega),
	\end{equation}
	is Borel-measurable (see \cite[Theorem 2.5]{neufeld2014measurability}). In this work, we use the following convention. For an arbitrary filtration $\G$ on $\Omega$ and $\P\in\fP(\Omega)$, an $\R^d$-valued, $\G$-adapted process $S=(S_t)_{t\in[0,\infty)}$ with c\`adl\`ag paths is referred to as a $(\G,\P)$--semi-martingale if there exist $\R^d$-valued, right-continuous, $\G$-adapted processes $M=(M_t)_{t\in[0,\infty)}$ and $A=(A_t)_{t\in[0,\infty)}$ such that $M_0=A_0=0$, $M$ is a $(\G,\P)$--local martingale, the paths of $A$ are $\P$--a.s. of locally finite variation, and
\[
    S = S_0 + M + A, \; \textnormal{$\P$--a.s.},
\]
holds. The precise choice of filtration in \eqref{eq_semimartingale_laws} is not important, namely, $S$ is an $(\F,\P)$--semi-martingale if and only if it is an $(\F_\smallertext{+},\P)$--semi-martingale, which is also equivalent to $S$ being an $(\F^\P_\smallertext{+},\P)$--semi-martingale; see \cite[Proposition~2.2]{neufeld2014measurability}.
	
	\medskip
	We denote by $\omega\otimes_\tau\tilde\omega \in \Omega$ the concatenation of two paths $(\omega,\tilde\omega) \in \Omega \times \Omega$ at an $\F$--stopping time $\tau$ defined by
	\begin{equation*}
		(\omega \otimes_\tau \tilde\omega)_s \coloneqq \omega_s \1_{\{0 \leq s < \tau(\omega)\}} + (\omega_{\tau(\omega)} + \tilde\omega_{s-\tau(\omega)})\1_{\{\tau(\omega) \leq s\}}, \; s \in [0,\infty),
	\end{equation*}
	and for any map $\xi$ on $\Omega$, we write $\xi^{\tau,\omega}$ for the function on $\Omega$ defined by
	\[
		\xi^{\tau,\omega}(\tilde{\omega}) \coloneqq \xi(\omega\otimes_\tau\tilde{\omega}), \; \omega^\prime \in \Omega.
	\]
	For $\P \in \fP(\Omega)$ and an $\F$--stopping time $\tau$, there exists a $\P$--a.s. unique stochastic kernel $(\P^\tau_\omega)_{\omega \in \Omega}$ on $(\Omega,\cF)$ given $(\Omega,\cF_\tau)$ such that, for all bounded and Borel-measurable functions $\xi$, the random variable $\E^{\P^\smalltext{\tau}_\smalltext{\cdot}}[\xi]$ is a $\P$-version of $\E^\P[\xi|\cF_\tau]$; we refer to \cite[Theorem 8.5]{kallenberg2021foundations} for its existence.
	Exactly as in \citeauthor*{stroock1997multidimensional} \cite[page 34]{stroock1997multidimensional}, we choose $(\P^\tau_\omega)_{\omega \in \Omega}$ in such a way that additionally
	\begin{equation*}
		\P^\tau_\omega\big[\tau = \tau(\omega), \; X_{\cdot \land \tau(\omega)} = \omega_{\cdot\land\tau(\omega)}\big] = 1, \; \omega \in \Omega,
	\end{equation*}
	holds. We then refer to $(\P^\tau_\omega)_{\omega \in \Omega}$ as the regular conditional probability distribution (r.c.p.d.) of $\P$ given $\cF_\tau$. For background and further details, see also \citeauthor*{elkaroui2013capacities} \cite[Section 4.1]{elkaroui2013capacities} and \citeauthor*{dellacherie1978probabilities} \cite[pages 145--152]{dellacherie1978probabilities}. For a finite $\F$--stopping time $\tau$ and every $\omega \in \Omega$, we then define $\P^{\tau,\omega} \in \fP(\Omega)$ by
	\[
		\P^{\tau,\omega}[A] \coloneqq \P^{\tau}_{\omega}[\omega\otimes_\tau A], \; A \in \cF,
	\]
	where $\omega\otimes_\tau A \coloneqq \{\omega\otimes_\tau\tilde\omega : \tilde\omega \in A\}$. Then $\E^{\P^{\tau,\omega}}[\xi^{\tau,\omega}] = \E^{\P^\smalltext{\tau}_\smalltext{\omega}}[\xi]$, $\omega \in \Omega$, which means that 
	\begin{equation}\label{eq_conditional_expectation_shifted_measure}
		\textnormal{$\E^{\P^{\tau,\cdot}}[\xi^{\tau,\cdot}]$ is a $\P$-version of $\E^\P[\xi|\cF_\tau]$.}
	\end{equation}
	Lastly, for two finite $\F$--stopping times $\tau$ and $\sigma$, a process $Z$ with index set $[0,\infty)$ or $[0,\infty]$, and $\omega \in \Omega$, we write $Z^{\tau,\omega}_{\sigma\smallertext{+}\smallertext{\cdot}} \coloneqq (Z_{\sigma\smallertext{+}\smallertext{\cdot}})^{\tau,\omega}$ for the process given by 
	\[
	(Z^{\tau,\omega}_{\sigma\smallertext{+}\smallertext{\cdot}})_r(\tilde\omega) = Z_{\sigma(\omega\otimes_\smalltext{\tau}\tilde\omega)\smallertext{+}r}(\omega\otimes_\tau\tilde\omega), \; (\tilde\omega,r) \in \Omega \times [0,\infty).
	\]
For $\sigma = \tau$, we have $\tau(\omega\otimes_\tau\tilde\omega) = \tau(\omega)$ by Galmarino's test (see \cite[Theorem IV.100]{dellacherie1978probabilities}), and thus $(Z^{\tau,\omega}_{\tau\smallertext{+}\smallertext{\cdot}})_r(\tilde\omega) = Z_{\tau(\omega)\smallertext{+}r}(\omega\otimes_\tau\tilde\omega)$.

	\begin{remark}\label{rem::conditional_law_future_increments}
		For a finite $\F$--stopping time $\tau$, it can be readily verified from \eqref{eq_conditional_expectation_shifted_measure} that
		\[
		\P^{\tau,\omega}[A] = \P\big[X_{\tau\smallertext{+}\smallertext{\cdot}}-X_\tau \in A\big|\cF_\tau\big](\omega), \; \textnormal{$\P$--a.e. $\omega \in \Omega$, $A \in \cF$,}
		\]
		that is, $(\mathbb{P}^{\tau, \omega})_{\omega \in \Omega}$ is the conditional distribution on $\Omega$ of the shifted process $X_{\tau\smallertext{+}\smallertext{\cdot}} - X_\tau = (X_{\tau\smallertext{+}t} - X_\tau)_{t \in [0, \infty)}$ given $\mathcal{F}_\tau$. 
	\end{remark}
	
	The next result concerns the conditioning and shifting of martingales. The proof of \cite[Lemma 3.3]{neufeld2016nonlinear} considers the raw filtration $\F$. Since handling $\F_\smallertext{+}$ requires additional care, we provide a proof in \Cref{sec::proofs_preliminaries} for completeness.
	\begin{lemma}\label{lem::conditioning_martingale2}
		Let $\P\in\fP(\Omega)$, and let $\tau$ be a finite $\F$--stopping time. Suppose that $M = (M_t)_{t \in [0,\infty)}$ is a real-valued, right-continuous, $\F_\smallertext{+}$-adapted, $(\F_\smallertext{+},\P)$--square-integrable martingale. Then $M^{\tau,\omega}_{\tau\smallertext{+}\smallertext{\cdot}}$ is a real-valued, right-continuous, $\F_\smallertext{+}$-adapted, $(\F_\smallertext{+},\P^{\tau,\omega})$--square-integrable martingale for \textnormal{$\P$--a.e.} $\omega \in \Omega$. The same assertion holds when $\F_\smallertext{+}$ is replaced by $\F$.
	\end{lemma}

	Fix a bounded and Borel-measurable function $h: \R^d \longrightarrow \R^d$ satisfying $h(x) = x$ in an open neighbourhood of the origin; such a function is called a truncation function in the theory of semi-martingales. For $\P \in \fP_\textnormal{sem}$, we denote the $(\F,\P)$--semi-martingale characteristics of $X$ relative to $h$ by $(\mathsf{B}^{\P},\mathsf{C}^{\P},\nu^{\P})$ consisting of $\R^d$- and $\S^d_\smallertext{+}$-valued processes (maps on $\Omega\times[0,\infty)$) $\mathsf{B}^{\P}$ and $\mathsf{C}^{\P}$, respectively, and $\nu^\P = \{\nu^\P(\omega;\d t, \d x): \omega \in \Omega\}$ is a collection of (nonnegative) measures on $[0,\infty)\times\R^d$ with $\nu^\P(\{0\}\times\R^d)\equiv 0$ satisfying the following: outside some common $\P$--null set, $\mathsf{B}^{\P} = (\mathsf{B}^\P_t)_{t \in [0,\infty)}$ coincides with the $\F$-predictable finite variation part in the canonical decomposition of the special $(\F,\P)$--semi-martingale $X - \sum_{s \in (0,\cdot]}(\Delta X_s - h(\Delta X_s))$ (see \cite[Corollary 7.2.8]{weizsaecker1990stochastic}), $\mathsf{C}^\P = (\mathsf{C}^\P_t)_{t \in [0,\infty)}$ coincides with the $(\F,\P)$--quadratic covariation of the $(\F,\P)$--continuous local martingale part of $X$ (see \Cref{lem::equivalence_loc_square_integrable_martingale}, \Cref{lem::martingale_decomposition}, and \cite[Example 4.2.2.(b) and Corollary 6.6.3]{weizsaecker1990stochastic}), and $\nu^\P$ is equal to the $(\F,\P)$--predictable compensator of $\mu^X$, the jump measure of $X$ on $[0,\infty) \times \R^d$ (see \Cref{lem::existence_predictable_compensator_mu}). We note that the characteristics relative to $\F$, $\F_\smallertext{+}$, and $\F^\P_\smallertext{+}$ coincide outside some $\P$--null set; see \cite[Proposition~2.2]{neufeld2014measurability}.

	\medskip
	Unless stated otherwise, we take the following assumption as given throughout this work, without explicitly mentioning it in the statements of our results.
	\begin{standing_assumption}
		$C = (C_t)_{t \in [0,\infty)}$ is a real-valued, right-continuous and non-decreasing, $\F$-predictable process starting at zero.
	\end{standing_assumption}

	For a finite $\F$--stopping time $\tau$, we set\footnote{Recall that $\mu \ll \nu$ means that the measure $\mu$ is absolutely continuous with respect to the measure $\nu$.}
	\begin{equation}\label{eq::semi_martingale_laws}
		\widehat{\Omega}^C_\tau \coloneqq \big\{(\omega,\P) \in \Omega \times \fP_\textnormal{sem} : (\mathsf{B}^{\P},\mathsf{C}^{\P},\nu^{\P}) \ll (C^{\tau,\omega}_{\tau\smallertext{+}\smallertext{\cdot}}-C_{\tau(\omega)}(\omega)), \; \text{$\P$--a.s.} \big\}.
	\end{equation}
	For the third characteristic, this means that $(|x|^2 \land 1)\ast\nu^\P \ll (C^{\tau,\omega}_{\tau\smallertext{+}\smallertext{\cdot}}-C_{\tau(\omega)}(\omega))$ up to a $\P$--null set. By \Cref{lem::absolute_continuity_nu}, this is equivalent to $\nu^{\P}(\d t, \d x)= \mathsf{K}^{\tau,\omega,\P}_{\cdot,t}(\d x)\d (C^{\tau,\omega}_{\tau\smallertext{+}\smallertext{\cdot}}-C_{\tau(\omega)}(\omega))_t$, $\P$--a.s., where $\mathsf{K}^{\tau,\omega,\P}$ is a kernel on $(\R^d,\cB(\R^d))$ given $(\Omega\times[0,\infty),\cP)$. For every $(\omega,\P)\in\widehat{\Omega}^C_\tau$, the corresponding derivatives $(\mathsf{b}^{\tau,\omega,\P},\mathsf{a}^{\tau,\omega,\P},\mathsf{K}^{\tau,\omega,\P})$ with respect to $C^{\tau,\omega}_{\tau\smallertext{+}\smallertext{\cdot}}-C_{\tau(\omega)}(\omega)$ are $\P \otimes \d (C^{\tau,\omega}_{\tau\smallertext{+}\smallertext{\cdot}}-C_{\tau(\omega)}(\omega))$--a.e. uniquely defined and referred to as the $(\F,\P)$--differential characteristics of $X$ relative to $C^{\tau,\omega}_{\tau\smallertext{+}\smallertext{\cdot}}-C_{\tau(\omega)}(\omega)$. 
	
	\medskip
	Our next result is reminiscent of \cite[Theorem 2.6]{neufeld2014measurability}, but adapted to the differential characteristics considered here. Let $\frM(\R^d)$ denote the collection of all (nonnegative) measures on $(\R^d,\cB(\R^d))$. The subset consisting of all L\'evy measures is
	\begin{equation*}
		\cL\coloneqq\bigg\{\kappa \in \frM(\R^d) : \int_{\R^\smalltext{d}}(|x|^2 \land 1) \kappa(\d x) < \infty \; \text{and} \; \kappa(\{0\}) = 0\bigg\}.
	\end{equation*}
	We endow $\cL$ with a metric $d_{\cL}$ and the corresponding Borel $\sigma$-algebra $\cB(\cL)$, such that, for any measurable space $(Y,\cY)$, a map $\kappa : Y \longrightarrow \cL$ is $(\cY,\cB(\cL))$-measurable if and only if for all bounded and Borel-measurable functions $f : \R^d \longrightarrow \R$, the map
	\begin{equation*}
		Y \ni y \longmapsto \int_{\R^\smalltext{d}} (|x|^2 \land 1)f(x)\kappa(y;\d x) \in \R,
	\end{equation*}
	is $(\cY,\cB(\R))$-measurable.
	We refer to \cite[Section 2.1]{neufeld2014measurability} for details and further references.

\medskip
We will repeatedly, and without further mention, use the following fact: if $S$ is a topological space, $T$ is a second-countable topological space (as is the case when $T$ is Polish), and $A\subseteq S$ and $B\subseteq T$ are endowed with the respective subspace topologies, then, with $A\times B$ endowed with the product topology, $\cB(A\times B)=\cB(A)\otimes\cB(B)$. This follows from the proof of \cite[Lemma 6.4.2.(i)]{bogachev2007measure}, since second countability is inherited by subspaces.\footnote{The cited result additionally assumes that $A$ is Hausdorff, but an inspection of its proof reveals that this assumption is not needed.}

\medskip
We provide the proof of the following result in \Cref{sec::proofs_preliminaries}. 

\begin{lemma}\label{lem::measurability_characteristics}
	Let $\tau$ be a finite $\F$--stopping time. The set $\widehat\Omega^C_\tau \subseteq \Omega \times \fP_\textnormal{sem}$ defined in \eqref{eq::semi_martingale_laws} is Borel-measurable. Moreover, there exists a Borel-measurable map
	\begin{equation*}
		\Omega \times \fP_\textnormal{sem} \times \Omega\times [0,\infty) \ni (\omega,\P,\tilde\omega,t) \longmapsto \big(\mathsf{b}^{\tau,\omega,\P}_t(\tilde\omega),\mathsf{a}^{\tau,\omega}_t(\tilde\omega),\mathsf{K}^{\tau,\omega,\P}_{\tilde\omega,t}\big) \in \R^d \times \S^{d}_\smallertext{+} \times \cL,
	\end{equation*}
	such that
	\begin{enumerate}
		\item[$(i)$] $\mathsf{b}^{\tau,\omega,\P}$ is $\F_\smallertext{+}$-progressive and $\F^\P_\smallertext{+}$-predictable{\rm;}
		\item[$(ii)$] $\mathsf{a}^{\tau,\omega}$ is $\F$-predictable{\rm;}
		\item[$(iii)$] $(\omega,\P,\tilde\omega,t,B) \longmapsto \mathsf{K}^{\tau,\omega,\P}_{\tilde\omega,t}(B)$ is a kernel on $(\R^d,\cB(\R^d))$ given $(\Omega\times\fP_\textnormal{sem} \times \Omega \times [0,\infty),\cF\otimes\cB(\fP_\textnormal{sem})\otimes \cP)$$;$
		\item[$(iv)$] for every $(\omega,\P) \in \widehat{\Omega}^C_\tau$, the collection $(\mathsf{b}^{\tau,\omega,\P},\mathsf{a}^{\tau,\omega},\mathsf{K}^{\tau,\omega,\P})$ constitutes the $(\F,\P)$--differential characteristic triplet of $X$ relative to $C^{\tau,\omega}_{\tau\smallertext{+}\smallertext{\cdot}}-C_{\tau(\omega)}(\omega)$, that is, outside a $\P$--null set,
			\begin{equation*}
			\mathsf{B}^{\P} = \int_0^\cdot \mathsf{b}^{\tau,\omega,\P}_t \d (C^{\tau,\omega}_{\tau\smallertext{+}\smallertext{\cdot}}-C_{\tau(\omega)}(\omega))_t, \; \mathsf{C}^{\P} = \int_0^\cdot \mathsf{a}^{\tau,\omega}_t \d (C^{\tau,\omega}_{\tau\smallertext{+}\smallertext{\cdot}}-C_{\tau(\omega)}(\omega))_t, \; \nu^{\P}(\d t,\d x) = \mathsf{K}^{\tau,\omega,\P}_t(\d x)\d (C^{\tau,\omega}_{\tau\smallertext{+}\smallertext{\cdot}}-C_{\tau(\omega)}(\omega))_t.
			\end{equation*}
	\end{enumerate}
\end{lemma}
	
For $\tau \equiv 0$, we omit the dependence on $\tau$ and $\omega$ in the notation of the differential characteristics.
	
\medskip
We recall the following result concerning the conditioning and products of semimartingale laws from \cite[Theorem 3.1 and Proposition 4.1]{neufeld2016nonlinear}. To formulate $(i)$, we use the fact that the compensator $\nu^\P$ admits a decomposition
\[
    \nu^\P(\omega;\d t,\d x)
    = \mathsf{K}_{\omega,t}(\d x)\d\mathsf{A}_t(\omega),
    \; \textnormal{$\P$--a.e. $\omega \in \Omega$,}
\] 
where $(\omega,t,B)\longmapsto\mathsf{K}_{\omega,t}(B)$ is a kernel on $(\R^d,\cB(\R^d))$ given $(\Omega\times[0,\infty),\cP)$, and $\mathsf{A}=(\mathsf{A}_t)_{t\in[0,\infty)}$ is a real-valued, right-continuous, $\P$--a.s. non-decreasing, $\F$-predictable process starting at zero; see \cite[Theorem II.1.8]{jacod2003limit} and \cite[Lemma 6.5.10]{weizsaecker1990stochastic}. The second assertion in $(i)$ follows immediately from \Cref{lem::measurability_characteristics}.$(iv)$.

	\begin{proposition}\label{prop::conditioning_characteristics2}
		Let $\P \in \fP_\textnormal{sem}$, and let $\tau$ be a finite $\F$--stopping time.
		\begin{enumerate}
			\item[$(i)$] Suppose that $(\mathsf{B},\mathsf{C},\mathsf{K}(\d x)\d\mathsf{A})$ are $(\F,\P)$--semi-martingale characteristics of $X$. Then $\P^{\tau,\omega} \in  \fP_\textnormal{sem}$ and 
			\begin{equation*}
				\big(\mathsf{B}^{\tau,\omega}_{\tau\smallertext{+}\smallertext{\cdot}}-\mathsf{B}_{\tau(\omega)}(\omega),\mathsf{C^{\tau,\omega}_{\tau\smallertext{+}\smallertext{\cdot}}}-\mathsf{C}_{\tau(\omega)}(\omega),\mathsf{K}_{\omega\otimes_\tau\cdot,\tau\smallertext{+}\smallertext{\cdot}}(\d x)\d(\mathsf{A}^{\tau,\omega}_{\tau\smallertext{+}\smallertext{\cdot}}-\mathsf{A}_{\tau(\omega)}(\omega))\big),
			\end{equation*}
			are $(\F,\P^{\tau,\omega})$--semi-martingale characteristics of $X$ for \textnormal{$\P$--a.e.} $\omega \in \Omega$. In particular, if $(\tilde\omega,\P) \in \Omega^C_s$ for some $s \in [0,\infty)$ and $\tau \geq s$, then $(\tilde\omega\otimes_s\omega,\P^{\tau,\omega}) \in \Omega^C_{\tau\smallertext{+}s}$ and
			\[
				\big(\mathsf{b}^{s,\omega,\P}_{\tau\smallertext{+}\smallertext{\cdot}}(\omega\otimes_\tau\cdot), \mathsf{a}^{s,\omega}_{\tau\smallertext{+}\smallertext{\cdot}}(\omega\otimes_\tau\cdot), \mathsf{K}^{s,\omega,\P}_{\omega\otimes_\tau\cdot,\tau\smallertext{+}\smallertext{\cdot}}(\d x)\big),
			\]
			are $(\F,\P^{\tau,\omega})$--differential characteristics of $X$ relative to $C^{\tau,\tilde\omega\otimes_s\omega}_{\tau\smallertext{+}s\smallertext{+}\smallertext{\cdot}}-C_{(\tau\smallertext{+}s)(\tilde\omega\otimes_s\omega)}(\tilde\omega\otimes_s\omega)$ for \textnormal{$\P$--a.e. $\omega \in \Omega$}.
			\item[$(ii)$] If $\Q$ is a stochastic kernel on $(\Omega,\cF)$ given $(\Omega,\cF_\tau)$ with $\Q(\omega) \in \fP_\textnormal{sem}$ for \textnormal{$\P$--a.e.} $\omega \in \Omega$, then
			\begin{align*}
				\overline{\P}[A] \coloneqq & \iint \big(\1_A\big)^{\tau,\omega}(\omega^\prime)\Q(\omega;\d\omega^\prime)\P(\d\omega), \; A \in \cF,
			\end{align*}
			belongs to $\fP_\textnormal{sem}$.
		\end{enumerate}
	\end{proposition}
	
	\begin{remark}
		For $s = 0$, the above result thus says that  $\mathsf{a}^{\tau,\omega} = \mathsf{a}_{\tau\smallertext{+}\smallertext{\cdot}}(\omega\otimes_\tau\cdot)$, \textnormal{$\d(C^{\tau,\omega}_{\tau\smallertext{+}\smallertext{\cdot}}-C_{\tau(\omega)}(\omega))$--a.e.}, \textnormal{$\P^{\tau,\omega}$--a.s.}, whenever the characteristics of $X$ relative to $(\F,\P)$ are absolutely continuous with respect to $C$.
	\end{remark}

	We denote by $\langle X^c \rangle$ the $\S^d_{\smallertext{+}}$-valued and $\F$-predictable process, which, for every $\P\in\fP_\textnormal{sem}$, is $\P$--a.s. non-decreasing and continuous, and coincides $\P$--a.s. with $\langle X^{c,\P} \rangle^{(\P)}$, the $(\F,\P)$--predictable quadratic covariation of the $(\F,\P)$--continuous local martingale part of $X$ (see \cite[Proposition 2.5(iii)]{neufeld2014measurability}). Here $\P$--a.s. non-decreasing means that $\langle X^c \rangle_t - \langle X^c \rangle_s$ is positive semi-definite for all $0 \leq s \leq t < \infty$ outside a $\P$--null set. 
	\begin{proposition}\label{cor::shif_quadratic_variation_continuous_martingale_part}
		Let $\tau$ be a finite $\F$--stopping time and $\P\in\fP_\textnormal{sem}$. If $X$ is an $(\F,\P)$--semi-martingale, then $(X^{c,\P})^{\tau,\omega}_{\tau\smallertext{+}\smallertext{\cdot}} - X^{c,\P}_{\tau(\omega)}(\omega)$ is a version of the $(\F,\P^{\tau,\omega})$--continuous local martingale part $X^{c,\P^{\tau,\omega}}$ of $X$, for \textnormal{$\P$--a.e.} $\omega \in \Omega$. Furthermore, $\langle X^c \rangle^{\tau,\omega}_{\tau\smallertext{+}\smallertext{\cdot}} - \langle X^c \rangle_{\tau(\omega)}(\omega) = \langle X^c\rangle$, \textnormal{$\P^{\tau,\omega}$--a.s.}, for \textnormal{$\P$--a.e.} $\omega \in \Omega$.
	\end{proposition}

	\subsection{Data of the BSDEs and corresponding weighted spaces}\label{sec::data}
	
	For each $(\omega,s) \in \Omega \times [0,\infty)$, we suppose that we are given a subset
	\[
		\fP(s,\omega) \subseteq \big\{\P\in \fP_\textnormal{sem} : (\mathsf{B}^{\P},\mathsf{C}^{\P},\nu^{\P}) \ll (C^{s,\omega}_{s\smallertext{+}\smallertext{\cdot}}-C_{s}(\omega)), \; \text{$\P$--a.s.} \big\},
	\]
	which is adapted in the sense that
	\begin{equation}\label{eq::adaptedness_probabilities}
		\fP(s,\omega) = \fP(s,\omega_{\cdot\land s}), \; s \in [0,\infty),\; \omega\in\Omega.
	\end{equation}
	Since each $\omega \in \Omega$ starts at zero, the collection $\fP(0,\omega)$ does not depend on $\omega$ and we therefore denote it by $\fP_0$. We suppose throughout that $\fP_0 \neq \varnothing$. 
	
	\medskip
	We recall that a subset of a Polish space is analytic if it is the image of a continuous map defined on some Polish space. Equivalently, it is the image of a Borel subset of a Polish space under a Borel-measurable map (see \cite[Propositions 8.2.3 and 8.2.6]{cohn2013measure}). For a $\sigma$-algebra $\cE$ on a set $E$, we denote by $\cE^\ast$ the universal completion of $\cE$, that is, $\cE^\ast \coloneqq \cap_\P \cE(\P)$ where the intersection is taken over all probability measures on $(E,\cE)$, and $\cE(\P)$ denotes the $\P$-completion of the $\sigma$-algebra $\cE$. Sets in $\cE^\ast$ are then said to be $\cE$--universally measurable, and a map defined on $E$ is then said to be $\cE$--universally measurable if it is measurable with respect to $\cE^\ast$. We note that every analytic subset of a Polish space $E$ is $\cB(E)$--universally measurable (see \cite[Corollary 8.4.3]{cohn2013measure}). Moreover, a measurable map between measurable spaces $(E,\cE)$ and $(H,\cH)$ is also ($\cE^\ast$,$\cH^\ast$)-measurable (see \cite[Lemma 8.4.6]{cohn2013measure}); we also use the terminology $(\cE, \cH)$-universally measurable at times. Lastly, a map $g : D \longrightarrow [-\infty,\infty]$, where $D\subseteq E$ and $E$ is Polish, is upper semi-analytic if $D$ and $\{x \in D : g(x) > c\}$ are analytic subsets of $E$ for every $c \in \R$. 
	
	\medskip
	The following conditions on the family $(\fP(s,\omega))_{(\omega,s)\in\Omega\times[0,\infty)}$ originate from \cite[Assumption 2.1]{nutz2013constructing} (see also \cite[Condition (A)]{neufeld2016nonlinear}); however, the weaker conditions in \cite[Condition (A)]{nutz2015robust} are sufficient for our purpose.
	\begin{assumption}\label{ass::probabilities2}
		For all $(s,t) \in [0,\infty)^2$ with $s \leq t$, $\bar\omega \in \Omega$, and $\P \in \fP(s,\bar{\omega})$,
		\begin{enumerate}
			\item[$(i)$] $\{(\omega,\P^\prime) :\omega \in \Omega,  \; \P^\prime \in \fP(s,\omega)\} \subseteq \Omega \times \fP(\Omega)$ is analytic$;$
			\item[$(ii)$] $\P^{t-s,\omega} \in \fP(t,\bar{\omega}\otimes_s \omega)$, for {\rm$\P$--a.e.} $\omega\in\Omega;$
			\item[$(iii)$] if $\Q$ is a stochastic kernel on $(\Omega,\cF)$ given $(\Omega,\cF_{t\smallertext{-}s})$ such that $\Q(\omega) \in \fP(t,\bar{\omega} \otimes_s\omega)$ for {\rm$\P$--a.e.} $\omega \in \Omega$, then the probability measure
			\begin{align*}
				\overline{\P}[A] \coloneqq & \iint \big(\1_A\big)^{t\smallertext{-}s,\omega}(\omega^\prime)\Q(\omega;\d\omega^\prime)\P(\d\omega), \; A \in \cF,
			\end{align*}
			belongs to $\fP(s,\bar{\omega})$.
		\end{enumerate}
	\end{assumption}
	
	\begin{example}\label{ex::differential_characteristics}
		This assumption is, for example, satisfied if $C_t \coloneqq t$ and $\fP(s,\omega) \equiv \fP_\Theta \subseteq \fP_\textnormal{sem}$ consists of all semi-martingale measures for which $X$ has differential characteristics $(\mathsf{b}_t,\mathsf{c}_t,\mathsf{K}_{\cdot,t})$ $($with respect to Lebesgue measure$)$ taking values in a Borel subset $\Theta \subseteq \R^d \times \S^d_\smallertext{+}\times \cL$$;$ see \textnormal{\cite[Theorem 2.1.$(i)$]{neufeld2016nonlinear}}. Further examples of such indexed families of probability measures satisfying the aforementioned assumptions can be found in {\rm\cite[Sections 3 and 4]{nutz2013constructing}}$;$ compare also with {\rm \citeauthor*{possamai2013robust} \cite{possamai2013robust}}.  
	\end{example}
	
	\begin{remark}\label{rem::measures_in_fP}
		$(i)$ By Galmarino's test $($see \textnormal{\cite[Theorem IV.100]{dellacherie1978probabilities}}$)$, the probability measure $\overline{\P}$ constructed in {\rm\Cref{ass::probabilities2}.$(iii)$} agrees with $\P$ on $\cF_{t\smallertext{-}s}$. Moreover, for $C \in \cF_{t-s}$ and $D \in \cF$,
		\[
			\overline{\P}\big[C \cap \{X_{(t-s)\smallertext{+}\smallertext{\cdot}} - X_{t-s} \in D\}\big] = \int_{C}\int_\Omega \1_D(\omega^\prime) \Q(\omega;\d\omega^\prime)\overline\P(\d\omega),
		\]
		which yields $\bar{\P}^{t\smallertext{-}s,\omega} = \Q(\omega)$ for \textnormal{$\P$--a.e.} $($and \textnormal{$\overline{\P}$--a.e.}$)$ $\omega \in \Omega$ by \textnormal{\Cref{rem::conditional_law_future_increments}} and the fact that $\cF$ is countably generated.
		
		\medskip
		$(ii)$ As already pointed out in {\rm\citeauthor*{nutz2013constructing} \cite[Remark 2.2.(b)]{nutz2013constructing}}, even if at the intuitive level one may think that \textnormal{`}$\fP(s,\omega) = \{\P^{s,\omega} : \P \in \fP_0\}$\textnormal{'}, this identity is formal since, for each $\P\in\fP_0$, the family $(\P^{s,\omega})_{\omega\in\Omega}$ is only \textnormal{$\P$--a.s.} uniquely determined.
	\end{remark}

	For $(\omega,s,K) \in \Omega \times [0,\infty)\times\cL$ and a Borel-measurable map $\cU : \R^d \longrightarrow \R$, we write
	\begin{align*}
		\|\cU(\cdot)\|^2_{\hat{\L}^\smalltext{2}_{\smalltext{\omega}\smalltext{,}\smalltext{s}}(K)} 
		& \coloneqq\int_{\R^\smalltext{d}}\bigg(\cU(x) - \int_{\R^\smalltext{d}}\cU(x)K(\d x) \Delta C_s(\omega)\bigg)^2 K(\d x) + (1-K(\R^d)\Delta C_s(\omega))^+\bigg(\int_{\R^\smalltext{d}} \cU(x) K(\d x)\bigg)^2 \Delta C_s(\omega).
	\end{align*}
	Then $\widehat\L^2_{\omega,s}(K)$ denotes the collection of $\cB(\R^d)$-measurable maps $\cU : \R^d \longrightarrow \R$ satisfying $\|\cU(\cdot)\|_{\hat{\L}^\smalltext{2}_{\smalltext{\omega}\smalltext{,}\smalltext{s}}(K)} < \infty$. This semi-normed space naturally arises from the predictable quadratic variation of stochastic integrals with respect to the compensated jump measure; see \textnormal{\Cref{sec::stochastic_integrals}} for details.
	
	\medskip
	The data $(T,\xi,f)$ satisfies the following conditions in our main results.

		\begin{assumption}\label{ass::generator2}
		$(i)$ $T$ is an $\F$--stopping time and $\xi : \Omega \longrightarrow \R$ is $\cF_T$-measurable.
		
		\medskip
		$(ii)$ The generator\footnote{The symbol $\bigsqcup$ denotes the disjoint union, and therefore each $f(t,\omega,\cdot,\cdot,\cdot,\cdot,\cdot,\cdot,K)$ is a map from $\R \times \R \times \R^d\times\R^d\times\S^d_\smallertext{+} \times{\widehat\L}^2_{\omega,t}(K)$ into $\R$.} $f:\bigsqcup_{ (\omega,t,K) \in \Omega \times [0,\infty) \times \cL} \R \times \R \times \R^d \times \R^d \times \S^d_{\smallertext{+}} \times {\widehat\L}^2_{\omega,t}(K) \longrightarrow \R$ satisfies
		\begin{align*}
			&\big|f\big(t,\tilde\omega,y,\mathrm{y},z,u(\cdot),b,a,K\big) - f\big(t,\tilde\omega,y^\prime,\mathrm{y}^\prime,z^\prime,u^{\prime}(\cdot),b,a,K\big)\big|^2 \\
			& \leq r_t(\tilde\omega) |y-y^\prime|^2 + \mathrm{r}_t(\tilde\omega) |\mathrm{y}-\mathrm{y}^\prime|^2 + \theta^X_t(\tilde\omega) (z-z^\prime)^\top a(z-z^\prime) + \theta^\mu_t(\tilde\omega)  \|u(\cdot) - u^{\prime}(\cdot)\|^2_{{\hat\L}^\smalltext{2}_{\smalltext{\tilde\omega}\smalltext{,}\smalltext{t}}(K)}, \; (\tilde\omega,t) \in \llparenthesis 0,T\rrbracket,
		\end{align*}
		for all $(y,y^\prime,\mathrm{y},\mathrm{y}^\prime,z,z^\prime,u(\cdot),u^\prime(\cdot),b,a) \in \R^2 \times \R^2 \times (\R^d)^2 \times \big(\widehat\L^2_{\tilde\omega,t}(K)\big)^2 \times \R^d \times \S^d_\smallertext{+}$, for some $\F$-predictable, $[0,\infty)$-valued processes $(r,\mathrm{r},\theta^X,\theta^\mu) = (r_t,\mathrm{r}_t,\theta^X_t,\theta^\mu_t)_{t \in [0,\infty)}$.
		
		\medskip
		$(iii)$ For any $s \in [0,\infty)$ and any $\cB(\widehat{\Omega}^C_s)\otimes\widetilde{\cP}$-measurable function
		\[
			\widehat{\Omega}^C_s \times \Omega \times [0,\infty) \times \R^d \ni (\omega,\P,\tilde{\omega},r,x) \longmapsto u^\P_r(\tilde{\omega};x) \in \R,
		\]
		with $u^\P_r(\tilde{\omega};\cdot) \in \widehat{\L}^2_{\omega\otimes_\smalltext{s}\tilde\omega,s+r}(\mathsf{K}^{s,\omega,\P}_{\tilde\omega,r})$, the function 
		\begin{equation*}
			\widehat\Omega^C_s \times \Omega \times [0,\infty) \times \R \times \R \times \R^d \times \R^d \times \S^d_\smallertext{+}\ni (\omega,\P,\tilde\omega,r,y,\mathrm{y},z,b,a) \longmapsto f\big(s+r,\omega\otimes_s\tilde\omega,y,\mathrm{y},z,u^\P_r(\tilde\omega;\cdot),b,a,\mathsf{K}^{s,\omega,\P}_{\tilde\omega,r}\big) \in \R,
		\end{equation*}
		is Borel-measurable, and for every fixed $(\omega,\P) \in \widehat\Omega^C_s$, the function
		\begin{equation*}
			\Omega \times [0,\infty) \times \R \times \R \times \R^d \times \R^d \times \S^d_\smallertext{+} \ni (\tilde\omega,r,y,\mathrm{y},z,b,a) \longmapsto f\big(s+r,\omega\otimes_s\tilde\omega,y,\mathrm{y},z,u^\P_r(\tilde\omega;\cdot),b,a,\mathsf{K}^{s,\omega,\P}_{\tilde\omega,r}\big) \in \R,
		\end{equation*}
		is $\textnormal{Prog}(\F_\smallertext{+})\otimes\cB(\R)\otimes\cB(\R)\otimes\cB(\R^d)\otimes\cB(\R^d)\otimes\cB(\S^d_\smallertext{+})$-measurable.
		
		\medskip
		$(iv)$ There exists $\hat\beta \in (0,\infty)$ such that for all $(\omega,s) \in \Omega\times[0,\infty)$,
		\begin{align}\label{eq::shifted_integrability}
			\sup_{\P\in\fP(s,\omega)}\E^\P \bigg[&\cE\big(\hat\beta A^{s,\omega}_{s\smallertext{+}\smallertext{\cdot}})_{(T\smallertext{-}s\land T)^{\smalltext{s}\smalltext{,}\smalltext{\omega}}}|\xi^{s,\omega}|^2 + \int_0^{(T\smallertext{-}s\land T)^{\smalltext{s}\smalltext{,}\smalltext{\omega}}} \cE(\hat\beta A^{s,\omega}_{s\smallertext{+}\smallertext{\cdot}})_r \frac{|f^{s,\omega,\P}_r(0,0,0,\mathbf{0})|^2}{|\alpha^{s,\omega}_r|^2} \d (C^{s,\omega}_{s\smallertext{+}\smallertext{\cdot}})_r\bigg] < \infty,
		\end{align}
		where
		\[
			f^{s,\omega,\P}_r(\tilde\omega,0,0,0,\mathbf{0}) \coloneqq f\big(s+r,\omega\otimes_s\tilde\omega,0,0,0,\mathbf{0},\mathsf{b}^{s,\omega,\P}_r(\tilde\omega),\mathsf{a}^{s,\omega}_r(\tilde\omega),\mathsf{K}^{s,\omega,\P}_{\tilde{\omega},r}(\d x)\big),
		\]
		 $\alpha^2 \coloneqq\max\{\sqrt{r},\sqrt{\mathrm{r}},\theta^X,\theta^\mu\} > 0$, the process $A = (A_t)_{t \in [0,\infty)}$ given by $A_t \coloneqq \int_0^{t \land T} \alpha^2_s\d C_s$ is assumed to be finite-valued, and
		\begin{equation}\label{eq::stieltjes_exponential2}
			\cE(\hat\beta A)_r \coloneqq 
			\begin{cases}
				1, & r = 0, \\[0.5em]
				 \mathrm{e}^{\hat\beta (A_\smalltext{r}\smallertext{-}A_\smalltext{0})} \prod_{s \in (0,r]} (1+\hat\beta\Delta A_s)\mathrm{e}^{\smallertext{-}\hat\beta\Delta A_\smalltext{s}}, & r \in (0,\infty).
			\end{cases}
		\end{equation}
		
		\medskip
		$(v)$ There exists $\Phi \in [0,1)$ such that $\Delta A \leq \Phi$ and $\widetilde M^\Phi_1(\hat\beta) <1$, where
		\begin{equation*}
			\widetilde M^\Phi_1(\beta) \coloneqq \ff^\Phi(\beta) + \frac{1}{\beta} + \max\bigg\{1,\frac{1+\beta\Phi}{\beta}\bigg\}\bigg(\frac{1}{\beta} + \beta\fg^\Phi(\beta)\bigg), \; \beta \in (0,\infty),
		\end{equation*}
		with
		\begin{align}\label{eq_frak_f_g}
			\ff^\Phi(\beta) 
			\coloneqq \frac{4(1+\beta\Phi)}{\beta^2},\;
			\fg^\Phi(\beta)
			\coloneqq \frac{4}{\beta^2} \mathbf{1}_{\{\Phi = 0\}} + \frac{\Phi^2\sqrt{1+\beta\Phi}}{\Big(1+\beta\Phi - \sqrt{1+\beta\Phi}\Big)\Big(\sqrt{1+\beta\Phi} - 1\Big)}\mathbf{1}_{\{\Phi > 0\}}.
		\end{align}
	\end{assumption}
	
	Since we are interested in the aggregation of solutions to BSDEs whose driving martingale is $X^{c,\P}$ and driving random measure is $\mu^X$, where the characteristics of $X$ are absolutely continuous relative to the fixed process $C$, we write
	\begin{equation}\label{eq::definition_f_s_P}
		f^{s,\omega,\P}_r\big(\tilde\omega,y,\mathrm{y},z,u(\cdot)\big) \coloneqq f\big(s+r,\omega\otimes_s\tilde\omega,y,\mathrm{y},z,u(\cdot),\mathsf{b}^{s,\omega,\P}_r(\tilde\omega),\mathsf{a}^{s,\omega}_r(\tilde\omega),\mathsf{K}^{s,\omega,\P}_{\tilde{\omega},r}(\d x)\big),
	\end{equation}
	for $u(\cdot) \in \widehat{\L}^2_{\omega\otimes_\smalltext{s}\tilde\omega,s+r}(\mathsf{K}^{s,\omega,\P}_{\tilde\omega,r})$ and $(\omega,\P) \in \widehat{\Omega}^C_s$. In case $s = 0$, we simply write $f^{\P}_r(\tilde\omega,y,\mathrm{y},z,u(\cdot))$.
	
	\medskip
	Some remarks on the structure and potentially awkward assumptions of our generator are appropriate here.

	\begin{remark}
		$(i)$ Note that $\cE(\hat\beta A^{s,\omega}_{s\smallertext{+}\smallertext{\cdot}}) = \cE(\hat\beta (A^{s,\omega}_{s\smallertext{+}\smallertext{\cdot}} - A_s(\omega)))$ and that, according to our conventions, $\d (C^{s,\omega}_{s\smallertext{+}\smallertext{\cdot}})$ is identical to $\d (C^{s,\omega}_{s\smallertext{+}\smallertext{\cdot}} - C_s(\omega))$, since $0$ is excluded from the domain of integration.
		
		\medskip
		$(ii)$ Although not our main focus in this work, we eventually aim to apply our results to stochastic control problems. Therefore, for each $(\omega, t)$, we need to be able to evaluate the generator along a function $u_t(\omega; \cdot) : \mathbb{R}^d \longrightarrow \mathbb{R}$, which appears in the stochastic integral of the compensated jump measure. Indeed, the generator associated to generic stochastic control problems will typically take the form
		\begin{align*}
			&f\big(t,\omega,y, \mathrm{y},z,u_t(\omega;\cdot), b,a,K_{\omega,t}(\d x)\big) \\
			&= \sup_{v  \in V} \bigg\{g_t(\omega,v) - \frac{d_t(\omega)}{1+d_t(\omega)\Delta C_t(\omega)}y +  z^\top a \beta_t(\omega,v) + \int_{\R^\smalltext{d}} u_t(\omega;x) ( \gamma_t(\omega,x,v) - 1) K_{\omega,t}(\d x)\bigg\},
		\end{align*}
		where $V$ denotes the space in which the control variables $v$ take values$;$ see also \textnormal{\citeauthor*{confortola2013backward} \cite[Equation 4.8]{confortola2013backward}} for the case where one solely controls the compensator of a marked point process. 
		This motivates the measurability condition imposed in \textnormal{\Cref{ass::generator2}.$(iii)$}.
		To clarify this, we revisit the assumptions imposed on the generator in \textnormal{\cite{possamai2024reflections}}, where the following issue arises: it is desirable to define the generator $f$ on a product space $\Omega \times [0,\infty) \times \R \times \R^d \times \fH$ such that the mapping
		\begin{equation}\label{eq::measurability_f}
			(\omega,t) \longmapsto f\big(t,\omega,Y_t(\omega),Z_t(\omega),U_t(\omega;\cdot)\big)
		\end{equation}
		is at least progressively measurable for the finite variation component in the dynamics of the \textnormal{BSDE} to be well-defined and adapted. However, since the function $U_t(\omega; \cdot) : \R^d \longrightarrow \R$ is passed to the generator, $\fH$ should consist of a collection of functions from $\R^d$ to $\R$. This raises the questions: which functions should the collection $\fH$ at least contain for our results to be applicable to control problems, and is there an appropriate $\sigma$-algebra on $\fH$ for which measurability of $f$ on its domain implies progressive measurability in \eqref{eq::measurability_f}?
		
		\medskip
		One approach is to let $\fH$ be the space of all Borel-measurable functions from $\R^d$ to $\R$, see \textnormal{\cite[Section 2.3]{confortola2014backward}}. We would then require that the map $\Omega \times [0,\infty) \ni (\omega,t) \longmapsto f\big(t,\omega,y,z,u_t(\omega;\cdot)\big) \in \R$ is progressively measurable for every $y$, $z$, and every predictable function $u : \Omega \times [0,\infty) \times \R^d \longrightarrow \R$. If the generator is additionally continuous in $y$ and $z$, then the map \eqref{eq::measurability_f} is progressively measurable.
		
		\medskip
		Alternatively, we note that the integrands $U$ passed to the generator satisfy $U_t(\omega; \cdot) \in \widehat{\L}^2(K_{\omega,t})$, where $K$ is the kernel in the disintegration of the compensator $\nu(\mathrm{d}t, \mathrm{d}x) = K_{\cdot,t}(\mathrm{d}x)\mathrm{d}C_t$. We could therefore impose that for each $(\omega, t, y, z)$, we are given a function $f(t, \omega, y, z, \cdot) : \widehat{\L}^2(K_{\omega,t}) \longrightarrow \R$ such that for every predictable function $u : \Omega \times [0,\infty) \times \R^d \longrightarrow \R$ with $u_t(\omega; \cdot) \in \widehat{\L}^2(K_{\omega,t})$, the map $(\omega,t) \longmapsto f(t,\omega,y,z,u_t(\omega; \cdot))$
		is progressively measurable. Continuity in $y$ and $z$ then ensures that \eqref{eq::measurability_f} is progressively measurable. In this case, $f$ can be seen as a function on a disjoint union of product spaces
		\begin{equation*}
			f : \bigsqcup_{(\omega,t) \in \Omega \times [0,\infty)} \R \times \R^d \times \widehat{\L}^2(K_{\omega,t}) \longrightarrow \R.
		\end{equation*}
		This approach was taken in \textnormal{\cite{confortola2013backward}} and in our previous work \textnormal{\cite{possamai2024reflections}}, and is more convenient for control applications since the norm associated with $\widehat{\L}^2(K_{\omega,t})$ is defined via the kernel in the disintegration $\nu(\omega; \mathrm{d}t, \mathrm{d}x) = K_{\omega,t}(\mathrm{d}x)\mathrm{d}C_t$, which naturally appears in control problems \textnormal{\cite[Equation 4.8]{confortola2013backward}}.
		
		\medskip
		It turns out that the measurability issue worsens in the current setting, since we want the---expectation of the---first component of the \textnormal{BSDE} solution to be measurable in the probability law $\P$. The solution to the \textnormal{BSDE} in \textnormal{\cite{possamai2024reflections}} is constructed via a fixed-point argument, making it the limit of Picard iterations. To ensure measurability in $\P$, we must assume that the process obtained by substituting the $\P$-dependent arguments into the generator $f$ is measurable in $\P$, and this measurability must be preserved at each step of the iteration. However, the spaces $\widehat{\L}^2(K_{\omega,t}^\P)$ depend on $\P$ since they are defined through the disintegration of the $\P$-compensator $\nu^\P$ relative to $C$. 
		Condition $(iii)$ ensures that each iteration in the fixed-point argument remains measurable with respect to $\P$.
		For the $z$-component in the generator---and similarly for the $(y,\mathrm{y})$-components---we do not need a condition involving $\P$, since the integrand $\cZ^\P$ takes values in $\R^d$. Therefore, the measurability of $(\P, \omega, t) \longmapsto \cZ_t^\P(\omega)$ ensures the measurability of $(\P, \omega, t) \longmapsto f(t,\omega, \ldots, \cZ_t^\P(\omega), \ldots)$ in $\P$, provided that $f$ is (Borel-)measurable in its remaining variables.
	\end{remark}
	
	We recall from \cite{possamai2024reflections} the weighted solution spaces used throughout. To define them in full generality, let $\cG$ be an arbitrary $\sigma$-algebra on $\Omega$ containing $\cF$, and let $\G = (\cG_t)_{t \in [0,\infty)}$ be an arbitrary filtration on $(\Omega,\cG)$ satisfying $\cF_t \subseteq \cG_t$ for every $t \in [0,\infty)$, and denote by $\cG_{0-}$ an additional $\sigma$-algebra contained in $\cG_0$. We use the conventions $\cG_{\infty\smallertext{+}} = \cG_\infty \coloneqq \cG_{\infty\smallertext{-}} \coloneqq \sigma (\cup_{t \in [0,\infty)} \cG_t)$. Let $\P$ be a probability measure on $(\Omega,\cG)$, and let $M$ be an $\R^d$-valued, right-continuous, $\G$-adapted, $(\G,\P)$--locally square-integrable martingale. Fix $(\omega,s) \in \Omega \times [0,\infty)$. Suppose that	$\langle M \rangle = \pi \bcdot (C^{s,\omega}_{s\smallertext{+}\smallertext{\cdot}}-C_{s}(\omega))$, $\P$--a.s., where $\pi = (\pi_t)_{t \in [0,\infty)}$ is an $\S^d_\smallertext{+}$-valued, $\G$-predictable process (see \Cref{sec::stochastic_integrals}), and that the $(\G,\P)$-compensator of $\mu^X$ is $\P$--a.s. of the form $\nu^\P(\d t,\d x) = K_{t}(\d x)\d (C^{s,\omega}_{s\smallertext{+}\smallertext{\cdot}}-C_{s}(\omega))_t$ for some kernel $K$ on $(\R^d,\cB(\R^d))$ given $(\Omega \times [0,\infty),\cP(\G))$.
	Moreover, we define $M^\P_{\mu^X}[W]\coloneqq\E^\P[W\ast\mu^X_\infty]$ for every nonnegative measurable function $W$ on $\Omega\times[0,\infty)\times\R^d$, where $W\ast \mu^X = \int_0^\cdot\int_{\R^d}W_s(x)\mu^X(\d s, \d x)$. For $\widetilde{\cP}(\G) \coloneqq \cP(\G) \otimes \cB(\R^d)$, we denote by $M^\P_{\mu^X}[W|\widetilde\cP(\G)]$ the `conditional expectation' of $W$ with respect to $\widetilde\cP(\G)$ under the measure induced by $M^\P_{\mu^X}$; see \Cref{sec::stochastic_integrals} for the precise definition and its extension to signed functions.
	
	\medskip
	For $\beta \in [0,\infty)$, we consider the following weighted spaces, in which two elements are identified whenever the seminorm of their difference is zero:

	\begin{itemize}[leftmargin=0.6cm]
		\item $\L^{2,s,\omega}_{T,\beta}(\cG,\P)$: Banach space of $\cG$-measurable random variables $\zeta : \Omega \longrightarrow \R$ satisfying 
		\[
		\|\zeta\|^{2}_{\L^{\smalltext{2}\smalltext{,}\smalltext{s}\smalltext{,}\smalltext{\omega}}_{\smalltext{T}\smalltext{,}\smalltext{\beta}}(\cG,\P)} \coloneqq \E^\P\Big[\big|\cE(\beta A^{s,\omega}_{s\smallertext{+}\smallertext{\cdot}})^{1/2}_{(T\smallertext{-}s\land T)^{\smalltext{s}\smalltext{,}\smalltext{\omega}}}\zeta\big|^2\Big] < \infty;
		\]
		\item $\cH^{2,s,\omega}_{T,\beta}(\G,\P)$: Banach space of real-valued, right-continuous, $(\G,\P)$--square-integrable martingales $L = (L_t)_{t \in [0,\infty)}$ with $L = L_{\cdot\land (T\smallertext{-}s\land T)^{s,\omega}}$ and 
		\[
		\|L\|^2_{\cH^{\smalltext{2}\smalltext{,}\smalltext{s}\smalltext{,}\smalltext{\omega}}_{\smalltext{T}\smalltext{,}\smalltext{\beta}}(\G,\P)} \coloneqq \E^\P[L^2_0] + \E^\P\bigg[\int_0^{(T\smallertext{-}s \land T)^{\smalltext{s}\smalltext{,}\smalltext{\omega}}} \cE(\beta A^{s,\omega}_{s\smallertext{+}\smallertext{\cdot}})_r \d \langle L \rangle_r\bigg] < \infty;
		\]
		\item $\cT^{2,s,\omega}_{T,\beta}(\G,\P)$: Banach space of real-valued, right-continuous, $\P$--a.s. c\`adl\`ag, $\G$-adapted processes $Y = (Y_t)_{t \in [0,\infty]}$ with $Y = Y_{\cdot \land (T\smallertext{-}s\land T)^{s,\omega}}$ , $(Y_t)_{t\in[0,\infty)}$ being $\G$-optional, and
		\begin{equation*}
			\|Y\|^2_{\cT^{\smalltext{2}\smalltext{,}\smalltext{s}\smalltext{,}\smalltext{\omega}}_{\smalltext{T}\smalltext{,}\smalltext{\beta}}(\G,\P)} 
			\coloneqq \sup_{\sigma \in \sT_{\smalltext{0}\smalltext{,}\smalltext{(}\smalltext{T}\tinytext{-}\smalltext{s}\smalltext{\land}\smalltext{T}\smalltext{)}^{\tinytext{s}\tinytext{,}\tinytext{\omega}}}(\G)}\E^\P\Big[ \big| \cE(\beta A^{s,\omega}_{s\smallertext{+}\smallertext{\cdot}})^{1/2}_\sigma Y_\sigma \big|^2\Big] < \infty,
		\end{equation*}
		where $\sT_{0,(T\smallertext{-}s\land T)^{\smalltext{s}\smalltext{,}\smalltext{\omega}}}(\G)$ denotes the collection of all $\G$--stopping times $\sigma$ satisfying $0 \leq \sigma \leq (T-s\land T)^{s,\omega}$;
		\item $\cS^{2,s,\omega}_{T,\beta}(\G,\P)$: Banach space of real-valued, right-continuous, $\P$--a.s. c\`adl\`ag, $\G$-adapted processes $Y = (Y_t)_{t \in [0,\infty]}$ with $Y = Y_{\cdot \land (T\smallertext{-}s\land T)^{\smalltext{s}\smalltext{,}\smalltext{\omega}}}$, $(Y_t)_{t \in [0,\infty)}$ being $\G$-optional, and\footnote{The measurability of the supremum inside the expectation follows from \cite[Proposition 2.21.(i)]{elkaroui2013capacities}. To be precise, the supremum is $\cG_\infty$--universally measurable, and the probability $\P$ uniquely extends to $\cG_\infty$--universally measurable sets.}
		\begin{equation*}
			\|Y\|^2_{\cS^{\smalltext{2}\smalltext{,}\smalltext{s}\smalltext{,}\smalltext{\omega}}_{\smalltext{T}\smalltext{,}\smalltext{\beta}}(\G,\P)} 
			\coloneqq \E^\P\bigg[ \sup_{r \in [0,(T\smallertext{-}s\land T)^{\smalltext{s}\smalltext{,}\smalltext{\omega}}]} \big| \cE(\beta A^{s,\omega}_{s\smallertext{+}\smallertext{\cdot}})^{1/2}_r Y_r \big|^2\bigg] < \infty;
		\end{equation*}
		\item $\H^{2,s,\omega}_{T,\beta}(\G,\P)$: Banach space of real-valued, $\G$-optional $\phi = (\phi_t)_{t \in [0,\infty]}$ with $\phi = \phi_{\cdot \land (T\smallertext{-}s\land T)^{s,\omega}}$ and
		\begin{equation*}
			\|\phi\|^2_{\H^{\smalltext{2}\smalltext{,}\smalltext{s}\smalltext{,}\smalltext{\omega}}_{\smalltext{T}\smalltext{,}\smalltext{\beta}}(\G,\P)} \coloneqq \E^\P\bigg[ \int_0^{(T\smallertext{-}s\land T)^{\smalltext{s}\smalltext{,}\smalltext{\omega}}} \cE(\beta A^{s,\omega}_{s\smallertext{+}\smallertext{\cdot}})_r |\phi_r|^2 \d (C^{s,\omega}_{s\smallertext{+}\smallertext{\cdot}})_r\bigg] < \infty;
		\end{equation*}
		\item $\H^{2,s,\omega}_{T,\beta}(M;\G,\P)$: Banach space of $\R^d$-valued, $\G$-predictable $Z = (Z_t)_{t \in [0,\infty)}$ with $Z = Z \mathbf{1}_{\llbracket 0, (T\smallertext{-}s\land T)^{\smalltext{s}\smalltext{,}\smalltext{\omega}} \rrbracket}$ and
		\begin{equation*}
			\|Z\|^2_{\H^{\smalltext{2}\smalltext{,}\smalltext{s}\smalltext{,}\smalltext{\omega}}_{\smalltext{T}\smalltext{,}\smalltext{\beta}}(M;\G,\P)} \coloneqq \E^\P\bigg[\int_0^{(T\smallertext{-}s\land T)^{s,\omega}} \cE(\beta A^{s,\omega}_{s\smallertext{+}\smallertext{\cdot}})_r Z^\top_r \pi_r Z_r \d (C^{s,\omega}_{s\smallertext{+}\smallertext{\cdot}})_r\bigg] < \infty;
		\end{equation*}
		\item $\H^{2,s,\omega}_{T,\beta}(\mu^X;\G,\P)$: Banach space of real-valued, $\widetilde{\cP}(\G)$-measurable functions $U$ with $U = U \mathbf{1}_{\llbracket 0, (T\smallertext{-}s\land T)^{\smalltext{s}\smalltext{,}\smalltext{\omega}} \rrbracket}$ and
		\begin{equation*}
			\|U\|^2_{\H^{\smalltext{2}\smalltext{,}\smalltext{s}\smalltext{,}\smalltext{\omega}}_{\smalltext{T}\smalltext{,}\smalltext{\beta}}(\mu^\smalltext{X};\G,\P)} \coloneqq \E^\P\bigg[\int_0^{(T-s\land T)^{s,\omega}} \cE(\beta A^{s,\omega}_{s\smallertext{+}\smallertext{\cdot}})_r \|U_r(\cdot)\|^2_{\hat\L^\smalltext{2}_{\smalltext{\omega}\smalltext{\otimes}_\tinytext{s}\smalltext{\cdot}\smalltext{,}\smalltext{s}\smalltext{+}\smalltext{r}}(K_\smalltext{r})} \d (C^{s,\omega}_{s\smallertext{+}\smallertext{\cdot}})_r\bigg] < \infty;
		\end{equation*}
		\item $\cH^{2,s,\omega,\perp}_{0,T,\beta}(M,\mu^X;\G,\P)$: subspace consisting of all martingales $N \in \cH^{2,s,\omega}_{T,\beta}(\G,\P)$ starting at zero with $\langle N, M\rangle^{(\P)} = 0$ and $M^\P_{\mu^\smalltext{X}}[\Delta N | \widetilde\cP(\G)] = 0$$.$
	\end{itemize}
	For $s = 0$, the spaces do not depend on $\omega$, and we will omit $(s,\omega)$ from the notation in the above spaces and norms. We will also adopt this simplification for $\beta = 0$.

\section{Main results}\label{sec::main_results}

In this section, we present our main results on the construction and measurability of a single value process obtained through the path regularisation of a pathwise-defined value function, together with the corresponding dynamic programming principle.

\medskip
For $(\omega,s) \in \Omega\times[0,\infty)$ and $\P \in \fP(s,\omega)$, we denote by $\cY^{s,\omega,\P}((T- s \land T)^{s,\omega},\xi^{s,\omega})$ the first component of the solution $(\cY,\cZ,\cU,\cN)$ to the following BSDE, with terminal time $(T-s\land T)^{s,\omega}$, terminal condition $\xi^{s,\omega}$ and generator $f^{s,\omega,\P}$ defined in \eqref{eq::definition_f_s_P},
\begin{align}\label{eq::P_BSDE2}
	\cY_t &= \xi^{s,\omega} + \int_t^{(T\smallertext{-}s\land T)^{\smalltext{s}\smalltext{,}\smalltext{\omega}}} f^{s,\omega,\P}_r\big(\cY_r,\cY_{r\smallertext{-}},\cZ_r,\cU_r(\cdot)\big)\d (C^{s,\omega}_{s\smallertext{+}})_r - \bigg(\int_t^{(T\smallertext{-}s\land T)^{\smalltext{s}\smalltext{,}\smalltext{\omega}}} \cZ_r \d X^{c,\P}_r\bigg)^{(\P)} \nonumber\\
	&\quad - \bigg(\int_t^{(T\smallertext{-}s\land T)^{\smalltext{s}\smalltext{,}\smalltext{\omega}}}\int_{\R^\smalltext{d}} \cU_r(x)\tilde\mu^{X,\P}(\d r, \d x)\bigg)^{(\P)} - \int_t^{(T\smallertext{-}s\land T)^{\smalltext{s}\smalltext{,}\smalltext{\omega}}}\d \cN_r, \; t \in [0,\infty], \; \text{$\P$--a.s.},
\end{align}
in the sense of \cite[Section 3]{possamai2024reflections}, relative to $(\F_\smallertext{+},\P)$, where $X^{c,\P}$ denotes the continuous local martingale part of $X$ relative to $(\F,\P)$ (see also \Cref{rem::equivalence_semimartingale}). The superscript $(\P)$ means that the stochastic integrals are constructed under $\P$; by \Cref{prop::good_version_stochastic_integral}, their values are independent of the choice of filtration $\G$ satisfying $\F\subseteq\G\subseteq\F^\P_+$.
For $s = 0$, the BSDE above does not depend on $\omega$; thus, we drop the dependence on $(\omega,s)$ from the notation and simply write $\cY^\P(T,\xi)$.

\medskip
Our main goal is to construct a measurable and c\`adl\`ag process $\widehat{\cY}^\smallertext{+} = \big\{\widehat{\cY}^\smallertext{+}_t(\omega)\big\}$ satisfying 
\[
	\widehat\cY^\smallertext{+}_t(T,\xi) = \underset{\bar{\P} \in \fP_\smalltext{0}(\cF_{t\smalltext{+}},\P)}{{\esssup}^\P} \cY^{\bar{\P}}_t(T,\xi), \; \textnormal{$\P$--a.s.}, \; t \in [0,\infty], \; \textnormal{$\P \in \fP_0$.}
\]
Here, $\fP_0(\cF_{t\smallertext{+}},\P)$ denotes the collection of probabilities $\overline\P$ in $\fP_0$ that coincide with $\P$ on $\cF_{t\smallertext{+}}$. 

\medskip
To summarise the main steps, let us consider the case $f \equiv 0$, in which case the previous identity reduces to
\[
	\widehat\cY^\smallertext{+}_t(T,\xi) = \underset{\bar{\P} \in \fP_\smalltext{0}(\cF_{t\smalltext{+}},\P)}{{\esssup}^\P} \E^{\bar\P} [ \xi | \cF_{t\smallertext{+}}], \; \textnormal{$\P$--a.s.}, \; t \in [0,\infty], \; \textnormal{$\P \in \fP_0$.}
\]
This case was considered by \citeauthor*{cohen2012quasi} \cite{cohen2012quasi}, \citeauthor*{nutz2012quasi} \cite{nutz2012quasi,nutz2013random,nutz2013constructing}, \citeauthor*{nutz2012superhedging} \cite{nutz2012superhedging}, \citeauthor*{soner2011quasi} \cite{soner2011quasi}, \citeauthor*{elkaroui2013capacities} \cite{elkaroui2013capacities, elkaroui2013capacities2}, and \citeauthor*{bartl2020conditional} \cite{bartl2020conditional}. The first step is to construct a single measurable process $\widehat\cY = \big\{\widehat\cY_t(\omega)\big\}$ satisfying
\begin{equation}\label{eq::aggregation_generator_zero}
	\widehat\cY_t = \underset{\bar\P \in \fP_\smalltext{0}(\cF_\smalltext{t},\P)}{{\esssup}^\P} \E^{\bar\P} [ \xi | \cF_t], \; \textnormal{$\P$--a.s.}, \; t \in [0,\infty), \; \text{for all $\P \in \fP_0$.}
\end{equation}
The discussion in \Cref{rem::measures_in_fP}.$(ii)$ suggests choosing the candidate as
\begin{equation*}
	\widehat\cY_t(\omega) 
	\coloneqq \sup_{\P \in \fP(t,\omega)}\E^{\P}[\xi^{t,\omega}] , \; (\omega,t) \in \Omega \times [0,\infty).
\end{equation*}
The assumptions imposed on the family $(\fP(t,\omega))_{(\omega,t) \in \Omega \times [0,\infty)}$ ensure that defining $\widehat\cY_t(\omega)$ exactly as above implies that \eqref{eq::aggregation_generator_zero} holds; see \cite[Theorem 2.3]{nutz2013constructing}.   The proof relies on the theory of analytic sets and the corresponding Jankov--von Neumann selection theorem. As a consequence, $\widehat{\cY}_t$ is upper semi-analytic and $\cF^\ast_t$-measurable (compare with \cite[Theorem~2.3]{nutz2013constructing}). For general $f$, we replace the conditional expectation $\E^{\bar\P}[\xi|\cF_t]$ with $\E^{\bar\P}[\cY^{\bar\P}_t(T,\xi)|\cF_t]$, where the term inside the conditional expectation is the first component of the solution to a BSDE. The correct replacement for $\E^\P[\xi^{t,\omega}]$ in defining $\widehat\cY_t$ is suggested by \Cref{lem::conditioning_bsde2} and is discussed in \Cref{sec::value_function}. 

\medskip
It follows from \eqref{eq::aggregation_generator_zero} that  
\begin{equation*}
	\E^\P[\widehat\cY_t| \cF_s] \leq \widehat\cY_s, \; 0 \leq s \leq t < \infty, \; \text{$\P$--a.s.}, \; \P \in \fP_0,
\end{equation*}
so that $\widehat\cY=(\widehat\cY_t)_{t\in[0,\infty)}$ is an $(\F^\ast,\P)$--super-martingale for every $\P\in\fP_0$, where $\F^\ast=(\cF^\ast_t)_{t\in[0,\infty]}$. It is therefore natural to wonder whether it admits a c\`adl\`ag modification. In general, it does not; see \cite[Example 4.6]{nutz2012superhedging} for a counterexample. 
Nonetheless, by following the proof of \cite[Proposition 4.5]{nutz2012superhedging}, one constructs, through path regularisation, a c\`adl\`ag process $\widehat\cY^\smallertext{+} = (\widehat\cY^\smallertext{+}_t)_{t\in[0,\infty]}$ satisfying
\begin{equation*}
	\widehat\cY^\smallertext{+}_t = \underset{\bar\P \in \fP_\smalltext{0}(\cG_{\smallertext{t}\tinytext{+}},\P)}{{\esssup}^\P} \E^{\bar\P} [ \xi | \cF_{t\smallertext{+}}], \; \textnormal{$\P$--a.s.}, \; t \in [0,\infty], \; \P \in \fP_0,
\end{equation*}
which is adapted to a filtration $\G_\smallertext{+} = (\cG_{t\smallertext{+}})$ satisfying $\F^\ast \subseteq \G\subseteq \F^\P$ (in the sense of filtrations) for every $\P\in\fP_0$. The precise filtration is given in \Cref{sec::regularisation}; the same representation holds with $\fP_0(\cF_{t\smallertext{+}},\P)$ in place of $\fP_0(\cG_{t\smallertext{+}},\P)$. When the generator is non-zero, the super-martingale property of the pathwise value function turns into a nonlinear super-martingale property. This, and the construction of the regularisation $\widehat\cY^\smallertext{+}$ are discussed in \Cref{sec::regularisation}.

\medskip
The value process $\widehat{\cY}^\smallertext{+}$ constructed in this work also constitutes the $Y$-component of the 2BSDE studied in our companion paper \cite{possamai2025mind}. There, we derive its semi-martingale decomposition and prove well-posedness of the corresponding 2BSDE system. 

\subsection{Construction and measurability of the pathwise value function}\label{sec::value_function}

The following is our first main result in this work. We defer the proof to \Cref{sec::proof_measurability}.

\begin{theorem}\label{thm::measurability2}
	Suppose that {\rm\Cref{ass::probabilities2}} and {\rm\ref{ass::generator2}} hold. For every $s \in [0,\infty)$, the function $\widehat\cY_s(T,\xi) : \Omega \longrightarrow [-\infty,\infty]$ defined by
	\begin{equation*}
		\widehat\cY_s(T,\xi)(\omega) \coloneqq
			\displaystyle \sup_{\P \in \fP(s,\omega)} \E^\P\big[\cY^{s,\omega,\P}_0((T-s\land T)^{s,\omega},\xi^{s,\omega})\big],
	\end{equation*} 
	is upper semi-analytic and $\cF_s$--universally measurable. Moreover, $\widehat\cY_s(T,\xi)(\omega) = \widehat\cY_{s\land T(\omega)}(T,\xi)(\omega)$ for all $\omega\in\Omega$, and
	\begin{equation}\label{eq::dynamic_programming_principle2}
		\widehat\cY_s(T,\xi) = \underset{\bar{\P} \in \fP_\smalltext{0}(\cF_\smallertext{s},\P)}{{\esssup}^\P} \E^{\bar{\P}} \big[ \cY^{\bar{\P}}_s(T,\xi)\big| \cF_s\big], \; \textnormal{$\P$--a.s.}, \; \P \in \fP_0,
	\end{equation}
	where $\fP_0(\cF_s,\P) \coloneqq \big \{\overline{\P} \in \fP_0 : \overline{\P} = \P \; \text{\rm on} \; \cF_s\big\}$.
\end{theorem}

\begin{remark}\label{rem::measurability_value_function}
	Note that $\widehat{\cY}_s(T,\xi)(\omega) = \widehat{\cY}_s(T,\xi)(\omega_{\cdot\land s})$ by construction, since $\fP(s,\omega) = \fP(s,\omega_{\cdot\land s})$ and each $\cY^{s,\omega,\P}_0((T-s\land T)^{s,\omega},\xi^{s,\omega})$ depends on $\omega$ only through $\omega_{\cdot\land s}$. Moreover, we write $\widehat{\cY}(T,\xi)$ to denote the associated process defined on $\Omega \times [0,\infty)$, or simply $\widehat{\cY}$ if no confusion arises regarding the terminal stopping time $T$ and the terminal condition $\xi$.
\end{remark}

\subsection{Path-regularisation of the value function and (pathwise) dynamic programming}\label{sec::regularisation}

Let $\G = (\cG_t)_{t \in [0,\infty)}$ be the filtration defined by
\begin{equation}\label{eq::filtration_G}
\cG_t \coloneqq \sigma(\cF^\ast_t \cup \sN^{\fP_\smalltext{0}}),
\end{equation}
where $\cF^\ast_t$ denotes the universal completion of $\cF_t$ and $\sN^{\fP_\smalltext{0}}$ denotes the collection of all subsets $A \subseteq \Omega$ that are $(\cF,\P)$--null sets for all $\P\in\fP_0$. We then implicitly adjoin $\cG_{0\smallertext{-}} \coloneqq \{\varnothing, \Omega\}$ and $\cG_{\infty\smallertext{+}} \coloneqq \cG_{\infty} \coloneqq \cG_{\infty\smallertext{-}} \coloneqq \sigma(\cup_{t \in [0,\infty)} \cG_t)$ to $\G$ whenever necessary.

\medskip
Under additional assumptions that we introduce below, the representation \eqref{eq::dynamic_programming_principle2} allows us to construct a c\`adl\`ag $\G_\smallertext{+}$-adapted process $\widehat{\cY}^\smallertext{+}(T,\xi)$, which is the right limit of $\widehat{\cY}(T,\xi)$ along $\D_\smallertext{+}$ outside a $\fP_0$--polar set. This regularisation then satisfies
 \[
	\widehat\cY^\smallertext{+}_t(T,\xi) = \underset{\bar{\P} \in \fP_\smalltext{0}(\cG_{t\smalltext{+}},\P)}{{\esssup}^\P} \cY^{\bar{\P}}_t(T,\xi), \; \textnormal{$\P$--a.s.}, \; t \in [0,\infty], \; \textnormal{$\P \in \fP_0$},
\]
which is precisely the desired property of our single value process. 

\medskip
The overarching idea is to show, through a linearisation argument, that $\widehat\cY(T,\xi) - \E^\P[\cY^\P_\cdot(T,\xi) | \cF_\cdot]$ is a nonlinear super-martingale relative to $(\F^\P,\P)$ for every $\P\in \fP_0$, where the conditional expectation process denotes the $\F$-optional projection of $\cY^\P(T,\xi)$. We then apply an up- and down-crossing inequality to a transformation of this nonlinear super-martingale. Due to the generality of our setting, the bound on the crossings does not follow from existing results in the literature. The proof of \cite[Lemma A.1]{bouchard2016general}, which provides the most general crossing result for nonlinear super-martingales, contains a gap that we address below; see also \Cref{rem::gap_regularisation}.

\medskip
Naturally, since linearisation and changes of measure are involved, we will need our generator to satisfy additional assumptions that are closely related to sufficient conditions for the comparison of solutions to BSDEs. To state them, we first note that the Lipschitz-continuity of the generator with respect to the $(\mathrm{y},z)$-variables allows us to write
\begin{align}\label{eq::lipschitz_linearisation}
	&f(t,\tilde\omega,y,\mathrm{y},z,u(\cdot),b,a,K) - f(t,\tilde\omega,y^\prime,\mathrm{y}^\prime,z^\prime,u^\prime(\cdot),b,a,K) \nonumber\\
	& \geq f(t,\tilde\omega,y,\mathrm{y},z,u(\cdot),b,a,K) - f(t,\tilde\omega,y^\prime,\mathrm{y},z,u(\cdot),b,a,K) \nonumber\\
	&\quad + \widehat\lambda^{\mathrm{y},\mathrm{y}^\smalltext{\prime}}_t(\tilde\omega)(\mathrm{y}-\mathrm{y}^\prime)  + (\eta^{z,z^\smalltext{\prime},a}_t(\tilde\omega))^\top a (z-z^\prime) + f(t,\tilde\omega,y^\prime,\mathrm{y}^\prime,z^\prime,u(\cdot),b,a,K) - f(t,\tilde{\omega},y^\prime,\mathrm{y}^\prime,z^\prime,u^\prime(\cdot),b,a,K),
\end{align}
for $(\tilde{\omega},t) \in \llparenthesis 0, T \rrbracket$, $(b,a,K) \in \R^d \times \S^d_\smallertext{+}\times\cL$, and $(y,y^\prime,\mathrm{y},\mathrm{y}^\prime,z,z^\prime,u(\cdot),u^\prime(\cdot)) \in (\R)^4 \times (\R^d)^2 \times  (\widehat{\L}^2_{\tilde\omega,t}\big(K)\big)^2$, where
\begin{equation}\label{eq::definition_lambda_hat_lambda}
	\widehat\lambda^{\mathrm{y},\mathrm{y}^\smalltext{\prime}}_t(\tilde\omega) \coloneqq - \sqrt{\mathrm{r}_t(\tilde\omega)}\sgn(\mathrm{y}-\mathrm{y}^\prime),\; 
	\eta^{z,z^\smalltext{\prime},a}_t(\tilde\omega) \coloneqq -\sqrt{\theta^{X}_t(\tilde\omega)} \frac{(z-z^\prime)}{| (z-z^\prime)^\top a(z-z^\prime)|^{1/2}}\1_{\{(z-z^\smalltext{\prime})^\smalltext{\top} a (z-z^\smalltext{\prime}) \neq 0\}}.
\end{equation}

\medskip
The following is our main additional assumption in this section. We will comment on its necessity further below.
\begin{assumption}\label{ass::crossing}
	For every $(\omega,s) \in \Omega \times [0,\infty)$ and every $\P \in \fP(s,\omega)$, the following hold
	\begin{enumerate}
	\item[$(i)$] for every $(\cY,\cY^{\prime},\cZ,\cU) \in \big(\cS^{2,s,\omega}_{T,\hat\beta}(\F_\smallertext{+},\P)\big)^2 \times \H^{2,s,\omega}_{T,\hat\beta}(X^{c,\P};\F,\P) \times \H^{2,s,\omega}_{T,\hat\beta}(\mu^X;\F,\P)$, there exists an $\F$-progressive process $\lambda = \lambda^{s,\omega,\P}\1_{\llparenthesis 0,(T-s\land T)^{\smallertext{s}\smallertext{,}\smallertext{\omega}} \rrbracket}$ such that $\lambda\Delta C^{s,\omega}_{s\smallertext{+}\smallertext{\cdot}} > -1$ holds \textnormal{$\P$--a.s.}, $|\lambda| \leq \sqrt{r^{s,\omega}_{s\smallertext{+}\smallertext{\cdot}}}$ holds \textnormal{$\P\otimes \d C^{s,\omega}_{s\smallertext{+}\smallertext{\cdot}}$--a.e.} on $\llparenthesis 0, (T-s\land T)^{s,\omega} \rrbracket$, and
	\begin{equation*}
		f^{s,\omega,\P}\big(\cY,\cY_{\smallertext{-}},\cZ,\cU(\cdot)\big) - f^{s,\omega,\P}\big(\cY^{\prime},\cY_{\smalltext{-}},\cZ,\cU(\cdot)\big) 
		\geq \lambda (\cY - \cY^{\prime}), \; \textnormal{$\P \otimes \mathrm{d} C^{s,\omega}_{s\smallertext{+}\smallertext{\cdot}}$--a.e. on $\llparenthesis 0, (T-s\land T)^{s,\omega} \rrbracket$$;$}
	\end{equation*}
	
	\item[$(ii)$] there exists $\mathfrak{C} = \mathfrak{C}(\omega,s,\P) \in [0,\infty)$ such that
	\begin{equation}\label{eq::bounded_lipschitz_constants}
		\int_0^{(T\smallertext{-}s\land T)^{\smalltext{s}\smalltext{,}\smalltext{\omega}}}\max\Big\{\sqrt{\mathrm{r}^{s,\omega}_{s\smallertext{+}t}}, (\theta^{X})^{s,\omega}_{s\smallertext{+}t}, (\theta^\mu)^{s,\omega}_{s\smallertext{+}t}\Big\} \d (C^{s,\omega}_{s\smallertext{+}\smallertext{\cdot}})_t \leq \mathfrak{C}, \; \textnormal{$\P$--a.s.};
	\end{equation}

	\item[$(iii)$] for every $(\cY,\cZ,\cU,\cU^{\prime}) \in \cS^{2,s,\omega}_{T,\hat\beta}(\F_\smallertext{+},\P) \times \H^{2,s,\omega}_{T,\hat\beta}(X^{c,\P};\F,\P) \times \big(\H^{2,s,\omega}_{T,\hat\beta}(\mu^X;\F,\P)\big)^2$, there exists $\rho \in \H^{2,s,\omega}_{T}(\mu^X;\F^\P_\smallertext{+},\P)$ such that $\Delta (\rho \ast\tilde{\mu}^{X,\P}) > -1$, \textnormal{$\P$--a.s.}, and
	\begin{equation}\label{eq::rho_bounded_theta}
		\frac{\d\langle \rho \ast\tilde\mu^{X,\P}\rangle^{(\P)}}{\d  C^{s,\omega}_{s\smallertext{+}\smallertext{\cdot}}} \leq (\theta^\mu)^{s,\omega}_{s\smallertext{+}\smallertext{\cdot}},
	\end{equation}
	\begin{equation*}
		f^{s,\omega,\P}\big(\cY,\cY_{\smallertext{-}},\cZ,\cU(\cdot)\big) - f^{s,\omega,\P}\big(\cY,\cY_{\smallertext{-}},\cZ,\cU^{\prime}(\cdot)\big) 
		\geq \frac{\d\langle \rho \ast\tilde\mu^{X,\P},(\cU-\cU^{\prime})\ast\tilde\mu^{X,\P}\rangle^{(\P)}}{\d C^{s,\omega}_{s\smallertext{+}\smallertext{\cdot}}},
	\end{equation*}
	both hold $\P \otimes \mathrm{d}C^{s,\omega}_{s\smallertext{+}\smallertext{\cdot}}$--{\rm a.e.} on $\llparenthesis 0, (T-s\land T)^{s,\omega} \rrbracket$.
	\end{enumerate}
	
	Moreover, for every $\P\in \fP_0$, the following holds
	\begin{enumerate}
	\item[$(iv)$] there exists $\rho^\dagger \in \H^{2}_{T}(\mu^X;\F,\P)$ with $\Delta (\rho^\dagger \ast\tilde{\mu}^{X,\P}) > -1$, \textnormal{$\P$--a.s.}, such that, for every $(\cY,\cZ,\cU,\cU^{\prime}) \in \cS^{2}_{T,\hat\beta}(\F_\smallertext{+},\P) \times \H^{2}_{T,\hat\beta}(X^{c,\P};\F,\P) \times \big(\H^{2}_{T,\hat\beta}(\mu^X;\F,\P)\big)^2$,
	\begin{equation}\label{eq::rho_dagger_bounded_theta}
		\frac{\d\langle \rho^\dagger \ast\tilde\mu^{X,\P}\rangle^{(\P)}}{\d  C} \leq \theta^\mu,
	\end{equation}
	\begin{equation*}
		f^{\P}\big(\cY,\cY_{\smallertext{-}},\cZ,\cU(\cdot)\big) - f^{\P}\big(\cY,\cY_{\smallertext{-}},\cZ,\cU^{\prime}(\cdot)\big) 
		\geq \frac{\d\langle \rho^\dagger \ast\tilde\mu^{X,\P},(\cU-\cU^{\prime})\ast\tilde\mu^{X,\P}\rangle^{(\P)}}{\d C},
	\end{equation*}
	hold $\P \otimes \mathrm{d}C$--{\rm a.e.} on $\llparenthesis 0, T \rrbracket$.
	\end{enumerate}
\end{assumption}

\begin{remark}\label{rem::ass_regularisation}
	$(i)$ The Lipschitz-continuity of $f^{s,\omega,\P}$ with respect to the $y$-variable yields
	\begin{equation*}
		- \sqrt{r^{s,\omega}_{s\smallertext{+}t}} |y-y^\prime| \leq f^{s,\omega,\P}_t\big(y,\mathrm{y},z,u(\cdot)\big) - f^{s,\omega,\P}_t\big(y^\prime,\mathrm{y},z,u(\cdot)\big) \leq \sqrt{r^{s,\omega}_{s\smallertext{+}t}} |y-y^\prime|.
	\end{equation*}
	If we now let $\lambda = (\lambda_t)_{t \in [0,\infty)}$ be
	\begin{equation*}
		\lambda \coloneqq -\sqrt{r^{s,\omega}_{s\smallertext{+}\smallertext{\cdot}}} \,\textnormal{\sgn}(\cY-\cY^\prime)\1_{\llparenthesis 0,(T-s\land T)^{\smallertext{s}\smallertext{,}\smallertext{\omega}} \rrbracket},
	\end{equation*}
	for solutions $\cY$ and $\cY^\prime$ to our \textnormal{BSDEs},  
	then the conditions of \textnormal{\Cref{ass::crossing}.$(i)$} are met, except that $\lambda$ might not necessarily be $\F$-progressive due to the $\F_\smallertext{+}$-adaptedness of $\cY$ and $\cY^{\prime}$. Since the measurability requirements in \textnormal{\Cref{ass::crossing}.$(i)$} are crucial in the proof of \textnormal{\Cref{lem::solv_bsde_cond}}, we need to impose them.
	
	\medskip
	$(ii)$ \textnormal{\Cref{ass::crossing}.$(ii)$--$(iii)$} are sufficient for a comparison principle to hold for our \textnormal{BSDEs}$;$ see \textnormal{\Cref{prop::comparison}}. These conditions are weaker than \textnormal{\cite[Assumption 7.1]{possamai2024reflections}} under which we showed a comparison principle, see \textnormal{\cite[Proposition 7.3]{possamai2024reflections}}. 
	
	\medskip
	$(iii)$ The reversed order of quantifiers in \textnormal{\Cref{ass::crossing}.$(iv)$} allows us to use the same function $\rho^\dagger$ throughout the down-crossing argument, while the bound \eqref{eq::rho_dagger_bounded_theta} is used to derive \eqref{eq::bounding_supermartingale_over_dyadics}$;$ see the proof of {\rm\Cref{thm::down-crossing}}.$(i)$.

\end{remark}

The result on the path regularisation of the value function then reads as follows.
\begin{theorem}\label{thm::down-crossing}
	Suppose that {\rm\Cref{ass::probabilities2}} and {\rm\ref{ass::generator2}} hold, that {\rm\Cref{ass::crossing}} holds for $s = 0$, and that
	\begin{equation}\label{eq::constant_phi}
		\phi^{2,\hat{\beta}}_{\xi,f} 
		\coloneqq \sup_{\P \in \fP_\smalltext{0}} \E^\P\Bigg[\sup_{s \in \D_\tinytext{+}} \underset{\bar{\P} \in \fP_\smalltext{0}(\cF_{\smalltext{s}},\P)}{{\esssup}^\P}\E^{\bar{\P}}\bigg[\cE(\hat\beta A)_T|\xi|^2 + \int_s^T \cE(\hat\beta A)_r \frac{|f^{\bar\P}_r(0,0,0,\mathbf{0})|^2}{\alpha^2_r} \d C_r\bigg| \cF_{s}\bigg] \Bigg] < \infty.
	\end{equation}
	Then the following hold
	\begin{enumerate}
		\item[$(i)$] there exists a real-valued, c\`adl\`ag, $\G_\smallertext{+}$-adapted process $\widehat\cY^\smallertext{+}(T,\xi) = (\widehat\cY^\smallertext{+}_t(T,\xi))_{t \in [0,\infty]}$ that satisfies $\widehat\cY^{\smallertext{+}}(T,\xi) = \widehat\cY^{\smallertext{+}}_{\cdot \land T}(T,\xi)$ and $\widehat\cY^{\smallertext{+}}_T(T,\xi) = \xi$ identically, and
		\begin{equation}\label{eq::hat_y_limit}
			\widehat\cY^\smallertext{+}_t(T,\xi) = \lim_{\D_\tinytext{+} \ni s \downarrow\downarrow t} \widehat\cY_s(T,\xi), \; t \in [0,\infty), \; \textnormal{$\fP_0$--q.s.;}
		\end{equation}
		\item[$(ii)$] there exists a constant $\mathfrak{C} \in (0,\infty)$ depending only on $\hat\beta$ and $\Phi$ such that
		\begin{equation*}
			\sup_{\P\in\fP_\smalltext{0}}\E^\P\bigg[\sup_{s \in \D_\tinytext{+}}\big|\cE(\hat\beta A)^{1/2}_{s}\widehat\cY_{s}(T,\xi)\big|^2\bigg] + \sup_{\P\in\fP_\smalltext{0}}\E^\P\bigg[\sup_{s \in [0,T]}\big|\cE(\hat\beta A)^{1/2}_{s}\widehat\cY^\smallertext{+}_s(T,\xi)\big|^2\bigg] \leq \mathfrak{C}\phi^{2,\hat{\beta}}_{\xi,f} < \infty,
		\end{equation*}
		\begin{equation}\label{eq::aggregation}
			\widehat\cY^\smallertext{+}_t(T,\xi) = \underset{\bar{\P} \in \fP_\smalltext{0}(\cG_{t\smalltext{+}},\P)}{{\esssup}^\P} \cY^{\bar{\P}}_t(T,\xi), \; \textnormal{$\P$--a.s.}, \; t \in [0,\infty], \; \textnormal{$\P \in \fP_0$;}
		\end{equation}
		\item[$(iii)$] if, in addition, {\rm\Cref{ass::crossing}} holds, then, for all $\G_\smallertext{+}$--stopping times $\sigma$ and $\tau$,
		\begin{equation}\label{eq::nonlinear_supermartingale_property}
			\widehat\cY^\smallertext{+}_{\sigma \land \tau \land T}(T,\xi) \geq \cY^\P_{\sigma \land \tau \land T}\big(\tau\land T,\widehat\cY^\smallertext{+}_{\tau\land T}(T,\xi)\big), \; \textnormal{$\P$--a.s.}, \; \text{\rm $\P \in \fP_0$.}
		\end{equation}
	\end{enumerate}
\end{theorem}

\begin{remark}
	Let $\F^\ast = (\cF^\ast_t)_{t \in [0,\infty)}$, that is, $\F^\ast$ is the filtration obtained by universally completing each $\sigma$-algebra in $\F$. Then the proof of \textnormal{\Cref{thm::down-crossing}.$(i)$} even shows that one could define $\widehat{\cY}^+(T,\xi)$ to be a real-valued, right-continuous, and $\F^\ast_\smallertext{+}$-optional process on $[0,\infty)$ at the expense of the paths then only being \textnormal{$\fP_0$--q.s.} c\`adl\`ag.
\end{remark}

The previous result plays a key role in establishing well-posedness of the corresponding second-order BSDE (2BSDE) in our companion paper \cite{possamai2025mind}.

\medskip
In the context of optimisation problems, we obtain the following relation between $\widehat{\cY}(T,\xi)$ and $\widehat{\cY}^\smallertext{+}(T,\xi)$.

\begin{corollary}\label{cor::optimisation}
	Suppose that {\rm\Cref{ass::probabilities2}} and {\rm\ref{ass::generator2}} hold, that {\rm\Cref{ass::crossing}} holds for $s = 0$, that $\phi^{2,\hat{\beta}}_{\xi,f} < \infty$, and that
	\begin{equation}\label{eq::implies_uniform_integrability}
		\E^\P\Bigg[\underset{\bar{\P} \in \fP_\smalltext{0}(\cF_{\smalltext{0}\tinytext{+}},\P)}{{\esssup}^\P}\E^{\bar{\P}}\bigg[\cE(\hat\beta A)_T|\xi_T|^2 + \int_0^T \cE(\hat\beta A)_r \frac{|f^{\bar\P}_r(0,0,0,\mathbf{0})|^2}{\alpha^2_r} \d C_r\bigg| \cF_{0\smallertext{+}}\bigg] \Bigg] < \infty, \; \P \in \fP_0.
	\end{equation}
	Let $\widehat{\cY}^\smallertext{+}(T,\xi)$ be the process constructed in \textnormal{\Cref{thm::down-crossing}}. Then
	\begin{equation}\label{eq::relation_to_optimisation}
		\widehat{\cY}_0(T,\xi) = \sup_{\P\in\fP_\smalltext{0}}\E^\P\big[\widehat{\cY}^\smallertext{+}_0(T,\xi)\big].
	\end{equation}
	Moreover, if $\P^\ast \in \fP_0$ satisfies $\widehat{\cY}_0(T,\xi) = \E^{\P^\smalltext{\star}}\big[\cY^{\P^\smalltext{\ast}}_0(T,\xi)\big]$, then $\widehat{\cY}^\smallertext{+}_0(T,\xi) = \cY^{\P^\smalltext{\ast}}_0(T,\xi)$,  \textnormal{$\P^\ast$--a.s.}, and
	\begin{equation}\label{eq::max_widehat_y_plus}
		 \sup_{\P\in\fP_\smalltext{0}}\E^\P\big[\widehat{\cY}^\smallertext{+}_0(T,\xi)\big] = \E^{\P^\smalltext{\ast}}\big[\widehat{\cY}^\smallertext{+}_0(T,\xi)\big].
	\end{equation}
	Conversely, if $\P^\ast\in\fP_0$ satisfies \eqref{eq::max_widehat_y_plus} and $\bar{\P}^\ast \in \fP_0(\cG_{0\smallertext{+}},\P^\ast)$ satisfies $\widehat{\cY}^\smallertext{+}_0(T,\xi) = \cY^{\bar\P^\smalltext{\ast}}_0(T,\xi)$, \textnormal{$\P^\ast$--a.s.}, then
	\[
		\widehat{\cY}_0(T,\xi) = \E^{\bar\P^\smalltext{\ast}}\big[\cY^{\bar\P^\smalltext{\ast}}_0(T,\xi)\big]
	\]
	holds.
\end{corollary}

\begin{remark}
	The optimisation problem involving $\widehat{\cY}_0(T,\xi)$ can be reduced to that involving $\widehat{\cY}_0^{\smallertext{+}}(T,\xi)$ and a subsequent conditional optimisation over $\fP_0(\cG_{0\smallertext{+}},\P^\ast)$. The latter has the advantage that $\widehat{\cY}^\smallertext{+}(T,\xi)$ is c\`adl\`ag.
\end{remark}

	\section{Proofs of the main results}\label{sec::proofs_main_results}

	In this section, we prove the main results stated in \Cref{sec::main_results}. For each main result, we introduce a series of auxiliary lemmata, whose proofs are provided in \Cref{sec::lemmas_main_results}.
	
	\subsection{Measurability of the value function}\label{sec::proof_measurability}

	In this subsection, we prove \Cref{thm::measurability2}. The proofs of all auxiliary lemmata deferred from this subsection are provided in \Cref{sec::proofs_lemmata_measurability}.
	
	\medskip
	It is useful to consider the following filtered space $(\Omega^\dagger,\cF^\dagger,\F^\dagger = (\cF^\dagger_t)_{t \in [0,\infty)})$, where
	\begin{equation*}
		\Omega^\dagger \coloneqq \Omega \times \fP(\Omega) \times \Omega, \; \cF^\dagger \coloneqq \cF \otimes \cB(\fP(\Omega)) \otimes \cF, \; \cF^\dagger_t \coloneqq \cF \otimes \cB(\fP(\Omega)) \otimes \cF_t.
	\end{equation*}
	The Borel $\sigma$-algebra on $\Omega^\dagger$ generated by the product topology coincides with $\cF^\dagger$; see \cite[Lemma 6.4.2.(i)]{bogachev2007measure}.
	We omit the proof of the following result, as it follows by a straightforward adaptation of the arguments in \cite[Section 3]{neufeld2014measurability}.
	\begin{lemma}\label{lem::measurable_martingale_modification}
$(i)$ Let $g : \Omega^\dagger \longrightarrow [-\infty,\infty]$ be Borel-measurable. Then $$\Omega \times \fP(\Omega) \ni (\omega,\P) \longmapsto \E^\P[g(\omega,\P,\cdot)] \in [-\infty,\infty]$$ is Borel-measurable. Moreover, for each $t \in [0,\infty)$, there exist versions of the conditional expectations $\E^\P[g(\omega,\P,\cdot)|\cF_t]$ and $\E^\P[g(\omega,\P,\cdot)|\cF_{t\smallertext{+}}]$ for which
			\begin{equation*}
				\Omega^\dagger \ni (\omega,\P,\tilde\omega) \longmapsto \E^\P[g(\omega,\P,\cdot)|\cF_t](\tilde\omega) \in [-\infty,\infty] \; \text{and} \; \Omega^\dagger \ni (\omega,\P,\tilde\omega) \longmapsto \E^\P[g(\omega,\P,\cdot)|\cF_{t\smallertext{+}}](\tilde\omega) \in [-\infty,\infty]
			\end{equation*}
			are measurable relative to $\cF^\dagger_t$ and $\cF^\dagger_{t\smallertext{+}}$, respectively, and, for fixed $(\omega,\P)\in\Omega\times\fP(\Omega)$,
			\begin{equation*}
				\Omega \ni \tilde\omega \longmapsto \E^\P[g(\omega,\P,\cdot)|\cF_t](\tilde\omega) \in [-\infty,\infty] 
				\; \text{and} \; \Omega \ni \tilde\omega \longmapsto \E^\P[g(\omega,\P,\cdot)|\cF_{t\smallertext{+}}](\tilde\omega) \in [-\infty,\infty]
			\end{equation*}
			are measurable relative to $\cF_t$ and $\cF_{t\smallertext{+}}$, respectively.
			
\medskip
$(ii)$ Let $g : \Omega^\dagger \times [0,\infty) \longrightarrow [-\infty,\infty]$, and suppose that $g(\cdot,\cdot,\cdot, t)$ is $\cF^\dagger_{t\smallertext{+}}$-measurable. Then $g(\omega,\P,\cdot, t)$ is $\cF_{t\smallertext{+}}$-measurable, and there exists a Borel-measurable and $\F^\dagger_{\smallertext{+}}$-optional function $\bar g : \Omega^\dagger \times [0,\infty) \longrightarrow \R$ satisfying the following properties
			\begin{enumerate}
				\item[$(a)$] $[0,\infty) \ni t \longmapsto \bar g(\omega,\P,\tilde\omega,t) \in \R$ is right-continuous for every $(\omega,\P,\tilde\omega) \in \Omega^\dagger$$;$
				\item[$(b)$] if $g(\omega,\P,\cdot,\cdot)$ for $(\omega,\P) \in \Omega \times \fP(\Omega)$ is an $\F_\smallertext{+}$-adapted, $(\F_\smallertext{+},\P)$--super-martingale such that the map $[0,\infty) \ni t \longmapsto \E^\P[g(\omega,\P,\cdot, t)] \in \R$ is right-continuous, then the process $\bar g(\omega,\P,\cdot,\cdot)$ is an $\F_\smallertext{+}$-optional $\P$-modification of $g(\omega,\P,\cdot,\cdot)$ and thus also an $(\F_\smallertext{+},\P)$--super-martingale.
			\end{enumerate}
	\end{lemma}

	\begin{lemma}\label{lem::measurable_decomposition}
		Fix $s \in [0,\infty)$. Let $\Omega \times \fP_\textnormal{sem} \times \Omega \times [0,\infty) \ni (\omega,\P,\tilde\omega,t) \longmapsto \cM^{\omega,\P}_t(\tilde\omega) \in \R$ be Borel-measurable. Suppose that, for each $(\omega,\P) \in \Omega \times \fP_\textnormal{sem}$, the process $\cM^{\omega,\P}$ is a right-continuous, $(\F_\smallertext{+},\P)$--square-integrable martingale. Then there exists a Borel-measurable map
		\begin{equation*}\label{eq::measurable_integrands2}
			\Omega \times \fP_\textnormal{sem} \times \Omega \times [0,\infty) \ni (\omega,\P,\tilde\omega,t) \longmapsto \cZ^{\omega,\P}_t(\tilde\omega) \in \R^d
		\end{equation*}
		and an $\cF\otimes\cB(\fP_\textnormal{sem})\otimes\widetilde{\cP}$-measurable map
		\[
			\Omega \times \fP_\textnormal{sem} \times \Omega \times [0,\infty) \times \R^d \ni (\omega,\P,\tilde\omega, t, x) \longmapsto \cU^{\omega,\P}_t(\tilde\omega;x) \in \R,
		\]
		such that,
		\begin{enumerate}
			\item[$(i)$] for each $(\omega,\P)\in\Omega\times\fP_\textnormal{sem}$, $(\cZ^{\omega,\P},\cU^{\omega,\P})\in\H^2(X^{c,\P};\F,\P) \times \H^2(\mu^X;\F,\P)$ and
			\begin{equation}\label{eq::martingale_representation2}
				\cN^{\omega,\P} \coloneqq \cM^{\omega,\P} - \cM^{\omega,\P}_0 - \bigg(\int_0^\cdot \cZ^{\omega,\P}_r \d X^{c,\P}_r\bigg)^{(\P)} -  \bigg(\int_0^\cdot\int_{\R^\smalltext{d}}\cU^{\omega,\P}_r(x)\tilde\mu^{X,\P}(\d r, \d x)\bigg)^{(\P)},
			\end{equation}
			belongs to $\cH^{2,\perp}_0(X^{c,\P},\mu^X;\F_{\smallertext{+}},\P)$, and
			\item[$(ii)$] if $(\omega,\P) \in \widehat{\Omega}^C_s \subseteq \Omega \times \fP_\textnormal{sem}$, then $\cU^{\omega,\P}_r(\tilde\omega;\cdot) \in \widehat{\L}^2_{\omega\otimes_\smalltext{s}\tilde\omega,s\smallertext{+}r}(\mathsf{K}^{s,\omega,\P}_{\tilde\omega, r})$ for every $(\tilde\omega, r) \in \Omega \times [0,\infty)$.
		\end{enumerate}
	\end{lemma}

	\begin{proposition}\label{prop::Borel-measurability_value_function}
		Suppose that {\rm\Cref{ass::probabilities2}} and {\rm\ref{ass::generator2}} hold. Let $s \in [0,\infty)$. The function
		\begin{equation}\label{eq::borel_measurability_value_function}
			(\omega,\P) \longmapsto \E^\P[\cY^{s,\omega,\P}_0((T-s\land T)^{s,\omega},\xi^{s,\omega})],
		\end{equation}
		defined on the subset $\{(\omega,\P) : \omega \in \Omega,\, \P \in \fP(s,\omega)\} \subseteq \Omega \times \fP(\Omega)$ is upper semi-analytic.
	\end{proposition}
	
	\begin{proof}
		We first explain why it suffices to consider the case of a bounded terminal condition and a bounded finite variation process in the dynamics of the BSDEs. For $n \in \N^\star$, set 
		\begin{align*}
			\sigma_n\coloneqq & \inf\{t \in [0,\infty) : C_t \geq n\} = \inf\{t \in [0,\infty) \cap \Q : C_t \geq n\}.
		\end{align*}
		Then, $(\sigma_n)_{n\in\N^\smalltext{\star}}$ is a sequence of strictly positive $\F$--predictable stopping times tending to infinity (see \cite[Proposition I.2.13]{jacod2003limit}). For the moment, suppose that we have shown that the map
		\begin{equation}\label{eq::truncated_measurability}
			(\omega,\P) \longmapsto \E^\P[\cY^{s,\omega,\P,n}_0((T-s\land T)^{s,\omega},(\xi^n)^{s,\omega})],
		\end{equation}
		defined on $\{(\omega,\P) : \omega \in \Omega ,\; \P \in \fP(s,\omega)\}$ is Borel-measurable. Here, $\cY^{s,\omega,\P,n}((T - s \land T)^{s,\omega}, (\xi^n)^{s,\omega})$ denotes the solution to the BSDE with generator $f^n \coloneqq \min\{\max\{f, -n\}, n\} \1_{\llbracket 0, \sigma_\smalltext{n} \rrparenthesis}$ and terminal condition $\xi^n \coloneqq \min\{\max\{\xi, -n\}, n\}$, before shifting by $(\omega, s)$. Note that each $f^n$ satisfies the same properties as $f$, and that the finite variation component of the BSDE is bounded in total variation by $n^2$. 
		As $\sigma_n$ tends to infinity and $(\xi^n, f^n)$ tends to $(\xi, f)$ for $n \longrightarrow \infty$, we deduce from the stability of solutions to BSDEs (see \Cref{cor::stability}) that
		\begin{equation*}
			\lim_{n \rightarrow \infty}\E^\P\big[\cY^{s,\omega,\P,n}_0((T-s\land T)^{s,\omega},(\xi^n)^{s,\omega})\big] = \E^\P\big[\cY^{s,\omega,\P}_0((T-s\land T)^{s,\omega},\xi^{s,\omega})\big],
		\end{equation*}
		pointwise on $\{(\omega,\P) : \omega \in \Omega,\; \P \in \fP(s,\omega)\}$, which implies the desired Borel-measurability.
		
		\medskip
		We fix $n \in \N^\star$ and turn to the Borel-measurability of \eqref{eq::truncated_measurability}.
		In what follows, all maps are implicitly extended to $\Omega \times \fP_\textnormal{sem} \times \Omega \times [0,\infty)$ by setting them to zero outside their original domains. To simplify the notation, we will write $f$ instead of $f^n$ and $\xi$ instead of $\xi^n$. The map
		\begin{equation*}
			\widehat{\Omega}^C_s \times \Omega \ni (\omega,\P,\tilde\omega) 
			\longmapsto 
			\xi^{s,\omega}(\tilde\omega) + \int_0^{(T\smallertext{-}s\land T)^{s,\omega}(\tilde\omega)} f^{s,\omega,\P}_r(\tilde\omega,0,0,0,\mathbf{0})\d(C^{s,\omega}_{s\smallertext{+}\smallertext{\cdot}})_r(\tilde\omega) \in \R,
		\end{equation*}
		is Borel-measurable and bounded by $n+n^2$. By \Cref{lem::measurable_martingale_modification}, there exists a Borel-measurable and $\F^\dagger_\smallertext{+}$-optional map
		\begin{equation*}
			\Omega^\dagger \times [0,\infty) \ni (\omega,\P,\tilde\omega,t) \longmapsto \cM^{\omega,\P,0}_t(\tilde\omega) \in \R,
		\end{equation*}
		such that for every $(\omega,\P) \in \widehat{\Omega}^C_s$, the process $\cM^{\omega,\P,0}$ is a right-continuous, $\F_\smallertext{+}$-optional, $(\F_\smallertext{+},\P)$-martingale bounded by $n+n^2$ satisfying
		\begin{equation*}
			\cM^{\omega,\P,0}_t = \E^\P\bigg[ \xi^{s,\omega} + \int_0^{(T\smallertext{-}s \land T)^{\smalltext{s}\smalltext{,}\smalltext{\omega}}} f^{s,\omega,\P}_r(0,0,0,\mathbf{0})\d(C^{s,\omega}_{s\smallertext{+}\smallertext{\cdot}})_r \bigg|\cF_{t\smallertext{+}}\bigg], \; \textnormal{$\P$--a.s.}, \; t \in [0,\infty).
		\end{equation*}
		Outside $\widehat{\Omega}^C_s$, we take $\cM^{\omega,\P,0}\equiv0$. For $(\omega,\P)\in \widehat{\Omega}^C_s$, we then let $\cY^{\omega,\P,0} : \Omega \times [0,\infty] \longrightarrow \R$ be defined by $\cY^{\omega,\P,0}_\infty \coloneqq \xi^{s,\omega}$ and
		\begin{equation*}
			\cY^{\omega,\P,0}_t \coloneqq \cM^{\omega,\P,0}_t - \int_0^{(T\smallertext{-}s\land T)^{\smalltext{s}\smalltext{,}\smalltext{\omega}}\land t} f^{s,\omega,\P}_r(0,0,0,\mathbf{0})\d(C^{s,\omega}_{s\smallertext{+}\smallertext{\cdot}})_r, \; t \in [0,\infty).
		\end{equation*}
		Note that the integral process above is c\`adl\`ag and $\F_\smallertext{+}$-adapted.
		Then the map
		\begin{equation*}
			\widehat{\Omega}^C_s \times \Omega \times [0,\infty] \ni (\omega,\P,\tilde\omega,t) \longmapsto \cY^{\omega,\P,0}_t(\tilde\omega) \in \R,
		\end{equation*}
		is Borel-measurable and the process $\cY^{\omega,\P,0}$ is right-continuous, $\P$--a.s. c\`adl\`ag, and $\F_\smallertext{+}$-optional. By \Cref{lem::measurable_decomposition}, there exists a Borel-measurable 
		\begin{equation*}
			\widehat{\Omega}^C_s \times \Omega \times [0,\infty) \ni (\omega,\P,\tilde\omega,t) \longmapsto \cZ^{\omega,\P,0}_t(\tilde\omega) \in \R^d
		\end{equation*}
		and a $\cB(\widehat{\Omega}^C_s)\otimes\widetilde{\cP}$-measurable function
		\[
			\widehat{\Omega}^C_s \times \Omega \times [0,\infty) \times \R^d \ni (\omega,\P,\tilde\omega, t, x) \longmapsto \cU^{\omega,\P,0}_t(\tilde\omega;x) \in \R
		\]
		such that $(\cZ^{\omega,\P,0},\cU^{\omega,\P,0}) \in \H^2(X^{c,\P};\F,\P) \times \H^2(\mu^X;\F,\P)$ with $\cU^{\omega,\P,0}_r(\tilde\omega;\cdot) \in \widehat{\L}^2_{\omega\otimes_\smalltext{s}\tilde\omega,s+r}(\mathsf{K}^{s,\omega,\P}_{\tilde\omega,r})$ for every $(\tilde\omega,r) \in \Omega \times [0,\infty)$, and
		\begin{equation*}
			\cN^{\omega,\P,0} \coloneqq \cM^{\omega,\P,0} - \cM^{\omega,\P,0}_0 - \bigg(\int_0^\cdot \cZ^{\omega,\P,0}_r \d X^{c,\P}_r\bigg)^{(\P)} -  \bigg(\int_0^\cdot\int_{\R^\smalltext{d}}\cU^{\omega,\P,0}_r(x)\tilde\mu^{X,\P}(\d r, \d x)\bigg)^{(\P)},
		\end{equation*}
		belongs to $\cH^{2,\perp}_0(X^{c,\P},\mu^X;\F_{\smallertext{+}},\P)$. Since $\cM^{\omega,\P,0} = \cM^{\omega,\P,0}_{\cdot \land (T-s\land T)^{\smalltext{s}\smalltext{,}\smalltext{\omega}}}$, $\P$--a.s., we choose the integrands such that $\cZ^{\omega,\P,0} = \cZ^{\omega,\P,0}\1_{\llparenthesis 0,(T-s\land T)^{\smalltext{s}\smalltext{,}\smalltext{\omega}}\rrbracket}$ and $\cU^{\omega,\P,0} = \cU^{\omega,\P,0}\1_{\llparenthesis 0,(T-s\land T)^{\smalltext{s}\smalltext{,}\smalltext{\omega}}\rrbracket}$ hold. This then yields
		\begin{align*}
			\cY^{\omega,\P,0}_t 
			&= \xi^{s,\omega} + \int_t^{(T\smallertext{-}s\land T)^{\smalltext{s}\smalltext{,}\smalltext{\omega}}} f^{s,\omega,\P}_r(0,0,0,\mathbf{0})\d(C^{s,\omega}_{s\smallertext{+}\smallertext{\cdot}})_r - \bigg(\int_t^{(T\smallertext{-}s\land T)^{\smalltext{s}\smalltext{,}\smalltext{\omega}}} \cZ^{\omega,\P,0}_r \d X^{c,\P}_r\bigg)^{(\P)} \\
			&\quad -  \bigg(\int_t^{(T\smallertext{-}s\land T)^{\smalltext{s}\smalltext{,}\smalltext{\omega}}}\int_{\R^\smalltext{d}}\cU^{\omega,\P,0}_r(x)\tilde\mu^{X,\P}(\d r, \d x)\bigg)^{(\P)} - \int_t^{(T\smallertext{-}s\land T)^{\smalltext{s}\smalltext{,}\smalltext{\omega}}}\d\cN^{\omega,\P,0}_r, \; t \in [0,\infty], \; \textnormal{$\P$--a.s.}
		\end{align*}
		This completes the first step of the Picard iteration, and
		by \Cref{lem::measurable_martingale_modification}, the map $(\omega,\P) \longmapsto \E^\P[\cY^{\omega,\P,0}_0]$	is Borel-measurable on $\widehat{\Omega}^C_s$.

		\medskip
		Since the restriction of $\cY^{\omega,\P,0}(\tilde\omega)$ to $(0,\infty)$ is only known to be c\`adl\`ag for $\P$--a.e. $\tilde\omega \in \Omega$, we define
		\[
			\cY^{\omega,\P,0}_{t\smallertext{-}} \coloneqq \widetilde{\cY}^{\omega,\P,0}_{t\smallertext{-}} \1_{\{\tilde{\cY}^{\smalltext{\omega}\smalltext{,}\smalltext{\P}\smalltext{,}\smalltext{0}}_{\smalltext{t}\smalltext{-}}\in\R\}} \1_{\{t > 0\}}, \; \textnormal{where} \; \widetilde{\cY}^{\omega,\P,0}_{t\smallertext{-}} \coloneqq \limsup_{n \rightarrow\infty} \cY^{\omega,\P,0}_{(t-1/n)\lor 0}, \; t \in [0,\infty).
		\]
		Then
		\begin{equation*}
			\widehat{\Omega}^C_s \times \Omega \times [0,\infty) \ni (\omega,\P,\tilde\omega,t) \longmapsto \cY^{\omega,\P,0}_{t\smallertext{-}}(\tilde\omega) \in \R,
		\end{equation*}
		is Borel-measurable, and, for fixed $(\omega,\P)\in \widehat{\Omega}^C_s$, the process $\cY^{\omega,\P,0}_{\smallertext{-}}$ is $\F$-predictable (see \cite[Theorem~IV.97.(b)]{dellacherie1978probabilities}) and coincides, outside some $\P$--null set, with the pathwise left limits of $\cY^{\omega,\P,0}$ on $(0,\infty)$.

		\medskip
		For the next Picard step, we replace $f^{s,\omega,\P}_r(0,0,0,\mathbf{0})$ by $f_r^{s,\omega,\P}(\cY^{\omega,\P,0}_{r\smallertext{-}},\cY_r^{\omega,\P,0},\cZ^{\omega,\P,0}_r,\cU_r^{\omega,\P,0})$ in the arguments above. This yields $(\cY^{\omega,\P,1},\cZ^{\omega,\P,1},\cU^{\omega,\P,1})$ having the same measurability and path-regularity properties as $(\cY^{\omega,\P,0},\cZ^{\omega,\P,0},\cU^{\omega,\P,0})$.
		We then iteratively obtain, for every $m \in \N^\star$, a triple $(\cY^{\omega,\P, m},\cZ^{\omega,\P, m},\cU^{\omega,\P, m})$ 
		such that $(\omega,\P) \longmapsto \E^\P[\cY^{\omega,\P, m}_0]$	is Borel-measurable on $\widehat{\Omega}^C_s$. This yields Borel-measurability of the map $\Y : \widehat{\Omega}^C_s \longrightarrow [-\infty,\infty]$ defined by
		\begin{equation}\label{eq::picard_limit_value_map}
			\Y(\omega,\P) \coloneqq \limsup_{m \rightarrow \infty} \E^\P[\cY^{\omega,\P, m}_0].
		\end{equation}
		In particular, the restriction of $\Y$ to $\{(\omega,\P) : \omega \in \Omega, \; \P \in \fP(s,\omega)\} \subseteq \Omega \times \fP(\Omega)$ is upper semi-analytic.
		The proof of \cite[Theorem 3.7]{possamai2024reflections}, which is based on Banach's fixed-point theorem, yields
		\[
			\E^\P\big[\cY^{s,\omega,\P}_0((T-s\land T)^{s,\omega},\xi^{s,\omega})\big]
			= \lim_{m \to \infty}\E^\P[\cY^{\omega,\P,m}_0] = \Y(\omega,\P), \; \omega\in\Omega, \; \P\in\fP(s,\omega),
		\]
		which concludes the proof.
	\end{proof}
	
	We need one last fact about the conditioning of BSDE solutions before proceeding to the proof of \Cref{thm::measurability2}.

	\begin{lemma}\label{lem::conditioning_bsde2}
		Suppose that {\rm\Cref{ass::probabilities2}} and {\rm\ref{ass::generator2}} hold. Let $0 \leq s \leq t < \infty$, $\bar{\omega} \in \Omega$, and $\P \in \fP(s,\bar\omega)$. Then there exists a $\P$--null set $\sN \in \cF_{t-s}$ such that
		\begin{align}\label{eq::conditioning_solution_bsde}
			\E^\P\big[\cY^{s,\bar\omega,\P}_{t-s}((T-s \land T)^{s,\bar\omega},\xi^{s,\bar\omega}) \big| \cF_{t\smallertext{-}s}\big](\omega)
			= \E^{\P^{\smalltext{t}\smalltext{-}\smalltext{s}\smalltext{,}\smalltext{\omega}}}\big[\cY^{t,\bar\omega\otimes_\smalltext{s}\omega,\P^{\smalltext{t}\smalltext{-}\smalltext{s}\smalltext{,}\smalltext{\omega}}}_0 ((T -t\land T)^{t,\bar\omega\otimes_\smalltext{s}\omega},\xi^{t,\bar\omega\otimes_\smalltext{s}\omega})\big],\; \omega \in \Omega\setminus\sN.
		\end{align}
	\end{lemma}
	
	\begin{proof}[Proof of \Cref{thm::measurability2}]
		Let $\textnormal{proj}_\Omega :\Omega \times \fP(\Omega) \ni (\omega,\P) \longmapsto \omega \in \Omega$. For every $\lambda \in \R$, we have
		\begin{align*}
			\big\{\omega \in \Omega : \widehat{\cY}_s(T,\xi)(\omega) > \lambda\big\} 
			&= \textnormal{proj}_\Omega\big(\big\{(\omega,\P) : \omega \in \Omega, \, \P \in \fP(s,\omega),\, \E^\P\big[\cY^{s,\omega,\P}_0((T-s\land T)^{s,\omega},\xi^{s,\omega})\big] > \lambda\big\}\big).
		\end{align*} 
		Since projections of analytic sets are analytic (due to the continuity of the projection operator), the upper semi-analyticity of $\widehat{\cY}_s(T,\xi)$ follows from \Cref{prop::Borel-measurability_value_function}. Furthermore, it then follows from \cite[Corollary 8.4.3]{cohn2013measure} that $\widehat{\cY}_s(T,\xi)$ is $\cF$--universally measurable. Consider the map $\iota_s : \Omega \ni \omega \longmapsto X_{\cdot\land s}(\omega) \in \Omega$, which satisfies $\iota_s(\omega) = X_{\cdot\land s}(\omega) = \omega_{\cdot \land s}$ and is $(\cF_s,\cF)$-measurable by \cite[Proposition 2.3.10]{weizsaecker1990stochastic}. Then, $\iota_s$ is also $(\cF_s,\cF)$--universally measurable by \cite[Lemma 8.4.6]{cohn2013measure}. For $\eta \coloneqq \widehat{\cY}_s(T,\xi)$, we have $\eta(\iota_s(\omega)) = \eta(\omega)$ for all $\omega \in \Omega$ by definition. Therefore, we conclude that $\widehat{\cY}_s(T,\xi)$ is $\cF_s$--universally measurable. To see that $\widehat{\cY}_{s \land T(\omega)}(T,\xi)(\omega) = \widehat{\cY}_{s}(T,\xi)(\omega)$, note that for $T(\omega) \leq s \in [0,\infty)$, we have $T(\omega\otimes_s\cdot) \equiv T(\omega)$ and $\xi(\omega\otimes_s\cdot) \equiv \xi(\omega)$ by Galmarino's test since $\xi$ is $\cF_T$-measurable (see \cite[Theorem IV.100.(a)]{dellacherie1978probabilities}), and therefore $(T - s \land T)^{s,\omega} \equiv 0$, which yields $\widehat{\cY}_{s \land T(\omega)}(T,\xi)(\omega) = \xi(\omega) = \widehat{\cY}_{s}(T,\xi)(\omega)$.
		
		\medskip
		We turn to \eqref{eq::dynamic_programming_principle2}, and follow closely the proof of \cite[Theorem 2.3]{nutz2013constructing}. Fix $\P \in \fP_0$ and $\varepsilon \in (0,\infty)$. By \cite[Proposition 7.50]{bertsekas1978stochastic}, there exists an analytically measurable\footnote{The analytic $\sigma$-algebra is generated by all analytic sets and is therefore contained in the universal completion.} map $\Q^\prime : \Omega \longrightarrow \fP(\Omega)$ such that 
		\begin{equation}\label{eq::analytic_selection2}
			\Q^\prime(\omega) \in \fP(s,\omega),
			\;
			\E^{\Q^\smalltext{\prime}(\omega)} \big[\cY^{s,\omega,{\Q^\smalltext{\prime}(\omega)}}_0((T-s\land T)^{s,\omega},\xi^{s,\omega})\big] 
			\geq \big(\widehat\cY_s(T,\xi)(\omega) - \varepsilon\big) \1_{\{\hat\cY_\smalltext{s}(T,\xi) < \infty\}}(\omega) + \frac{1}{\varepsilon}\1_{\{\hat\cY_\smalltext{s}(T,\xi) = \infty\}}(\omega),
		\end{equation}
		for each $\omega \in \Omega$ for which $\fP(s,\omega) \neq \varnothing$. The map $\widetilde\Q : \Omega \longrightarrow \fP(\Omega)$ given by $\widetilde\Q(\omega) \coloneqq \Q^\prime({\omega_{\cdot \land s}})$ is $\cF^\ast_s$-measurable and also satisfies \eqref{eq::analytic_selection2} for each $\omega \in \Omega$ with $\fP(s,\omega)\neq \varnothing$. We now choose an $\cF_s$-measurable map $\Q : \Omega \longrightarrow \fP(\Omega)$ satisfying $\Q(\cdot) = \widetilde\Q(\cdot)$, $\P$--a.s., see \cite[Lemma 1.27]{kallenberg2021foundations}.\footnote{Polish spaces are Borel spaces by \cite[Theorem 1.8]{kallenberg2021foundations}.} Since $\fP(s,\omega) \neq \varnothing$ for $\P$--a.e. $\omega \in\Omega$ by \Cref{ass::probabilities2}.$(ii)$, we have that $\Q(\omega)$ satisfies \eqref{eq::analytic_selection2} for $\P$--a.e. $\omega \in \Omega$. Let
		\begin{equation*}
			\overline{\P}[A] \coloneqq  \iint \big(\1_A\big)^{s,\omega}(\omega^\prime)\Q(\omega;\d\omega^\prime)\P(\d\omega), \; A \in \cF.
		\end{equation*}
		Then, $\overline\P \in \fP_0$ by \Cref{ass::probabilities2}.$(iii)$, $\overline{\P} = \P$ on $\cF_s$ and $\Q(\omega) = \bar\P^{s,\omega}$ for $\P$--a.e. $\omega \in \Omega$ (see \Cref{rem::measures_in_fP}). Using \Cref{lem::conditioning_bsde2}, we obtain
		\begin{equation*}
			\E^{\bar\P}\big[\cY^{\bar\P}_s(T,\xi)\big|\cF_s\big](\omega) = \E^{\Q(\omega)}\big[\cY^{s,\omega,\Q(\omega)}_0((T-s\land T)^{s,\omega},\xi^{s,\omega})\big] 
			\geq 
			\big(\widehat\cY_s(T,\xi)(\omega) - \varepsilon\big) \1_{\{\hat\cY_\smalltext{s}(T,\xi) < \infty\}}(\omega) + \frac{1}{\varepsilon}\1_{\{\hat\cY_\smalltext{s}(T,\xi) = \infty\}}(\omega),
		\end{equation*}
		for $\P$--a.e. $\omega \in \Omega$. This then yields
		\begin{equation*}
			\underset{\P^\prime \in \fP_\smalltext{0}(\cF_{s},\P)}{{\esssup}^\P} \E^{\P^\smalltext{\prime}} \big[ \cY^{\P^\smalltext{\prime}}_s(T,\xi)\big| \cF_s\big] 
			\geq 
			\big(\widehat\cY_s(T,\xi) - \varepsilon\big) \1_{\{\hat\cY_\smalltext{s}(T,\xi) < \infty\}} + \frac{1}{\varepsilon}\1_{\{\hat\cY_\smalltext{s}(T,\xi) = \infty\}}, \; \text{$\P$--a.s.},
		\end{equation*}
		and since $\varepsilon\in(0,\infty)$ is arbitrary, this yields one inequality.
		
		\medskip
		The converse inequality can be established as follows. Fix $\P \in \fP_0$, and let $\tilde{\P} \in \fP_0(\cF_s, \P)$. From \Cref{ass::probabilities2}.$(ii)$ and \Cref{lem::conditioning_bsde2}, we obtain
		\begin{equation*}
			\E^{\tilde\P}\big[\cY^{\tilde\P}_s(T,\xi)\big| \cF_s \big](\omega) = \E^{\tilde\P^{\smalltext{s}\smalltext{,}\smalltext{\omega}}}\big[\cY^{s,\omega,\tilde\P^{\smalltext{s}\smalltext{,}\smalltext{\omega}}}_0((T-s\land T)^{s,\omega},\xi^{s,\omega})\big] \leq \widehat\cY_{s}(T,\xi)(\omega), \; \text{for $\tilde\P$--a.e. $\omega \in \Omega$.}
		\end{equation*}
		Note that the left-hand side is $\cF_s$-measurable, while the right-hand side is $\cF_s$--universally measurable. Thus, since $\P = \tilde{\P}$ on $\cF_s$, the inequality above also holds for $\P$--a.e. $\omega \in \Omega$. It remains to take the essential supremum over $\tilde\P \in \fP_0(\cF_s,\P)$ under $\P$ on the left-hand side, which completes the proof.
	\end{proof}

	\subsection{Path-regularisation of the value function}\label{sec::proof_regularisation}
	
	In this subsection, we prove \Cref{thm::down-crossing}. The proofs of all auxiliary lemmata deferred from this subsection are provided in \Cref{sec::proofs_lemmata_regularisation}.
		
	\begin{lemma}\label{lem::solv_bsde_cond}
		Suppose that {\rm Assumptions \ref{ass::probabilities2}} and {\rm \ref{ass::generator2}} hold. Let $s \in [0,\infty)$, $\omega \in \Omega$, $\P\in\fP(s,\omega)$, and $\tau$ be an $\F$--stopping time satisfying $\tau \geq s$. Then
		\begin{align*}
			&\cY^{s,\omega,\P}_{\cdot\land (\tau \land T\smallertext{-}s \land T)^{\smalltext{s}\smalltext{,}\smalltext{\omega}}}((T-s\land T)^{s,\omega},\xi^{s,\omega})
			= \cY^{s,\omega,\P}\big((\tau \land T-s \land T)^{s,\omega},\cY^{s,\omega,\P}_{(\tau\land T\smallertext{-}s\land T)^{s,\omega}}((T-s\land T)^{s,\omega},\xi^{s,\omega})\big), \; \textnormal{$\P$--a.s.}
		\end{align*}
		If, in addition, {\rm\Cref{ass::crossing}}.$(i)$--$(iii)$ hold, then
		\begin{align*}
			&\cY^{s,\omega,\P}\big((\tau \land T-s \land T)^{s,\omega},\cY^{s,\omega,\P}_{(\tau\land T\smallertext{-}s\land T)^{s,\omega}}((T-s\land T)^{s,\omega},\xi^{s,\omega})\big) \\
			&= \cY^{s,\omega,\P}\Big((\tau \land T-s \land T)^{s,\omega},\E^\P\big[\cY^{s,\omega,\P}_{(\tau\land T\smallertext{-}s\land T)^{s,\omega}}((T-s\land T)^{s,\omega},\xi^{s,\omega})\big|\cF_{(\tau\land T\smallertext{-}s\land T)^{s,\omega}}\big]\Big)
		\end{align*}
		on $[0,(\tau\land T-s\land T)^{s,\omega})$ outside some \textnormal{$\P$--null set}.
		Furthermore, for $s = 0$, the same assertion holds if \textnormal{\Cref{ass::crossing}}.$(i)$--$(iii)$ is satisfied only at $s = 0$.
	\end{lemma}

	\begin{lemma}\label{lem::stopping_value_function}
		Suppose that {\rm Assumptions \ref{ass::probabilities2}}, {\rm\ref{ass::generator2}}, and {\rm\ref{ass::crossing}} hold. Let $0 \leq s \leq t < \infty$, and let $\bar\omega \in \Omega$. Then $ \widehat\cY_{t}(T,\xi)^{s,\bar\omega} = \widehat\cY_{t \land T^{\smalltext{s}\smalltext{,}\smalltext{\bar\omega}}}(T,\xi)^{s,\bar\omega}$ is $\cF^\ast_{(t \land T\smallertext{-}s\land T)^{s,\bar\omega}}$-measurable,
		\begin{gather}\label{eq::integrability_y_hat_tau}
			\sup_{\P \in \fP(s,\bar\omega)}\E^\P\Big[\big|\cE\big(\hat\beta A^{s,\bar\omega}_{s\smallertext{+}\smallertext{\cdot}})\big)^{1/2}_{(t \land T\smallertext{-}s\land T)^{s,\bar\omega}} \widehat\cY_{t \land T^{\smalltext{s}\smalltext{,}\smalltext{\bar\omega}}}(T,\xi)^{s,\bar\omega}\big|^2\Big] < \infty,\\
	\label{eq::dpp_y_hat_sup}
			\widehat\cY_s(T,\xi)(\bar\omega) =  \sup_{\P\in\fP(s,\bar\omega)}\E^\P\Big[\cY^{s,\bar\omega,\P}_0\big((t\land T-s\land T)^{s,\bar\omega},\widehat{\cY}_{t \land T^{\smalltext{s}\smalltext{,}\smalltext{\bar\omega}}} (T,\xi)^{s,\bar\omega}\big)\Big].
		\end{gather}
		\end{lemma}

	\begin{lemma}\label{lem::linearising_bsde}
		Suppose that  {\rm\Cref{ass::generator2}} holds, and that {\rm\Cref{ass::crossing}}.$(ii)$ holds for $s = 0$. Let $\P \in \fP_0$, and let $\eta \in \H^2_T(X^{c,\P};\F,\P)$ and $\rho \in \H^2_T(\mu^X;\F,\P)$ with $\Delta(\rho\ast\tilde\mu^{X,\P}) > - 1$, \textnormal{$\P$--a.s.}, be such that
		\begin{equation}\label{eq::integrability_lipschitz_eta_rho}
			\frac{\d\langle\eta\bcdot X^{c,\P}\rangle^{(\P)}}{\d C} \leq \theta^{X}, \; \text{\rm and} \; \frac{\d\langle\rho\ast\tilde\mu^{X,\P} \rangle^{(\P)}}{\d C} \leq \theta^{\mu}, \; \textnormal{$\P \otimes \mathrm{d}C$--a.e.}
		\end{equation}
		Let $\zeta \in \L^2(\cF_T,\P)$, and let $(\cY,\cZ,\cU,\cN)$ and $(\sY,\sZ,\sU,\sN)$ be the solutions to the {\rm BSDEs}
		\begin{align}
			\cY_t = \zeta &+ \int_t^T \bigg( \frac{\d\langle \rho \ast\tilde\mu^{X,\P}, \cU \ast\tilde\mu^{X,\P}\rangle^{(\P)}_s}{\d C_s}  
			+ \eta^\top_s \mathsf{a}_s \cZ_s\bigg) \d C_s \nonumber\\
			&\quad - \bigg(\int_t^T \cZ_s \d X^{c,\P}_s\bigg)^{(\P)} - \bigg(\int_t^T\int_{\R^\smalltext{d}} \cU_s(x)\tilde\mu^{X,\P}(\d s, \d x)\bigg)^{(\P)} - \int_t^T \d \cN_s, \; t \in [0,\infty], \; \textnormal{$\P$--a.s.}, \label{eq::linear_bsde} \\
			\sY_t = \zeta &+ \int_t^T \bigg( \frac{\d\langle \rho \ast\tilde\mu^{X,\P}, \sU \ast\tilde\mu^{X,\P}\rangle^{(\P)}_s}{\d C_s} - \sqrt{\theta^{X}_s}\|\mathsf{a}_s^{1/2} \sZ_s\|   \bigg)  \d C_s \nonumber\\
			&\quad - \bigg(\int_t^T \sZ_s \d X^{c,\P}_s\bigg)^{(\P)} - \bigg(\int_t^T\int_{\R^\smalltext{d}} \sU_s(x)\tilde\mu^{X,\P}(\d s, \d x)\bigg)^{(\P)} - \int_t^T \d \sN_s, \; t \in [0,\infty], \; \textnormal{$\P$--a.s.}, \label{eq::lipschitz_linear_bsde}
		\end{align}
		respectively, in the sense of {\rm\cite[Theorem 3.7]{possamai2024reflections}}. Let\footnote{Recall that $A^\oplus$ denotes the Moore--Penrose pseudo-inverse of a matrix $A$.}
		\begin{equation*}
			\eta^\prime_s \coloneqq -\sqrt{\theta^{X}_s} (\mathsf{a}^{1/2}_s)^{\oplus}  
			\frac{\mathsf{a}^{1/2}_s \sZ_s}{\big\|\mathsf{a}^{1/2}_s \sZ_s\big\|}\1_{\R^\smalltext{d}\setminus\{0\}}(\mathsf{a}^{1/2}_s\sZ_s), \; s \in [0,\infty).
		\end{equation*}
		Then $\eta^\prime \in \H^2_T(X^{c,\P};\F,\P)$, $\langle\eta^\prime\bcdot X^{c,\P}\rangle^{(\P)}$ is $\P$--essentially bounded, the random variables
		\begin{equation*}
			\frac{\d\mathscr{Q}}{\d\P} \coloneqq \cE\Big( \eta^\prime \bcdot X^{c,\P} + \rho\ast\tilde\mu^{X,\P}\Big)_T,
			\;
			\frac{\d\mathcal{Q}}{\d\P} \coloneqq \cE\Big( \eta \bcdot X^{c,\P} + \rho\ast\tilde\mu^{X,\P}\Big)_T,
		\end{equation*}
		are both well-defined, strictly positive, and $\P$--square-integrable probability densities relative to $\P$, and
		\begin{equation*}
			\E^{\mathscr{Q}}[\zeta|\cF_{t\smallertext{+}}] = \sY_t \leq \cY_t = \E^{\mathcal{Q}}[\zeta | \cF_{t\smallertext{+}}], \; \textnormal{$\P$--a.s.}, \; t \in [0,\infty].
		\end{equation*}
	\end{lemma}

\begin{remark}
Since both generators in \textnormal{\Cref{lem::linearising_bsde}} are independent of the $(y,\mathrm{y})$-arguments, we may take $\alpha^2\coloneqq\max\{\theta^X,\theta^\mu\}$. Consequently, \eqref{eq::integrability_lipschitz_eta_rho} and \textnormal{\Cref{ass::crossing}.$(ii)$} imply that $A_T$ is $\P$--essentially bounded. Hence the corresponding weighted and unweighted norms are equivalent. This implies that every $\zeta\in\L^2(\cF_T,\P)$ yields unique solutions to \eqref{eq::linear_bsde} and \eqref{eq::lipschitz_linear_bsde}$;$ see \textnormal{\cite[Section 2.4]{possamai2024reflections}}.
\end{remark}
	
\begin{proof}[Proof of \Cref{thm::down-crossing}.$(i)$]
	In this part, we prove the existence of $\widehat\cY^\smallertext{+}(T,\xi)$ satisfying the claimed properties, except for its representation \eqref{eq::aggregation}, which will be obtained after \Cref{rem::gap_regularisation}. We show the following: the set
	\begin{equation}\label{eq::set_of_crossings}
			\Omega_0 \coloneqq \bigg\{\sup_{t \in \D_\tinytext{+} \cap [0,K]}|\widehat{\cY}_t(T,\xi)| < \infty \; \textnormal{and} \; D_a^b(\widehat{\cY}(T,\xi); \D_\smallertext{+}\cap[0,K]) < \infty \; \textnormal{for all $(a,b) \in \D^2$ with $a < b$, and $K \in \N$}\bigg\},
	\end{equation}
	which belongs to $\sigma(\cup_{0\leq t < \infty} \cF^\ast_t)$, satisfies $\P[\Omega_0] = 1$ for all $\P\in\fP_0$. Here, $\D$ denotes all dyadic numbers $k2^{-n}$ for integer $k$ and nonnegative integer $n$, $\D_\smallertext{+} = \D \cap [0,\infty)$, and $D_a^b(\widehat{\cY}(T,\xi); \D_\smallertext{+}\cap[0,K])$ denotes the number of down-crossings of $[a,b]$ by $\widehat{\cY}(T,\xi)$ on $\D_\smallertext{+}\cap[0,K]$. Note that the number of down-crossings and up-crossings can differ by at most one; therefore, they are either both finite or both infinite. We then let
	\[
		\textstyle \tau^{b}_{a} \coloneqq \inf\big\{r \in \D_\smallertext{+} : D^b_a(\widehat{\cY}(T,\xi); \D_\smallertext{+} \cap [0,r]) = \infty \big\}, 
		\; 
		\sigma \coloneqq \inf\big\{r \in \D_\smallertext{+} : \sup_{\{s\in\D_{\tinytext{+}}: s \in[0, r]\}} |\widehat{\cY}_s(T,\xi)| = \infty\big\},
	\]
	and then
	\[
		\rho \coloneqq \sigma \land \inf_{\{(a,b) \in \D^\smalltext{2} :a < b\}} \tau^b_a,
	\]
	which is an $\F^\ast_\smallertext{+}$--stopping time.
	Since $\P[\Omega_0] = 1$ for all $\P\in\fP_0$, $\rho = \infty$ holds $\fP_0$--quasi-surely. Therefore, defining $\widehat{\cY}^\smallertext{+}(T,\xi) = (\widehat{\cY}^\smallertext{+}_t(T,\xi))_{t \in [0,\infty]}$ as
	\[
		\widehat{\cY}^\smallertext{+}_t(T,\xi) \coloneqq \bigg(\limsup_{\D_\tinytext{+} \ni s \downarrow\downarrow t} \widehat{\cY}_s(T,\xi)\bigg)\1_{\{t < \rho\}} \1_{\{t < T\}} + \xi \1_{\{T \leq t\}} = \lim_{\D_\tinytext{+} \ni s \downarrow\downarrow t} \widehat{\cY}_s(T,\xi)\1_{\{t < \rho\}} \1_{\{t < T\}} + \xi \1_{\{T \leq t\}},
	\]
	yields a real-valued, right-continuous and $\fP_0$--q.s. c\`adl\`ag, $\F^\ast_\smallertext{+}$-optional process on $[0,\infty)$ (see \cite[Lemma 3.16]{legall2016brownian} and \cite[Remark VI.5.(a), pages 70--71]{dellacherie1982probabilities}); here $\F^\ast = (\cF^\ast_t)_{t \in [0,\infty)}$. We can thus now redefine $\widehat{\cY}^\smallertext{+}_t(T,\xi)$ to be zero on $(\Omega\setminus\Omega_0)\cap\{t < T\}$ for every $t \in [0,\infty]$, which then yields a c\`adl\`ag and $\G_\smallertext{+}$-adapted process on $[0,\infty]$ meeting the requirements described in $(i)$.
	
	\medskip
	We now prove that $\P[\Omega_0] = 1$ for all $\P\in\fP_0$, where we adapt the proof of \cite[Lemma 4.8]{soner2013dual} (and \cite[Theorem 6]{chen2000general}). Fix $\P \in\fP_0$ and decompose 
	\[
		\widehat\cY_t(T,\xi) \eqqcolon V_t + W_t, \; \textnormal{where} \; W_t \coloneqq \E^\P[\cY^\P_t(T,\xi)|\cF_t], \; t \in [0,\infty).
	\]
	We show that $V = (V_t)_{t \in [0,\infty)}$ satisfies $\P[\Omega_1] = 1$,  where $\Omega_1 \in \sigma(\cup_{0\leq t < \infty} \cF^\ast_t)$ is defined analogously to \eqref{eq::set_of_crossings} but for $V$ instead of $\widehat{\cY}$. Since the process $W = (W_t)_{t \in [0,\infty)}$ is merely a difference of two nonnegative $(\F,\P)$--super-martingales, it follows that $\Omega_2 \in \sigma(\cup_{0\leq t < \infty} \cF_t)$, defined analogously to \eqref{eq::set_of_crossings} but for $W$, satisfies $\P[\Omega_2] = 1$; see, for example, the proof of \cite[Theorem 3.17]{legall2016brownian} or \cite[Theorem VI.2]{dellacherie1982probabilities}. Then, it is straightforward to check that $\Omega_1\cap\Omega_2 \subseteq \Omega_0$, which then yields $\P[\Omega_0] = 1$. 
	It therefore remains to prove that $\P[\Omega_1] = 1$. 
	
	\medskip
	To simplify the notation, we omit the reference to $(T,\xi)$ and $\P$ whenever it does not cause confusion. We start with the boundedness of the paths. By \Cref{thm::measurability2} and \Cref{prop::stability}, there exists a constant $\mathfrak{C}^\prime \in (0,\infty)$ depending only on $\hat\beta$ and $\Phi$ such that for every $s \in [0,\infty)$,
	\begin{align*}
		|V_s|^2 \leq 2 |\widehat\cY_s|^2 + 2 \E^\P[|\cY^\P_s|^2|\cF_s]
		&\leq 4 \underset{\bar{\P} \in \fP_\smalltext{0}(\cF_{s},\P)}{{\esssup}^\P}\E^{\bar{\P}}\big[|\cY^{\bar{\P}}_s|^2\big|\cF_s\big] \\
		&\leq \mathfrak{C}^\prime \underset{\bar{\P} \in \fP_\smalltext{0}(\cF_{s},\P)}{{\esssup}^\P}\E^{\bar{\P}}\bigg[ \frac{\cE(\hat\beta A)_T}{\cE(\hat\beta A)_{s\land T}} |\xi|^2 +\int_s^T \frac{\cE(\hat\beta A)_r}{\cE(\hat\beta A)_s} \frac{|f^{\bar\P}_r(0,0,0,\mathbf{0})|^2}{\alpha^2_r} \d C_r \bigg| \cF_{s}\bigg]	, \; \text{$\P$--a.s.}
	\end{align*}
	We then obtain from \eqref{eq::constant_phi} that
	\[
		\bigg(\sup_{s \in \D_\tinytext{+}} |V_s|\bigg)^2 = \sup_{s \in \D_\tinytext{+}} |V_s|^2 
		\leq \mathfrak{C}^\prime \sup_{s \in \D_\tinytext{+}} \underset{\bar{\P} \in \fP_\smalltext{0}(\cF_{s},\P)}{{\esssup}^\P}\E^{\bar{\P}}\bigg[ \frac{\cE(\hat\beta A)_T}{\cE(\hat\beta A)_{s\land T}} |\xi|^2 +\int_s^T \frac{\cE(\hat\beta A)_r}{\cE(\hat\beta A)_s} \frac{|f^{\bar\P}_r(0,0,0,\mathbf{0})|^2}{\alpha^2_r} \d C_r \bigg| \cF_{s}\bigg]  < \infty, \; \textnormal{$\P$--a.s.}
	\]
	
	\medskip
	We now turn to the down-crossings. We choose, for each $t \in \D_\smallertext{+}$, a $\P$-modification of $V_t$ that is real-valued, nonnegative, and $\cF_{t\land T}$-measurable. This is possible since $\widehat{\cY}_t = \widehat{\cY}_{t\land T}$ is $\cF^\ast_{t\land T}$-measurable (see the arguments in the proof of \cite[Lemma~2.5]{nutz2013constructing}) and $\cY^\P_{t\land T}$ is $\cF_T$-measurable (see \cite[Theorem~IV.56.(c)]{dellacherie1978probabilities}). 
	For each $(n,i) \in \N^2$, let $t^n_i = i 2^{-n}$, and let $(\cY^{n,i},\cZ^{n,i},\cU^{n,i},\cN^{n,i})$ be the solution to the following well-posed BSDE relative to $\P$ and with time horizon $t^n_i \land T$:
	\begin{align*}
		\cY^{n,i}_t 
		= \widehat\cY_{t^\smalltext{n}_\smalltext{i}} 
		&\quad + \int_t^{t^\smalltext{n}_\smalltext{i} \land T} f^{\P}_s\big(\cY^{n,i}_s,\cY^{n,i}_{s\smallertext{-}}, \cZ^{n,i}_s, \cU^{n,i}_s(\cdot)\big) \d C_s  \\
		&- \int_t^{t^\smalltext{n}_\smalltext{i} \land T} \cZ^{n,i}_s\d X^{c}_s 
		- \int_t^{t^\smalltext{n}_\smalltext{i} \land T}\int_{\R^\smalltext{d}} \cU^{n,i}_s(x)\tilde\mu^{X}(\d s, \d x) - \int_t^{t^\smalltext{n}_\smalltext{i} \land T}\d \cN^{n,i}_s, \; t \in [0,\infty].
	\end{align*}
	Recall that $\widehat\cY_{t^\smalltext{n}_\smalltext{i}} = \widehat\cY_{t^\smalltext{n}_\smalltext{i} \land T}$ (identically) satisfies the integrability condition necessary for the well-posedness of this BSDE (see \Cref{lem::stopping_value_function}).
	We then define $(\widetilde\cY^{n,i}, \widetilde\cZ^{n,i}, \widetilde\cU^{n,i}, \widetilde\cN^{n,i})$ by
	\begin{gather*}
		\widetilde\cY^{n,i} \coloneqq \cY^{n,i} - \cY^\P(t^n_i \land T,\E^\P[\cY^\P_{t^\smalltext{n}_\smalltext{i} \land T}|\cF_{t^\smalltext{n}_\smalltext{i} \land T}]), 
		\;
		\widetilde\cZ^{n,i} \coloneqq \cZ^{n,i} - \cZ^\P(t^n_i \land T,\E^\P[\cY^\P_{t^\smalltext{n}_\smalltext{i} \land T}|\cF_{t^\smalltext{n}_\smalltext{i} \land T}]),
		\\
		\widetilde\cU^{n,i} \coloneqq \cU^{n,i} - \cU^\P(t^n_i \land T,\E^\P[\cY^\P_{t^\smalltext{n}_\smalltext{i} \land T}|\cF_{t^\smalltext{n}_\smalltext{i} \land T}]),
		\;
		\widetilde\cN^{n,i} \coloneqq \cN^{n,i} - \cN^\P(t^n_i \land T,\E^\P[\cY^\P_{t^\smalltext{n}_\smalltext{i} \land T}|\cF_{t^\smalltext{n}_\smalltext{i} \land T}]).
	\end{gather*}
	Then, since $\cY^\P_{t^\smalltext{n}_\smalltext{i}} = \cY^\P_{t^\smalltext{n}_\smalltext{i} \land T}$ is $\cF_T$-measurable (see \cite[Lemma 2.2.4]{weizsaecker1990stochastic}), we find
	\[
		\E^\P[\cY^\P_{t^\smalltext{n}_\smalltext{i} \land T}|\cF_{t^\smalltext{n}_\smalltext{i}\land T}] = \E^\P\big[ \E^\P[ \cY^\P_{t^\smalltext{n}_\smalltext{i} \land T} \big| \cF_T] \big| \cF_{t^\smalltext{n}_\smalltext{i}} \big] = \E^\P\big[ \cY^\P_{t^\smalltext{n}_\smalltext{i} \land T} | \cF_{t^\smalltext{n}_\smalltext{i}}] = \E^\P\big[ \cY^\P_{t^\smalltext{n}_\smalltext{i}} | \cF_{t^\smalltext{n}_\smalltext{i}}], \; \textnormal{$\P$--a.s.},
	\]
	and thus obtain
	\begin{align*}
		\widetilde\cY^{n,i}_t = V_{t^\smalltext{n}_\smalltext{i}} 
		&+ \int_t^{t^\smalltext{n}_\smalltext{i} \land T} F^{n,i}_s\big(\widetilde\cY^{n,i}_s,\widetilde\cY^{n,i}_{s\smallertext{-}}, \widetilde\cZ^{n,i}_s, \widetilde\cU^{n,i}_s(\cdot)\big) \d C_s \\
		&\quad - \int_t^{t^\smalltext{n}_\smalltext{i} \land T} \widetilde\cZ^{n,i}_s\d X^{c}_s- \int_t^{t^\smalltext{n}_\smalltext{i} \land T}\int_{\R^\smalltext{d}} \widetilde\cU^{n,i}_s(x)\tilde\mu^{X}(\d s, \d x) - \int_t^{t^\smalltext{n}_\smalltext{i} \land T}\d \widetilde\cN^{n,i}_s, \; t \in [0,\infty], \; \textnormal{$\P$--a.s.,}
	\end{align*}
	where the generator
	\begin{align*}
		&F^{n,i}_s\big(\omega, y,\mathrm{y}, z, u_s(\omega;\cdot)\big) \\
		&\coloneqq f^{\P}_s\big(y + \cY^\P_s(t^n_i \land T,\E[\cY^\P_{t^\smalltext{n}_\smalltext{i} \land T}|\cF_{t^\smalltext{n}_\smalltext{i} \land T}])(\omega),\mathrm{y} + \cY^\P_{s\smallertext{-}}(t^n_i \land T,\E[\cY^\P_{t^\smalltext{n}_\smalltext{i} \land T}|\cF_{t^\smalltext{n}_\smalltext{i} \land T}])(\omega), z + \cZ^\P_s(t^n_i \land T,\E[\cY^\P_{t^\smalltext{n}_\smalltext{i} \land T}|\cF_{t^\smalltext{n}_\smalltext{i} \land T}])(\omega), \\
		&\qquad\quad u_s(\omega;\cdot) + \cU^\P_s(t^n_i \land T,\E[\cY^\P_{t^\smalltext{n}_\smalltext{i} \land T}|\cF_{t^\smalltext{n}_\smalltext{i} \land T}])(\omega;\cdot)\big) \\
		&\quad 
		- f^{\P}_s\big(\cY^\P_s(t^n_i \land T,\E[\cY^\P_{t^\smalltext{n}_\smalltext{i} \land T}|\cF_{t^\smalltext{n}_\smalltext{i} \land T}])(\omega),\cY^\P_{s\smallertext{-}}(t^n_i \land T,\E[\cY^\P_{t^\smalltext{n}_\smalltext{i} \land T}|\cF_{t^\smalltext{n}_\smalltext{i} \land T}])(\omega),\cZ^\P_s(t^n_i \land T,\E[\cY^\P_{t^\smalltext{n}_\smalltext{i} \land T}|\cF_{t^n_i \land T}])(\omega), \\
		&\qquad\quad\cU^\P_s(t^n_i \land T,\E[\cY^\P_{t^\smalltext{n}_\smalltext{i} \land T}|\cF_{t^\smalltext{n}_\smalltext{i} \land T}])(\omega;\cdot)\big),
	\end{align*}
	satisfies $F^{n,i}_s(0,0,0,\mathbf{0}) = 0$.
	From \Cref{lem::solv_bsde_cond}, it follows that, for $i \in \N^\ast$,
	\[
		\cY^\P_{t^\smalltext{n}_{\smalltext{i}\smalltext{-}\smalltext{1}}}(t^n_i \land T,\E^\P[\cY^\P_{t^\smalltext{n}_\smalltext{i} \land T}|\cF_{t^\smalltext{n}_\smalltext{i} \land T}])
		= \cY^\P_{t^\smalltext{n}_{\smalltext{i}\smalltext{-}\smalltext{1}}}, \; \textnormal{$\P$--a.s.}
	\]

	Therefore
	\begin{align}\label{eq::inequality_tilde_y_vP}
		\E^{\P}[\widetilde\cY^{n,i}_{t^\smalltext{n}_{\smalltext{i}\smalltext{-}\smalltext{1}}}|\cF_{t^\smalltext{n}_{\smalltext{i}\smalltext{-}\smalltext{1}}}]
		&=  \E^{\P}[\cY^{n,i}_{t^\smalltext{n}_{\smalltext{i}\smalltext{-}\smalltext{1}}}|\cF_{t^\smalltext{n}_{\smalltext{i}\smalltext{-}\smalltext{1}}}] - \E^{\P}\big[\cY^\P_{t^\smalltext{n}_{\smalltext{i}\smalltext{-}\smalltext{1}}}\big(t^n_i \land T,\E[\cY^\P_{t^\smalltext{n}_\smalltext{i} \land T}|\cF_{t^\smalltext{n}_\smalltext{i} \land T}]\big)\big|\cF_{t^\smalltext{n}_{\smalltext{i}\smalltext{-}\smalltext{1}}}\big] \nonumber\\
		&= \E^{\P}[\cY^\P_{t^\smalltext{n}_{\smalltext{i}\smalltext{-}\smalltext{1}}}(t^n_i \land T,\widehat\cY_{t^\smalltext{n}_\smalltext{i} \land T})|\cF_{t^\smalltext{n}_{\smalltext{i}\smalltext{-}\smalltext{1}}}] - \E^{\P}[\cY^\P_{t^\smalltext{n}_{\smalltext{i}\smalltext{-}\smalltext{1}}}|\cF_{t^\smalltext{n}_{\smalltext{i}\smalltext{-}\smalltext{1}}}] 
		\leq \widehat \cY_{t^\smalltext{n}_{\smalltext{i}\smalltext{-}\smalltext{1}}} - \E^\P[\cY^\P_{t^\smalltext{n}_{\smalltext{i}\smalltext{-}\smalltext{1}}}|\cF_{t^\smalltext{n}_{\smalltext{i}\smalltext{-}\smalltext{1}}}] 
		= V_{t^\smalltext{n}_{\smalltext{i}\smalltext{-}\smalltext{1}}}, \; \text{$\P$--a.s.},
	\end{align} 
	where the inequality follows from \Cref{ass::probabilities2}.$(ii)$, \Cref{lem::conditioning_bsde2}, and \Cref{lem::stopping_value_function}.
	
	\medskip
	We denote by $(\overline\cY^{n,i}, \overline\cZ^{n,i}, \overline\cU^{n,i}, \overline\cN^{n,i})$ the solution to the following well-posed BSDE relative to $\P$ and with time horizon $t^n_i \land T$:
	\begin{align*}
		\overline\cY^{n,i}_t 
		=
		V_{t^\smalltext{n}_\smalltext{i}} 
		&+ 
		\int_t^{t^\smalltext{n}_\smalltext{i} \land T} \Big( -\sqrt{r_s}\big|\overline\cY^{n,i}_s\big| - \sqrt{\mathrm{r}_s}\big|\overline\cY^{n,i}_{s\smallertext{-}}\big| - \sqrt{\theta^{X}_s} \big\| \mathsf{a}_s^{1/2} \overline\cZ^{n,i}_s \big\| + F^{n,i}_s\big(0,0, 0, \overline\cU^{n,i}_s(\cdot) \big) \Big)  \d C_s \\
		&\quad - 
		\int_t^{t^\smalltext{n}_\smalltext{i} \land T} \overline\cZ^{n,i}_s\d X_s^{c}
		- \int_t^{t^\smalltext{n}_\smalltext{i} \land T} \int_{\R^\smalltext{d}} \overline\cU^{n,i}_s(x)\tilde\mu^{X}(\d s, \d x) - \int_t^{t^\smalltext{n}_\smalltext{i} \land T}\d \overline\cN^{n,i}_s, \; t \in [0,\infty].
	\end{align*}
	The comparison principle in the form of \Cref{prop::comparison} yields $\overline\cY^{n,i} \leq \widetilde\cY^{n,i}$, $\P$--a.s., and $\overline\cY^{n,i} \geq 0$, $\P$--a.s., since $V_{t^\smalltext{n}_\smalltext{i}} \geq 0$, $\P$--a.s., and $F^{n,i}_s(0,0,0,\mathbf{0}) = 0$. We thus write
	\begin{align*}
		\overline\cY^{n,i}_t 
		=
		V_{t^\smalltext{n}_\smalltext{i}} 
		&+ 
		\int_t^{t^\smalltext{n}_\smalltext{i} \land T} \Big( \lambda_s\overline\cY^{n,i}_s 
		+ \widehat\lambda_s\overline\cY^{n,i}_{s\smallertext{-}} 
		+ (\eta^{n,i}_s)^\top \mathsf{a}_s \overline\cZ^{n,i}_s  
		+ F^{n,i}_s\big(0,0, 0, \overline\cU^{n,i}_s(\cdot) \big) \Big)  \d C_s \\
		&\quad - 
		\int_t^{t^\smalltext{n}_\smalltext{i} \land T} \overline\cZ^{n,i}_s\d X_s^{c}
		- \int_t^{t^\smalltext{n}_\smalltext{i} \land T}\int_{\R^\smalltext{d}} \overline\cU^{n,i}_s(x)\tilde\mu^{X}(\d s, \d x) - \int_t^{t^\smalltext{n}_\smalltext{i} \land T}\d \overline\cN^{n,i}_s, \; t \in [0,\infty], \; \text{$\P$--a.s.},
	\end{align*}
	where
	\begin{equation*}
		\lambda \coloneqq -\sqrt{r}, 
		\; 
		\widehat\lambda \coloneqq - \sqrt{\mathrm{r}}, 
		\; 
		\textnormal{and}
		\;
		\eta^{n,i} \coloneqq -\sqrt{\theta^{X}} (\mathsf{a}^{1/2})^\oplus  
		\frac{\mathsf{a}^{1/2} \overline\cZ^{n,i}}{\big\|\mathsf{a}^{1/2} \overline\cZ^{n,i}\big\|}\1_{\R^\smalltext{d}\setminus\{0\}}(\mathsf{a}^{1/2}\overline{\cZ}^{n,i}) \1_{\llparenthesis 0, t^\smalltext{n}_\smalltext{i} \land T \rrbracket}.
	\end{equation*}
	
	Note that $\eta^{n,i} \in \H^2_{t^\smalltext{n}_\smalltext{i} \land T}(X^{c};\F,\P)$ and $\langle \eta^{n,i} \bcdot X^{c} \rangle \leq \int_{0}^{\cdot\land t^\smalltext{n}_\smalltext{i} \land T} \theta^{X}_s \d C_s$. By \Cref{ass::crossing}.$(iv)$, there exists $\rho^\dagger \in \H^2_{T}(\mu^X;\F,\P)$ satisfying $\Delta (\rho^\dagger \ast\tilde\mu^X) > -1$, $\P$--a.s., and
	\begin{equation*}
		F^{n,i}_s\big(0,0, 0, \overline\cU^{n,i}_s(\cdot) \big) 
		\geq 
		\frac{\d\langle \rho^\dagger \ast\tilde\mu^{X} , \overline\cU^{n,i} \ast\tilde\mu^{X}\rangle_s}{\d C_s}
		, \; \text{$\P \otimes \mathrm{d}C$--a.e. on $\llparenthesis 0, t^n_i \land T \rrbracket$, for all $(n,i)\in\N^2$}.
	\end{equation*}
	Hence
	\begin{align*}
		\overline\cY^{n,i}_t 
		\geq V_{t^\smalltext{n}_\smalltext{i}} 
		&+ 
		\int_t^{t^\smalltext{n}_\smalltext{i} \land T} \bigg( \lambda_s\overline\cY^{n,i}_s + \widehat\lambda_s\overline\cY^{n,i}_{s\smallertext{-}} + (\eta^{n,i}_s)^\top \mathsf{a}_s \overline\cZ^{n,i}_s 
		+ \frac{\d\langle \rho^\dagger \ast\tilde\mu^{X} , \overline\cU^{n,i} \ast\tilde\mu^{X}\rangle_s}{\d C_s} \bigg) \d C_s \\
		& \quad - 
		\int_t^{t^\smalltext{n}_\smalltext{i} \land T} \overline\cZ^{n,i}_s\d X_s^{c,\P}
		- \int_t^{t^\smalltext{n}_\smalltext{i} \land T}\int_{\R^\smalltext{d}} \overline\cU^{n,i}_s(x) \tilde\mu^{X}(\d s, \d x)
		- \int_t^{t^\smalltext{n}_\smalltext{i} \land T}\d \overline\cN^{n,i}_s, \; t \in [0,\infty], \; \text{$\P$--a.s.}
	\end{align*}
	By following the arguments that lead to \cite[Equation~7.7]{possamai2024reflections}, we obtain
	\begin{equation}\label{eq::ineq_conditional_exp_stoch_exp}
		\cE(w)_{t^\smalltext{n}_{\smalltext{i}\smalltext{-}\smalltext{1}}}\cE(v)_{t^\smalltext{n}_{\smalltext{i}\smalltext{-}\smalltext{1}}}\overline\cY^{n,i}_{t^\smalltext{n}_{\smalltext{i}\smalltext{-}\smalltext{1}}} \geq \E^{\tilde\Q^{\smalltext{n}\smalltext{,}\smalltext{i}}}\big[ \cE(w)_{t^\smalltext{n}_\smalltext{i}}\cE(v)_{t^\smalltext{n}_\smalltext{i}} V_{t^\smalltext{n}_\smalltext{i}}   \big| \cF_{t^\smalltext{n}_{\smalltext{i}\smalltext{-}\smalltext{1}}\smallertext{+}}\big], \; \text{$\P$--a.s.}, \; i \in \N^\ast,
	\end{equation}
	where $\d\widetilde\Q^{n,i} \coloneqq \cE(\widetilde{L}^{n,i})_{t^\smalltext{n}_\smalltext{i}}\d\P$ on $(\Omega,\cF_{t^\smalltext{n}_\smalltext{i}})$ with 
	\begin{equation*}
		\widetilde{L}^{n,i} \coloneqq \int_{0}^{\cdot \land t^\smalltext{n}_\smalltext{i} \land T} \eta^{n,i}_s \d X^{c}_s + \int_{0}^{\cdot \land t^\smalltext{n}_\smalltext{i} \land T}\d (\rho^\dagger \ast\tilde\mu^{X})_s,\;
		w \coloneqq \int_{0}^{\cdot \land T} \lambda_s \d C_s,
		\; 
		\text{and} 
		\; 
		v \coloneqq \int_{0}^{\cdot \land T} \frac{\widehat\lambda_s}{1-\widehat\lambda_s\Delta C_s} \d C_s.
	\end{equation*}
	By \Cref{ass::crossing}.$(ii)$,$(iv)$ and \cite[Lemma 7.4]{possamai2024reflections}, $\cE(\widetilde L^{n,i})_{t^\smalltext{n}_\smalltext{i}}$ is strictly positive and in $\L^2(\P)$ with $\P$-expectation one, and thus defines a probability density.
	Moreover, both $w$ and $v$ are non-increasing, and $\Phi < 1$ implies $\Delta w > -1$ and $\Delta v > -1$. Therefore $\cE(w)$ and $\cE(v)$ are positive and non-increasing. Rearranging the terms in \eqref{eq::ineq_conditional_exp_stoch_exp}, and then applying Bayes's formula for conditional expectation yields
	\begin{equation}\label{eq::ineq_conditional_exp_stoch_exp_2}
		\overline\cY^{n,i}_{t^\smalltext{n}_{\smalltext{i}\smalltext{-}\smalltext{1}}} \geq \E^\P\big[ \cE(L^{n,i})_{t^\smalltext{n}_\smalltext{i}} \cE(w^{n,i})_{t^\smalltext{n}_\smalltext{i}}\cE(v^{n,i})_{t^\smalltext{n}_\smalltext{i}} V_{t^\smalltext{n}_\smalltext{i}}   \big| \cF_{t^\smalltext{n}_{\smalltext{i}\smalltext{-}\smalltext{1}}\smallertext{+}}\big], \; \text{$\P$--a.s.},
	\end{equation}
	where $\d\Q^{n,i} \coloneqq \cE(L^{n,i})_{t^\smalltext{n}_\smalltext{i}}\d\P = (\cE(\widetilde L^{n,i})_{t^\smalltext{n}_\smalltext{i}}/\cE(\widetilde L^{n,i})_{t^\smalltext{n}_{\smalltext{i}\smalltext{-}\smalltext{1}}})\d\P$ on $(\Omega,\cF_{t^\smalltext{n}_\smalltext{i}})$ with
	\begin{gather}\label{eq::formulas_tilde_v_w}
		L^{n,i} \coloneqq \widetilde L^{n,i} - \widetilde L^{n,i}_{\cdot \land t^\smalltext{n}_{\smalltext{i}\smalltext{-}\smalltext{1}} \land T} = \int_{t^\smalltext{n}_{\smalltext{i}\smalltext{-}\smalltext{1}} \land T}^{\cdot \land t^\smalltext{n}_\smalltext{i} \land T} \eta^{n,i}_s \d X^{c,\P}_s + \int_{t^\smalltext{n}_{\smalltext{i}\smalltext{-}\smalltext{1}} \land T}^{\cdot \land t^\smalltext{n}_\smalltext{i} \land T}\d (\rho^\dagger \ast\tilde\mu^{X})_s, \nonumber\\
		w^{n,i} \coloneqq \int_{t^\smalltext{n}_{\smalltext{i}\smalltext{-}\smalltext{1}} \land T}^{\cdot \land t^\smalltext{n}_\smalltext{i} \land T} \lambda_s \d C_s, \; v^{n,i} \coloneqq \int_{t^\smalltext{n}_{\smalltext{i}\smalltext{-}\smalltext{1}}\land T}^{\cdot \land t^\smalltext{n}_\smalltext{i} \land T} \frac{\widehat\lambda_s}{1-\widehat\lambda_s\Delta C_s} \d C_s.
	\end{gather}
	Taking conditional expectation with respect to $\cF_{t^\smalltext{n}_{\smalltext{i}\smalltext{-}\smalltext{1}}}$ in \eqref{eq::ineq_conditional_exp_stoch_exp_2} implies
	\begin{equation*}
		\E^{\P} 
		\big[\overline\cY^{n,i}_{t^\smalltext{n}_{\smalltext{i}\smalltext{-}\smalltext{1}}} \big|\cF_{t^\smalltext{n}_{\smalltext{i}\smalltext{-}\smalltext{1}}} \big] 
		\geq 
		\E^{\P} 
		\big[ \cE(L^{n,i})_{t^\smalltext{n}_\smalltext{i}} \cE(w^{n,i})_{t^\smalltext{n}_\smalltext{i}}\cE(v^{n,i})_{t^\smalltext{n}_\smalltext{i}} V_{t^\smalltext{n}_\smalltext{i}}   \big| \cF_{t^\smalltext{n}_{\smalltext{i}\smalltext{-}\smalltext{1}}}\big], \; \text{$\P$--a.s.}
	\end{equation*}
	By \eqref{eq::inequality_tilde_y_vP} and the fact that $\overline\cY^{n,i} \leq \widetilde\cY^{n,i}$, $\P$--a.s., we then obtain
	\begin{equation*}
		V_{t^\smalltext{n}_{\smalltext{i}\smalltext{-}\smalltext{1}}} 
		\geq \E^{\P}\big[\widetilde\cY^{n,i}_{t^\smalltext{n}_{\smalltext{i}\smalltext{-}\smalltext{1}}} \big|\cF_{t^\smalltext{n}_{\smalltext{i}\smalltext{-}\smalltext{1}}} \big] 
		\geq \E^{\P}\big[\overline\cY^{n,i}_{t^\smalltext{n}_{\smalltext{i}\smalltext{-}\smalltext{1}}} \big|\cF_{t^\smalltext{n}_{\smalltext{i}\smalltext{-}\smalltext{1}}} \big] 
		\geq 
		\E^{\P} 
		\big[ \cE(L^{n,i})_{t^\smalltext{n}_\smalltext{i}} \cE(w^{n,i})_{t^\smalltext{n}_\smalltext{i}}\cE(v^{n,i})_{t^\smalltext{n}_\smalltext{i}} V_{t^\smalltext{n}_\smalltext{i}}   \big| \cF_{t^\smalltext{n}_{\smalltext{i}\smalltext{-}\smalltext{1}}}\big], \; \text{$\P$--a.s.}
	\end{equation*} 
	Fix $K \in \N^\star$, and define $\d\Q^{n,K} \coloneqq \cE(L^{n,K})_{K}\d\P$, where
	\begin{equation*}
		L^{n,K} \coloneqq \sum_{i = 1}^{K 2^{\smalltext{n}}}\bigg(\int_{t^\smalltext{n}_{\smalltext{i}\smalltext{-}\smalltext{1}} \land T}^{\cdot \land t^\smalltext{n}_\smalltext{i} \land T} \eta^{n,i}_s \d X^{c}_s + \int_{t^\smalltext{n}_{\smalltext{i}\smalltext{-}\smalltext{1}} \land T}^{\cdot \land t^\smalltext{n}_\smalltext{i} \land T}\d (\rho^\dagger \ast\tilde\mu^{X})_s\bigg).
	\end{equation*}
	Then $\displaystyle\cE(L^{n,i})_{t^\smalltext{n}_\smalltext{i}} = \frac{\cE(L^{n,K})_{t^\smalltext{n}_\smalltext{i}}}{\cE(L^{n,K})_{t^\smalltext{n}_{\smalltext{i}\smalltext{-}\smalltext{1}}}}$ for $i \in\{1,\dots, K 2^n\}$, and therefore
	\begin{equation}\label{eq::ineq_conditional_exp_stoch_exp_3}
		V_{t^\smalltext{n}_{\smalltext{i}\smalltext{-}\smalltext{1}}} 
		\geq 
		\E^{\P} \bigg[ \frac{\cE(L^{n,K})_{t^\smalltext{n}_\smalltext{i}}}{\cE(L^{n,K})_{t^\smalltext{n}_{\smalltext{i}\smalltext{-}\smalltext{1}}}} \cE(w^{n,i})_{t^\smalltext{n}_\smalltext{i}}\cE(v^{n,i})_{t^\smalltext{n}_\smalltext{i}} V_{t^\smalltext{n}_\smalltext{i}} \bigg| \cF_{t^\smalltext{n}_{\smalltext{i}\smalltext{-}\smalltext{1}}}\bigg], \; \text{$\P$--a.s.}, \; i \in\{1,\dots, K 2^n\}.
	\end{equation}
	By choosing a $\P$-version of $\cE(L^{n,K})$ which is $\F$-adapted (see \Cref{prop::good_version_stochastic_integral}) and by \eqref{eq::formulas_tilde_v_w}, we then obtain
	\begin{equation*}
		\cE(L^{n,K})_{t^\smalltext{n}_{\smalltext{i}\smalltext{-}\smalltext{1}}} \cE(w)_{t^\smalltext{n}_{\smalltext{i}\smalltext{-}\smalltext{1}}}\cE(v)_{t^\smalltext{n}_{\smalltext{i}\smalltext{-}\smalltext{1}}} V_{t^\smalltext{n}_{\smalltext{i}\smalltext{-}\smalltext{1}}} 
		\geq \E^{\P} \big[ \cE(L^{n,K})_{t^\smalltext{n}_\smalltext{i}} \cE(w)_{t^\smalltext{n}_\smalltext{i}}\cE(v)_{t^\smalltext{n}_\smalltext{i}} V_{t^\smalltext{n}_\smalltext{i}}   \big| \cF_{t^\smalltext{n}_{\smalltext{i}\smalltext{-}\smalltext{1}}}\big], \; \text{$\P$--a.s.}, \; i \in\{1,\dots, K 2^n\}.
	\end{equation*}
	
	This implies that the nonnegative, discrete-time process $S^{n,K} = (S^{n,K}_{t^\smalltext{n}_\smalltext{i}})_{i\in\{0,\dots, K2^\smalltext{n}\}}$, defined by 
	\[
	S^{n,K}_{t^\smalltext{n}_\smalltext{i}} \coloneqq \cE(w)_{t^\smalltext{n}_\smalltext{i}}\cE(v)_{t^\smalltext{n}_\smalltext{i}} V_{t^\smalltext{n}_\smalltext{i}},
	\] 
	satisfies the $\Q^{n,K}$--super-martingale property relative to the discrete-time filtration $(\cF_{t^\smalltext{n}_\smalltext{i}})_{i\in\{0,\dots, K2^\smalltext{n}\}}$. Since $V_0 \in \L^2(\P)$ and $\d\Q^{n,K}/\d\P \in \L^2(\P)$ (\cite[Lemma 7.4]{possamai2024reflections}), we obtain
	\[
	\E^{\Q^{\smalltext{n}\smalltext{,}\smalltext{K}}} [S^{n,K}_{t^\smalltext{n}_\smalltext{i}} | \cF_{t^\smalltext{n}_\smalltext{0}}] \leq S^{n,K}_0 = V_0 \in \L^1(\Q^{n,K}),
	\]
	which then implies the required integrability for $S^{n,K}$ to be a (true) $\Q^{n,K}$--super-martingale.

	\medskip
	Let $0 \leq a < b < \infty$. Define the processes $(\ell_{t^\smalltext{n}_\smalltext{i}})_{ i\in\{0,1,\dots, K2^\smalltext{n}\}}$ and $(u_{t^\smalltext{n}_\smalltext{i}})_{i\in\{0,1,\dots,K2^\smalltext{n}\}}$ by $\ell_{t^\smalltext{n}_\smalltext{i}} \coloneqq a \cE(w)_{t^\smalltext{n}_\smalltext{i}}\cE(v)_{t^\smalltext{n}_\smalltext{i}}$ and $u_{t^\smalltext{n}_\smalltext{i}} \coloneqq b\cE(w)_{t^\smalltext{n}_\smalltext{i}}\cE(v)_{t^\smalltext{n}_\smalltext{i}}$. By definition of $w$ and $v$, both $\ell$ and $u$ are non-increasing processes. Let $\D^n_\smallertext{+} \coloneqq \{j 2^{-n} : j \in \N\}$. We denote by $D^u_\ell(S^{n,K};\D^n_\smallertext{+} \cap [0,K])$ the number of down-crossings of $[\ell, u]$ by the process $S^{n,K}$ on $\D^n_\smallertext{+} \cap [0,K]$. Since $\langle L^{n,K} \rangle^{(\P)}$ is $\P$--essentially bounded by a constant $\mathfrak{C} \in (0,\infty)$ independent of $n$ and $K$ (see \Cref{ass::crossing}.$(ii)$,$(iv)$), the arguments in the proof of \cite[Lemma 7.4]{possamai2024reflections} imply that
	\begin{equation*}
		\E^\P\Bigg[\bigg(\frac{\d\Q^{n,K}}{\d\P}\bigg)^2\Bigg] 
		= \E^\P\Big[\big(\cE(L^{n,K})_{K\land T}\big)^2 \Big] 
		\leq 4\mathrm{e}^{\mathfrak{C}}.
	\end{equation*}
	Since, with analogous notation, we have $D^u_{\ell} (S^{n,K};\D^n_\smallertext{+} \cap [0,K]) = D^b_a(V;\D^n_\smallertext{+} \cap [0,K])$, it follows from Doob's down-crossing inequality\footnote{Like most inequalities in martingale theory, this also holds when conditioning at time zero.} \cite[Equation 12.3, page 446]{doob1984classical} that
	\begin{align}\label{eq::bounding_supermartingale_over_dyadics}
		\E^{\Q^{n,K}}[\cE(w)_{K}\cE(v)_{K}D^b_a(V;\D^n_\smallertext{+}\cap[0,K])] 
		= \E^{\Q^{\smalltext{n}\smalltext{,}\smalltext{K}}}[\cE(w)_{K}\cE(v)_{K}D^u_\ell(S^{n,K};\D^n_\smallertext{+}\cap[0,K])] 
		\leq \frac{b}{b-a}.
	\end{align}
	Bayes's formula for conditional expectation then yields
	\begin{align}\label{eq::bayes_conditioned_F0}
		\E^{\P}\Big[\E^{\Q^{\smalltext{n}\smalltext{,}\smalltext{K}}}\big[\cE(w)_{K}\cE(v)_{K}D^b_a(V;\D^n_\smallertext{+}\cap[0,K])\big| \cF_{0\smallertext{+}}\big]\Big]
		&=\E^{\P}\Bigg[\E^\P\bigg[\frac{\d\Q^{n,K}}{\d\P}\cE(w)_{K}\cE(v)_{K}D^b_a(V;\D^n_\smallertext{+}\cap[0,K])\bigg| \cF_{0\smallertext{+}}\bigg]\Bigg] \nonumber \\
		&= \E^{\Q^{n,K}}\big[\cE(w)_{K}\cE(v)_{K}D^b_a(V;\D^n_\smallertext{+}\cap[0,K])\big] \leq \frac{b}{b-a}.
	\end{align}
	For $k \in \N^\ast$, let $\xi^{n,K,k} \coloneqq \cE(\hat\beta A)^{-1/2}_{K \land T}\mathbf{1}_{\{\cE(w)_{\smalltext{K}}\cE(v)_{\smalltext{K}}D^\smalltext{b}_\smalltext{a}(V;\D^\smalltext{n}_\smalltext{+} \cap [0,K]) > k\}} \in \L^2_{\hat\beta}(\cF_{K \land T}).$ Let $\sY(K \land T,\xi^{n,K,k})$ be the first component of the solution to \eqref{eq::lipschitz_linear_bsde} with terminal time $K \land T$ and terminal condition $\xi^{n,K,k}$ and where $\rho \coloneqq \rho^\dagger$. Then
	\begin{align*}
		\sY_0(K \land T,\xi^{n,K,k}) 
		\leq \E^{\Q^{\smalltext{n}\smalltext{,}\smalltext{K}}}[\xi^{n,K,k}|\cF_{0\smallertext{+}}] 
		&\leq \E^{\Q^{\smalltext{n}\smalltext{,}\smalltext{K}}}[\mathbf{1}_{\{\cE(w)_{\smalltext{K}}\cE(v)_{\smalltext{K}}D^\smalltext{b}_\smalltext{a}(V;\D^\smalltext{n}_\smalltext{+}\cap[0,K]) > k\}}|\cF_{0\smallertext{+}}] \\
		&\leq \frac{1}{k} \E^{\Q^{\smalltext{n}\smalltext{,}\smalltext{K}}}[\cE(w)_{K}\cE(v)_{K}D^b_a(V;\D^n_\smallertext{+}\cap[0,K])|\cF_{0\smallertext{+}}], \; \text{$\P$--a.s.},
	\end{align*}
	where we use \Cref{lem::linearising_bsde} in the first inequality and a conditional version of Markov's inequality in the last line. Taking expectation under $\P$ together with \eqref{eq::bayes_conditioned_F0} yields
	\begin{equation}\label{eq::cond_ineq_script_y_1}
		\E^\P\big[\sY_0(K \land T,\xi^{n,K,k})\big] \leq  \frac{b}{k(b-a)}.
	\end{equation}
	Let $\xi^{K,k} \coloneqq \lim_{n \rightarrow \infty} \xi^{n,K,k} = \cE(\hat\beta A)^{-1/2}_{K \land T}\mathbf{1}_{\{\cE(w)_{\smalltext{K}}\cE(v)_{\smalltext{K}}D^\smalltext{b}_\smalltext{a}(V;\D_\tinytext{+} \cap[0,K]) > k\}}$; the equality holds since the dyadic grids are increasing and every finite sequence of dyadic down-crossings of $[a,b]$ is contained in some $\D^n_\smallertext{+}$ for $n$ large enough. 
	\Cref{prop::stability} implies that there exists a constant $\mathfrak{C}$ only depending on $\hat{\beta}$ and $\Phi$ such that
	\[
		|\sY_0(K \land T,\xi^{n,K,k})| \leq \mathfrak{C}, \; \textnormal{$\P$--a.s.},
	\]
	and
	\[
		\lim_{n \rightarrow \infty} \sY_0(K \land T,\xi^{n,K,k}) = \sY_0(K \land T,\xi^{K,k}), \; \textnormal{$\P$--a.s.}
	\]
	We then obtain by dominated convergence
	\begin{equation}\label{eq::cond_ineq_script_y_2}
		\E^\P\big[\sY_0(K \land T,\xi^{K,k})\big] = \lim_{n \rightarrow \infty} \E^\P\big[\sY_0(K \land T,\xi^{n,K,k})\big] \leq  \frac{b}{k(b-a)}.
	\end{equation}
	Since $\xi^K \coloneqq \lim_{k \rightarrow \infty} \xi^{K,k} = \cE(\hat\beta A)^{-1/2}_{K \land T}\mathbf{1}_{\{\cE(w)_{\smalltext{K}}\cE(v)_{\smalltext{K}}D^\smalltext{b}_\smalltext{a}(V;\D_\tinytext{+}\cap[0,K]) = \infty\}},$
	an analogous argument then yields
	\begin{equation*}
		\E^\P\big[\sY_0(K \land T,\xi^K)\big] = \lim_{k \rightarrow\infty} \E^\P\big[\sY_0(K \land T,\xi^{K,k})\big] \leq 0.
	\end{equation*}
	By \Cref{lem::linearising_bsde}, we know that $\sY_0(K \land T,\xi^K) = \E^{\mathscr{Q}}[\xi^K|\cF_{0\smallertext{+}}]$, $\P$--a.s., for some probability $\mathscr{Q}$ equivalent to $\P$. Hence
	\begin{equation*}
		\E^\P\big[\E^{\mathscr{Q}}[\xi^K|\cF_{0\smallertext{+}}]\big] = \E^\P\big[\sY_0(K \land T,\xi^K)\big] \leq 0.
	\end{equation*}
	Since $\xi^K \geq 0$, and thus $\E^{\mathscr{Q}}[\xi^K|\cF_{0\smallertext{+}}] \geq 0$, $\P$--a.s., we have $\E^{\mathscr{Q}}[\xi^K|\cF_{0\smallertext{+}}] = 0$, $\P$--a.s. and $\mathscr{Q}$--a.s. This in turn implies that $\xi^K = 0$, $\mathscr{Q}$--a.s. and $\P$--a.s. Since $\cE(\hat\beta A)_{K \land T}$, $\cE(w)_{K}$ and $\cE(v)_{K}$ are $(0,\infty)$-valued, we obtain
	\begin{equation*}
		\P\big[D^b_a(V;\D_\smallertext{+} \cap [0,K]) < \infty\big] = 1,
	\end{equation*}
	and then
	\begin{equation*}
		\P\Bigg[ \bigcap_{K \in \N^\smalltext{\star}} \bigcap_{\{(a,b) \in \D^2_\tinytext{+}: a < b\}} \big\{D^b_a(V;\D_\smallertext{+} \cap [0,K]) < \infty\big\}\Bigg] = 1.
	\end{equation*}
	Since $V_t \geq 0$, $\P$--a.s. for every $t \in \D_\smallertext{+}$, this yields $\P[\Omega_1] = 1$ and completes the proof of assertion $(i)$.
\end{proof}

\begin{remark}\label{rem::gap_regularisation}
	The path-regularisation of the value function has previously appeared in \textnormal{\cite[Lemma 3.2]{possamai2018stochastic}} in the context of \textnormal{2BSDEs} driven by continuous processes. In that proof, following a transformation of $\widehat{\cY}$ into a super-martingale under an equivalent measure $($see \textnormal{\cite[page 575]{possamai2018stochastic}}$)$, the proof was completed by invoking the arguments used in the proof of \textnormal{\cite[Lemma A.1]{bouchard2016general}}. However, the process $u$ appearing in the latter argument is not necessarily a super-martingale, as required by the cited down-crossing inequality \textnormal{\cite[page 446]{doob1984classical}}.
	Nevertheless, our transformations outlined in the preceding proof avoid this issue.
	
	\medskip
	The preceding proof is related to the argument used in the proof of \textnormal{\cite[Lemma 4.8]{soner2013dual}}, which invokes \textnormal{\cite[Theorem 6]{chen2000general}}. However, the proof of \textnormal{\cite[Theorem 6]{chen2000general}} asserts, in the notation there, that $y^{(j)}_t \geq 0$, which does not appear to follow without assuming that $g_s(0,0) \equiv 0$. This does not affect the application of \textnormal{\cite[Theorem 6]{chen2000general}} in the proof of \textnormal{\cite[Lemma 4.8]{soner2013dual}}, as the generator $f^\P$ considered in \textnormal{\cite[Lemma 4.8]{soner2013dual}} satisfies $f^\P_s(0,0) \equiv 0$ and the relevant process is nonnegative.
\end{remark}

\begin{proof}
	[Proof of \Cref{thm::down-crossing}.$(ii)$]
	We start with the stated integrability. By \Cref{thm::measurability2} and \Cref{prop::stability}, there exists a constant $\mathfrak{C}^\prime \in (0,\infty)$, depending only on $\hat\beta$ and on $\Phi$, such that for every $\P \in \fP_0$ and $s \in [0,\infty)$,
	\begin{align}\label{eq_inequality_y_hat}
		|\widehat\cY_s|^2 
		&\leq \underset{\bar{\P} \in \fP_\smalltext{0}(\cF_{s},\P)}{{\esssup}^\P}\E^{\bar{\P}}\big[|\cY^{\bar{\P}}_s(T,\xi)|^2\big|\cF_s\big] \nonumber\\
		&\leq \mathfrak{C}^\prime \underset{\bar{\P} \in \fP_\smalltext{0}(\cF_{s},\P)}{{\esssup}^\P}\E^{\bar{\P}}\bigg[ \frac{\cE(\hat\beta A)_T}{\cE(\hat\beta A)_{s}} |\xi|^2 +\int_s^T \frac{\cE(\hat\beta A)_r}{\cE(\hat\beta A)_s} \frac{|f^{\bar\P}_r(0,0,0,\mathbf{0})|^2}{\alpha^2_r} \d C_r \bigg| \cF_{s}\bigg]	, \; \text{$\P$--a.s.}
	\end{align}
	The $\F$-adaptedness of $\cE(\hat\beta A) = \cE(\hat\beta A)_{\cdot \land T}$ together with $\phi^{2,\hat{\beta}}_{\xi, f} < \infty$, then immediately yields
	\begin{equation*}
		\sup_{\P \in \fP_\smalltext{0}}\E^\P\bigg[\sup_{s \in \D_\tinytext{+}}\big|\cE(\hat\beta A)^{1/2}_{s}\widehat\cY_{s}(T,\xi)\big|^2\bigg] 
		\leq \mathfrak{C}^\prime \phi^{2,\hat{\beta}}_{\xi, f} < \infty,
	\end{equation*}
	and then
	\begin{equation*}
		\sup_{\P \in \fP_\smalltext{0}}\E^\P\bigg[\sup_{s \in [0,T]}\big|\cE(\hat\beta A)^{1/2}_s\widehat\cY^\smallertext{+}_s(T,\xi)\big|^2\bigg] \leq \sup_{\P \in \fP_\smalltext{0}}\E^\P\bigg[\sup_{s \in \D_\tinytext{+}}\big|\cE(\hat\beta A)^{1/2}_{s \land T}\widehat\cY_{s \land T}(T,\xi)\big|^2 + \cE(\hat\beta A)_T|\xi|^2\bigg] \leq (\mathfrak{C}^{\prime}+1) \phi^{2,\hat{\beta}}_{\xi, f} < \infty.
	\end{equation*}
	We then let $\mathfrak{C} \coloneqq 2\mathfrak{C}^\prime + 1$.
	
	\medskip
	We turn to the representation \eqref{eq::aggregation}, and thus fix $\P \in \fP_0$ and $t \in [0,\infty)$; for $t = \infty$, the stated representation holds immediately. We write $\widehat{\cY}^\smallertext{+}$ instead of $\widehat{\cY}^\smallertext{+}(T,\xi)$ for simplicity. We first prove the following auxiliary result: for any sequence $(t_n)_{n \in \N}$ of dyadic numbers converging to $t$ from above, there exists a subsequence $(t_{n_\smalltext{k}})_{k \in \N}$ such that
	\begin{equation}\label{eq::stability_terminal_condition}
		\lim_{k \rightarrow \infty} \sup_{s \in [0,t]}|\cY^\P_s(t_{n_\smalltext{k}} \land T,\widehat\cY_{t_{\smalltext{n}_\tinytext{k}} \land T}) - \cY^\P_s(t \land T,\widehat\cY^\smallertext{+}_{t \land T})| = 0, \; \textnormal{$\P$--a.s.}
	\end{equation}
	This can be argued as follows. We write
	\begin{align}\label{eq::convergence_t_n}
		&\cY^\P_s(t_n \land T,\widehat\cY_{t_\smalltext{n} \land T}) 
		- \cY^\P_s(t \land T,\widehat\cY^\smallertext{+}_{t \land T}) \nonumber\\
		&= \cY^\P_s(t_n \land T,\widehat\cY_{t_\smalltext{n} \land T})  - \widetilde\cY^\P_s\big(t_n \land T,\widehat\cY^\smallertext{+}_{t \land T}\cE(\hat\beta A)^{1/2}_{t \land T}/\cE(\hat\beta A)^{1/2}_{t_\smalltext{n} \land T}\big) \nonumber\\
		&\quad + \widetilde\cY^{\P}_s\big(t \land T,\widetilde\cY^\P_{t \land T}(t_n \land T,\widehat\cY^\smallertext{+}_{t \land T}\cE(\hat\beta A)^{1/2}_{t \land T}/\cE(\hat\beta A)^{1/2}_{t_\smalltext{n} \land T})\big) - \cY^\P_s(t \land T,\widehat\cY^\smallertext{+}_{t \land T}), \; s \in [0,t],\; \textnormal{$\P$--a.s.},
	\end{align}
	where the generic notation $\widetilde\cY^{\P}(S, \zeta)$ denotes the first component of the solution $(\widetilde\cY^\P(S,\zeta),\widetilde\cZ^\P(S,\zeta),\widetilde\cU^\P(S,\zeta),\widetilde\cN^\P(S,\zeta))$ to the BSDE with generator $f^\P \1_{\llparenthesis 0,t\rrbracket}$, terminal time $S$ and terminal condition $\zeta$, if well-posed according to \cite[Section 3.2]{possamai2024reflections}. The second difference after the equality in \eqref{eq::convergence_t_n} converges $\P$--a.s.\ (along a subsequence if necessary) uniformly in $s \in [0,t]$ to zero by \Cref{cor::stability}; hence, we focus on the first difference after the equality. To ease the notation, we write in the following lines $(\widetilde\cY^{\P,n},\widetilde\cZ^{\P,n},\widetilde\cU^{\P,n},\widetilde\cN^{\P,n})$ for the solution of the BSDE with generator $f^\P\1_{\llparenthesis 0,t\rrbracket}$, terminal time $t_n \land T$ and terminal condition $\zeta^{n} \coloneqq \widehat\cY^\smallertext{+}_{t \land T}\cE(\hat\beta A)^{1/2}_{t \land T}/\cE(\hat\beta A)^{1/2}_{t_\smalltext{n} \land T}$. By \Cref{cor::stability}, there exists a constant $\mathfrak{C}^\prime \in (0,\infty)$, only depending on $\Phi$ and on $\hat\beta$, such that
	\begin{align*}
		&\E^\P\bigg[\sup_{s \in [0,\infty]}\big|\cE(\hat\beta A)^{1/2}_{s \land t \land T} \big(\cY^\P_{s \land t \land T}(t_n \land T, \widehat{\cY}_{t_\smalltext{n}\land T}) - \widetilde{\cY}^\P_{s \land t \land T}(t_n \land T, \zeta^{n})  \big)\big|^2\bigg] \\
		&= \E^\P\bigg[\sup_{s \in [0,\infty]}\Big|\cE(\hat\beta A)^{1/2}_{s \land t \land T} \Big(\cY^\P_{s \land t \land T}\big(t \land T,\cY^\P_{t \land T}(t_n\land T, \widehat{\cY}_{t_\smalltext{n}\land T})\big) - \widetilde{\cY}^\P_{s \land t \land T}\big(t \land T, \widetilde{\cY}^\P_{t \land T}(t_n\land T, \zeta^{n})\big)\Big)\Big|^2\bigg] \\
		&\leq \mathfrak{C}^\prime \E^\P\Big[ \cE(\hat\beta A)_{t \land T} \big|\cY^\P_{t \land T}(t_n \land T,\widehat\cY_{t_\smalltext{n} \land T})  - \widetilde\cY^\P_{t \land T}(t_n \land T,\zeta^{n})\big|^2 \Big].
	\end{align*}
	It therefore suffices to show that the last term converges to zero as $n$ tends to infinity to deduce \eqref{eq::stability_terminal_condition}.
	
	\medskip
	By the stability result in \Cref{prop::stability}, there exists a constant $\mathfrak{C}$ only depending on $\hat\beta$ and on $\Phi$ such that
	\begin{align}\label{eq::convergence_terminal_time_t}
		& \E^\P\Big[\cE(\hat\beta A)_{t \land T}\big|\cY^\P_t(t_n \land T,\widehat\cY_{t_\smalltext{n} \land T})  - \widetilde\cY^\P_t(t_n \land T,\zeta^{n})\big|^2\Big] \nonumber\\
		&\leq \mathfrak{C} \Bigg( \E^\P\bigg[\cE(\hat\beta A)_{t_\smalltext{n} \land T} \big| \widehat\cY_{t_\smalltext{n} \land T} - \zeta^{n} \big|^2 + \int_t^{t_n \land T} \cE(\hat\beta A)_{r} \frac{|f^\P\big(\widetilde\cY^{\P,n}_r,\widetilde\cY^{\P,n}_{r\smallertext{-}},\widetilde\cZ^{\P,n}_r,\widetilde\cU^{\P,n}_r(\cdot)\big)|^2}{\alpha^2_r} \d C_r\bigg] \Bigg) \nonumber\\
		&\leq 2 \mathfrak{C} \E^\P\bigg[\cE(\hat\beta A)_{t_\smalltext{n} \land T} \big| \widehat\cY_{t_\smalltext{n} \land T} - \zeta^{n} \big|^2 + \int_t^{t_n \land T}\cE(\hat\beta A)_r \big|\alpha_r\widetilde\cY^{\P,n}_r\big|^2 \d C_r \nonumber\\
		&\quad + \int_t^{t_n \land T} \cE(\hat\beta A)_r \big|\alpha_r\widetilde\cY^{\P,n}_{r\smallertext{-}}\big|^2 \d C_r   
		+ \int_t^{t_n \land T} \cE(\hat\beta A)_r \d\langle\widetilde\eta^{\P,n}\rangle_r + \int_t^{t_n \land T} \cE(\hat\beta A)_r \frac{|f^\P_r(0,0,0,\mathbf{0})|^2}{\alpha^2_r} \d C_r \bigg],
	\end{align}
		where
	\begin{equation*}
		\widetilde\eta^{\P,n} \coloneqq \int_{t}^{\cdot \land t_\smalltext{n} \land T} \widetilde\cZ^{\P,n}_r\d X^{c,\P}_r + \int_{t}^{\cdot \land t_\smalltext{n} \land T} \d (\widetilde\cU^{\P,n}\ast\tilde{\mu}^{X,\P})_r + \int_t^{\cdot\land t_n \land T} \d\widetilde{\cN}^{\P,n}_r.
	\end{equation*}
	We show that the expectation after the last inequality in \eqref{eq::convergence_terminal_time_t}  converges to zero as $n$ tends to infinity.
	Note that
	\begin{equation*}
		\widetilde{\cY}^{\P,n}_u = \E^\P\big[\zeta^{n}\big|\cF_{u\smallertext{+}}\big] \; \textnormal{and} \; \big(\widetilde\cY^{\P,n}_{u} - \zeta^{n}\big)^2 = \bigg(\int_{u}^{t_\smalltext{n} \land T} \d\widetilde\eta^{\P,n}_r\bigg)^2, \; \textnormal{$\P$--a.s.}, \; u \in [t,\infty],
	\end{equation*}
	implies
	\begin{align*}
		& \E^\P\bigg[\int_u^{t_\smalltext{n}\land T}\d\langle\widetilde\eta^{\P,n}\rangle^{(\F_\smalltext{+},\P)}_r \bigg| \cF_{u\smallertext{+}}\bigg]
		= \E^\P\bigg[\bigg(\int_{u}^{t_\smalltext{n} \land T} \d\widetilde\eta^{\P,n}_r\bigg)^2 \bigg| \cF_{u\smallertext{+}}\bigg] 
		= \E^\P\big[	\big(\widetilde\cY^{\P,n}_{u} - \zeta^{n}\big)^2 \big| \cF_{u\smallertext{+}}\big] \\
		&= \big(\widetilde\cY^{\P,n}_{u}\big)^2 - 2 \widetilde\cY^{\P,n}_{u} \E^\P[\zeta^{n}|\cF_{u\smallertext{+}}] + \E^\P\big[\big(\zeta^{n}\big)^2\big|\cF_{u\smallertext{+}}\big] = \E^\P\big[\big(\zeta^{n}\big)^2 - \big(\widetilde\cY^{\P,n}_{u}\big)^2\big|\cF_{u\smallertext{+}}\big], \; \textnormal{$\P$--a.s.}, \; u \in [t,\infty],
	\end{align*}
	and (see \cite[Lemma C.5]{possamai2024reflections})
	\begin{align*}
		\E^\P\bigg[\int_{u\smallertext{-}}^{t_\smalltext{n} \land T} \d\langle\widetilde\eta^{\P,n}\rangle^{(\F_\smalltext{+},\P)}_r \bigg| \cF_{u\smallertext{-}}\bigg]
		= \E^\P\bigg[\bigg(\int_{u\smallertext{-}}^{t_\smalltext{n} \land T} \d\widetilde\eta^{\P,n}_r\bigg)^2 \bigg| \cF_{u\smallertext{-}}\bigg] 
		&= \big(\widetilde\cY^{\P,n}_{u\smallertext{-}}\big)^2 - 2 \widetilde\cY^{\P,n}_{u\smallertext{-}} \E^\P[\zeta^{n}|\cF_{u\smallertext{-}}] + \E^\P\big[\big(\zeta^{n}\big)^2\big|\cF_{u\smallertext{-}}\big] \\
		&= \E^\P\big[\big(\zeta^{n}\big)^2 - \big(\widetilde\cY^{\P,n}_{u\smallertext{-}}\big)^2\big|\cF_{u\smallertext{-}}\big], \; \text{$\P$--a.s.}, \; u \in (t,\infty].
	\end{align*}	
	Using $\cE(\hat\beta A)_r = \cE(\hat\beta A)_{t} + \int_t^r \cE(\hat\beta A)_{u-}\hat\beta \d A_u$, $r \in [t,\infty)$, the predictable projection (see \cite[Theorem VI.57]{dellacherie1982probabilities}), and Tonelli's theorem, we obtain
	\begin{align*}
		\E^\P\bigg[ \int_t^{t_n \land T} \cE(\hat\beta A)_r \d\langle\widetilde\eta^{\P,n}\rangle^{(\F_\smalltext{+},\P)}_r \bigg]
		&= \E^\P\big[\cE(\hat\beta A)_{t \land T}\langle\widetilde\eta^{\P,n}\rangle^{(\F_\smalltext{+},\P)}_{t_\smalltext{n} \land T} \big] + \hat\beta \E^\P\bigg[\int_t^{t_\smalltext{n} \land T} \int_t^r\cE(\hat\beta A)_{u\smallertext{-}}\d A_u \d\langle\widetilde\eta^{\P,n}\rangle^{(\F_\smalltext{+},\P)}_r \bigg] \nonumber\\
		&= \E^\P\big[\cE(\hat\beta A)_{t \land T}\langle\widetilde\eta^{\P,n}\rangle^{(\F_\smalltext{+},\P)}_{t_\smalltext{n} \land T}\big] + \hat\beta \E^\P\bigg[\int_{t}^{t_\smalltext{n}\land T} \cE(\hat\beta A)_{u\smallertext{-}} \int_{u\smallertext{-}}^{t_\smalltext{n} \land T} \d\langle\widetilde\eta^{\P,n}\rangle^{(\F_\smalltext{+},\P)}_r \d A_u \bigg] \nonumber\\
		&= \E^\P\big[\cE(\hat\beta A)_{t \land T}\langle\widetilde\eta^{\P,n}\rangle^{(\F_\smalltext{+},\P)}_{t_\smalltext{n} \land T} \big] \\
		&\quad+ \hat\beta \E^\P\bigg[\int_{t}^{t_\smalltext{n}\land T} \cE(\hat\beta A)_{u\smallertext{-}} \E^\P\bigg[\int_{u\smallertext{-}}^{t_\smalltext{n} \land T} \d\langle\widetilde\eta^{\P,n}\rangle^{(\F_\smalltext{+},\P)}_r \bigg| \cF_{u\smallertext{-}} \bigg] \d A_u \bigg] \nonumber\\
		&= \E^\P\bigg[\cE(\hat\beta A)_{t \land T} \big(\widetilde\cY^{\P,n}_{t \land T} - \zeta^{n}\big)^2 + \hat\beta \int_{t}^{t_\smalltext{n}\land T} \cE(\hat\beta A)_{u\smallertext{-}} \E^\P\big[\big(\zeta^{n}\big)^2 - \big(\widetilde\cY^{\P,n}_{u\smallertext{-}}\big)^2\big|\cF_{u\smallertext{-}}\big]\d A_u \bigg] \nonumber\\
		&= \E^\P\bigg[ \cE(\hat\beta A)_{t \land T} \big(\widetilde\cY^{\P,n}_{t \land T} - \zeta^{n}\big)^2 + \hat\beta \int_{t}^{t_\smalltext{n}\land T} \cE(\hat\beta A)_{u\smallertext{-}} \big(\big(\zeta^{n}\big)^2 - \big(\widetilde\cY^{\P,n}_{u\smallertext{-}}\big)^2\big)\d A_u \bigg].
	\end{align*}
	We rearrange the terms, use $\cE(\hat\beta A) = \cE(\hat\beta A)_{\smallertext{-}}(1+\hat\beta \Delta A) \leq \cE(\hat\beta A)_{\smallertext{-}}(1+\hat\beta\Phi)$, and find
	\begin{align}\label{eq::convergence_terminal_eta}
		\E^\P\bigg[ \int_t^{t_n \land T} \cE(\hat\beta A)_r \d\langle\widetilde\eta^{\P,n}\rangle^{(\F_\smalltext{+},\P)}_r \bigg] 
		+ & \frac{\hat\beta}{(1+\hat\beta \Phi)} \E^\P\bigg[\int_{t}^{t_\smalltext{n}\land T} \cE(\hat\beta A)_{u} \big(\widetilde\cY^{\P,n}_{u\smallertext{-}}\big)^2\d A_u \bigg] \nonumber\\
		& \leq \E^\P\bigg[ \int_t^{t_n \land T} \cE(\hat\beta A)_r \d\langle\widetilde\eta^{\P,n}\rangle^{(\F_\smalltext{+},\P)}_r 
		+ \hat\beta \int_{t}^{t_\smalltext{n}\land T} \cE(\hat\beta A)_{u\smallertext{-}} \big(\widetilde\cY^{\P,n}_{u\smallertext{-}}\big)^2\d A_u \bigg] \nonumber\\
		& = \E^\P\bigg[\cE(\hat\beta A)_{t \land T} \big(\widetilde\cY^{\P,n}_{t \land T} - \zeta^{n}\big)^2 + \hat\beta \int_{t}^{t_\smalltext{n}\land T} \cE(\hat\beta A)_{u\smallertext{-}} \big(\zeta^{n}\big)^2\d A_u \bigg] \nonumber\\
		&= \E^\P\bigg[\cE(\hat\beta A)_{t \land T} \big(\widetilde\cY^{\P,n}_{t \land T} - \zeta^{n}\big)^2 + \big(\zeta^{n}\big)^2 \big(\cE(\hat\beta A)_{t_\smalltext{n} \land T} - \cE(\hat\beta A)_{t \land T}\big)  \bigg] \nonumber\\
		&= \E^\P\Bigg[\cE(\hat\beta A)_{t \land T}(\widehat\cY^\smallertext{+}_{t \land T})^2\Bigg(\E^\P\bigg[\frac{\cE(\hat\beta A)^{1/2}_{t \land T}}{\cE(\hat\beta A)^{1/2}_{t_\smalltext{n} \land T}}\bigg|\cF_{t+}\bigg] - \frac{\cE(\hat\beta A)^{1/2}_{t \land T}}{\cE(\hat\beta A)^{1/2}_{t_\smalltext{n} \land T}}\Bigg)^2  \nonumber\\
		&\quad+ \cE(\hat\beta A)_{t \land T}(\widehat\cY^\smallertext{+}_{t \land T})^2 \bigg(1 - \frac{\cE(\hat\beta A)_{t \land T}}{\cE(\hat\beta A)_{t_\smalltext{n} \land T}}\bigg) \Bigg].
	\end{align}

	Similarly
	\begin{align}\label{eq::convergence_terminal_y}
		\E^\P\bigg[ \int_t^{t_n \land T} \cE(\hat\beta A)_r \big|\alpha_r\widetilde\cY^{\P,n}_r\big|^2 \d C_r \bigg]
		&\leq \E^\P\bigg[\int_{t}^{t_\smalltext{n} \land T} \cE(\hat\beta A)_r (\zeta^{n})^2 \d A_r \bigg] \nonumber\\
		&= \E^\P\bigg[(\zeta^{n})^2 \big(\cE(\hat\beta A)_{t_\smalltext{n} \land T} - \cE(\hat\beta A)_{t \land T}\big) \bigg] \nonumber\\
		&= \E^\P\bigg[ \cE(\hat\beta A)_{t \land T} (\widehat\cY^\smallertext{+}_{t \land T})^2 \bigg(1 - \frac{\cE(\hat\beta A)_{t \land T}}{\cE(\hat\beta A)_{t_\smalltext{n} \land T}}\bigg) \bigg],
	\end{align}
	and
	\begin{align}\label{eq::cond_equality_terminal}
			\E^\P\Big[\cE(\hat\beta A)_{t_\smalltext{n} \land T}  \big| \widehat\cY_{t_\smalltext{n} \land T} - \zeta^{n} \big|^2 \Big] 
			= \E^\P\bigg[\Big( \cE(\hat\beta A)^{1/2}_{t_\smalltext{n} \land T} \widehat\cY_{t_\smalltext{n} \land T} - \cE(\hat\beta A)_{t \land T}^{1/2}\widehat{\cY}^\smallertext{+}_{t\land T} \Big)^2 \bigg].
	\end{align}
	Now we substitute \eqref{eq::convergence_terminal_eta}--\eqref{eq::cond_equality_terminal} into \eqref{eq::convergence_terminal_time_t}, and end up with
	\begin{align}\label{eq::convergence_terminal_final}
		\lim_{n \rightarrow \infty}\E^\P\Big[ \cE(\hat\beta A)_{t \land T} \big|\cY^\P_t(t_n \land T,\widehat\cY_{t_\smalltext{n} \land T})  - \widetilde\cY^\P_t(t_n \land T,\zeta^{n})\big|^2 \Big] = 0.
	\end{align}
	Here, we use the integrability established at the beginning together with dominated convergence. This yields \eqref{eq::stability_terminal_condition} along a suitable subsequence.

	\medskip
	We now go back to proving the representation \eqref{eq::aggregation}. Let $\overline{\P} \in \fP_0(\cG_{t\smallertext{+}},\P)$, and let $(t_n)_{n \in \N}$ be a sequence of dyadic numbers converging to $t$ from above. We write $\cY^{\bar{\P}}$ for $\cY^{\bar{\P}}(T,\xi)$. Since
	\begin{equation*}
		\E^{\bar{\P}}\big[ \cY^{\bar{\P}}_{t_\smalltext{n}} \big|\cF_{t_\smalltext{n}}\big] \leq \widehat\cY_{t_\smalltext{n}}, \; \text{$\overline{\P}$--a.s.}, \; n \in \N,
	\end{equation*}
	and
	\begin{equation*}
		\E^{\bar{\P}}\big[|\E^{\bar{\P}}[\cY^{\bar{\P}}_{t_\smalltext{n}}|\cF_{t_\smalltext{n}}] -\cY^{\bar{\P}}_{t}| \big] \leq  \E^{\bar{\P}}\big[|\cY^{\bar{\P}}_{t_\smalltext{n}} - \cY^{\bar{\P}}_{t}| \big] \xrightarrow{n\to\infty} 0,
	\end{equation*}
	we obtain (up to choosing a subsequence of $(t_n)_{n \in \N}$ if necessary) that
	\begin{equation*}
		\cY^{\bar{\P}}_{t} \leq \widehat\cY^\smallertext{+}_{t}, \; \text{${\overline{\P}}$--a.s.}
	\end{equation*}
	Since both sides are $\cG_{t\smallertext{+}}$--measurable, the inequality also holds $\P$--a.s., and since $\bar{\P}$ was arbitrary, we obtain
	\begin{equation*}
		\underset{\P^\smalltext{\prime} \in \fP_\smalltext{0}(\cG_{\smalltext{t}\tinytext{+}},\P)}{{\esssup}^\P} \cY^{\P^\smalltext{\prime}}_t \leq \widehat\cY^\smallertext{+}_t, \; \textnormal{$\P$--a.s.}
	\end{equation*}

	We turn to the reverse inequality. 
	Recall from \eqref{eq::dynamic_programming_principle2} that
	\begin{equation*}
		\widehat\cY_{t_\smalltext{n}} = \underset{\bar{\P} \in \fP_\smalltext{0}(\cF_{t_\tinytext{n}},\P)}{{\esssup}^\P} \E^{\bar{\P}}\big[\cY^{\overline{\P}}_{t_\smalltext{n}}\big|\cF_{t_\smalltext{n}}\big], \; \textnormal{$\P$--a.s.}, \; n \in \N.
	\end{equation*}
	Since $\cY^\P_{t_\smalltext{n}} = \cY^\P_{t_\smalltext{n}\land T}$ is $\cF_T$-measurable, which can be deduced from \cite[Theorem~IV.56]{dellacherie1978probabilities},
	we obtain with \cite[Remark (c), p. 119]{dellacherie1978probabilities} that
	\begin{equation}\label{eq::conditiong_on_F_T}
		\E^{\P}\big[\cY^{\P}_{t_\smalltext{n}}\big|\cF_{t_\smalltext{n}}\big] = \E^{\P}\Big[\E^\P\big[\cY^{\P}_{t_\smalltext{n}}\big|\cF_{T}\big]\Big|\cF_{t_\smalltext{n}}\Big] = \E^{\P}\big[\cY^{\P}_{t_\smalltext{n}}\big|\cF_{t_\smalltext{n} \land T}\big] = \E^{\P}\big[\cY^{\P}_{t_\smalltext{n} \land T}\big|\cF_{t_\smalltext{n} \land T}\big], \; \text{$\P$--a.s.},
	\end{equation}
	and, together with the fact that $\widehat\cY^\smallertext{+} = \widehat\cY^\smallertext{+}_{\cdot \land T}$, then also
	\begin{align*}
		\widehat\cY_{t_\smalltext{n} \land T} 
		= \widehat\cY_{t_\smalltext{n}} 
		= \underset{\bar{\P} \in \fP_\smalltext{0}(\cF_{t_\tinytext{n}},\P)}{{\esssup}^\P} \E^{\bar{\P}}\big[\cY^{\bar{\P}}_{t_\smalltext{n}}\big|\cF_{t_\smalltext{n}}\big]
		= \underset{\bar{\P} \in \fP_\smalltext{0}(\cF_{t_\tinytext{n}},\P)}{{\esssup}^\P} \E^{\bar{\P}}\big[\cY^{\bar{\P}}_{t_\smalltext{n} \land T}\big|\cF_{t_\smalltext{n} \land T}\big], \; \textnormal{$\P$--a.s.}
	\end{align*}
	
	\medskip
	Suppose, for the moment, that the collection
	\begin{equation}\label{eq::upward_directed}
		\Big\{\E^{\bar{\P}}\big[\cY^{\bar{\P}}_{t_\smalltext{n}}\big|\cF_{t_\smalltext{n}}\big] : \overline{\P} \in \fP_0(\cF_{t_\smalltext{n}},\P) \Big\},
	\end{equation}
	is $\P$--upward directed for each $n \in \N$. We can then choose sequences $(\P^m_n)_{m \in \N} \subseteq \fP_0(\cF_{t_\smalltext{n}},\P)$ (see \cite[Proposition VI-1-1, page 121]{neveu1975discrete}) such that 
	\begin{equation*}
		\E^\P\big[\cY^\P_{t_\smalltext{n}}\big|\cF_{t_\smalltext{n}}\big] \leq \E^{\P^\smalltext{m}_\smalltext{n}}\big[\cY^{\P^\smalltext{m}_\smalltext{n}}_{t_\smalltext{n}}\big|\cF_{t_\smalltext{n}}\big] \leq \E^{\P^{\smalltext{m}\smalltext{+}\smalltext{1}}_\smalltext{n}}\big[\cY^{\P^{\smalltext{m}\smalltext{+}\smalltext{1}}_\smalltext{n}}_{t_\smalltext{n}}\big|\cF_{t_\smalltext{n}}\big] \xrightarrow{m\rightarrow\infty}\widehat\cY_{t_\smalltext{n}}, \; \text{$\P$--a.s.}, \; n \in \N
	\end{equation*}
	Since
	\begin{equation*}
		\E^{\P}\big[\cY^{\P}_{t_\smalltext{n}}\big|\cF_{t_\smalltext{n}}\big] \leq \E^{\P^\smalltext{m}_\smalltext{n}}\big[\cY^{\P^\smalltext{m}_\smalltext{n}}_{t_\smalltext{n}}\big|\cF_{t_\smalltext{n}}\big] = \E^{\P^\smalltext{m}_\smalltext{n}}\big[\cY^{\P^\smalltext{m}_\smalltext{n}}_{t_\smalltext{n} \land T}\big|\cF_{t_\smalltext{n} \land T}\big] \leq \widehat\cY_{t_\smalltext{n}}, \; \text{$\P$--a.s.},
	\end{equation*}
	it follows from the integrability established at the beginning of this proof, by dominated convergence, that
	\begin{equation}\label{eq::m_n_for_stability}
		\lim_{m \rightarrow \infty}\E^\P\Big[\cE(\hat\beta A)_{t_\smalltext{n}\land T}\big|\E^{\P^m_n}\big[\cY^{\P^m_n}_{t_n \land T}\big|\cF_{t_n \land T}\big] - \widehat\cY_{t_n \land T}\big|^2\Big] = 0.
	\end{equation}
	Then, since $\P^m_n = \P$ on $\cF_{t_\smalltext{n}}$, and thus also $\cG_{t\smallertext{+}}$, we find by choosing a suitable subsequence of $(t_n)_{n \in \N}$ and then of each family $(\P^m_n)_{m \in \N}$ if necessary, that
	\begin{align}\label{eq::lim_upward_directed}
		\widehat\cY^\smallertext{+}_{t} = \widehat\cY^\smallertext{+}_{t \land T} = \lim_{n \rightarrow\infty} \cY^\P_{t \land T}(t_n \land T,\widehat\cY_{t_\smalltext{n} \land T}) 
		&= \lim_{n \rightarrow\infty} \cY^\P_{t \land T}\Big(t_n \land T,\lim_{m \rightarrow \infty}\E^{\P^\smalltext{m}_\smalltext{n}}\big[\cY^{\P^\smalltext{m}_\smalltext{n}}_{t_\smalltext{n} \land T}\big|\cF_{t_\smalltext{n} \land T}\big]\Big) \nonumber\\
		&= \lim_{n \rightarrow\infty} \lim_{m \rightarrow \infty} \cY^\P_{t \land T}\Big(t_n \land T,\E^{\P^\smalltext{m}_\smalltext{n}}\big[\cY^{\P^\smalltext{m}_\smalltext{n}}_{t_\smalltext{n} \land T}\big|\cF_{t_\smalltext{n} \land T}\big]\Big) \nonumber\\
		&= \lim_{n \rightarrow\infty} \lim_{m \rightarrow \infty} \cY^{\P^\smalltext{m}_\smalltext{n}}_{t \land T}\Big(t_n \land T,\E^{\P^\smalltext{m}_\smalltext{n}}\big[\cY^{\P^\smalltext{m}_\smalltext{n}}_{t_\smalltext{n} \land T}\big|\cF_{t_\smalltext{n} \land T}\big]\Big) \nonumber\\
		&= \lim_{n \rightarrow\infty} \lim_{m \rightarrow \infty} \cY^{\P^\smalltext{m}_\smalltext{n}}_{t \land T}\big(t_n \land T,\cY^{\P^\smalltext{m}_\smalltext{n}}_{t_\smalltext{n} \land T}\big) \nonumber\\
		&= \lim_{n \rightarrow\infty} \lim_{m \rightarrow \infty} \cY^{\P^\smalltext{m}_\smalltext{n}}_t 
		\leq \underset{\bar\P \in \fP_\smalltext{0}(\cG_{\smalltext{t}\tinytext{+}},\P)}{{\esssup}^\P} \cY^{\bar\P}_t, \; \text{$\P$--a.s.}
	\end{align}
	Here, the first equality follows from \eqref{eq::stability_terminal_condition}, the third equality follows from the stability result of BSDEs in \Cref{cor::stability} together with \eqref{eq::m_n_for_stability}, the fourth equality follows from the fact that $\P^m_n = \P$ on $\cF_{t_\smalltext{n}}$, and the fifth and sixth equalities follow from \Cref{lem::solv_bsde_cond}.

	\medskip
	It remains to show that the family \eqref{eq::upward_directed} is $\P$--upward directed. Let $(\P_1,\P_2) \in (\fP_0(\cF_{t_\smalltext{n}},\P))^2$, and define
	\begin{equation*}
		\Q(\omega;A) \coloneqq\P^{t_\smalltext{n},\omega}_1[A]\1_{B}(\omega) + \P^{t_\smalltext{n},\omega}_2[A]\1_{B^\smalltext{c}}(\omega), \; \omega \in \Omega,
	\end{equation*}
	where
	\begin{equation*}
		B \coloneqq \big\{\E^{\P_\smalltext{1}}[\cY^{\P_\smalltext{1}}_{t_\smalltext{n}}|\cF_{t_\smalltext{n}}] > \E^{\P_\smalltext{2}}[\cY^{\P_\smalltext{2}}_{t_\smalltext{n}}|\cF_{t_\smalltext{n}}]\big\} \in \cF_{t_\smalltext{n}}.
	\end{equation*}
	Then $\Q(\omega;\d\omega^\prime)$ is a kernel on $(\Omega,\cF)$ given $(\Omega,\cF_{t_\smalltext{n}})$ satisfying $\Q(\omega;\d\omega^\prime) \in \fP(t_n,\omega)$ for $\P$--a.e. $\omega \in \Omega$ by \Cref{ass::probabilities2}.$(ii)$.
	The probability measure
	\begin{align*}
		\overline{\P}[A] \coloneqq & \iint \big(\1_A\big)^{t_\smalltext{n},\omega}(\omega^\prime)\Q(\omega;\d\omega^\prime)\P(\d\omega)
				= \E^\P[\P_1[A|\cF_{t_\smalltext{n}}]\1_B + \P_2[A|\cF_{t_\smalltext{n}}]\1_{B^\smalltext{c}}] = \P_1[A \cap B] + \P_2[A \cap B^c], \; A \in \cF,
	\end{align*}
	is then an element of $\fP_0$ by \Cref{ass::probabilities2}.$(iii)$. Since $\overline{\P}$ agrees with $\P$ on $\cF_{t_\smalltext{n}}$ and $\bar{\P}^{t_\smalltext{n},\omega}(\d\omega^\prime) = \Q(\omega,\d\omega^\prime)$ for $\P$--a.e. $\omega \in \Omega$, we obtain
	\begin{align*}
		\E^{\bar{\P}}\big[\cY^{\bar{\P}}_{t_n}\big|\cF_{t_\smalltext{n}}\big](\omega)
		&= \E^{\Q(\omega)}[\cY^{t_n,\omega,\Q(\omega)}_{0}((T-t_n\land T)^{t_\smalltext{n},\omega},\xi^{t_\smalltext{n},\omega})] \\
		&= \E^{\Q(\omega)}[\cY^{t_n,\omega,\Q(\omega)}_{0}((T-t_n\land T)^{t_\smalltext{n},\omega},\xi^{t_\smalltext{n},\omega})]\1_B(\omega) + \E^{\Q(\omega)}[\cY^{t_n,\omega,\Q(\omega)}_{0}((T-t_n\land T)^{t_\smalltext{n},\omega},\xi^{t_\smalltext{n},\omega})]\1_{B^\smalltext{c}}(\omega) \\
		&= \E^{\P^{{\smalltext{t}_\tinytext{n}}\smalltext{,}\smalltext{\omega}}_\smalltext{1}}[\cY^{t_n,\omega,\P^{{\smalltext{t}_\tinytext{n}}\smalltext{,}\smalltext{\omega}}_\smalltext{1}}_{0}((T-t_n\land T)^{t_\smalltext{n},\omega},\xi^{t_\smalltext{n},\omega})]\1_B(\omega) + \E^{\P^{{\smalltext{t}_\tinytext{n}}\smalltext{,}\smalltext{\omega}}_\smalltext{2}}[\cY^{t_n,\omega,\P^{{\smalltext{t}_\tinytext{n}}\smalltext{,}\smalltext{\omega}}_\smalltext{2}}_{0}((T-t_n\land T)^{t_\smalltext{n},\omega},\xi^{t_\smalltext{n},\omega})]\1_{B^\smalltext{c}}(\omega) \\
		&= \E^{\P_\smalltext{1}}[\cY^{\P_\smalltext{1}}_{t_\smalltext{n}}|\cF_{t_\smalltext{n}}](\omega)\1_B(\omega) + \E^{\P_\smalltext{2}}[\cY^{\P_\smalltext{2}}_{t_\smalltext{n}}|\cF_{t_\smalltext{n}}](\omega)\1_{B^\smalltext{c}}(\omega) \\
		&\geq \max\big\{\E^{\P_\smalltext{1}}[\cY^{\P_\smalltext{1}}_{t_\smalltext{n}}|\cF_{t_\smalltext{n}}](\omega),\E^{\P_\smalltext{2}}[\cY^{\P_\smalltext{2}}_{t_\smalltext{n}}|\cF_{t_\smalltext{n}}](\omega)\big\}, \; \text{$\P$--a.e. $\omega \in \Omega$.}
	\end{align*}
	Here, we used \Cref{lem::conditioning_bsde2} in the first and fourth equalities.
	Since $\E^{\bar{\P}}\big[\cY^{\bar{\P}}_{t_\smalltext{n}}\big|\cF_{t_\smalltext{n}}\big]$ is an element of the family \eqref{eq::upward_directed}, this completes the proof.
\end{proof}

	We turn to the proof of the last assertion.

	\begin{proof}[Proof of \Cref{thm::down-crossing}.$(iii)$]
		For simplicity, we write $\widehat{\cY}$ and $\widehat{\cY}^\smallertext{+}$ instead of $\widehat{\cY}(T,\xi)$ and $\widehat{\cY}^\smallertext{+}(T,\xi)$, respectively. We fix $ \P \in \fP_0 $ and first prove the result for deterministic times. By \Cref{lem::conditioning_bsde2} and \Cref{lem::stopping_value_function}, we obtain
		\begin{equation*}
			\E^\P\big[ \cY^\P_s(t \land T,\widehat\cY_{t \land T}) \big|\cF_s\big](\bar{\omega}) 
			= \E^{\P^{\smalltext{s}\smalltext{,}\smalltext{\bar{\omega}}}} [\cY^{s,\bar{\omega},\P^{\smalltext{s}\smalltext{,}\smalltext{\bar{\omega}}}}_0((t \land T - s \land T)^{s,\bar{\omega}},\widehat\cY^{s,\bar{\omega}}_{t \land T^{\smalltext{s}\smalltext{,}\smalltext{\bar{\omega}}}})] \leq \widehat\cY_{s}(\bar{\omega}), \; \text{$\P$--a.e. $\bar{\omega} \in \Omega$}, \; 0 \leq s \leq t < \infty,
		\end{equation*}
		since $\P^{s,\bar{\omega}} \in \fP(s,\bar{\omega})$ for $\P$--a.e. $\bar{\omega} \in \Omega$ by \Cref{ass::probabilities2}.$(ii)$.
		Now let us fix $0 \leq s < t < \infty$, and let $(s_m)_{m \in \N}$ and $(t_n)_{n \in \N}$ be sequences of dyadic numbers converging to $s$ and $t$ from above, respectively, with $s_m < t$ for all $m$. Since $\cY^\P(t_n \land T,\widehat\cY_{t_\smalltext{n} \land T}) \in \cS^2_{t_\smalltext{n} \land T,\hat\beta}(\F_\smallertext{+},\P)$ by construction, we find by dominated convergence that
		\begin{equation*}
			\E^\P\big[|\E^\P[\cY^\P_{s_\smalltext{m}}(t_n \land T,\widehat\cY_{t_\smalltext{n} \land T})|\cF_{s_\smalltext{m}}] -\cY^\P_{s}(t_n \land T,\widehat\cY_{t_\smalltext{n} \land T})|^2 \big] \leq  \E^\P\big[|\cY^\P_{s_\smalltext{m}}(t_n \land T,\widehat\cY_{t_\smalltext{n} \land T}) - \cY^\P_{s}(t_n \land T,\widehat\cY_{t_\smalltext{n} \land T})|^2 \big] \xrightarrow{m\rightarrow\infty} 0.
		\end{equation*}
		By considering a $\P$--a.s. convergent subsequence if necessary, we find, together with $\E^\P\big[ \cY^\P_{s_\smalltext{m}}(t_n \land T,\widehat\cY_{t_\smalltext{n} \land T}) \big|\cF_{s_\smalltext{m}}\big] \leq \widehat\cY_{s_\smalltext{m}}$, $\P$--a.s., for all $(m, n) \in \N^2$, that
		\begin{equation*}
			\cY^\P_{s}(t_n \land T,\widehat\cY_{t_\smalltext{n} \land T}) \leq \widehat\cY^\smallertext{+}_{s}, \; \text{$n \in \N$}, \; \text{$\P$--a.s.}
		\end{equation*}
		By \eqref{eq::stability_terminal_condition}, and after passing to a subsequence which we still denote by $(t_n)_{n \in \N}$, it follows that
		\begin{align*}
			\cY^\P_s(t \land T,\widehat\cY^\smallertext{+}_{t \land T}) = \lim_{n \rightarrow \infty}\cY^\P_s(t_n \land T,\widehat\cY_{t_\smalltext{n} \land T}) \leq \widehat\cY^\smallertext{+}_s, \; \text{$\P$--a.s.}
		\end{align*}
		The inequality above holds immediately for $t = \infty$ by assertion $(ii)$ since $\widehat{\cY}^\smallertext{+}_T = \xi$, and it follows automatically for $s = t$ as well since $\widehat{\cY}^\smallertext{+}_{t \land T} = \widehat{\cY}^\smallertext{+}_t$. Thus, we have established that
		\begin{align*}
			\cY^\P_{s \land t \land T}(t \land T,\widehat\cY^\smallertext{+}_{t \land T}) = \cY^\P_s(t \land T,\widehat\cY^\smallertext{+}_{t \land T}) \leq \widehat\cY^\smallertext{+}_s = \widehat\cY^\smallertext{+}_{s \land T}, \; \text{$\P$--a.s.}, \; 0 \leq s \leq t \leq \infty.
		\end{align*}

		We now show that the above inequality remains valid when $t$ is replaced by an $\F_\smallertext{+}$--stopping time $\tau$ and $s$ is replaced by $\tau \land s$. We proceed similarly to the proof of \cite[Lemma 2.1]{chen2001continuous}. We begin by fixing $s \in [0,\infty]$ and initially assume that $\tau$ takes at most finitely many values $0 = t_0 < t_1 < \cdots < t_{n-1} < t_n \leq \infty$. Suppose $t_{n-1} \leq s$. In this case, $\tau = \tau\1_{\{\tau\leq t_{\smalltext{n}\smalltext{-}\smalltext{1}}\}} + t_n\1_{\{\tau = t_\smalltext{n}\}}$, where both $\1_{\{\tau\leq t_{\smalltext{n}\smalltext{-}\smalltext{1}}\}}$ and $\1_{\{\tau = t_\smalltext{n}\}} = \1_{\Omega\setminus\{\tau \leq t_{\smalltext{n}\smalltext{-}\smalltext{1}}\}}$ are $\cF_{s\smallertext{+}}$-measurable. 
		It follows from \Cref{prop::stability} that
		\begin{equation}\label{eq::locality_bsde_tau_tn}
			\cY^\P_s(\tau\land T, \widehat{\cY}^\smallertext{+}_{\tau\land T}) = \cY^\P_s(t_n\land T, \widehat{\cY}^\smallertext{+}_{t_\smalltext{n}\land T}), \; \textnormal{$\P$--a.s. on $\{\tau = t_n\}$}.
		\end{equation}
		Indeed, the flow property of solutions to BSDEs yields, for $T^\prime \coloneqq \tau \land t_n \land T$ satisfying $s\land T^\prime \leq T^\prime \leq (\tau \land T) \land (t_n \land T)$,
		\begin{gather*}
			\cY^\P_{s\land \tau\land t_n \land T}(\tau\land T, \widehat{\cY}^\smallertext{+}_{\tau\land T}) 
			= \cY^\P_{s\land T^\smalltext{\prime}}(T^\prime, \cY^\P_{T^\prime}(\tau\land T,\widehat{\cY}^\smallertext{+}_{\tau\land T})), \\	
			\cY^\P_{s\land \tau\land t_n \land T}(t_n\land T, \widehat{\cY}^\smallertext{+}_{t_n\land T}) 
			= \cY^\P_{s\land T^\smalltext{\prime}}(T^\prime, \cY^\P_{T^\prime}(t_n\land T,\widehat{\cY}^\smallertext{+}_{t_n\land T})), \; \textnormal{$\P$--a.s.}
		\end{gather*}
		By \Cref{prop::stability}, there exists $\mathfrak{C} \in (0,\infty)$ such that
		\begin{align*}
			\cE(\hat{\beta}A)_{s\land T^\smalltext{\prime}} & \big|\cY^\P_{s\land T^\smalltext{\prime}}(T^\prime, \cY^\P_{T^\prime}(\tau\land T,\widehat{\cY}^\smallertext{+}_{\tau\land T})) - \cY^\P_{s\land T^\smalltext{\prime}}(T^\prime, \cY^\P_{T^\prime}(t_n\land T,\widehat{\cY}^\smallertext{+}_{t_n\land T}))\big|^2 \\
			& \leq \mathfrak{C}\E^\P\bigg[\cE(\hat{\beta}A)_{T^\smalltext{\prime}} \big| \cY^\P_{T^\prime}(\tau\land T,\widehat{\cY}^\smallertext{+}_{\tau\land T})- \cY^\P_{T^\prime}(t_n\land T,\widehat{\cY}^\smallertext{+}_{t_n\land T})\big|^2 \bigg|\cF_{s\smallertext{+}}\bigg], \; \textnormal{$\P$--a.s.},
		\end{align*}
		and multiplying both sides by $\1_{\{\tau = t_n\}}$, which is in $\cF_{s\smallertext{+}}$, yields \eqref{eq::locality_bsde_tau_tn}. We then obtain
		\begin{align}
			\cY^\P_{s \land \tau \land T}\big(\tau \land T,\widehat\cY^\smallertext{+}_{\tau \land T}\big) = \cY^\P_{s}\big(\tau \land T,\widehat\cY^\smallertext{+}_{\tau \land T}\big) 
			&= \cY^\P_{s}\big(\tau \land T,\widehat\cY^\smallertext{+}_{\tau \land T}\big)\1_{\{\tau\leq t_{\smalltext{n}\smalltext{-}\smalltext{1}}\}} + \cY^\P_{s}\big(\tau \land T,\widehat\cY^\smallertext{+}_{\tau \land T}\big)\1_{\{\tau = t_{\smalltext{n}}\}} \nonumber\\
			&= \cY^\P_{s}\big(\tau \land T,\widehat\cY^\smallertext{+}_{\tau \land T}\big)\1_{\{\tau\leq t_{\smalltext{n}\smalltext{-}\smalltext{1}}\}} + \cY^\P_{s}\big(t_n \land T,\widehat\cY^\smallertext{+}_{t_\smalltext{n} \land T}\big)\1_{\{\tau = t_{\smalltext{n}}\}} \nonumber\\
			&= \widehat\cY^\smallertext{+}_{\tau \land T}\1_{\{\tau\leq t_{\smalltext{n}\smalltext{-}\smalltext{1}}\}} + \cY^\P_{s}\big(t_n \land T,\widehat\cY^\smallertext{+}_{t_\smalltext{n} \land T}\big)\1_{\{\tau = t_\smalltext{n}\}}, \nonumber\\
			&\leq \widehat\cY^\smallertext{+}_{\tau \land T}\1_{\{\tau\leq t_{\smalltext{n}\smalltext{-}\smalltext{1}}\}} + \widehat\cY^\smallertext{+}_{s \land t_\smalltext{n} \land T}\1_{\{\tau = t_\smalltext{n}\}} = \widehat\cY^\smallertext{+}_{s \land \tau \land T}, \; \text{$\P$--a.s.}
		\end{align}
		
		Suppose now that $t_{k-1}\leq s <t_{k}$ for some $k \in \{1,\ldots,n-1\}$, and note that each $t_j \land \tau$ for $j\in\{0,\dots,n\}$, is an $\F_\smallertext{+}$--stopping time with values in $\{t_0,t_1,\ldots, t_j\}$. We repeatedly use the comparison principle and time-consistency of the BSDEs (see \Cref{lem::solv_bsde_cond}) together with the above arguments and find
		\begin{align*}
			\cY^\P_s\big(\tau \land T,\widehat\cY^\smallertext{+}_{\tau \land T}\big)
			= \cY^\P_{s \land \tau \land T}\big(\tau \land T,\widehat\cY^\smallertext{+}_{\tau \land T}\big) 
			&= \cY^\P_{s \land \tau \land T}\big(t_{n-1} \land \tau\land T,\cY^\P_{t_{\smalltext{n}\smalltext{-}\smalltext{1}} \land \tau\land T}(\tau \land T,\widehat\cY^\smallertext{+}_{\tau \land T})\big) \\
			&\leq \cY^\P_{s \land \tau \land T}\big(t_{n-1} \land \tau \land T,\widehat\cY^+_{t_{n-1} \land \tau \land T}\big) \\
			&= \cY^\P_{s \land \tau \land T}\big(t_{n-2} \land \tau \land T,\cY^\P_{t_{n-2} \land \tau \land T}(t_{n-1} \land \tau \land T,\widehat\cY^\smallertext{+}_{t_{\smalltext{n}\smalltext{-}\smalltext{1}} \land \tau\land T})\big) \\
			&\leq \cY^\P_{s \land \tau \land T}\big(t_{n-2} \land \tau \land T,\widehat\cY^{\smallertext{+}}_{t_{\smalltext{n}\smalltext{-}\smalltext{2}} \land \tau \land T}\big) \\
			& \leq \cdots \\
			&\leq \cY^\P_{s \land \tau \land T}\big(t_{k} \land \tau \land T,\widehat\cY^\smallertext{+}_{t_{\smalltext{k}} \land \tau \land T}\big) \leq \widehat\cY^\smallertext{+}_{s \land \tau \land T}, \; \text{$\P$--a.s.}
		\end{align*}
		We therefore obtain
		\begin{equation*}
			\cY^\P_s\big(\tau \land T,\widehat\cY^\smallertext{+}_{\tau \land T}\big) 
			= \cY^\P_{s \land \tau \land T}\big(\tau \land T,\widehat\cY^\smallertext{+}_{\tau \land T}\big) 
			\leq \widehat\cY^\smallertext{+}_{s \land \tau \land T} = \widehat\cY^\smallertext{+}_{s \land \tau}, \; s \in [0,\infty], \; \text{$\P$--a.s.}
		\end{equation*}
		Now, suppose that $\tau$ is a general $\F_\smallertext{+}$--stopping time. \textcolor{black}{Let $\D^n_\smallertext{+} \coloneqq \{k2^{-n} : k \in\{0,1,\ldots,2^{2n}\}\}$},	 and define
		\begin{equation*}
			\tau^n \coloneqq \inf\big\{t \in \D^n_\smallertext{+}: t > \tau + 1/2^n\big\} = \sum_{k = 1}^{2^{2n}} k 2^{-n} \1_{\{(k-1)2^{-n} \leq \tau + 1/2^n < k2^{-n}\}} + \infty \1_{\{\tau + 1/2^n \geq 2^n\}}.
		\end{equation*} 
		It follows that $\tau^n$ is an $\F$--predictable stopping time (see \cite[Theorem IV.57.(a), Theorem IV.71.(a), and Comment IV.72]{dellacherie1978probabilities}) that converges decreasingly to $\tau$. Furthermore, each $\tau^n$ takes at most finitely many values in $\D^n_\smallertext{+} \cup \{\infty\}$. From the preceding considerations, we deduce that
		\begin{equation*}
			\cY^\P_{s}\big(\tau^n \land T,\widehat\cY^\smallertext{+}_{\tau^\smalltext{n} \land T}\big) = \cY^\P_{s \land \tau^\smalltext{n} \land T}\big(\tau^n \land T,\widehat\cY^\smallertext{+}_{\tau^\smalltext{n} \land T}\big) \leq \widehat\cY^\smallertext{+}_{s \land \tau^\smalltext{n} \land T} = \widehat\cY^\smallertext{+}_{s \land \tau^\smalltext{n}}, \; s \in [0,\infty], \; n \in \N, \; \text{$\P$--a.s.}
		\end{equation*}
		Next, let us fix $s \in [0,\infty)$. We express
		\begin{align}\label{eq::convergence_tau_n}
			&\cY^\P_{s \land \tau}\big(\tau^n \land T,\widehat\cY^\smallertext{+}_{\tau^\smalltext{n} \land T}\big) - \cY^\P_{s \land \tau}\big(\tau \land T,\widehat\cY^\smallertext{+}_{\tau \land T}\big) \nonumber\\
			&= \cY^\P_{s \land \tau}\big(\tau^n \land T,\widehat\cY^\smallertext{+}_{\tau^\smalltext{n} \land T}\big) - \widetilde\cY^\P_{s \land \tau}\big(\tau^n \land T,\zeta^n\big)  + \widetilde\cY^{\P}_{s \land \tau}\big(\tau \land T,\widetilde\cY^\P_{\tau \land T}(\tau^n \land T,\zeta^n)\big) - \cY^\P_{s \land \tau}\big(\tau \land T,\widehat\cY^\smallertext{+}_{\tau \land T}\big),
		\end{align}
		where
		\[
			\zeta^n \coloneqq \widehat\cY^\smallertext{+}_{\tau \land T}\frac{\cE(\hat\beta A)^{1/2}_{\tau \land T}}{\cE(\hat\beta A)^{1/2}_{\tau^\smalltext{n} \land T}} \1_{\{\cE(\hat{\beta}A)_{\smalltext{\tau}^\tinytext{n}\smalltext{\land}\smalltext{T}} < \infty\}}
		\]
		and $\widetilde{\cY}^\P(\cdot,\cdot)$ now denotes the solution of the BSDE whose generator is $f^\P \1_{\llparenthesis 0, \tau \rrbracket}$. By arguing as in the proof of \Cref{thm::down-crossing}.$(ii)$, more specifically proceeding analogously to the arguments following \eqref{eq::convergence_t_n}, we deduce that both differences on the right-hand side of \eqref{eq::convergence_tau_n} converge $\P$--a.s. to zero along a subsequence $(\tau^{n_\smalltext{k}})_{k \in \N}$ of $(\tau^n)_{n \in \N}$. 
		This yields
		\begin{equation*}
			\cY^\P_{s}\big(\tau \land T,\widehat\cY^\smallertext{+}_{\tau \land T}\big) 
			= \cY^\P_{s \land \tau}\big(\tau \land T,\widehat\cY^\smallertext{+}_{\tau \land T}\big)  
			= \lim_{k \rightarrow \infty}\cY^\P_{s \land \tau}\big(\tau^{n_\smalltext{k}} \land T,\widehat\cY^\smallertext{+}_{\tau^{\smalltext{n}_\tinytext{k}} \land T}\big)
			\leq \lim_{k \rightarrow \infty} \widehat\cY^\smallertext{+}_{s \land \tau \land \tau^{\smalltext{n}_\tinytext{k}}} = \widehat\cY^\smallertext{+}_{s\land \tau}, \; \text{$\P$--a.s.}
		\end{equation*}
		As the above equality holds trivially for $s = \infty$, the right-continuity of $\cY^\P(\tau \land T,\widehat{\cY}^\smallertext{+}_{\tau \land T})$ and $\widehat{\cY}^\smallertext{+}$ yield \eqref{eq::nonlinear_supermartingale_property} for arbitrary $\F_\smallertext{+}$--stopping times $\sigma$ and $\tau$. Since $\G_\smallertext{+}\subseteq\F^\P_\smallertext{+}$, any $\G_\smallertext{+}$--stopping times $\sigma^\prime$ and $\tau^\prime$ are $\P$--a.s. equal to  $\F_\smallertext{+}$--stopping times $\sigma$ and $\tau$, respectively; see \cite[Theorem IV.59]{dellacherie1978probabilities}. Hence, the stated result follows, which concludes the proof.
	\end{proof}

\begin{proof}[Proof of \Cref{cor::optimisation}]
		We begin with \eqref{eq::relation_to_optimisation} and, for simplicity, omit $(T,\xi)$ from the notation. For fixed $\P \in \fP_0$ and $t = 0$, we denote by $(\cY^{\P^\smalltext{m}_\smalltext{n}}_0)_{(m, n) \in \N^2}$ the family used in \eqref{eq::lim_upward_directed}, which is bounded by a $\P$-integrable random variable; this follows from \eqref{eq::implies_uniform_integrability} together with \Cref{prop::stability}. We then apply Fatou's lemma twice together with \eqref{eq::lim_upward_directed} and find 
		\begin{align*}
		\E^\P\big[\widehat{\cY}^\smallertext{+}_0\big]
		= \E^\P\bigg[\lim_{n\rightarrow\infty} \lim_{m \rightarrow \infty}\cY^{\P^\smalltext{m}_\smalltext{n}}_0\bigg]
		&\leq \liminf_{n\rightarrow\infty} \E^\P\bigg[\lim_{m \rightarrow \infty}\cY^{\P^\smalltext{m}_\smalltext{n}}_0\bigg]\\
		&\leq \liminf_{n\rightarrow\infty}\liminf_{m \rightarrow \infty} \E^\P\Big[\cY^{\P^\smalltext{m}_\smalltext{n}}_0\Big]
		= \liminf_{n\rightarrow\infty}\liminf_{m \rightarrow \infty} \E^{\P^\smalltext{m}_\smalltext{n}}\Big[\cY^{\P^\smalltext{m}_\smalltext{n}}_0\Big]
		\leq \sup_{\P\in\fP_0}\E^\P\big[\cY^\P_0\big] = \widehat{\cY}_0.
		\end{align*}
		It remains to take the supremum over $\fP_0$ on the left-hand side. The converse inequality follows immediately from the representation \eqref{eq::aggregation}, since for any $\P \in \fP_0$, we have $\cY^\P_0 \leq \widehat{\cY}^\smallertext{+}_0$, $\P$--a.s., which then yields
		\[
		\widehat{\cY}_0 = \sup_{\P\in\fP_0} \E^\P\big[ \cY^\P_0 \big] \leq \sup_{\P\in\fP_0} \E^\P\big[ \widehat{\cY}^\smallertext{+}_0 \big].
		\]
	This yields \eqref{eq::relation_to_optimisation}.
		
	\medskip
	Suppose now that $\P^\ast \in \fP_0$ satisfies $\widehat{\cY}_0 = \E^{\P^\smalltext{\ast}}\big[\cY^{\P^\smalltext{\ast}}_0\big]$. Since $\cY^{\P^\smalltext{\ast}}_0 \leq \widehat{\cY}^\smallertext{+}_0$, $\P^\ast$--a.s., we have
	\[
		\widehat{\cY}_0 = \E^{\P^\smalltext{\ast}}\big[\cY^{\P^\smalltext{\ast}}_0\big] \leq \E^{\P^\smalltext{\ast}}\big[\widehat{\cY}^\smallertext{+}_0\big] \leq \widehat{\cY}_0,
	\]
	from which $\cY^{\P^\smalltext{\ast}}_0 = \widehat{\cY}^\smallertext{+}_0$, $\P^\ast$--a.s., and \eqref{eq::max_widehat_y_plus} follow.
	
	\medskip
	Lastly, if $\P^\ast\in\fP_0$ satisfies \eqref{eq::max_widehat_y_plus} and $\bar{\P}^\ast \in \fP_0(\cG_{0\smallertext{+}},\P^\ast)$ satisfies $\widehat{\cY}^\smallertext{+}_0 = \cY^{\bar\P^\smalltext{\ast}}_0$, \textnormal{$\P^\ast$--a.s.}, then
	\[
		\widehat{\cY}_0 
		= \sup_{\P\in\fP_\smalltext{0}}\E^\P\big[\widehat{\cY}^\smallertext{+}_0\big] 
		= \E^{\P^\smalltext{\ast}}\big[\widehat{\cY}^\smallertext{+}_0\big] 
		= \E^{\P^\smalltext{\ast}}\big[\cY^{\bar\P^\smalltext{\ast}}_0\big] 
		= \E^{\bar\P^\smalltext{\ast}}\big[\cY^{\bar\P^\smalltext{\ast}}_0\big],
	\]
	where the first equality follows from \eqref{eq::relation_to_optimisation}, and the last one from $\P^\ast = \bar{\P}^\ast$ on $\cG_{0\smallertext{+}}$. This concludes the proof.
\end{proof}

\appendix
\crefalias{section}{appendix}

\section{Raw-filtration stochastic calculus}\label{app::raw_filtration_stochastic_calculus}

In the first part of this section, let
$(\Omega,\cG,\G = (\cG_t)_{t \in [0,\infty)},\P)$
be an arbitrary filtered probability space. We assume that we are given an additional $\sigma$-algebra $\cG_{0\smallertext{-}}$ contained in $\cG_0$ and included in $\G$ whenever necessary. We define $\cG_{\infty\smallertext{+}} \coloneqq \cG_{\infty} \coloneqq \cG_{\infty\smallertext{-}} \coloneqq \sigma\big(\cup_{t \in [0,\infty)}\cG_t\big)$. The right-continuous version $\G_\smallertext{+} = (\cG_{t\smallertext{+}})_{t \in [0,\infty)}$ of $\G$ is defined as $\cG_{t\smallertext{+}} \coloneqq \cap_{s \in (t,\infty)} \cG_s$, with $\cG_{(0\smallertext{+})\smallertext{-}}= \cG_{0\smallertext{-}}$ also added to $\G_\smallertext{+}$ when necessary. Similarly, the left-continuous version $\G_\smallertext{-} \coloneqq (\cG_{t\smallertext{-}})_{t \in [0,\infty)}$ is defined as $\cG_{t\smallertext{-}} \coloneqq \sigma(\cup_{s\in[0,t)}\cG_s)$ for $t \in (0,\infty)$. Then $\cG_\infty = \sigma\left(\cup_{t \in [0,\infty)}\cG_{t\smallertext{+}}\right) =  \sigma\left(\cup_{t \in [0,\infty)}\cG_{t\smallertext{-}}\right)$. We denote by $\cP(\G)$ the $\G$-predictable $\sigma$-algebra on $\Omega \times [0,\infty)$, generated by all real-valued, $\G_\smallertext{-}$-adapted processes that are left-continuous on $(0,\infty)$. Note that then $\cP(\G) = \cP(\G_\smallertext{+})$. 
	
\medskip
For $t \in \{0{-}\}\cup[0,\infty)$, we denote by $\cG^\P_t$ the $\sigma$-algebra generated by $\cG_t$ and the $(\cG,\P)$--null sets, and we then write $\G^\P = (\cG^\P_t)_{t \in [0,\infty)}$, adding $\cG^\P_{0\smallertext{-}}$ to it whenever necessary. We then write $\cP(\G)^\P \coloneqq \cP(\G^\P)$. The $\G$-optional $\sigma$-algebra $\cO(\G)$ on $\Omega\times[0,\infty)$ is generated by all $\G$-adapted, real-valued processes which are right-continuous at zero and c\`adl\`ag on $(0,\infty)$. We denote by $\textnormal{Prog}(\G)$ the progressive $\sigma$-algebra on $\Omega \times [0,\infty)$ consisting of those subsets $A \subseteq \Omega \times [0,\infty)$ such that $\1_A$ is $\G$-progressive, that is, $\Omega \times [0,t] \ni (\omega,s) \longmapsto \1_A(\omega,s) \in \R$ is $\cG_t \otimes \cB([0,t])$-measurable for every $t \in [0,\infty)$.
	
\medskip
For two maps $S: \Omega \longrightarrow [0,\infty]$ and $T : \Omega \longrightarrow [0,\infty]$, we denote by $\llparenthesis S,T\rrbracket$ the stochastic interval $\{(\omega,t) \in \Omega \times [0,\infty) \,|\,  S(\omega) < t \leq T(\omega)\}$. The stochastic intervals $\llbracket S,T\rrparenthesis$, $\llbracket S,T\rrbracket$ and $\llparenthesis S,T\rrparenthesis$ are defined analogously. We denote by $\cG_{S\smallertext{-}}$ the $\sigma$-algebra generated by $\cG_{0\smallertext{-}}$ and all sets of the form $A \cap \{t < S\}$, where $A \in \cG_t$ and $t \in [0,\infty)$. If we define $\H = (\cH_t)_{t \in [0,\infty)}$ by $\cH_t \coloneqq \cG_{t\smallertext{+}}$ and $\cH_{0\smallertext{-}} \coloneqq \cG_{0\smallertext{-}}$, then $\cH_{S\smallertext{-}} = \cG_{S\smallertext{-}}$. If $\llbracket S, \infty \rrparenthesis$ belongs to $\cP(\G)$, then $S$ is referred to as a $\G$-predictable stopping time. 
	
\medskip
In this work, $\G$--stopping time means that $\{S \leq t\} \in \cG_t$ for every $t \in [0,\infty)$. In \cite{weizsaecker1990stochastic}, this is referred to as a `strict stopping time', while in \cite{dellacherie1978probabilities}, it is additionally referred to as an optional time. Note that a $\G$--predictable stopping time is a $\G$--stopping time. A $\G$--stopping time is said to be finite(-valued) if it never attains the value $\infty$. For a $\G$--stopping time $S$, we denote by $\cG_S$ the $\sigma$-algebra consisting of all $A \in \cG_\infty$ for which $A \cap \{S \leq t\} \in \cG_t$ holds for all $t \in [0,\infty)$. If $S$ is an $\H$--stopping time, then $\cH_{S}$ will be denoted by $\cG_{S\smallertext{+}}$.

\medskip
Lastly, a sequence $(\tau_n)_{n \in \N}$ of $\G$--stopping times is a $(\G,\P)$--localising sequence if $\P[\tau_n \uparrow \infty] = 1$, that is, $(\tau_n(\omega))_{n \in \N}$ is a non-decreasing sequence of numbers in $[0,\infty]$ that converges to $\infty$ for $\P$--a.e. $\omega \in \Omega$.

\subsection{Semi-martingales and their characteristics}\label{sec::semimartingales}

We recall facts on semimartingales, characteristics, and stochastic integrals
needed later for integration with respect to the raw filtration. Let $M = (M_t)_{t \in [0,\infty)}$ be a real-valued, right-continuous and $\G$-adapted process. Then $M$ is a $(\G,\P)$--square-integrable martingale if $M$ is a $(\G,\P)$-martingale such that $\sup_{t \in [0,\infty)}|M_{t}|$ is $\P$--square-integrable.
We refer to $M$ as a $(\G,\P)$--locally square-integrable (resp. $(\G,\P)$--local) martingale, if there exists $(\G,\P)$--localising sequence $(\tau_n)_{n \in \N}$ such that, for each $n \in \N$, the stopped process $M_{\cdot\land\tau_n}$ is a $(\G,\P)$--square-integrable (resp. $(\G,\P)$--uniformly integrable) martingale. 

\medskip
In case $M$ is a $(\G,\P)$--locally square-integrable martingale, we denote by $\langle M \rangle = (\langle M \rangle_t)_{t \in [0,\infty)}$ the predictable quadratic variation of $M - M_0$ relative to $(\G,\P)$ (see \cite[Corollary 6.6.3]{weizsaecker1990stochastic}), that is, $\langle M \rangle$ is the, up to $\P$-indistinguishability, unique right-continuous, $\G$-predictable, $\P$--a.s. non-decreasing process such that $M^2-\langle M \rangle$ is a $(\G,\P)$--local martingale. If the filtration and probability measure need to be emphasised, we also write $\langle M\rangle^{(\G,\P)}$; see, however, \Cref{lem::equivalence_loc_square_integrable_martingale}.

\begin{remark}\label{rem::quadratic_variation}\label{rem::G_local_martingale}
	\textnormal{A priori}, \textnormal{\cite[Corollary 6.6.3]{weizsaecker1990stochastic}} gives the existence of a $(\G_\smallertext{+},\P)$--localising sequence $(\tau_n)_{n\in\N}$ such that $M^2_{\cdot\land \tau_n} - \langle M \rangle_{\cdot\land \tau_n}$ is a $(\G,\P)$--uniformly integrable martingale. However, one can find a corresponding $(\G,\P)$--localising sequence of strictly positive $\G$-predictable stopping times$;$ see \textnormal{\Cref{lem::predictable_localisation}}.
\end{remark}

\begin{lemma}\label{lem::predictable_localisation}
	Let $A = (A_t)_{t \in [0,\infty)}$ be a real-valued, right-continuous and non-decreasing, $\G^\P_\smallertext{+}$-predictable process starting at zero. Set $A_\infty \coloneqq \lim_{t \uparrow\uparrow \infty} A_t$. There exists a $(\G,\P)$--localising sequence of strictly positive $\G$--predictable stopping times $(\sigma_n)_{n \in \N}$ such that $A_{\sigma_\smalltext{n}}$ is $\P$--essentially bounded.
\end{lemma}

\begin{proof}
Let $\tau_n \coloneqq \inf\{t \in [0,\infty) : A_t \geq n\}$ for $n \in \N^\ast$. Then $\tau_n > 0$, by right-continuity of $A$, and $\llparenthesis 0, \tau_n \rrparenthesis = \{A < n\}$. Hence, every $\tau_n$ is a $\G^\P$--predictable stopping time. For every $n \in \N^\ast$, consider a sequence $(\tau^m_n)_{m \in \N^\ast}$ of $\G^\P$--predictable stopping times such that, $\P$--a.s., $\tau^m_n < \tau_n$ and $\tau^m_n \uparrow \tau_n$; see \cite[Theorem~6.4.5.(c) and the subsequent Remark]{weizsaecker1990stochastic}. Choose $\G$-predictable and strictly positive versions of the stopping times $\{\tau^m_n : (m,n) \in (\N^\ast)^2\}$ (see \cite[Theorem~6.4.5]{weizsaecker1990stochastic}) and enumerate them as $(\kappa_j)_{j \in \N}$. Let $\sigma_n \coloneqq \max_{1 \leq j \leq n} \kappa_j$ for $n \in \N$. Then $(\sigma_n)_{n \in \N}$ is a non-decreasing sequence of $\G$-predictable stopping times that converges $\P$--a.s. to infinity. Moreover, $A_{\sigma_n} \leq \max\{A_{\kappa_1},\ldots,A_{\kappa_n}\}$ is $\P$--essentially bounded. This completes the proof.
\end{proof}

\begin{lemma}\label{lem::equivalence_loc_square_integrable_martingale}
	Let $M = (M_t)_{t \in [0,\infty)}$ be a real-valued, right-continuous, $\G$-adapted process. The following conditions are equivalent$:$
	\begin{enumerate}
		\item[$(i)$] there exists a $(\G,\P)$--localising sequence of strictly positive $\G$--predictable stopping times $(\tau_n)_{n \in \N}$ such that each $M_{\cdot \land \tau_n}$ is a $(\G,\P)$--square-integrable $($resp. $(\G,\P)$--uniformly integrable$)$ martingale$;$
		\item[$(ii)$] $M$ is a $(\G,\P)$--locally square-integrable martingale $($resp. $(\G,\P)$--local martingale$);$
		\item[$(iii)$] $M$ is a $(\G_\smallertext{+},\P)$--locally square-integrable martingale $($resp. $(\G_\smallertext{+},\P)$--local martingale$);$
		\item[$(iv)$] $M$ is a $(\G^\P_\smallertext{+},\P)$--locally square-integrable martingale $($resp. $(\G^\P_\smallertext{+},\P)$--local martingale$)$.
	\end{enumerate}
	Moreover, in the locally square-integrable case, the predictable quadratic variations relative to the above filtrations coincide up to $\P$-evanescence.
\end{lemma}

\begin{proof}
	That $(i)$ implies $(ii)$, $(ii)$ implies $(iii)$, and $(iii)$ implies $(iv)$ is immediate. We show that $(iv)$ implies $(i)$. In the locally square-integrable case, denote by $\langle M \rangle^{(\G^\smalltext{\P}_\smalltext{+},\P)}$ the $(\G^\P_\smallertext{+},\P)$--predictable quadratic variation of $M$. By \Cref{lem::predictable_localisation}, there exists a sequence $(\sigma_n)_{n \in \N}$ of strictly positive, $\G$--predictable stopping times such that $\langle M \rangle^{(\G^\smalltext{\P}_\smalltext{+},\P)}_{\cdot\land\sigma_n}$ is $\P$--essentially bounded for every $n \in \N$. By the Burkholder--Davis--Gundy-inequality for exponent $p = 2$, $\sup_{t \in [0,\infty)}|M_{t\land\sigma_n}|$ is $\P$--square-integrable, and thus $M_{\cdot\land\sigma_n} = (M_{t\land\sigma_n})_{t\in[0,\infty)}$ is a $(\G^\P_\smallertext{+},\P)$--square-integrable martingale. Since $M_{\cdot\land\sigma_n}$ is $\G$-adapted, it is also a $(\G,\P)$--square-integrable martingale, which yields $(i)$. The fact that the predictable quadratic variations coincide follows from the characterising property, together with \Cref{rem::quadratic_variation}.
	
	\medskip
	In the local martingale case, we consider the $(\G^\P_\smallertext{+},\P)$--optional quadratic variation $[M]^{(\G^\smalltext{\P}_\smalltext{+},\P)}$ instead and note that $([M]^{(\G^\smalltext{\P}_\smalltext{+},\P)})^{1/2}$ admits a $(\G^\P_\smallertext{+},\P)$--predictable compensator, say $A^p = (A^p_t)_{t \in [0,\infty)}$, by \cite[Theorem I.3.17 and Corollary I.4.55.a)]{jacod2003limit}. By \Cref{lem::predictable_localisation}, there exists a sequence, denoted again by $(\sigma_n)_{n \in \N}$, of strictly positive, $\G$--predictable stopping times such that $A^p_{\cdot\land\sigma_\smalltext{n}}$ is $\P$--essentially bounded. The Burkholder--Davis--Gundy-inequality for exponent $p = 1$ implies that $\sup_{t \in [0,\infty)}|M_{t\land\sigma_\smalltext{n}}|$ is $\P$--integrable, and therefore $M_{\cdot\land \sigma_\smalltext{n}}$ is a $(\G^\P_\smallertext{+},\P)$--uniformly integrable martingale. Since $M_{\cdot\land \sigma_\smalltext{n}}$ is $\G$-adapted, it is also a $(\G,\P)$--uniformly integrable martingale. This completes the proof.
\end{proof}

\medskip
If $M$ and $N$ are both $(\G,\P)$--locally square-integrable martingales, we define $\langle M, N \rangle$ as usual through polarisation. In case $M$ is multidimensional, $\langle M \rangle$ denotes the matrix-valued process whose $(i,j)$-th entry is $\langle M^i, M^j\rangle$, for $i$ and $j$ ranging through the appropriate set of integers.

\medskip
We turn to semi-martingales. We fix an $\R^d$-valued, c\`adl\`ag, $\G$-adapted process $X = (X_t)_{t \in [0,\infty)}$ for the time being. We suppose that $X$ is a $(\G,\P)$--semi-martingale in the following sense: there exists an $\R^d$-valued, right-continuous, $\G$-adapted, $(\G,\P)$--local martingale $M = (M_t)_{t \in [0,\infty)}$ and an $\R^d$-valued, right-continuous, $\G$-adapted process $A = (A_t)_{t \in [0,\infty)}$ whose paths are $\P$--a.s. of locally finite variation with $M_0 = A_0 = 0$ and such that
\begin{equation*}
	X = X_0 + M + A, \; \text{$\P$--a.s.}
\end{equation*}
We fix a measurable and bounded map $h : \R^d \longrightarrow \R^d$, known as a truncation map, satisfying $h(x) = x$ in an open neighbourhood of the origin. Since all paths of $X$ are c\`adl\`ag, the process
	\begin{equation*}
		X^\prime \coloneqq X - X_0 - \sum_{s \in (0,\cdot]} \big(\Delta X_s - h(\Delta X_s)\big),
	\end{equation*}
	is well-defined. The semi-martingale $X^\prime$ has bounded jumps and is therefore a special $(\G,\P)$--semi-martingale (see \cite[Corollary 7.2.8]{weizsaecker1990stochastic}). It thus admits a unique decomposition
	\begin{equation*}
		X^\prime = M^\prime + B^\prime, \; \text{$\P$--a.s.},
	\end{equation*}
	where $M^\prime = (M^\prime_t)_{t \in [0,\infty)}$ is an $\R^d$-valued, right-continuous, $\G$-adapted, $(\G,\P)$--local martingale starting at zero, and $B^\prime = (B^\prime_t)_{t \in [0,\infty)}$ is an $\R^d$-valued, right-continuous, $\G$-predictable process starting at zero, whose paths are $\P$--a.s. of locally finite variation (see \cite[Theorem 7.2.6]{weizsaecker1990stochastic} and \Cref{lem::equivalence_loc_square_integrable_martingale}). The process $B^\prime$ forms the first component of the characteristic triplet of $X$. To describe the second and third components, we need the following two results.

	\begin{lemma}\label{lem::martingale_decomposition}
		Let $M = (M_t)_{t \in [0,\infty)}$ be a real-valued, right-continuous, $\G$-adapted, $(\G,\P)$--local martingale. There exists a, up to $\P$--indistinguishability, unique pair $(M^c,M^d) = (M^c_t,M^d_t)_{t \in [0,\infty)}$ consisting of two real-valued, right-continuous, $\G$-adapted, $(\G,\P)$--local martingales starting at zero such that  
		\begin{equation*}
			M = M_0 + M^c + M^d, \; \text{{\rm$\P$--a.s.}},
		\end{equation*}
		$M^c$ has {\rm$\P$--a.s.} continuous paths, and $M^d$ is a purely discontinuous local martingale in the sense that $M^d N$ is a $(\G,\P)$--local martingale for each real-valued, right-continuous, $\G$-adapted, $(\G,\P)$--local martingale $N = (N_t)_{t \in [0,\infty)}$ with {\rm$\P$--a.s.} continuous paths. Furthermore, for a $(\G,\P)$--semi-martingale $X = (X_t)_{t \in [0,\infty)}$, there exists a, up to $\P$-indistinguishability, unique real-valued, right-continuous, $\G$-adapted, $(\G,\P)$--local martingale $X^c = (X^c_t)_{t \in [0,\infty)}$ with {\rm$\P$--a.s.} continuous paths such that any semi-martingale decomposition $X = X_0 + M + A$ satisfies $X^c = M^c$, up to $\P$-indistinguishability.
	\end{lemma}

\begin{proof}[Proof of \Cref{lem::martingale_decomposition}]
	By \cite[Lemma I.4.18]{jacod2003limit} and \cite[Lemma 4.3.5]{weizsaecker1990stochastic}, there exists a real-valued, right-continuous, $\P$--a.s. continuous, $(\G_\smallertext{+},\P)$--local martingale $\bar{M}^c$ starting at zero such that $M - M_0 - \bar{M}^c$ is a $(\G_\smallertext{+},\P)$--purely discontinuous local martingale. Let $M^c = (M^c_t)_{t \in [0,\infty)}$ be defined by
	\[
		M^c_t \coloneqq \tilde{M}^c_t\1_{\{\tilde{M}^c_t \in \R\}}, \; \textnormal{where} \; \tilde{M}^c_t \coloneqq \limsup_{n\rightarrow\infty}\bar{M}^c_{(t-1/n)\lor 0}, \; t \in [0,\infty).
	\]
	Then $M^c_t$ is $\cG_t$-measurable for every $t \in [0,\infty)$ and $\tilde{M}^c = \bar{M}^c$ up to $\P$--indistinguishability by the $\P$--a.s. continuity of $\bar{M}^c$. By \cite[Lemma 4.3.5]{weizsaecker1990stochastic}, we can modify $M^c$ in such a way that all its paths become additionally right-continuous. Then $M^c$ is real-valued, right-continuous, $\G$-adapted, and a $(\G_\smallertext{+},\P)$--local martingale, and thus a $(\G,\P)$--local martingale (see \Cref{lem::equivalence_loc_square_integrable_martingale}) which is $\P$-indistinguishable from $\bar{M}^c$. Hence $M^d \coloneqq M - M_0 - M^c$ is a $(\G_\smallertext{+},\P)$--purely discontinuous local martingale. That $M^d$ is also a $(\G,\P)$--purely discontinuous local martingale follows by applying \Cref{lem::equivalence_loc_square_integrable_martingale} twice.
	Uniqueness can be argued as in the proof of \cite[Lemma I.4.18]{jacod2003limit}. For the remaining claim, we refer to \cite[Proposition I.4.27]{jacod2003limit} for details.
\end{proof}

	The process $X^c$ in the previous result is the $(\G,\P)$--continuous local martingale part of the $(\G,\P)$--semi-martingale $X$. When the underlying probability measure matters, we write $X^{c,\P}$.
	
	\medskip
	Let $\mu^X$ be the jump measure on $[0,\infty) \times \R^d$ of $X$ defined through
	\begin{equation*}
		\mu^X(\omega; \d t, \d x) \coloneqq \sum_{s \in (0,\infty)} \1_{\{\Delta X_\smalltext{s}(\omega) \neq 0\}} \boldsymbol{\delta}_{(s,\Delta X_\smalltext{s}(\omega))}(\d t, \d x).
	\end{equation*}
	
	\begin{lemma}\label{lem::existence_predictable_compensator_mu}
		Let $X = (X_t)_{t \in [0,\infty)}$ be an $\R^d$-valued, c\`adl\`ag, $\G$-adapted process. There exists a random measure $\nu(\omega;\d t, \d x)$ on $[0,\infty)\times\R^d$ such that, for every $\cP(\G)\otimes\cB(\R^d)$-measurable $W : \Omega \times [0,\infty) \times \R^d \longrightarrow [0,\infty]$,
		\begin{enumerate}
			\item[$(i)$] the process $\displaystyle W\ast\nu \coloneqq \int_{(0,\cdot]\times\R^\smalltext{d}}W_s(x)\nu(\d s,\d x)$ is $\G$-predictable$,$ and
			\item[$(ii)$] $\E^\P[W\ast\mu^X_\infty] = \E^\P[W\ast\nu_\infty]$.
		\end{enumerate}
		Moreover, $(i)$ and $(ii)$ uniquely characterise the random measure $\nu$ up to a $\P$--null set, and $\nu$ can be constructed to be of the form
		\begin{equation}\label{eq::representation_nu}
			\nu(\omega;\d t, \d x) = K_{\omega,t}(\d x)\sum_{\ell = 1}^\infty \d A^\ell_t(\omega), \; \omega \in \Omega,
		\end{equation}
		where $K$ is a kernel on $(\R^d,\cB(\R^d))$ given $(\Omega \times [0,\infty),\cP(\G))$, and each $A^\ell = (A^\ell_t)_{t \in [0,\infty)}$, $\ell\in\N^\star$, is a $[0,\infty)$-valued, right-continuous and non-decreasing, $\G$-predictable process starting at zero. Furthermore, $\sum_{\ell = 1}^\infty A^\ell$ is $\P$-indistinguishable from a real-valued, right-continuous, \textnormal{$\P$--a.s.} non-decreasing, $\G$-predictable process $A = (A_t)_{t \in [0,\infty)}$ starting at zero. There also exists a {\rm`}good version{\rm'} of $\nu$ that additionally satisfies $\nu(\omega;\{t\}\times \R^d) \leq 1$ identically and such that  
		\begin{equation}\label{eq::definition_J}
			J \coloneqq \big\{(\omega,t) \in \Omega \times [0,\infty) : \nu(\omega ; \{t\} \times \R^d) > 0 \big\} = \bigcup_{n = 1}^\infty \llbracket \tau_n \rrbracket,
		\end{equation}
		identically, where $(\tau_n)_{n \in \N}$ are $\G$--predictable stopping times with disjoint graphs.
	\end{lemma}

	\begin{proof}[Proof of \Cref{lem::existence_predictable_compensator_mu}]
		Instead of repeating the whole argument in the proof of \cite[Theorem II.1.8]{jacod2003limit}, we merely want to point out that one just has to replace in that proof the process $V$ by the one constructed in \cite[Lemma 6.5]{neufeld2014measurability} and then the process $(V\ast\mu)^p$ by the dual predictable projection (predictable compensator) constructed in \cite[Appendix I, Theorem 12, page 405]{dellacherie1982probabilities}.
		
		\medskip
		As in the proof of \cite[Proposition II.1.17]{jacod2003limit}, we construct a `good' version of $\nu$ as follows. First, the set $D \coloneqq \{\Delta X \neq 0\}$ is $\G$-optional and has countable $\omega$-sections. By \cite[Theorem B, page xiii, and Remark E, page xvii]{dellacherie1982probabilities} and then \cite[Theorem IV.88.(a), page 139]{dellacherie1978probabilities}, the $\G$-optional set $D$ is the countable union of disjoint graphs of $\G$--stopping times $(\tau_n)_{n \in \N}$. We now follow exactly the steps in the proof of \cite[Proposition II.1.17.b)]{jacod2003limit} together with an application of \cite[Theorem 88.(b), page 139]{dellacherie1978probabilities} to find a `good' version of $\nu$ that additionally satisfies $\nu(\omega;\{t\}\times \R^d) \leq 1$ identically and such that the $\G$-predictable set  $J \coloneqq \big\{(\omega,t) \in \Omega \times [0,\infty) : \nu(\omega ; \{t\} \times \R^d) > 0 \}$ is equal to a countable union of disjoint graphs of predictable $\G$--stopping times and is the $\G$-predictable support of $D$.
	\end{proof}

	The random measure $\nu$ constructed in the preceding lemma is the $(\G,\P)$--predictable compensator of $\mu^X$. When it is necessary to emphasise the probability measure, we write $\nu^\P$.
	
	\medskip
	A triple $(\sfB, \sfC, \nu)$, consisting of an $\R^d$-valued process $\sfB = (\sfB_t)_{t \in [0,\infty)}$, an $\R^{d \times d}$-valued process $\sfC = (\sfC_t)_{t \in [0,\infty)}$, and a random measure $\nu$ on $[0,\infty) \times \R^d$, is referred to as the $(\G,\P)$–semi-martingale characteristics of $X$ (relative to $h$) if $(\sfB, \sfC)$ is $\P$-indistinguishable from $(B^\prime, \langle X^c \rangle)$, and $\nu$ coincides $\P$--a.s. with the $(\G,\P)$--predictable compensator of $\mu^X$.
	
	\begin{remark}\label{rem::equivalence_semimartingale}
		Although the notion of a semi-martingale and its characteristics depend on the filtration, for processes adapted to $\G$, the various notions and their characteristics relative to $\G$, $\G_\smallertext{+}$, or $\G^\P_\smallertext{+}$ coincide$;$ see {\rm\cite[Proposition 2.2]{neufeld2014measurability}}. This actually extends to any filtration $\H = (\cH_t)_{t \in [0,\infty)}$ satisfying $\cG_t \subseteq \cH_t \subseteq \cG^\P_{t\smallertext{+}}$, as follows immediately from a close examination of the proof of the aforementioned result. 
		Moreover, by {\rm\cite[Theorem I.4.47.a)]{jacod2003limit}} and {\rm\cite[Theorem 4.3.3]{weizsaecker1990stochastic}}, the optional quadratic variation may and will be chosen to be right-continuous and $\G$-adapted.
		Furthermore, the continuous local martingale parts of $X$ relative to $\G$, $\G_\smallertext{+}$, and $\G^\P_\smallertext{+}$ are $\P$-indistinguishable.
	\end{remark}
	
	Lastly, the following result provides equivalent descriptions of the disintegration of the compensator $\nu$ relative to an auxiliary process, helping to simplify the formulation of the semi-martingale laws in \eqref{eq::semi_martingale_laws}.
	\begin{lemma}\label{lem::absolute_continuity_nu}
		Let $C = (C_t)_{t \in [0,\infty)}$ be a real-valued, right-continuous and $\P$--{\rm a.s.} non-decreasing, $\G$-predictable process starting at zero. Suppose that $X$ is an $\R^d$-valued, c\`adl\`ag, $(\G,\P)$--semi-martingale, and let $\nu$ be the $(\G,\P)$--predictable compensator of its jump measure $\mu^X$. There exists a $(\G,\P)$--localising sequence $(\tau_n)_{n \in \N}$ of $\G$--predictable stopping times satisfying $\E^\P[(|x|^2 \land 1)\ast\nu_{\tau_\smalltext{n}}] < \infty$ for each $n \in \N$. Moreover, the following are equivalent$:$
		
		\begin{enumerate}
		\item[$(i)$] $\displaystyle (|x|^2 \land 1) \ast\nu = \int_0^\cdot\int_{\R^\smalltext{d}}(|x|^2 \land 1)\nu(\d s, \d x) \ll C$, \textnormal{$\P$--a.s.}$;$
		
		\item[$(ii)$] $\nu(\d t, \d x) = \mathsf{K}_{t}(\d x) \d C_t$, $\P$--{\rm a.s.}, for a kernel $\mathsf{K}$ on $(\R^d,\cB(\R^d))$ given $(\Omega\times[0,\infty),\cP(\G))$$;$
		
		\item[$(iii)$] if $\nu(\d t, \d x) = K_{t}(\d x)\d A_t$, \textnormal{$\P$--a.s.}, for a kernel $K$ on $(\R^d,\cB(\R^d))$ given $(\Omega\times[0,\infty),\cP(\G))$, and a real-valued, right-continuous and \textnormal{$\P$--a.s.} non-decreasing, $\G$-predictable process $A = (A_t)_{t \in [0,\infty)}$ starting at zero, then the compensator satisfies $\nu(\d t, \d x) = K_{t}(\d x)(\d A^\textnormal{ac}/\d C)_t\d C_t$, $\P$--{\rm a.s.}, where $\d A^\textnormal{ac}/\d C$ denotes the, \textnormal{$\P$--a.s.} defined, Radon--Nikod\'ym derivative of the absolutely continuous component of $A$ relative to $C$.
		\end{enumerate}
	\end{lemma}

	\begin{proof}[Proof of \Cref{lem::absolute_continuity_nu}]
		The existence of the $(\G_\smallertext{+},\P)$--localising sequence $(\tau_n)_{n \in \N}$ satisfying $\E^\P[(|x|^2 \land 1)\ast\nu_{\tau_\smalltext{n}}] < \infty$ for each $n \in \N$ follows from \cite[II.2.13, page 77]{jacod2003limit}. 
		We then choose corresponding $\G$--predictable stopping times by applying \Cref{lem::predictable_localisation}.
		Note that this implies that the Lebesgue--Stieltjes measure induced by $(|x|^2 \land 1)\ast\nu$ on $\cB([0,\infty))$ is $\sigma$-finite up to a $\P$--null set. 
		
		\medskip
		Next, that $(ii)$ implies $(i)$ is clear. We thus suppose that $(i)$ holds, and with \eqref{eq::representation_nu}, we write
		\begin{equation*}
			\nu(\d s, \d x) = K_s(\d x)\d A_s, \; \text{$\P$--a.s.},
		\end{equation*}
		and then
		\begin{equation*}
			(|x|^2 \land 1)\ast\nu = \int_0^\cdot\int_{\R^\smalltext{d}}(|x|^2 \land 1)\nu(\d s,\d x) = \int_0^\cdot\int_{\R^\smalltext{d}} (|x|^2 \land 1) K_s(\d x)\d A_s, \; \text{$\P$--a.s.}
		\end{equation*}
		Let us denote by $\d A^\textnormal{ac}$ and $\d A^\textnormal{si}$ the absolutely continuous and singular components, respectively, of the measure $\d A$ relative to $\d C$ which are defined for $\P$--a.e. $\omega \in \Omega$. We have
		\begin{equation*}
			\nu(\d s, \d x) = K_s(\d x) \frac{\d A^\textnormal{ac}_s}{\d C_s} \d C_s + K_s(\d x) \d A^\textnormal{si}_s, \; \textnormal{$\P$--a.s.},
		\end{equation*}
		and then
		\begin{equation*}
			(|x|^2 \land 1)\ast\nu = \int_0^\cdot\int_{\R^\smalltext{d}}(|x|^2 \land 1)K_s(\d x) \frac{\d A^\textnormal{ac}_s}{\d C_s} \d C_s + \int_0^\cdot\int_{\R^\smalltext{d}}(|x|^2 \land 1)K_s(\d x)  \d A^\textnormal{si}_s, \; \text{$\P$--a.s.}
		\end{equation*}
		Since $(i)$ holds, 
		the measure induced by the second term on the right-hand side is $\P$--a.s. absolutely continuous and singular to $\d C$, thus
		\begin{equation}\label{eq::integral_decomposition_singular}
			\int_0^\infty\int_{\R^\smalltext{d}}(|x|^2 \land 1)K_s(\d x)  \d A^\textnormal{si}_s = 0, \; \text{$\P$--a.s.}
		\end{equation}
		From
		\begin{equation*}
			\int_0^\infty\int_{\R^\smalltext{d}} \1_{\{(|x|^2 \land 1) = 0\}}K_s(\d x)\d A^\textnormal{si}_s \leq \int_0^\infty\int_{\R^\smalltext{d}} \1_{\{(|x|^2 \land 1) = 0\}}\nu(\d s,\d x) = \int_0^\infty\int_{\R^\smalltext{d}} \1_{\{x = 0\}}\nu(\d s,\d x), \; \text{$\P$--a.s.},
		\end{equation*}
		and $\E^\P[\1_{\{x = 0\}}\ast\nu_\infty] = \E^\P[\1_{\{x = 0\}}\ast\mu^X_\infty] = 0$, we deduce that
		\begin{equation*}
			(|x|^2 \land 1) > 0, \; \text{$K_s(\d x)\d A^\textnormal{si}_s$--a.e.}, \; \text{$\P$--a.s.}
		\end{equation*}
		This, however, together with \eqref{eq::integral_decomposition_singular}, implies $K_s(\d x) \d A^\textnormal{si}_s = 0$, $\P$--a.s., and we therefore must have
		\begin{equation*}
			\nu(\d s, \d x) = K_s(\d x) \frac{\d A^\textnormal{ac}_s}{\d C_s} \d C_s, \; \text{$\P$--a.s.}
		\end{equation*}
		Hence $(iii)$ holds. That $(iii)$ implies $(i)$ is immediate. Lastly, that $(iii)$ implies $(ii)$ follows from defining
		\begin{equation*}
				\mathsf{K}_{\omega,t}(\d x) \coloneqq K_{\omega,t}(\d x) \mathsf{a}_t(\omega),
			\end{equation*}
		where
		\begin{equation*}
				\mathsf{a}_t \coloneqq \hat{\mathsf{a}}_t \1_{[0,\infty)}(\hat{\mathsf{a}}_t), \; \hat{\mathsf{a}}_t \coloneqq \limsup_{n \rightarrow \infty} \frac{A_t - A_{(t-1/n)\lor 0}}{C_t - C_{(t-1/n)\lor 0}}
			\end{equation*}
		is $\cP(\G)$-predictable; we use the convention $\lambda/0=0$ for any $\lambda \in \R$. Then $\mathsf{a} = \d A^\textnormal{ac}/\d C$, $\d C$--a.e., $\P$--a.s., by \cite[Section X.4, page 159]{doob1994measure}, which yields $(ii)$ and completes the proof.
	\end{proof}

	Uniqueness of the kernel $\mathsf{K}$ in \Cref{lem::absolute_continuity_nu}.$(ii)$ is ensured by the following result. The assumptions are satisfied since there exists a $\cP(\G)\otimes\cB(\R^d)$--predictable function $V > 0$ satisfying $\E^\P[V\ast\nu_\infty] = \E^\P[V\ast\mu^X_\infty] \leq 1$ (see \cite[Lemma 6.5]{neufeld2014measurability}).
	\begin{lemma}\label{lem::uniqueness_of_kernel}
		Let $\mu$ be a $\sigma$-finite measure on a measurable space $(\Omega,\cF)$ and $(K,K^\prime)$ be kernels on a measurable space $(\overline{\Omega},\overline{\cF})$ given $(\Omega,\cF)$ that are $\mu$--{\rm a.e.} $\sigma$-finite. Suppose that $\overline{\cF}$ is countably generated. If $\mu\otimes K = \mu\otimes K^\prime$, then $K = K^\prime$ outside some $\mu$--null set.
	\end{lemma}

	\begin{proof}[Proof of \Cref{lem::uniqueness_of_kernel}]
		Let $(A,B) \in \cF \times \overline{\cF}$. Then
		\[
		\int_A \mu(\d \omega) K(\omega,B) = \mu\otimes K (A\times B) = \mu\otimes K^\prime (A\times B) = \int_A \mu(\d\omega) K^\prime(\omega,B).
		\]
		So that $\sigma$-finiteness of $\mu$ implies $K(\cdot,B) = K^\prime(\cdot,B)$, $\mu$--a.e.; alternatively, this follows from the uniqueness in the Radon--Nikod\'ym theorem or from \cite[Satz IV.4.5]{elstrodt2018mass}. Then \cite[Corollary 1.6.4]{cohn2013measure} (or \cite[Korollar II.5.7]{elstrodt2018mass}) and the $\mu$--a.e. $\sigma$-finiteness of the kernels, together with the separability of $\overline{\cF}$, then yields
		\[
		K(\omega,B) = K^\prime(\omega,B), \; B \in \overline{\cF}, \; \textnormal{$\mu$--a.e. $\omega \in \Omega$},
		\]
		which completes the proof.
	\end{proof}

\subsection{Stochastic integrals without the usual conditions}\label{sec::stochastic_integrals}
	
	Let $M = (M_t)_{t \in [0,\infty)}$ be an $\R^d$-valued, right-continuous, $\G$-adapted, $(\G,\P)$--locally square-integrable martingale starting at zero. Let $C = (C_t)_{t \in [0,\infty)}$ be a right-continuous and $\P$--a.s. non-decreasing, $\G$-predictable process starting at zero satisfying $\langle M \rangle \ll C$, $\P$--a.s., component-wise. We then choose a factorisation $\langle M \rangle = \pi \bcdot C$, $\P$--a.s., where $\pi = (\pi_t)_{t \in [0,\infty)}$ is an $\S^d_\smallertext{+}$-valued, $\G$-predictable process (see \cite[Section 3.1]{shiryaev2002vector} or the proof of \cite[Proposition II.2.9]{jacod2003limit}). We denote by $\H^2_{\rm loc}(M;\G,\P)$ the space of $\R^d$-valued, $\G$-predictable processes $Z = (Z_t)_{t \in [0,\infty)}$ for which there exists a $(\G,\P)$--localising sequence $(\tau_n)_{n \in \N}$ such that
	\begin{equation*}
		\E^\P\bigg[\int_0^{\tau_\smalltext{n}} Z^\top_r \pi_r Z_r \d C_r \bigg] < \infty, \; n \in \N,
	\end{equation*}
	and then denote by $Z \bcdot M$ or $\int_0^\cdot Z_s \d M_s$ the vector stochastic integral of $Z$ with respect to $M$ in the sense of \cite[Theorem III.6.4]{jacod2003limit}. 
	We note here that this space could equivalently be defined by $(\G_\smallertext{+},\P)$--localising sequences 
	or even by $(\G,\P)$--localising sequences of $\G$--predictable stopping times 
	(see \Cref{lem::predictable_localisation}). 
	
	\medskip
	The stochastic integral $Z\bcdot M$ is \emph{a priori} a real-valued, $\G_\smallertext{+}$-adapted, $(\G_\smallertext{+},\P)$--locally square-integrable martingale with $\P$--a.s. c\`adl\`ag paths starting at zero. The properties that the vector stochastic integral ought to satisfy describe it uniquely up to $\mathbb{P}$-indistinguishability; therefore, we can always implicitly choose one representative---also referred to as $\P$-version in what follows. 
	Because $M$ is right-continuous and $\G$-adapted, we will see in \Cref{prop::good_version_stochastic_integral} that there exists a `good' representative of the vector stochastic integral that is $\G$-adapted, with all of its paths being right-continuous. 
	
	\medskip
	We further denote by $\H^2(M;\G,\P)$ the subset of $\H^2_\text{loc}(M;\G,\P)$ consisting of those processes $Z$ satisfying
	\begin{equation*}
		\|Z\|^2_{\H^{\smalltext{2}}(M;\G,\P)} \coloneqq \E^\P\bigg[\int_0^\infty Z^\top_r \pi_r Z_r \d C_r \bigg] < \infty,
	\end{equation*}
	which then yields a true $(\G,\P)$--square-integrable martingale $Z \bcdot M$. Moreover, the space $\H^2(M;\G,\P)$ together with $\|\cdot\|_{\H^{\smalltext{2}}(M;\G,\P)}$ forms a complete semi-normed space such that $\|Z - Z^\prime\|_{\H^{\smalltext{2}}(M;\G,\P)} = 0$ implies $Z \bcdot M = Z^\prime \bcdot M$ up to $\P$-indistinguishability. For additional background, we refer to \cite[Section III.6]{jacod2003limit}.

	\medskip
	We turn to the construction of the stochastic integral relative to the compensated jump measure $\mu^X - \nu$ of an $\R^d$-valued, càdlàg, $\G$-adapted process $X$. First, we note that by \cite[Theorem B, page xiii, and Remark E, page xvii]{dellacherie1982probabilities}, and subsequently \cite[Theorem IV.88.(a), page 139]{dellacherie1978probabilities}, the optional set $D \coloneqq \{\Delta X \neq 0\}$ is the union of a sequence of disjoint graphs of $\G$--stopping times. Let $C$ be a right-continuous and $\P$--a.s. non-decreasing, $\G$-predictable process starting at zero\footnote{The same process $C$ can be chosen to satisfy the required properties for both the vector stochastic integral and the stochastic integral with respect to a compensated random measure.} satisfying $\nu(\d t, \d x) = K_{t}(\d x)\d C_t$, $\P$--a.s., for some transition kernel $K$ on $(\R^d, \cB(\R^d))$ given $(\Omega \times [0,\infty), \cP(\G))$; compare with \eqref{eq::representation_nu}. We denote by $\nu$ a `good version' of the $(\G,\P)$--predictable compensator as described in \Cref{lem::existence_predictable_compensator_mu}. For an $\cF\otimes\cB([0,\infty))\otimes\cB(\R^d)$-measurable, $[-\infty,\infty]$-valued function $U$, we define
	\begin{equation*}
		\widehat U_t(\omega) \coloneqq \int_{\{t\} \times \R^\smalltext{d}} U_t(\omega;x) \nu(\omega; \d t, \d x), \; (\omega,t) \in \Omega \times [0,\infty),
	\end{equation*}
	and then $\widetilde U_t(\omega) \coloneqq U_t(\omega;\Delta X_t(\omega))\1_{\{\Delta X_\smalltext{t}(\omega) \neq 0\}} - \widehat U_t(\omega)$. Recall that we are using the convention $\infty-\infty = -\infty$ throughout. Note that $\widetilde U$ is a $\G$-optional process with\footnote{See \eqref{eq::definition_J} for the definition of $J$.} $\{\widetilde U \neq 0\} \subseteq D \cup J$, and therefore $\{\widetilde U \neq 0\}$
	is the countable union of disjoint graphs of $\G$--stopping times by \cite[Theorem IV.88.(a), page 139]{dellacherie1978probabilities}. Thus, $\sum_{s \in (0,\infty)} |\widetilde U_s|^2$ is well-defined and $\cG_{\infty\smallertext{-}}$-measurable. 
	
	\medskip
	We denote by $\H^2(\mu^X;\G,\P)$ the linear space of real-valued and $\cP(\G)\otimes\cB(\R^d)$-measurable functions $U$ satisfying 
	\[
		\E^\P \Bigg[ \sum_{s \in (0,\infty)} |\widetilde U_s|^2 \Bigg] < \infty.
	\]
	For each $U \in \H^2(\mu^X;\G,\P)$, we denote by $U \ast\tilde\mu^X$ the stochastic integral of $U$ with respect to the compensated random measure $\mu^X - \nu$ (see \cite[Definition II.1.27]{jacod2003limit}), that is, $U \ast\tilde\mu^X$ is the, up to $\P$-indistinguishability, unique real-valued, right-continuous, $\G_\smallertext{+}$-adapted, $(\G_\smallertext{+},\P)$--purely discontinuous local martingale with $\Delta (U \ast\tilde\mu^X) = \widetilde U$ up to $\P$-evanescence. 
	Since $[U\ast\tilde\mu^X] = \sum_{s \in (0,\cdot]}(\widetilde U_s)^2$ is $\P$-integrable, $U\ast\tilde\mu^X$ is even a $(\G_\smallertext{+},\P)$--square-integrable martingale. Moreover
	\[
	(U + U^\prime) \ast\tilde\mu^X = U \ast\tilde\mu^X + U^\prime \ast\tilde\mu^X,\; \P\text{\rm--a.s.}, \; \text{for}\; \text{any}\; (U,U^\prime) \in \big(\H^2(\mu^X;\G,\P)\big)^2. 
	\]
	
	By \cite[Theorem II.1.33.a)]{jacod2003limit}, the predictable quadratic variation relative to $(\G_\smallertext{+},\P)$ of the process $U \ast\tilde\mu^X$ is given by
	\begin{align*}
		\langle U \ast\tilde\mu^{X} \rangle^{(\G_\tinytext{+},\P)}_t(\omega) 
		&= \int_{(0,t]} \Bigg( \int_{\R^\smalltext{d}}\bigg(U_s(\omega; x) - \int_{\R^\smalltext{d}} U_s(\omega;x)K_{\omega,s}(\d x)\Delta C_s(\omega)\bigg)^2 K_{\omega,s}(\d x)  \\
		&\quad+ \big(1 - K_{\omega,s}(\R^d)\Delta C_s(\omega)\big) \bigg( \int_{\R^\smalltext{d}} U_s(\omega;x) K_{\omega,s}(\d x) \bigg)^2 \Delta C_s(\omega) \Bigg) \d C_s(\omega), \; t \in [0,\infty), \; \text{for $\P$--a.e. $\omega \in \Omega$}.
	\end{align*}
	Since $C$ is only $\P$--a.s. non-decreasing, we denote by $\Delta C = (\Delta C_t)_{t \in [0,\infty)}$ the $\G$-predictable process
	\begin{equation*}
		\Delta C_t \coloneqq 
		\begin{cases}
			\displaystyle\limsup_{n \rightarrow \infty}\big\{C_t - C_{(t\smallertext{-}1/n)\lor 0}\big\}, \;\textnormal{if this is finite}, \\
			0, \;\textnormal{otherwise.}
		\end{cases}
	\end{equation*}
	For $(\omega,s) \in \Omega \times [0,\infty)$, a (nonnegative) measure $F$ on $(\R^d,\cB(\R^d))$, and a $\cB(\R^d)$-measurable map $\cU : \R^d \longrightarrow \R$, we write
	\begin{align*}
		\|\cU(\cdot)\|^2_{\hat{\L}^\smalltext{2}_{\smalltext{\omega}\smalltext{,}\smalltext{s}}(F)} \coloneqq\int_{\R^\smalltext{d}}\bigg(\cU(x) - \int_{\R^\smalltext{d}}\cU(x)F(\d x) \Delta C_s(\omega)\bigg)^2 F(\d x) + (1-F(\R^d)\Delta C_s(\omega))^+\bigg(\int_{\R^\smalltext{d}} \cU(x) F(\d x)\bigg)^2 \Delta C_s(\omega).
	\end{align*}
	Then $\widehat\L^2_{\omega,s}(F)$ denotes the collection of $\cB(\R^d)$-measurable maps $\cU : \R^d \longrightarrow \R$ satisfying $\|\cU(\cdot)\|_{\hat{\L}^\smalltext{2}_{\smalltext{\omega}\smalltext{,}\smalltext{s}}(F)} < \infty$. Note that on $\widehat\L^2_{\omega,s}(F)$, we have $\|\cU(\cdot)\|^2_{\hat{\L}^\smalltext{2}_{\smalltext{\omega}\smalltext{,}\smalltext{s}}(F)} = \langle \cU(\cdot),\cU(\cdot)\rangle_{\hat{\L}^\smalltext{2}_{\smalltext{\omega}\smalltext{,}\smalltext{s}}(F)}$, where
	\begin{align*}
		&\langle \cU(\cdot),\cV(\cdot)\rangle_{\hat{\L}^\smalltext{2}_{\smalltext{\omega}\smalltext{,}\smalltext{s}}(F)} \\
		&\coloneqq 
		\begin{cases}
			\displaystyle \int_{\R^d} \cU(x)\cV(x)F(\d x) - \Delta C_s(\omega) \bigg(\int_{\R^\smalltext{d}}\cU(x)F(\d x)\bigg)\bigg(\int_{\R^\smalltext{d}}\cV(x)F(\d x)\bigg), & \textnormal{if $F(\R^d)\Delta C_s(\omega) \leq 1$,} \\[1em]
			\displaystyle \int_{\R^\smalltext{d}}\bigg(\cU(x) - \int_{\R^\smalltext{d}}\cU(x)F(\d x) \Delta C_s(\omega)\bigg)\bigg(\cV(x) - \int_{\R^\smalltext{d}}\cV(x)F(\d x) \Delta C_s(\omega)\bigg) F(\d x), & \textnormal{if $F(\R^d)\Delta C_s(\omega) > 1$,}
		\end{cases}
	\end{align*}
	is a positive-semidefinite, symmetric bilinear form. Hence, we have that $\|\,\cdot\,\|_{\hat{\L}^\smalltext{2}_{\smalltext{\omega}\smalltext{,}\smalltext{s}}(F)}$ is a semi-norm on $\widehat\L^2_{\omega,s}(F)$.	Since $K_{\omega,s}(\R^d)\Delta C_s(\omega) = \nu(\omega;\{s\}\times\R^d) \leq 1$, $s \in [0,\infty)$, $\P$--a.s., we have
	\begin{equation*}
		\langle U \ast\tilde\mu^{X} \rangle^{(\G_\tinytext{+},\P)}(\omega) = \int_0^\cdot \|U_s(\omega;\cdot)\|^2_{\hat\L^\smalltext{2}_{\smalltext{\omega}\smalltext{,}\smalltext{s}}(K_{\smalltext{\omega}\smalltext{,}\smalltext{s}})} \d C_s(\omega), \; \text{for $\P$--a.e. $\omega \in \Omega$},
	\end{equation*}
	which implies that $U_s(\omega;\cdot) \in \widehat\L^2_{\omega,s}(K_{\omega,s})$ for $\P\otimes \mathrm{d}C$--a.e. $(\omega,s) \in \Omega \times [0,\infty)$, whenever $U \in \H^2(\mu^X;\G,\P)$. Conversely, if we are given a $\cP(\G)\otimes\cB(\R^d)$-measurable function $U$, the map $(\omega, s) \longmapsto \|U_s(\omega; \cdot)\|^2_{\hat\L^2_{\omega, s}(K_{\omega, s})}$ is $\G$-predictable. Furthermore, if 
	\begin{equation*}
		\|U\|^2_{\H^\smalltext{2}(\mu^\smalltext{X}; \G, \P)} \coloneqq \E^\P \bigg[ \int_{(0,\infty)} \|U_s(\cdot)\|^2_{\hat\L^\smalltext{2}_\smalltext{s}(K_\smalltext{s})}  \d C_s \bigg] < \infty,
	\end{equation*}
	then $U \in \H^2(\mu^X; \G, \P)$ (see \cite[Theorem II.1.33.a)]{jacod2003limit}).
	The space $\H^2(\mu^X; \G, \P)$, together with the semi-norm $\|\cdot\|_{\H^\smalltext{2}(\mu^\smalltext{X}; \G, \P)}$, forms a complete semi-normed space, such that $\|U - U^\prime\|^2_{\H^\smalltext{2}(\mu^\smalltext{X}; \G, \P)} = 0$ implies $U \ast \tilde{\mu}^X = U^\prime \ast \tilde{\mu}^X$ up to $\P$-indistinguishability. For $U \in \H^2(\mu^X;\G,\P)$, define $U^\prime_s(\omega;x) \coloneqq U_s(\omega;x)\1_{\Omega\times[0,\infty)\setminus N}(\omega,s)$, where
	\begin{equation*}
		N \coloneqq \big\{(\omega,s) \in \Omega \times [0,\infty) : \|U_s(\omega;\cdot)\|_{\hat{\L}^\smalltext{2}_{\smalltext{\omega}\smalltext{,}\smalltext{s}}(K_{\smalltext{\omega}\smalltext{,}\smalltext{s}})} = \infty\big\} \in \cP(\G).
	\end{equation*}
	Then $\|U - U^\prime\|^2_{\H^2(\mu^X;\G,\P)} = 0$ and $U^\prime_s(\omega;\cdot) \in \widehat\L^2_{\omega,s}(K_{\omega,s})$ for each $(\omega,s) \in \Omega \times [0,\infty)$. Therefore, we will implicitly assume that any $U \in \H^2(\mu^X;\G,\P)$ satisfies $U_s(\omega;\cdot) \in \widehat\L^2_{\omega,s}(K_{\omega,s})$ for every $(\omega,s) \in \Omega \times [0,\infty)$.

	\begin{proposition}\label{prop::good_version_stochastic_integral}
		Let $M = (M_t)_{t \in [0,\infty)}$ be an $\R^d$-valued, right-continuous, $\G$-adapted, $(\G, \P)$--locally square-integrable martingale, and let $\mu^X$ be the random jump measure on $[0,\infty) \times \R^d$ of an $\R^d$-valued, c\`adl\`ag, $\G$-adapted process $X = (X_t)_{t \in [0,\infty)}$. For $Z \in \H^2_\textnormal{loc}(M; \G^\P_\smallertext{+}, \P)$ and $U \in \H^2(\mu^X; \G^\P_\smallertext{+}, \P)$, there exist real-valued, $\G$-adapted and right-continuous $\P$-modifications of $(Z \bcdot M)^{(\P)}$ and $(U \ast \tilde{\mu}^X)^{(\P)}$.
	\end{proposition}

	\begin{proof}[Proof of \Cref{prop::good_version_stochastic_integral}]
		If $Z = (Z_t)_{t \in [0,\infty)}$ is $\G^\P_\smallertext{+}$-predictable, there exists a $\G$-predictable process $Z^\prime =(Z^\prime_t)_{t \in [0,\infty)}$ such that $Z^\prime$ and $Z$ are $\P$-indistinguishable. This follows by applying \cite[Remark IV.74, Theorem IV.78]{dellacherie1978probabilities} together with a monotone class argument. An analogous reasoning also yields a $\widetilde{\cP}(\G)$-measurable $U^\prime$ such that $U_t(\omega;x) = U^\prime_t(\omega;x)$ for all $(t,x) \in [0,\infty) \times \R^d$, for $\P$--a.e. $\omega \in \Omega$. It then follows from the characterising properties of the stochastic integrals that $(Z\bcdot M)^{(\G^\P_\smallertext{+},\P)} = (Z^\prime\bcdot M)^{(\G_\smallertext{+},\P)}$ and $(U\ast\tilde\mu^X)^{(\G^\P_\smalltext{+},\P)} = (U^\prime\ast\tilde\mu^X)^{(\G_\smalltext{+},\P)}$ up to $\P$-indistinguishability. We thus suppose, without loss of generality, that $Z$ is $\G$-predictable and $U$ is $\widetilde{\cP}(\G)$-measurable.

		\medskip
		We turn to the stochastic integral of $M$. By \cite[Theorem III.6.4.a)]{jacod2003limit} and \cite[Theorem 4.3.3]{weizsaecker1990stochastic}, we can suppose without loss of generality that $M$ is one-dimensional and $Z \in \H^2_\text{loc}(M;\G,\P)$. Let $(\tau_n)_{n \in \N}$ be a localising sequence of $\G_\smallertext{+}$--stopping times such that $M^{\tau_\smalltext{n}}$ is a square-integrable $(\G_\smallertext{+},\P)$-martingale and $Z \in \H^2(M^{\tau_\smalltext{n}};\G,\P)$ for each $n \in \N$. Then $(Z\bcdot M)^{\tau_\smalltext{n}} = Z\bcdot (M^{\tau_\smalltext{n}})$, up to $\P$-indistinguishability, and thus the sequence $Z \bcdot (M^{\tau_\smalltext{n}})$ converges uniformly on compacts in $\P$-probability to $Z \bcdot M$. Since $M^{\tau_\smalltext{n}}$ is $\G$-adapted by \cite[Proposition 2.3.11.(b)]{weizsaecker1990stochastic}, another application of \cite[Theorem 4.3.3]{weizsaecker1990stochastic} thus implies that it is enough to consider the case where $M$ is $\G$-adapted and a square-integrable $(\G_\smallertext{+},\P)$-martingale and $Z \in \H^2(M;\G,\P)$. In this case, we can find by the proof of \cite[Theorem I.4.40]{jacod2003limit}, a sequence of elementary $\G_\smallertext{+}$-predictable processes $(Z^n)_{n \in \N}$ in the sense of \cite[Definition 4.4.1, Proposition 4.4.2.(b)]{weizsaecker1990stochastic} such that
		\begin{equation*}
			\E^\P\bigg[\sup_{t \in [0,\infty)}|Z^n\bcdot M_t - Z\bcdot M_t|^2\bigg]
			\leq 4 \E^\P\bigg[\int_0^\infty (Z^n_u-Z_u)^2\d \langle M \rangle_u\bigg] \xrightarrow{n\rightarrow\infty} 0, 
		\end{equation*}
		as $n$ tends to infinity. Here the inequality follows from Doob's martingale inequality. Since $Z^n\bcdot M$ is right-continuous and $\G$-adapted, it follows from Markov's inequality and from \cite[Theorem 4.3.3]{weizsaecker1990stochastic}, that there is a version of $Z\bcdot M$ that is right-continuous and $\G$-adapted.
		
		\medskip
		We now turn to the construction of $U \ast\tilde\mu^X$. We denote by $\nu$ the `good version' of the compensator constructed in \Cref{lem::existence_predictable_compensator_mu}, and we let $D \coloneqq \{\Delta X \neq 0\}$ and $J \coloneqq \big\{(\omega,t) \in \Omega \times [0,\infty) : \nu(\omega ; \{t\} \times \R^d) > 0 \}$. As noted in the proof of \Cref{lem::existence_predictable_compensator_mu}, we can write $D$ (resp. $J$) as the countable union of disjoint graphs of $\G$--stopping times (resp. predictable $\G$--stopping times.)  Let $\widetilde U_t(\omega) \coloneqq U(\omega,t,\Delta X_t(\omega))\1_D(\omega,t) - \int_{\R^\smalltext{d}}U(\omega,t,x)\nu(\omega; \{t\} \times \d x)$.	 Then $\widetilde U$ is $\G$-optional and satisfies $D^\prime \coloneqq \{\widetilde U\neq0\} \subseteq D \cup J$. Hence, by \cite[Theorem 88.(a), page 139]{dellacherie1978probabilities}, the set $D^\prime$ is the countable union of disjoint graphs of $\G$--stopping times. Let 
		\begin{equation*}
			S(\widetilde U^2) \coloneqq \sum_{0 < s \leq \cdot} (\widetilde U_s)^2.
		\end{equation*}
		Since $U \in \H^2(\mu^X;\G,\P)$, we have $\E^\P[S(\widetilde U^2)_\infty] < \infty$ (see the proof of \cite[Theorem II.1.33.a)]{jacod2003limit}). We now follow the proof of \cite[Theorem I.4.56.a)]{jacod2003limit} rather closely to construct a purely discontinuous, square-integrable $(\G_\smallertext{+},\P)$-martingale $M$ satisfying $\Delta M = \widetilde U$ up to $\P$-indistinguishability. First, let us denote by $J^\prime$ the predictable support of $D^\prime$. By the proof of \cite[Proposition I.2.34]{jacod2003limit} and an application of \cite[Theorem 88.(b), page 139]{dellacherie1978probabilities}, we can choose a $\P$-version of $J^\prime$ such that it is exactly the countable union of disjoint graphs of predictable $\G$--stopping times $(T_n)_{n \in \N}$. Then by another application of \cite[Theorem 88.(a), page 139]{dellacherie1978probabilities}, the set $D^\prime \setminus J^\prime = D^\prime \setminus (\cup_{n\in\N} \llbracket T_n\rrbracket)$ is itself equal to a countable union of disjoint graphs of $\G$--stopping times $(S_n)_{n \in \N}$. We thus have that $D^\prime \subseteq (\bigcup_{n \in \N} \llbracket S_n\rrbracket)\cup(\bigcup_{n \in \N} \llbracket T_n\rrbracket)$ and the graphs of $\{S_n: n \in \N\}\cup\{T_n : n \in \N\}$ are pairwise disjoint. Since $J^\prime$ is the $\G$-predictable support of $D^\prime$ it follows that the $\G$-predictable support of $D^\prime \setminus J^\prime$ is $\P$-evanescent. This in turn implies that each $S_n$ is totally $\P$-inaccessible (see \cite[Remark I.2.33]{jacod2003limit}). Now let $A^n \coloneqq \widetilde U_{S_\smalltext{n}}\1_{\llbracket S_\smalltext{n},\infty \rrparenthesis}$. Since $A^n$ is right-continuous, $\G$-adapted, and of integrable variation due to $\E^\P[S(\widetilde U^2)_\infty] < \infty$, the compensator $A^{n, p}$ of $A^n$ can be chosen to be right-continuous and $\G$-predictable by \cite[Theorem 6.6.1]{weizsaecker1990stochastic}. Let $M^n \coloneqq A^n - A^{n, p}$, $N^n \coloneqq \widetilde U_{T_\smalltext{n}}\1_{\llbracket T_\smalltext{n},\infty\rrparenthesis}$ and $Y^n \coloneqq \sum_{m=1}^n(M^m + N^m)$. The proof of \cite[Lemma I.4.51]{jacod2003limit} then shows that $(Y^n)_{n \in \N}$ converges in the space of square-integrable $(\G_\smallertext{+},\P)$-martingales to some purely discontinuous, square-integrable $(\G_\smallertext{+},\P)$-martingale $Y$ that satisfies $\Delta Y = \widetilde U$ up to $\P$-evanescence. We then let  $U \ast\tilde\mu \coloneqq Y$. Since
		\begin{equation*}
			\E^\P\bigg[\sup_{t \in [0,\infty)} |Y^n_t - Y_t|^2\bigg]\xrightarrow{n\rightarrow\infty} 0,
		\end{equation*}
		and since each $Y^n$ is right-continuous and $\G$-adapted, an application of Markov's inequality together with \cite[Theorem 4.3.3]{weizsaecker1990stochastic} implies that the limit $Y = U \ast\tilde\mu$ can be chosen to be right-continuous and $\G$-adapted. This completes the proof.
	\end{proof}

	We will always choose right-continuous and $\G$-adapted $\P$-modifications for both types of stochastic integrals. However, when it is necessary to emphasise both the filtration and probability measure, we will write $(Z \bcdot M)^{(\G, \P)}$ and $(U \ast \tilde{\mu}^X)^{(\G, \P)}$. If only the probability measure needs to be highlighted and the underlying filtration is clear, we will use the simplified notation $(Z \bcdot M)^{(\P)}$ and $(U \ast \tilde{\mu}^X)^{(\P)}$.

	\medskip
	Lastly, we turn to orthogonal decompositions of martingales. We borrow the notation from \cite[Chapters II and III]{jacod2003limit}. Let $\widetilde\cP(\G) \coloneqq \cP(\G)\otimes\cB(\R^d)$, and let $M^\P_{\mu^\smalltext{X}}$ be the measure on $\cF\otimes\cB([0,\infty))\otimes\cB(\R^d)$ defined by
	\begin{equation*}
		M^\P_{\mu^\smalltext{X}}[A] \coloneqq \E^\P[\1_A\ast\mu^X],\; A\in \cF\otimes\cB([0,\infty))\otimes\cB(\R^d).
	\end{equation*}
	We also denote by $M^\P_{\mu^\smalltext{X}}[W]$ the integral with respect to $M^\P_{\mu^\smalltext{X}}$ of an $\cF\otimes\cB([0,\infty))\otimes\cB(\R^d)$-measurable, nonnegative function $W$. Then $M^\P_{\mu^\smalltext{X}}[W|\widetilde\cP(\G)]$ denotes the, up to a $M^\P_{\mu^\smalltext{X}}$--null set, unique $\widetilde\cP(\G)$-measurable function $W^\prime$ satisfying
	\begin{equation*}
		M^\P_{\mu^\smalltext{X}}[W U] = M^\P_{\mu^\smalltext{X}}[W^\prime U],\; \text{for every bounded, nonnegative and $\widetilde\cP(\G)$-measurable function $U$.}
	\end{equation*}
	For a $[-\infty,\infty]$-valued, $\cF\otimes\cB([0,\infty))\otimes\cB(\R^d)$-measurable function $W$, we let $M^\P_{\mu^\smalltext{X}}[W|\widetilde\cP(\G)] \coloneqq M^\P_{\mu^\smalltext{X}}[W^\smallertext{+}|\widetilde\cP(\G)] - M^\P_{\mu^\smalltext{X}}[W^\smallertext{-}|\widetilde\cP(\G)]$, where we again used the convention $\infty - \infty = -\infty$. The following result appears as Proposition 2.6 in \citeauthor*{possamai2024reflections} \cite{possamai2024reflections}.
	\begin{proposition}
		Let $m$ be a positive integer, let $M = (M_t)_{t \in [0,\infty)}$ be an $\R^m$-valued, right-continuous, $\G$-adapted, $(\G,\P)$--locally square-integrable martingale, and let $X = (X_t)_{t \in [0,\infty)}$ be an $\R^d$-valued, c\`adl\`ag, $\G$-adapted process. Suppose that $M^\P_{\mu^\smalltext{X}}[\Delta M^i|\widetilde\cP(\G)] = 0$ for $i \in \{1,\ldots,m\}$. For every real-valued, right-continuous, $\G_\smallertext{+}$-adapted, $(\G_\smallertext{+},\P)$--square-integrable martingale $L$, there exists a unique pair $(Z,U) \in \H^2(M;\G,\P)\times\H^2(\mu^X;\G,\P)$ such that
		\begin{equation*}
			N \coloneqq L - L_0 - Z\bcdot M - U\ast\tilde\mu^X,
		\end{equation*}
		satisfies $\langle N, M^i\rangle^{(\G_\tinytext{+},\P)} = 0$, for each $i \in \{1,\ldots,m\}$, as well as $M^\P_{\mu^\smalltext{X}}[\Delta N|\widetilde\cP(\G)] = 0$.
	\end{proposition}

	\section{Proofs of Section \ref{sec::preliminaries}}\label{sec::proofs_preliminaries}
	
	In this part, we provide proofs of the results mentioned in \Cref{sec::preliminaries}.

		\begin{proof}[Proof of \Cref{lem::conditioning_martingale2}]
			We follow the arguments in the proof of \cite[Lemma 3.3]{neufeld2016nonlinear}. It follows from Galmarino's test (see \cite[Theorem IV.101.(b)]{dellacherie1978probabilities}) that $M^{\tau,\omega}_{\tau\smallertext{+}\smallertext{\cdot}}$ is $\F_\smallertext{+}$-adapted for $\omega \in \Omega$. The square-integrability follows from
			\begin{equation*}
				\E^{\P^{\smalltext{\tau}\smalltext{,}\smalltext{\omega}}}\bigg[\sup_{t \in [0,\infty)} |M^{\tau,\omega}_{\tau\smallertext{+}t}|^2 \bigg]
				= \E^{\P^{\smalltext{\tau}\smalltext{,}\smalltext{\omega}}}\bigg[\bigg(\sup_{t \in [0,\infty)} |M_{\tau\smallertext{+}t}|^2 \bigg)^{\tau,\omega}\bigg]
				= \E^{\P}\bigg[\sup_{t \in [0,\infty)} |M_{\tau\smallertext{+}t}|^2\bigg|\cF_\tau\bigg](\omega) < \infty, \; \textnormal{$\P$--a.e. $\omega \in \Omega$.}
			\end{equation*}
			We turn to the martingale property. Fix $0 \leq s < s +\varepsilon < t < \infty$, a bounded, $\cF_{s\smallertext{+}\varepsilon}$-measurable function $g$, and define $\tilde{g}(\omega)\coloneqq g(\omega_{\tau(\omega)\smallertext{+}\smallertext{\cdot}}-\omega_{\tau(\omega)})$. Then $\tilde{g}$ is $\cF_{\tau\smallertext{+}s\smallertext{+}\varepsilon}$-measurable by Galmarino's test, and $\tilde{g}^{\tau,\omega} = g$. By the optional sampling theorem under $\P$ (see \cite[Corollary 3.2.8]{weizsaecker1990stochastic}), we have
			\begin{equation*}
				\E^{\P^{\smalltext{\tau}\smalltext{,}\smalltext{\omega}}}\big[ \big(M^{\tau,\omega}_{\tau\smallertext{+}t}-M^{\tau,\omega}_{\tau\smallertext{+}s\smallertext{+}\varepsilon} \big) g\big] 
				= \E^{\P}\big[ \big(M_{\tau\smallertext{+}t}-M_{\tau\smallertext{+}s\smallertext{+}\varepsilon} \big) \tilde{g} \big| \cF_{\tau}\big](\omega)
				= \E^{\P}\Big[ \E^\P\big[ \big(M_{\tau\smallertext{+}t}-M_{\tau\smallertext{+}s\smallertext{+}\varepsilon} \big) \big| \cF_{(\tau\smallertext{+}s\smalltext{+}\varepsilon)\smallertext{+}} \big] \tilde{g} \Big| \cF_{\tau}\Big](\omega) = 0,
			\end{equation*}
			for $\P$--a.e. $\omega \in \Omega$. A functional monotone class argument, together with the fact that $\cF_{s\smallertext{+}\varepsilon}$ is countably generated, then implies that
			\begin{equation*}
				\E^{\P^{\smalltext{\tau}\smalltext{,}\smalltext{\omega}}}\big[ \big(M^{\tau,\omega}_{\tau\smallertext{+}t}-M^{\tau,\omega}_{\tau\smallertext{+}s\smallertext{+}\varepsilon} \big) g\big]=0,
			\end{equation*}
			holds simultaneously for all $\cF_{s\smallertext{+}\varepsilon}$-measurable and bounded functions $g$, for every $\omega \in \Omega\setminus N$ with $N$ being an $(\cF,\P)$--null set. This implies that
			\begin{equation*}
				\E^{\P^{\smalltext{\tau}\smalltext{,}\smalltext{\omega}}}\big[ M^{\tau,\omega}_{\tau\smallertext{+}t}\big|\cF_{s\smallertext{+}\varepsilon}\big] = \E^{\P^{\smalltext{\tau}\smalltext{,}\smalltext{\omega}}}\big[M^{\tau,\omega}_{\tau\smallertext{+}s\smallertext{+}\varepsilon}\big|\cF_{s\smallertext{+}\varepsilon}\big], \; \textnormal{$\P^{\tau,\omega}$--a.s.}, \; \textnormal{$\P$--a.e. $\omega \in \Omega$}.
			\end{equation*}
			Conditioning with respect to $\cF_{s\smallertext{+}}$, and by right-continuity and $\F_\smallertext{+}$-adaptedness of $M^{\tau,\omega}_{\tau\smallertext{+}\smallertext{\cdot}}$, we find by letting $\varepsilon$ tend to zero along a fixed rational sequence, that
			\begin{equation*}
				\E^{\P^{\smalltext{\tau}\smalltext{,}\smalltext{\omega}}}\big[ M^{\tau,\omega}_{\tau\smallertext{+}t}\big|\cF_{s\smallertext{+}}\big] = \E^{\P^{\smalltext{\tau}\smalltext{,}\smalltext{\omega}}}\big[ M^{\tau,\omega}_{\tau\smallertext{+}s\smallertext{+}\varepsilon}\big|\cF_{s\smallertext{+}}\big] 
				\xrightarrow{\varepsilon\downarrow\downarrow 0} \E^{\P^{\smalltext{\tau}\smalltext{,}\smalltext{\omega}}}\big[M^{\tau,\omega}_{\tau\smallertext{+}s}\big|\cF_{s\smallertext{+}}\big] = M^{\tau,\omega}_{\tau\smallertext{+}s}, \; \textnormal{$\P^{\tau,\omega}$--a.s.}, \; \textnormal{$\P$--a.e. $\omega \in \Omega$.}
			\end{equation*}
			Thus 
			\begin{equation*}
				\E^{\P^{\smalltext{\tau}\smalltext{,}\smalltext{\omega}}}\big[ M^{\tau,\omega}_{\tau\smallertext{+}t}\big|\cF_{s\smallertext{+}}\big] = M^{\tau,\omega}_{\tau\smallertext{+}s}, \; \textnormal{$\P^{\tau,\omega}$--a.s.,}\; \textnormal{$\P$--a.e. $\omega \in \Omega$.}
			\end{equation*}
			The above martingale property can be extended to all rational times $0 \leq s < t < \infty$ on the complement of a $\P$--null set, and by an application of the backward martingale convergence theorem \cite[Theorem V.33]{dellacherie1982probabilities} together with the right-continuity of $M^{\tau,\omega}_{\tau\smallertext{+}\smallertext{\cdot}}$ the martingale property then also extends to  $[0,\infty)$, for $\P$--a.e. $\omega \in \Omega$.
			
			\medskip
			The result for the filtration $\F$ is now immediate: using the backward martingale convergence theorem, it follows that $M$ is an $(\F_\smallertext{+}, \P)$--square-integrable martingale, and Galmarino's test (see \cite[Theorem IV.100(b)]{dellacherie1978probabilities}) implies that $M^{\tau,\omega}_{\tau\smallertext{+}\smallertext{\cdot}}$ is adapted to both $\F$ and $\F_\smallertext{+}$ for $\omega \in \Omega$ (see \cite[Theorems IV.100(b) and IV.101(b)]{dellacherie1978probabilities}). The preceding considerations then imply that $M^{\tau,\omega}_{\tau\smallertext{+}\smallertext{\cdot}}$ is an $(\F_\smallertext{+}, \P^{\tau,\omega})$--square-integrable martingale for $\P$--a.e. $\omega \in \Omega$. The $\F$--adaptedness of $M^{\tau,\omega}_{\tau\smallertext{+}\cdot}$ then implies that this also holds with respect to $\F$. This completes the proof.
		\end{proof}
	
\begin{proof}[Proof of \Cref{lem::measurability_characteristics}]
We closely follow the arguments in the proof of \cite[Lemma 7.1]{neufeld2014measurability}, with only minor adjustments at specific points. Note that $\Omega\times\Omega\times[0,\infty) \ni (\omega,\tilde{\omega},t) \longmapsto (\omega\otimes_\tau\tilde{\omega},\tau(\omega)+t) \in \Omega \times [0,\infty)$ is $\cF\otimes\cP$--$\cP$-measurable by \cite[Theorem IV.97.(a)]{dellacherie1978probabilities}.
Hence
\[
\Omega \times \Omega \times [0,\infty) \ni (\omega,\tilde\omega,t) \longmapsto C^{\tau,\omega}_{\tau+t}(\tilde\omega) = C_{\tau(\omega)+t}(\omega\otimes_\tau\tilde\omega) \in \R
\]
is $\cF\otimes\cP$-measurable and thus also Borel.

\medskip
We turn to the Borel-measurability of $\widehat{\Omega}^C_\tau \subseteq \Omega\times\fP_\textnormal{sem}$. Let $(\mathsf{B}^\P,\mathsf{C},\nu^\P)$ be the versions of the $(\F,\P)$-characteristics of $X$ for $\P \in \fP_\textnormal{sem}$ which are measurable in the probability parameter $\P$ and constructed in \cite[Theorem 2.5]{neufeld2014measurability}.	Let
\[
R^\P \coloneqq \sum_{i = 1}^d \textnormal{Var}(\mathsf{B}^{\P,i}) + \sum_{i,j = 1}^d \textnormal{Var}(\mathsf{C}^{i,j}) + (|x|^2\land1)\ast\nu^\P,
\]
where
\[
\textnormal{Var}(f)_t \coloneqq \lim_{n \rightarrow \infty}\sum_{k = 1}^{2^n} |f_{kt/2^n} - f_{(k-1)t/2^n}|
\]
for any $f: [0,\infty) \longrightarrow \R$; in case $f$ is right-continuous, $\textnormal{Var}(f)_t$ exactly corresponds to the total variation of $f$ on $[0,t]$. It follows from \cite[Theorem 2.5]{neufeld2014measurability} that $\fP_\textnormal{sem}\times\Omega\times[0,\infty)\ni (\P,\omega,t) \longmapsto R^\P_t(\omega) \in [0,\infty]$ is Borel-measurable and that $R^\P$ is $\P$--a.s. finite-valued and right-continuous. Moreover, $\mathsf{B}^\P$, $\mathsf{C}$ and $(|x|^2\land 1)\ast \nu^\P$ are $\P$--a.s. absolutely continuous (component-wise) relative to $R^\P$. Define (with as usual for us $0/0 = 0$ and $\infty - \infty = -\infty$)
\[
\varphi^{\omega,\P,n}_t(\tilde\omega) \coloneqq \sum_{k = 0}^\infty \frac{\big(R^\P_{(k+1)2^{\smalltext{-}\smalltext{n}}}(\tilde\omega) - R^\P_{k2^{\smalltext{-}\smalltext{n}}}(\tilde\omega)\big)}{\big(C^{\tau,\omega}_{\tau\smallertext{+}(k\smallertext{+}1)2^{\smalltext{-}\smalltext{n}}}(\tilde\omega) - C^{\tau,\omega}_{\tau\smallertext{+}k2^{\smalltext{-}\smalltext{n}}}(\tilde\omega)\big)}\mathbf{1}_{(k2^{\smalltext{-}\smalltext{n}},(k+1)2^{\smalltext{-}\smalltext{n}}]}(t), \; (\omega,\P,\tilde\omega,t) \in \Omega\times\fP_\textnormal{sem}\times\Omega\times[0,\infty),
\]
and then
\begin{equation*}
	\varphi^{\omega,\P}_t(\tilde\omega) \coloneqq \limsup_{n \rightarrow \infty} \varphi^{\omega,\P,n}_t(\tilde\omega), \; (\omega,\P,\tilde\omega,t) \in \Omega\times\fP_\textnormal{sem}\times\Omega\times[0,\infty).
\end{equation*}
Then $\Omega\times\fP_\textnormal{sem}\times\Omega\times[0,\infty) \ni (\omega,\P,\tilde\omega,t) \longmapsto \varphi^{\omega,\P}_t(\tilde\omega) \in [-\infty,\infty]$ is Borel-measurable and for each $(\omega,\P) \in \Omega\times\fP_\textnormal{sem}$, the function $\varphi^{\omega,\P}$ is, $\P$--a.s., the Radon--Nikod\'ym derivative with respect to $\d(C^{\tau,\omega}_{\tau\smallertext{+}\smallertext{\cdot}}-C^{\tau,\omega}_{\tau(\omega)}(\omega))$ of the absolutely continuous component of $\d R^\P$ relative to $\d(C^{\tau,\omega}_{\tau\smallertext{+}\smallertext{\cdot}}-C_{\tau(\omega)}(\omega))$ (see \cite[Section XI.17, pages 199--201]{doob1994measure} or \cite[Theorem V.58]{dellacherie1982probabilities}, as well as the remark afterwards). It follows that
\begin{equation}\label{eq::borel_measurability_hat_omega}
	\widehat{\Omega}^C_\tau = \big\{(\omega,\P) \in \Omega\times\fP_\textnormal{sem} : \E^\P\big[\1_{G}(\omega,\P,\cdot)\big] = 1 \big\},
\end{equation}
where
\[
G = \bigg\{(\omega,\P,\tilde\omega) \in \Omega\times\fP_\textnormal{sem}\times\Omega :R^\P_t(\tilde\omega) = \int_0^t \varphi^{\omega,\P}_s(\tilde\omega)\d(C^{\tau,\omega}_{\tau\smallertext{+}\smallertext{\cdot}}-C_{\tau(\omega)}(\omega))_s(\tilde{\omega}), \, \textnormal{$t \in \Q\cap[0,\infty)$}\bigg\}.
\]
The Borel measurability of the set $G$ follows from Fubini's theorem for kernels (see \cite[Proposition 6.9, page 40]{cinlar2011probability}). Moreover, for any bounded and Borel-measurable function $h$ defined on $\Omega\times\fP_\textnormal{sem}\times\Omega$, the map
\[
\Omega\times\fP_\textnormal{sem} \ni (\omega,\P) \longmapsto \E^\P[h(\omega,\P,\cdot)] \in \R,
\]
is Borel-measurable; this follows from a functional monotone class argument by considering $h(\omega,\P,\tilde\omega) = f(\omega)g(\P,\tilde\omega)$ for bounded $f$ and $g$, noting that $\cB(\Omega\times\fP_\textnormal{sem}) = \cB(\Omega)\otimes\cB(\fP_\textnormal{sem})$ (see \cite[Lemma 6.4.2.(i)]{bogachev2007measure}), and the fact that $\fP_\textnormal{sem}\ni \P \longmapsto \E^\P[g(\P,\cdot)]$ is Borel-measurable (see \cite[Lemma 3.1]{neufeld2014measurability}). Thus \eqref{eq::borel_measurability_hat_omega} implies that  $\widehat{\Omega}^C_\tau \in \cB(\Omega\times \fP_\textnormal{sem}) \subseteq \cB(\Omega\times\fP(\Omega)) = \cB(\Omega)\otimes\cB(\fP(\Omega))$.

\medskip
We turn to the differential characteristics. For $(i)$ and $(ii)$, we define component-wise
\begin{gather*}
	\mathsf{b}^{\tau,\omega,\P}_t \coloneqq \tilde{\mathsf{b}}^{\tau,\omega,\P}_t \1_{\{\tilde{\mathsf{b}}^{\smalltext{\tau}\smalltext{,}\smalltext{\omega}\smalltext{,}\smalltext{\P}}_\smalltext{t} \in \R^\smalltext{d}\}}, \; \text{where} \; \tilde{\mathsf{b}}^{\tau,\omega,\P}_t \coloneqq \limsup_{n\rightarrow\infty}\frac{\mathsf{B}^{\P}_t - \mathsf{B}^{\P}_{(t-1/n)\lor 0}}{C^{\tau,\omega}_{\tau+t} - C^{\tau,\omega}_{\tau+(t-1/n) \lor 0}}, \; (\omega,\P,t) \in \Omega\times\fP_\textnormal{sem}\times[0,\infty),\\
	\mathsf{a}^{\tau,\omega}_t \coloneqq \tilde{\mathsf{a}}^{\tau,\omega}_t \1_{\{\tilde{\mathsf{a}}^{\smalltext{\tau}\smalltext{,}\smalltext{\omega}}_\smalltext{t}\in\S^\smalltext{d}_\tinytext{+}\}}, \; \text{where} \; \tilde{\mathsf{a}}^{\tau,\omega}_t \coloneqq \limsup_{n\rightarrow\infty}\frac{\mathsf{C}_t - \mathsf{C}_{(t-1/n)\lor 0}}{C^{\tau,\omega}_{\tau+t} - C^{\tau,\omega}_{\tau+(t-1/n) \lor 0}}, \; (\omega,t) \in \Omega\times[0,\infty),
\end{gather*}
which then satisfy the desired properties in the statement.
		
\medskip
We turn to the third characteristic and apply the following two results.

\begin{lemma}\label{lem::measurable_predictable_projection}
Let $\fP(\Omega)\times\Omega\times[0,\infty)\ni (\P,\omega,t) \longmapsto V^\P_t(\omega) \in \R$ be Borel-measurable and either bounded or nonnegative. There exists a $\cB(\fP(\Omega))\otimes\cP$-measurable, and thus Borel, map $\fP(\Omega)\times\Omega\times[0,\infty)\ni (\P,\omega,t) \longmapsto \prescript{p}{}{V}^\P_t(\omega) \in [-\infty,\infty]$ such that, for each $\P\in\fP(\Omega)$, the process $\prescript{p}{}{V}^\P$ is the $(\F,\P)$--predictable projection of $V^\P$, that is, for every $\P\in\fP(\Omega)$, $\prescript{p}{}{V}^\P$ is $\F$-predictable and, for every $\F$--predictable stopping time $\tau$,
\[
	\E^\P\big[V^\P_\tau \1_{\{\tau < \infty\}}\big|\cF_{\tau\smallertext{-}}\big] = \prescript{p}{}{V}^\P_\tau \1_{\{\tau < \infty\}}, \; \textnormal{$\P$--a.s.}
\]
\end{lemma}

\begin{proof}
This follows from \cite[Lemma 3.5]{neufeld2014measurability} and the proof of \cite[Theorem VI.43]{dellacherie1982probabilities}.
\end{proof}

\begin{lemma}\label{lem::measurable_disintegration_compensator}
For every $\P\in\fP_\textnormal{sem}$, there exists a decomposition
\[
	\nu^\P(\d t, \d x) = \mathsf{K}^\P_t (\d x) \d \mathsf{A}^\P_t, \; \textnormal{$\P$--a.s.},
\]
where $(\P,\omega,t) \longmapsto \mathsf{K}^\P_{\omega,t}(\d x)$ is a kernel on $(\R^d,\cB(\R^d))$ given $(\fP_\textnormal{sem}\times\Omega\times[0,\infty), \cB(\fP_\textnormal{sem})\otimes\cP)$, and $(\P,\omega,t) \longmapsto \mathsf{A}^\P_t(\omega)$ is $[0,\infty)$-valued and $\cB(\fP_\textnormal{sem})\otimes\cP$-measurable $($thus Borel on $\fP_\textnormal{sem}\times\Omega\times[0,\infty)$$)$ and for every $\P\in\fP_\textnormal{sem}$, the process $\mathsf{A}^\P$ is an $\F$-predictable process starting at zero with \textnormal{$\P$--a.s.} right-continuous and non-decreasing paths, and $\P$-integrable $($$\E^\P[\mathsf{A}^\P_\infty] < \infty$$)$.
\end{lemma}

\begin{proof}
	By the proof of \cite[Proposition~6.4]{neufeld2014measurability}, there exist a kernel $(\P,\omega,t) \longmapsto \mathsf{K}^\P_{\omega,t}(\d x)$ on $(\R^d,\cB(\R^d))$ given $(\fP_\textnormal{sem}\times\Omega\times[0,\infty),\cB(\fP_\textnormal{sem})\otimes\cP)$ and a real-valued, Borel-measurable map $(\P,\omega,t) \longmapsto A^\P_t(\omega)$ on $\fP_\textnormal{sem}\times\Omega\times[0,\infty)$ such that $A^\P$ is $\P$-integrable, $\F^\P_\smallertext{+}$-predictable, $\F_\smallertext{+}$-adapted, right-continuous and $\P$--a.s. non-decreasing, $A^\P_0 = 0$, and
  	\begin{equation*}
		\nu^\P(\,\cdot\,;\d t, \d x) = \mathsf{K}^\P_{\cdot,t}(\d x)\d A^\P_t, \; \textnormal{$\P$--a.s.}, \; \P\in\fP_\textnormal{sem}.
	\end{equation*}
	Replacing $A^\P$ by $(A^\P)^+$, which does not change it up to $\P$--indistinguishability, we may assume that $A^\P$ is $[0,\infty)$-valued. It follows from \Cref{lem::measurable_predictable_projection} that there exists a $\cB(\fP_\textnormal{sem})\otimes\cP$-measurable map $\fP_\textnormal{sem} \times \Omega \times [0,\infty) \ni (\P,\omega,t) \longmapsto \mathsf{A}^\P_t(\omega) \in [0,\infty]$ such that $\mathsf{A}^\P$ is the $(\F,\P)$--predictable projection of $A^\P$. Since $A^\P$ is $\F^\P_\smallertext{+}$-predictable, it is $\P$--indistinguishable from an $\F$-predictable process (see \cite[Appendix I, Lemma 7]{dellacherie1982probabilities}), and by uniqueness of the predictable projection, we obtain that $\mathsf{A}^\P = A^\P$ up to $\P$--indistinguishability. This implies that $\mathsf{A}^\P$ inherits the $\P$-integrability and path properties (outside some $\P$--negligible set) of $A^\P$. Replacing $\mathsf{A}^\P$ by $\mathsf{A}^\P\1_{\{\mathsf{A}^\P<\infty\}}$ and then setting $\mathsf{A}^\P_0\coloneqq0$ concludes the proof.
\end{proof}

For each $\P\in\fP_\textnormal{sem}$, let $(\mathsf{K}^{\P},\mathsf{A}^{\P})$ denote the disintegration from \Cref{lem::measurable_disintegration_compensator}. Let
\begin{equation*}
	\alpha^{\tau,\omega,\P}_t \coloneqq \tilde{\alpha}^{\tau,\omega,\P}_t\1_{\{\tilde{\alpha}^{\smalltext{\tau}\smalltext{,}\smalltext{\omega}\smalltext{,}\smalltext{\P}}_\smalltext{t} \in [0,\infty)\}}, \; 
	\text{where} \;
	\tilde{\alpha}^{\tau,\omega,\P}_t \coloneqq \limsup_{n \rightarrow \infty}\frac{\mathsf{A}^{\P}_t - \mathsf{A}^{\P}_{(t-1/n)\lor 0}}{C^{\tau,\omega}_{\tau+t} - C^{\tau,\omega}_{\tau+(t-1/n)\lor 0}}, \; (\omega,\P,t) \in \Omega\times\fP_\textnormal{sem}\times[0,\infty),
\end{equation*}
and then
\begin{equation*}
	D \coloneqq \bigg\{(\omega,\P,\tilde\omega,t) \in \Omega \times \fP_\textnormal{sem} \times \Omega \times [0,\infty): \int_{\R^d}(|x|^2 \land 1) \alpha^{\tau,\omega,\P}_t(\tilde\omega)\mathsf{K}^{\P}_{\tilde\omega,t}(\d x) < \infty \; \text{and} \; \alpha^{\tau,\omega,\P}_t(\tilde\omega)\mathsf{K}^{\P}_{\tilde{\omega},t}(\{0\}) = 0\bigg\}.
\end{equation*}
The set $D$ is in $\cF\otimes\cB(\fP_\textnormal{sem})\otimes\cP$ and thus Borel-measurable.
Define
\begin{equation*}
	\mathsf{K}^{\tau,\omega,\P}_{\tilde\omega,t}(\d x) \coloneqq \1_{D}(\omega,\P,\tilde\omega,t) \alpha^{\tau,\omega,\P}_t(\tilde\omega) \mathsf{K}^{\P}_{\tilde\omega,t}(\d x)  , \; (\omega,\P,\tilde\omega,t) \in \Omega\times\fP_\textnormal{sem}\times\Omega\times[0,\infty),
\end{equation*}
which satisfies $\mathsf{K}^{\tau,\omega,\P}_{\tilde\omega,t}(\d x) \in \cL$ by construction. Note also that $(\omega,\P,\tilde\omega,t)\longmapsto \mathsf{K}^{\tau,\omega,\P}_{\tilde\omega,t}(\d x)$ is a kernel on $(\R^d,\cB(\R^d))$ given $(\Omega \times \fP_\textnormal{sem} \times \Omega \times [0,\infty), \cF\otimes\cB(\fP_\textnormal{sem}) \otimes \cP)$.
The $\cF\otimes\cB(\fP_\textnormal{sem})\otimes\cP$-measurability, and thus Borel-measurability, of 
\begin{equation*}
	\Omega \times \fP_\textnormal{sem} \times \Omega \times [0,\infty) \ni (\omega,\P,\tilde\omega,t) \longmapsto \mathsf{K}^{\tau,\omega,\P}_{\tilde\omega,t}(\d x) \in \cL,
\end{equation*}
follows from an application of \cite[Lemma 2.4]{neufeld2014measurability}. That this indeed is the correct third differential characteristic can be argued as follows. For fixed $(\omega,\P) \in \widehat{\Omega}^C_\tau$, it follows from \Cref{lem::absolute_continuity_nu} that
\[
	\nu^\P(\d t, \d x) = \alpha^{\tau,\omega,\P}_t \mathsf{K}^{\P}_t(\d x)   \d (C^{\tau,\omega}_{\tau\smallertext{+}\smallertext{\cdot}}-C_{\tau(\omega)}(\omega))_t, \; \text{$\P$--a.s.},
\]
and therefore also $\1_D(\omega,\P,\cdot,\cdot) = 1$ up to a $\P\otimes\d(C^{\tau,\omega}_{\tau\smallertext{+}\smallertext{\cdot}}-C_{\tau(\omega)}(\omega))$--null set. Thus
\[
	\nu^\P(\d t, \d x) = \alpha^{\tau,\omega,\P}_t \mathsf{K}^{\P}_t(\d x)   \d (C^{\tau,\omega}_{\tau\smallertext{+}\smallertext{\cdot}}-C_{\tau(\omega)}(\omega))_t = \mathsf{K}^{\tau,\omega,\P}_{t}(\d x) \d (C^{\tau,\omega}_{\tau\smallertext{+}\smallertext{\cdot}}-C_{\tau(\omega)}(\omega))_t, \; \textnormal{$\P$--a.s.}, \; (\omega,\P) \in \widehat{\Omega}^C_\tau.
\]
This completes the proof.
\end{proof}
	
\begin{proof}[Proof of \Cref{cor::shif_quadratic_variation_continuous_martingale_part}]
We only prove the first assertion, as the second one follows immediately from \Cref{prop::conditioning_characteristics2}.$(i)$ or \cite[Theorem 3.1]{neufeld2016nonlinear}. We also suppose, without loss of generality, that $X$ is one-dimensional; otherwise, we argue component-wise. First, the process $(X^{c,\P})^{\tau,\omega}_{\tau\smallertext{+}\smallertext{\cdot}} - X^{c,\P}_{\tau(\omega)}(\omega)$ is a right-continuous, $(\F,\P^{\tau,\omega})$--local martingale with $\P^{\tau,\omega}$--a.s. continuous paths for $\P$--a.e. $\omega \in \Omega$ by \cite[Lemma 3.4]{neufeld2016nonlinear}. Since $X^{c,\P}$ is the continuous local martingale part of $X$ relative to $(\F,\P)$, we have
\[
[X-X^{c,\P}]^{(\F,\P)} = \sum_{0 < s \leq \cdot} (\Delta X_s)^2, \; \textnormal{$\P$--a.s.},
\]
which then yields, using $X^{\tau,\omega}_{\tau\smallertext{+}\smallertext{\cdot}} - X_{\tau(\omega)}(\omega) = X$,
\begin{align*}
	\big[X - \big((X^{c,\P})^{\tau,\omega}_{\tau\smallertext{+}\smalltext{\cdot}} - X^{c,\P}_{\tau(\omega)}(\omega)\big) \big]^{(\F,\P^{\smalltext{\tau}\smalltext{,}\smalltext{\omega}})} 
	&= \big[X^{\tau,\omega}_{\tau\smallertext{+}\smalltext{\cdot}} - X_{\tau(\omega)}(\omega) - \big((X^{c,\P})^{\tau,\omega}_{\tau\smallertext{+}\smalltext{\cdot}} - X^{c,\P}_{\tau(\omega)}(\omega)\big) \big]^{(\F,\P^{\smalltext{\tau}\smalltext{,}\smalltext{\omega}})} \\
	&= [X - X^{c,\P}]^{(\F,\P)}_{\tau\smallertext{+}\smalltext{\cdot}}(\omega\otimes_\tau\cdot) - [X-X^{c,\P}]^{(\F,\P)}_{\tau(\omega)}(\omega) \\
	&= \sum_{\tau(\omega) < s \leq \tau(\omega)\smallertext{+}\cdot} (\Delta X_s(\omega\otimes_\tau\cdot))^2 \\
	&= \sum_{0 < s \leq \cdot} (\Delta X_{\tau\smallertext{+}s}(\omega\otimes_\tau\cdot))^2 = \sum_{0 < s \leq \cdot} (\Delta X_{s})^2, \; \textnormal{$\P^{\tau,\omega}$--a.s.},
\end{align*}
for $\P$--a.e. $\omega \in \Omega$. Here we used \cite[Theorem I.4.47.a)]{jacod2003limit} to justify the second equality. Therefore, by \cite[Theorem I.4.52]{jacod2003limit}, we have
$\big\langle X^{c,\P^{\smalltext{\tau}\smalltext{,}\smalltext{\omega}}} - \big((X^{c,\P})^{\tau,\omega}_{\tau\smallertext{+}\smallertext{\cdot}} - X^{c,\P}_{\tau(\omega)}(\omega)\big)\big\rangle^{(\F,\P^{\smalltext{\tau}\smalltext{,}\smalltext{\omega}})} = 0$, $\P^{\tau,\omega}$--a.s., for $\P$--a.e. $\omega \in \Omega$; the continuous local martingale part of the sum of two semi-martingales is the sum of the respective continuous local martingale parts. Since both local martingales start at zero, this completes the proof.
\end{proof}
	
	\section{Proofs of auxiliary lemmata from Section \ref{sec::proofs_main_results}}\label{sec::lemmas_main_results}
	
	This part is devoted to the proofs of the lemmata introduced in \Cref{sec::proofs_main_results}. In \Cref{sec::proofs_lemmata_measurability} and \ref{sec::proofs_lemmata_regularisation}, we provide the proofs to the lemmata of \Cref{sec::proof_measurability} and \ref{sec::proof_regularisation}, respectively. 
	
	\subsection{Auxiliary lemmata from Section \ref{sec::proof_measurability}}\label{sec::proofs_lemmata_measurability}
	Before proving \Cref{lem::measurable_decomposition}, we need to establish the following two results.
	
	\begin{lemma}\label{lem::measurability_cond_M_tilde_P2}
		Suppose that $\Omega \times \fP(\Omega) \times \Omega \times [0,\infty) \times \R^d \ni (\omega,\P,\tilde\omega,t,x) \longmapsto \cW^{\omega,\P}_t(\tilde\omega; x) \in [-\infty,\infty]$ is Borel-measurable. For each $(\omega,\P) \in \Omega\times\fP(\Omega)$, there exists a version of $M^\P_{\mu^\smalltext{X}}\big[\cW^{\omega,\P}\big|\widetilde\cP\big]$ such that
		\begin{equation*}
			\Omega \times \fP(\Omega) \times \Omega \times [0,\infty) \times \R^d \ni (\omega,\P,\tilde\omega,t,x) \longmapsto M^\P_{\mu^\smalltext{X}}\big[\cW^{\omega,\P}\big|\widetilde\cP\big](\tilde\omega,t,x) \in [-\infty,\infty],
		\end{equation*}
		is $\cB(\Omega)\otimes\cB(\fP(\Omega))\otimes\widetilde\cP$-measurable.	
	\end{lemma}
	
	\begin{proof}
		We closely follow the proof of \cite[Lemma 3.1]{neufeld2014measurability} and make appropriate modifications. By \cite[Lemma 6.5]{neufeld2014measurability}, there exists a positive, $\widetilde\cP$-measurable function $V$ satisfying $0\leq V\ast\mu^X_\infty \leq 1$ identically. Let $\mu^V(\omega;\d s, \d x) \coloneqq V_s(\omega,x)\mu^X(\omega;\d s, \d x)$, and suppose, without loss of generality, that the map $(\omega,\P,\tilde\omega,t,x) \longmapsto \cW^{\omega,\P}_t(\tilde\omega;x)$ is nonnegative and bounded; the nonnegative case follows by monotone convergence. 
		Let $(A_n)_{n\in \N}$ be a sequence of sets generating $\widetilde{\cP}$; the $\sigma$-algebra $\cP$ is countably generated by \cite[Lemma 6.3]{neufeld2014measurability}, thus so is $\widetilde{\cP}$. 
		For any $n\in\N$, let $(A^m_n)_{m \in\{0,\dots,k_\smalltext{n}\}}$ be a finite partition generating $\cA_n \coloneqq \sigma(A_1,\ldots,A_n)$.
		Then
		\begin{align*}
			M^{\P}_{\mu^\smalltext{V}}\big[\cW^{\omega,\P}\big|\widetilde\cP\big] \coloneqq
			\frac{1}{M^\P_{\mu^\smalltext{V}}[\mathbf{1}_{\tilde\Omega}]} \limsup_{n \rightarrow\infty}\sum_{m = 0}^{k_\smalltext{n}}\frac{M^\P_{\mu^\smalltext{V}}\big[\cW^{\omega,\P}\mathbf{1}_{A^\smalltext{m}_\smalltext{n}}\big]}{M^\P_{\mu^\smalltext{V}}[\mathbf{1}_{A^\smalltext{m}_\smalltext{n}}]/M^\P_{\mu^\smalltext{V}}[\mathbf{1}_{\tilde\Omega}]}\mathbf{1}_{A^\smalltext{m}_\smalltext{n}}, \; (\omega,\P) \in \Omega \times \fP(\Omega),
		\end{align*} 
		with convention $\lambda / 0 \coloneqq 0$ for any $\lambda \in \R$, is a version of the Radon--Nikodým derivative of the finite measure $A \longmapsto M^{\P}_{\mu^\smalltext{V}}[\1_A \cW^{\omega,\P}]$ with respect to the finite measure $A \longmapsto M^{\P}_{\mu^\smalltext{V}}[\1_A]$ on $(\widetilde{\Omega},\widetilde\cP)$; see \cite[V.56, pages 50--51]{dellacherie1982probabilities} or \cite[XI.17, pages 199--201]{doob1994measure}. Note that $(\omega,\P) \longmapsto M^{\P}_{\mu^\smalltext{V}}[\cW^{\omega,\P}] = \E^\P[\cW^{\omega,\P}\ast\mu^V]$ is measurable by \Cref{lem::measurable_martingale_modification}.$(i)$. We then extend this to the case of unbounded, nonnegative $(\omega,\P,\tilde\omega,t,x) \longmapsto \cW^{\omega,\P}_t(\tilde\omega;x)$ by setting
		\begin{gather*}
			M^\P_{\mu^\smalltext{V}}\big[\cW^{\omega,\P}\big|\widetilde\cP\big] \coloneqq \limsup_{n \rightarrow \infty} M^\P_{\mu^\smalltext{V}}\big[\cW^{\omega,\P} \land n\big|\widetilde\cP\big], \; (\omega,\P) \in \Omega \times \fP(\Omega),\\
			M^\P_{\mu^\smalltext{X}}\big[\cW^{\omega,\P}\big|\widetilde\cP\big] \coloneqq V M^\P_{\mu^\smalltext{V}}\big[\cW^{\omega,\P}/V\big|\widetilde\cP\big], \; (\omega,\P) \in \Omega \times \fP(\Omega).
		\end{gather*}
		It is then straightforward to verify that
		\begin{equation*}
			M^\P_{\mu^\smalltext{X}}\big[\cU M^\P_{\mu^\smalltext{X}}[\cW^{\omega,\P}|\widetilde\cP] \big] 
			= M^\P_{\mu^\smalltext{V}}\big[\cU M^\P_{\mu^\smalltext{V}}[\cW^{\omega,\P}/V|\widetilde\cP]\big] = M^\P_{\mu^\smalltext{V}}\big[\cU \cdot \cW^{\omega,\P}/V\big] = M^\P_{\mu^\smalltext{X}}\big[\cU \cdot \cW^{\omega,\P}\big], \; (\omega,\P) \in \Omega \times \fP(\Omega),
		\end{equation*}
		holds for each bounded, nonnegative, $\widetilde\cP$-measurable function $\cU$. 
		
		\medskip
		The map $(\omega,\P,\tilde\omega,t,x) \longmapsto M^\P_{\mu^\smalltext{X}}\big[\cW^{\omega,\P}|\widetilde\cP\big](\tilde\omega,t,x)$ is clearly $\cB(\Omega)\otimes\cB(\fP(\Omega))\otimes\widetilde{\cP}$-measurable. For a general Borel-measurable map $(\omega,\P,\tilde\omega,t,x) \longmapsto \cW^{\omega,\P}_t(\tilde\omega;x)$, we define 
		\begin{equation*}
			M^\P_{\mu^\smalltext{X}}\big[\cW^{\omega,\P}\big|\widetilde\cP\big] 
			\coloneqq M^\P_{\mu^\smalltext{X}}\big[(\cW^{\omega,\P})^{+}\big|\widetilde\cP\big] - M^\P_{\mu^\smalltext{X}}\big[(\cW^{\omega,\P})^{-}\big|\widetilde\cP\big], \; (\omega,\P) \in \Omega \times \fP(\Omega),
		\end{equation*}
		where we use our usual convention $\infty - \infty = -\infty$. This completes the proof.
	\end{proof}

	\begin{lemma}\label{lem::borel_quadratic_variation2}
		Let $\Omega \times \fP_\textnormal{sem} \times \Omega \times [0,\infty) \ni (\omega,\P,\tilde\omega,t) \longmapsto \cM^{\omega,\P}_t(\tilde\omega) \in \R$ be Borel-measurable. Suppose that, for every $(\omega,\P) \in \Omega \times \fP_\textnormal{sem}$, the process $\cM^{\omega,\P}$ is a right-continuous, $(\F_\smallertext{+},\P)$--square-integrable martingale. Then there exists a Borel-measurable function $\Omega \times \fP_\textnormal{sem} \times \Omega \times [0,\infty) \ni (\omega,\P,\tilde\omega,t) \longmapsto \langle \cM,X^c\rangle^{\omega,\P}_t(\tilde\omega) \in \R^d$ such that, for every $(\omega,\P) \in \Omega \times \fP_\textnormal{sem}$, the process $\langle \cM, X^c\rangle^{\omega,\P}$ is $\F$-predictable and $\langle \cM,X^c\rangle^{\omega,\P} = \langle \cM^{\omega,\P},X^{c,\P}\rangle^{(\F_\tinytext{+},\P)}$, $\P\text{\rm--a.s.}$
	\end{lemma}

	\begin{proof}
		As we cannot directly apply \cite[Proposition 6.6]{neufeld2014measurability}, we will adapt the arguments in its proof. Moreover, we suppose without loss of generality that $X$ is one-dimensional; otherwise we argue component-wise. Suppose we are given a Borel-measurable map
		\begin{equation*}
			\Omega \times \fP_\textnormal{sem} \times \Omega \times [0,\infty) \ni (\omega,\P,\tilde\omega,t) \longmapsto (X^{\omega,\P}_t(\tilde\omega),Z^{\omega,\P}_t(\tilde\omega)) \in \R \times \R,
		\end{equation*}
		such that every $Z^{\omega,\P}$ is right-continuous, $\P$--a.s. c\`adl\`ag, and $\F_\smallertext{+}$-adapted, and every $X^{\omega,\P}$ is a right-continuous, $\F_\smallertext{+}$-adapted, $(\F_\smallertext{+},\P)$--semi-martingale, that is, there exist right-continuous, $\F_\smallertext{+}$-adapted processes $M^{\omega,\P}$ and $A^{\omega,\P}$ with $M^{\omega,\P}_0 = A^{\omega,\P}_0 = 0$ such that $M^{\omega,\P}$ is an $(\F_\smallertext{+},\P)$--local martingale, $\P$--almost all paths of $A^{\omega,\P}$ are of locally finite variation, and $$X^{\omega,\P} = X^{\omega,\P}_0 + M^{\omega,\P} + A^{\omega,\P}, \; \textnormal{$\P$--a.s.}$$ We will, in a first step, show that there exists a measurable map
		\begin{equation*}
			\Omega \times \fP_\textnormal{sem} \times \Omega \times [0,\infty) \ni (\omega,\P,\tilde\omega,t) \longmapsto I(Z,X)^{\omega,\P}_t(\tilde\omega) \in [-\infty,\infty],
		\end{equation*}
		such that every $I(Z,X)^{\omega,\P}$ is $\F_\smallertext{+}$-progressive and satisfies
		\begin{equation*}
			I(Z,X)^{\omega,\P}_t = \bigg(\int_0^t Z^{\omega,\P}_{r\smallertext{-}} \d X^{\omega,\P}_r\bigg)^{(\F_\tinytext{+},\P)}, \; t \in [0,\infty), \; \text{$\P$--a.s.}
		\end{equation*}
		For $n \in \N$ and $(\omega,\P) \in \Omega \times \fP_\textnormal{sem}$, we let $\tau^{\omega,\P,n}_0 = 0$ and then inductively define, for $\ell \in \N$,
		\begin{align*}
			\tau^{\omega,\P,n}_{\ell + 1} 
			\coloneqq &\, \inf\big\{t > \tau^{\omega,\P,n}_\ell : |Z^{\omega,\P}_t - Z^{\omega,\P}_{\tau^{\smalltext{\omega}\smalltext{,}\smalltext{\P}\smalltext{,}\smalltext{n}}_\smalltext{\ell}}| \geq 2^{-n} \; \text{or} \; {\limsup}_{\D_\tinytext{+} \ni s \uparrow\uparrow t}|Z^{\omega,\P}_{s} - Z^{\omega,\P}_{\tau^{\smalltext{\omega}\smalltext{,}\smalltext{\P}\smalltext{,}\smalltext{n}}_\smalltext{\ell}}| \geq 2^{-n}\big\} \\
			= &\, \inf\big\{t > \tau^{\omega,\P,n}_\ell : |Z^{\omega,\P}_t - Z^{\omega,\P}_{\tau^{\smalltext{\omega}\smalltext{,}\smalltext{\P}\smalltext{,}\smalltext{n}}_\smalltext{\ell}}| \geq 2^{-n} \; \text{or} \; |Z^{\omega,\P}_{t\smallertext{-}} - Z^{\omega,\P}_{\tau^{\smalltext{\omega}\smalltext{,}\smalltext{\P}\smalltext{,}\smalltext{n}}_\smalltext{\ell}}| \geq 2^{-n}\big\}, \; \textnormal{$\P$--a.s.}
		\end{align*} 
		By noting as in the proof of \cite[Theorem IV.64, page 123]{dellacherie1978probabilities}, that
		\begin{align*}
			\big\{\tau^{\omega,\P,n}_1 \leq t\big\}
			&= \bigcap_{k \in \N^\star} \bigcup_{r \in \{t\} \cup((0,t) \cap \D_\tinytext{+})} \big\{|Z^{\omega,\P}_{r} - Z^{\omega,\P}_0| > 2^{-n} -1/k \big\} \in \cF_{t\smallertext{+}}, \; t \in [0,\infty),
		\end{align*}
		it follows that $\tau^{\omega,\P,n}_1$ is an $\F_\smallertext{+}$--stopping time and similarly that the map $\Omega \times \fP_\textnormal{sem} \times \Omega \ni (\omega,\P,\tilde\omega) \longmapsto \tau^{\omega,\P,n}_1(\tilde\omega) \in [0,\infty]$ is Borel-measurable. By induction, this then also holds for every $\tau^{\omega,\P,n}_\ell$. We note that since $Z^{\omega,\P}$ is $\P$--a.s. c\`adl\`ag, and in particular, $\P$--a.s. left-limited, we obtain $\lim_{\ell \rightarrow\infty}\tau^{\omega,\P,n}_\ell(\tilde\omega) = \infty$ for $\P$--a.e. $\tilde\omega \in \Omega$. We then define
		\begin{equation*}
			I^{\omega,\P,n}_t \coloneqq \limsup_{\ell \rightarrow \infty} \sum_{k = 0}^{\ell} Z^{\omega,\P}_{\tau^{\smalltext{\omega}\smalltext{,}\smalltext{\P}\smalltext{,}\smalltext{n}}_k} (X^{\omega,\P}_{\tau^{\smalltext{\omega}\smalltext{,}\smalltext{\P}\smalltext{,}\smalltext{n}}_{\smalltext{k}\smalltext{+}\smalltext{1}}\land t} - X^{\omega,\P}_{\tau^{\smalltext{\omega}\smalltext{,}\smalltext{\P}\smalltext{,}\smalltext{n}}_{\smalltext{k}}\land t}),
		\end{equation*}
		which is $\F_\smallertext{+}$-progressive as it is the limsup of right-continuous and $\F_\smallertext{+}$-adapted processes. Hence,
		\begin{equation*}
			I(Z,X)^{\omega,\P}_t(\tilde{\omega}) \coloneqq \limsup_{n \rightarrow \infty}I^{\omega,\P,n}_t(\tilde{\omega}),
		\end{equation*}
		is $\F_\smallertext{+}$-progressive. Note that
		\begin{equation*}
			\Omega \times \fP_\textnormal{sem} \times \Omega \times [0,\infty) \ni (\omega,\P,\tilde\omega,t) \longmapsto I(Z,X)^{\omega,\P}_t(\tilde\omega) \in [-\infty,\infty],
		\end{equation*}
		is Borel-measurable. By the proof of \cite[Theorem 2]{karandikar1995pathwise}, we obtain
		\begin{equation*}
			I(Z,X)^{\omega,\P}_t = \bigg(\int_0^t Z^{\omega,\P}_{r\smallertext{-}}\d X^{\omega,\P}_r\bigg)^{(\F_\tinytext{+},\P)}, \; t \in [0,\infty), \; \text{$\P$--a.s.},
		\end{equation*}
		where
		\[
			Z^{\omega,\P}_{t\smallertext{-}} \coloneqq \widetilde{Z}^{\omega,\P}_{t\smallertext{-}} \1_{\{\tilde{Z}^{\smalltext{\omega}\smalltext{,}\smalltext{\P}}_{\smalltext{t}\smalltext{-}}\in\R\}} \1_{\{t > 0\}}, \; \textnormal{where} \; \widetilde{Z}^{\omega,\P}_{t\smallertext{-}} \coloneqq \limsup_{n \rightarrow\infty} Z^{\omega,\P}_{(t-1/n)\lor 0}, \; t \in [0,\infty).
		\]
		is $\F$-predictable by \cite[Theorem IV.97.(b)]{dellacherie1978probabilities} and satisfies 
		\[
			Z^{\omega,\P}_{t\smallertext{-}} = \lim_{s\uparrow\uparrow t} Z^{\omega,\P}_s, \; t \in (0,\infty), \; \textnormal{$\P$--a.s.}
		\]
		We now use the map $I(Z,X)$ to construct a jointly Borel-measurable version of $\langle\cM^{\omega,\P},X^{c,\P}\rangle^{(\F_\smallertext{+},\P)}$. Denote by $X^i$ the $i$th component of the canonical process $X$. Let
		\begin{align*}
			[\cM,X^i]^{\omega,\P} 
			&\coloneqq \cM^{\omega,\P} X^i - \cM^{\omega,\P}_0 X^i_0 - I(\cM,X^i)^{\omega,\P} - I(X^i,\cM)^{\omega,\P} \\
			&= \cM^{\omega,\P} X^i - \cM^{\omega,\P}_0 X^i_0 - \bigg(\int_0^\cdot \cM^{\omega,\P}_{r\smallertext{-}}\d X^i_r\bigg)^{(\F,\P)} - \bigg(\int_0^\cdot X^i_{r\smallertext{-}}\d \cM^{\omega,\P}_r\bigg)^{(\F_\tinytext{+},\P)} = [\cM^{\omega,\P},X^i]^{(\F_\tinytext{+},\P)}, \; \text{$\P$--a.s.}
		\end{align*}
		Define $\Delta\cM^{\omega,\P}$ as $\Delta\cM^{\omega,\P}_t \coloneqq \limsup_{n\rightarrow \infty}(\cM^{\omega,\P}_t - \cM^{\omega,\P}_{(t-1/n)\lor 0})$ since every $\cM^{\omega,\P}$ is only $\P$--a.s. c\`adl\`ag. The $\F$-optional set $\{\Delta X^i \neq 0\}$ is the countable union of disjoint graphs of $\F$--stopping times $(\sigma^i_n)_{n \in \N}$ by \cite[Theorem B, page xiii, and Remark E, page xvii]{dellacherie1982probabilities} and then \cite[Theorem IV.88.(a), page 139]{dellacherie1978probabilities}. We then define the $\F_\smallertext{+}$-optional process
		\begin{equation*}
			S(\cM,X^i)^{\omega,\P} \coloneqq \limsup_{k \rightarrow \infty} \sum_{n = 1}^k \Delta \cM^{\omega,\P}_{\sigma^\smalltext{i}_\smalltext{n}} \Delta X^i_{\sigma^\smalltext{i}_\smalltext{n}} \1_{\llbracket \sigma^\smalltext{i}_\smalltext{n},\infty \rrparenthesis} = \sum_{s \in (0,\cdot]}\Delta\cM^{\omega,\P}_s\Delta X^i_s, \; \textnormal{$\P$--a.s.},
		\end{equation*}
		see also \cite[Definition 7.39]{he1992semimartingale}, and then let
		\begin{equation*}
			Q(\cM,X^i)^{\omega,\P} \coloneqq [\cM,X^i]^{\omega,\P} - S(\cM,X^i)^{\omega,\P} = [\cM^{\omega,\P},X^i]^{(\F_\tinytext{+},\P)} - \sum_{s \in (0,\cdot]}\Delta\cM^{\omega,\P}_s\Delta X^i_s = \langle \cM^{\omega,\P},(X^{c,\P})^i\rangle^{(\F_\tinytext{+},\P)}, \; \text{$\P$--a.s.},
		\end{equation*}
		and $Q(\cM,X^i)^{\omega,\P}$ is thus $\F_\smalltext{+}$-progressive and $\P$--a.s. continuous. We then define component-wise
		\begin{equation*}
			\langle \cM,X^{c} \rangle^{\omega,\P} \coloneqq \widetilde{Q}(\cM,X)^{\omega,\P} \1_{\R^\smalltext{d}}\big(\widetilde{Q}(\cM,X)^{\omega,\P}\big), \; \text{where} \; \widetilde{Q}(\cM,X)^{\omega,\P}_t \coloneqq \limsup_{n \rightarrow \infty} Q(\cM,X)^{\omega,\P}_{(t - 1/n) \lor 0}.
		\end{equation*}
		Then $\langle \cM,X^{c} \rangle^{\omega,\P}$ is $\F$--predictable by \cite[Theorem IV.97.(b)]{dellacherie1978probabilities}, coincides with the $(\F_\smallertext{+},\P)$-predictable quadratic co-variation of $\cM^{\omega,\P}$ with $X^{c,\P}$ and is thus in particular also $\P$--a.s. continuous for each $(\omega,\P) \in \Omega \times \fP_\textnormal{sem}$. Moreover, the Borel-measurability of
		\begin{equation*}
			\Omega \times \fP_\textnormal{sem} \times \Omega \times [0,\infty) \ni (\omega,\P,\tilde\omega,t) \longmapsto \langle \cM, X^c\rangle^{\omega,\P}_t(\tilde\omega) \in \R^d,
		\end{equation*}
		is preserved along the way, which completes the proof.
	\end{proof}	

	\begin{proof}[Proof of \Cref{lem::measurable_decomposition}]
		For each $(\omega,\P) \in \Omega \times \fP_\textnormal{sem}$, we find by \cite[Proposition 2.6]{possamai2024reflections}  a unique triple $(\overline\cZ^{\omega,\P},\overline\cU^{\omega,\P},\overline\cN^{\omega,\P})$ in $\H^2(X^{c,\P};\F,\P) \times \H^2(\mu^X;\F,\P) \times \cH^{2,\perp}_0(X^{c,\P},\mu^X,\F_\smallertext{+},\P)$ for which \eqref{eq::martingale_representation2} holds. We show that $\overline{\cZ}^{\omega,\P}$ and $\overline{\cU}^{\omega,\P}$ admit versions $\cZ^{\omega,\P}$ and $\cU^{\omega,\P}$ measurable with respect to the stated $\sigma$-algebras.
		
		\medskip
		We start with $(\overline\cZ^{\omega,\P})_{(\omega,\P) \in \Omega\times\fP_\smalltext{\textnormal{sem}}}$. By \cite[Proposition 6.6]{neufeld2014measurability}, there exists an $\F$-predictable, $\S^d_\smallertext{+}$-valued process $\mathsf{C}$ which coincides with the second characteristic of $X$, up to $\P$-evanescence for each $\P\in\fP_\textnormal{sem}$. Let $\mathsf{A} \coloneqq \textnormal{Tr}[\mathsf{C}]$ be the trace process of $\mathsf{C}$. Then $\d\mathsf{C} = \mathsf{c}\d\mathsf{A}$, $\fP_\textnormal{sem}$--q.s., where
			\[
				\mathsf{c}_t \coloneqq \mathsf{c}^\prime_t \1_{\{\mathsf{c}^\prime_t \in \S^d_\smallertext{+}\}}, \; \textnormal{with} \; \mathsf{c}^\prime_t \coloneqq \limsup_{n \rightarrow \infty}\frac{\mathsf{C}_t - \mathsf{C}_{(t-1/n)\lor 0}}{\mathsf{A}_t - \mathsf{A}_{(t-1/n)\lor 0}}, \; t \in [0,\infty).
			\]
			Here the limit on the right-hand side is taken component-wise, and we use the convention $0/0 = 0$. For each $(\omega,\P) \in \Omega\times\fP_\textnormal{sem}$, the process $\overline\cZ^{\omega,\P}$ satisfies (see \cite[Theorem III.6.4.b)]{jacod2003limit})
			\begin{equation*}
				\langle \cM^{\omega,\P},(X^{c,\P})^j\rangle^{(\F_\tinytext{+},\P)} = \bigg(\sum_{i = 1}^d(\overline\cZ^{\omega,\P})^i  \mathsf{c}^{i,j}\bigg) \bcdot \mathsf{A} , \; j \in \{1,\ldots,d\}, \; \text{$\P$--a.s.},
			\end{equation*}
			or, equivalently component-wise
			\begin{equation*}
				\frac{\d\langle \cM^{\omega,\P},X^{c,\P}\rangle^{(\F_\tinytext{+},\P)}}{\d \mathsf{A}} = \mathsf{c} \overline\cZ^{\omega,\P}, \; \text{$\d \mathsf{A}$--a.e.}, \; \text{$\P$--a.s.}
			\end{equation*}
			By \Cref{lem::borel_quadratic_variation2}, there exists a Borel-measurable map $\Omega \times \fP_\textnormal{sem} \times \Omega \times [0,\infty) \ni (\omega,\P,\tilde\omega,t) \longmapsto \langle \cM,X^{c}\rangle^{\omega,\P}_t(\tilde\omega) \in \R^d$ such that, for each $(\omega,\P) \in \Omega\times\fP_\textnormal{sem}$, the process $\langle \cM, X^{c}\rangle^{\omega,\P}$ is $\F$-predictable and
			\begin{equation*}
				\langle \cM,X^{c}\rangle^{\omega,\P} = \langle \cM^{\omega,\P},X^{c,\P}\rangle^{(\F_\tinytext{+},\P)}, \; \text{$\P$--a.s.}
			\end{equation*}	
			We now define the $\F$-predictable process $\cZ^{\omega,\P}$ by
			\begin{equation*}
				\cZ^{\omega,\P}_t \coloneqq \mathsf{c}_t^\oplus \mathsf{Z}^{\omega,\P}_t\1_{\{\mathsf{Z}^{\smalltext{\omega}\smalltext{,}\smalltext{\P}}_\smalltext{t} \in \R^\smalltext{d}\}}, 
				\; 
				\text{where} 
				\; 
				\mathsf{Z}^{\omega,\P}_t \coloneqq \limsup_{n \rightarrow \infty} \frac{\langle\cM,X^{c}\rangle^{\omega,\P}_t - \langle \cM,X^{c}\rangle^{\omega,\P}_{(t-1/n)\lor 0}}{\mathsf{A}_{t} - \mathsf{A}_{(t-1/n) \lor 0}},\; t \in [0,\infty),
			\end{equation*}
			where the limit on the right-hand side is taken component-wise, and where $\mathsf{c}_t^\oplus$ denotes the Moore--Penrose pseudo-inverse of $\mathsf{c}_t$. Then $\cZ^{\omega,\P} = \mathsf{c}^\oplus \mathsf{c} \overline\cZ^{\omega,\P}$, $\d \mathsf{A}$--a.e., $\P$--a.s., which then implies, together with $\mathsf{c}\mathsf{c}^\oplus\mathsf{c} = \mathsf{c}$ and $(\mathsf{c}^\oplus)^\top = (\mathsf{c}^\top)^\oplus = \mathsf{c}^\oplus$, that
			\begin{equation*}
				(\cZ^{\omega,\P} - \overline\cZ^{\omega,\P})^\top \mathsf{c} (\cZ^{\omega,\P} - \overline\cZ^{\omega,\P}) = 0, \; \text{$\d \mathsf{A}$--a.e.}, \; \text{$\P$--a.s.}
			\end{equation*}
			Therefore $\cZ^{\omega,\P} = \overline\cZ^{\omega,\P}$ in $\H^2(X^{c,\P};\F,\P)$ and thus $(\cZ^{\omega,\P}\bcdot X^{c,\P})^{(\P)} = (\overline\cZ^{\omega,\P}\bcdot X^{c,\P})^{(\P)}$, $\P$--a.s., by \cite[Theorem III.6.4.c)]{jacod2003limit}, for each $(\omega,\P) \in \Omega \times \fP_\textnormal{sem}$. Moreover, this construction of the family $(\cZ^{\omega,\P})_{(\omega,\P) \in \Omega \times \fP_{\smalltext{\textnormal{sem}}}}$ immediately implies that
			\begin{equation*}
				\Omega \times \fP_\textnormal{sem} \times \Omega \times [0,\infty) \ni (\omega,\P,\tilde\omega,t) \longmapsto \cZ^{\omega,\P}_t(\tilde\omega) \in \R^d,
			\end{equation*}
			is Borel-measurable.
		
		\medskip
		We turn to $(\overline\cU^{\omega,\P})_{(\omega,\P) \in \Omega\times\fP_{\smalltext{\textnormal{sem}}}}$. We denote by $\nu^\P(\d s, \d x) = \mathsf{K}^\P_t (\d x)\d \mathsf{A}^\P_t$ the disintegration obtained in \Cref{lem::measurable_disintegration_compensator}.
		Let $(\P,\omega,t)\longmapsto \mathsf{a}^\P_t(\omega)$ be the $\cB(\fP_\textnormal{sem})\otimes\cP$-measurable map defined by $\mathsf{a}^{\P}_t \coloneqq \mathsf{K}^{\P}_{t}(\R^d)\Delta \mathsf{A}^\P_t$, where
		\[
			\Delta \mathsf{A}^\P_t \coloneqq \mathsf{A}^{\P,\Delta}_t \1_{\{\mathsf{A}^{\smalltext{\P}\smalltext{,}\smalltext{\Delta}}_{\smalltext{t}} \in [0,\infty)\}}, \; \textnormal{for} \; \mathsf{A}^{\P,\Delta}_t \coloneqq \limsup_{n \rightarrow \infty} \big(\mathsf{A}^\P_t - \mathsf{A}^\P_{(t-1/n)\lor 0}\big).
		\]
		For $(\omega,\P) \in \Omega\times\fP_\textnormal{sem}$, we let $\cU^{\omega,\P} \coloneqq \mathsf{U}^{\omega,\P} \mathbf{1}_{\{\mathsf{U}^{\smalltext{\omega}\smalltext{,}\smalltext{\P}} \in \R\}}$, where
		\begin{equation*}
			\mathsf{U}^{\omega,\P}_t(\tilde\omega;x) \coloneqq \cW^{\omega,\P}_t(\tilde\omega;x) + \frac{1}{1-\mathsf{a}^{\P}_t(\tilde\omega)} \mathbf{1}_{\{\mathsf{a}^{\smalltext{\P}}_\smalltext{t}(\tilde\omega) < 1\}}\int_{\R^\smalltext{d}}\cW^{\omega,\P}_t(\tilde\omega;x)\mathsf{K}^{\P}_{\tilde\omega,t}(\d x)\Delta \mathsf{A}^\P_t(\tilde\omega), \; (\tilde\omega,t,x) \in \Omega\times[0,\infty) \times \R^d,
		\end{equation*}
		and $\cW^{\omega,\P} \coloneqq M^\P_{\mu^\smalltext{X}}\big[(\Delta\cM^{\omega,\P})\big| \widetilde\cP\big] = M^\P_{\mu^\smalltext{X}}\big[\Delta(\cM^{\omega,\P}-(\cZ^{\omega,\P}\bcdot X^{c,\P})^{(\F_\tinytext{+},\P)})\big|\widetilde\cP\big]$, $M^\P_{\mu^\smalltext{X}}$--a.e., is defined using \Cref{lem::measurability_cond_M_tilde_P2} with
		\begin{equation*}
			(\Delta\cM^\P)^{\omega,\P}_t \coloneqq \cM^{\omega,\P,\Delta}_t \mathbf{1}_{\{\cM^{\smalltext{\omega}\smalltext{,}\smalltext{\P}\smalltext{,}\smalltext{\Delta}}_\smalltext{t} \in \R\}}, \; \text{where} \; \cM^{\omega,\P,\Delta}_t \coloneqq \limsup_{n \rightarrow \infty}\big(\cM^{\omega,\P}_t - \cM^{\omega,\P}_{(t-1/n) \lor 0}\big).
		\end{equation*}
		Then $\cU^{\omega,\P} \in \H^2(\mu^X;\F,\P)$ (see the proof of \cite[Theorem III.4.20]{jacod2003limit}) and
		\begin{equation*}
			\cN^{\omega,\P} \coloneqq \cM^{\omega,\P} - \cM^{\omega,\P}_0 - (\cZ^{\omega,\P}\bcdot X^{c,\P})^{(\P)}- (\cU^{\omega,\P}\ast\tilde\mu^{X,\P})^{(\P)},
		\end{equation*}
		satisfies $M^{\P}_{\mu^\smalltext{X}}\big[\Delta\cN^{\omega,\P}\big|\widetilde\cP\big] = M^\P_{\mu^\smalltext{X}}\big[\Delta(\cN^{\omega,\P}+(\cZ^{\omega,\P}\bcdot X^{c,\P})^{(\F,\P)})\big|\widetilde\cP\big]=0,$ and
		\begin{align*}
			\langle \cN^{\omega,\P},X^{c,\P}\rangle^{(\F_\tinytext{+},\P)} &= \langle \cN^{\omega,\P} + (\cU^{\omega,\P}\ast\tilde\mu^{X,\P})^{(\P)},X^{c,\P}\rangle^{(\F_\tinytext{+},\P)} = \langle \cM^{\omega,\P} - (\cZ^{\omega,\P}\bcdot X^{c,\P})^{(\P)}, X^{c,\P} \rangle^{(\F_\tinytext{+},\P)} \\
			&= \big\langle \cM^{\omega,\P} - ({\overline\cZ}^{\omega,\P}\bcdot X^{c,\P})^{(\P)}, X^{c,\P} \big\rangle^{(\F_\tinytext{+},\P)} = \big\langle ({\overline\cU}^{\omega,\P}\ast\tilde\mu^{X,\P})^{(\P)} + {\overline\cN}^{\omega,\P}, X^{c,\P} \big\rangle^{(\F_\tinytext{+},\P)} = 0.
		\end{align*}
		Here, we used the $(\F_\smallertext{+},\P)$-orthogonality of $\overline\cN^{\omega,\P}$ with respect to $X^{c,\P}$ and $\mu^X$ in the last equality.
		This yields the $(\F_\smallertext{+},\P)$-orthogonality of $\cN^{\omega,\P}$ and $X^{c,\P}$. It follows by uniqueness of the decomposition of $\cM^{\omega,\P}$, that $\cU^{\omega,\P} = \overline\cU^{\omega,\P}$ in $\H^2(\mu^X;\F_\smallertext{+},\P)$ and thus $\cN^{\omega,\P} = \overline\cN^{\omega,\P}$ in $\cH^{2,\perp}(X^{c,\P},\mu^X;\F_\smallertext{+},\P)$. This shows that $(i)$ is satisfied
		
		\medskip
		For $(ii)$ to hold, we need to slightly modify the integrand $\cU^{\omega,\P}$. Since $(\omega,\P,\tilde\omega,r) \longmapsto \|\cU^{\omega,\P}_r(\tilde\omega;\cdot)\|_{\hat\L^\smalltext{2}_{\smalltext{\omega}\smalltext{\otimes}_\tinytext{s}\smalltext{\tilde\omega}\smalltext{,}\smalltext{s}\smalltext{+}\smalltext{r}}(\mathsf{K}^{\smalltext{s}\smalltext{,}\smalltext{\omega}\smalltext{,}\smalltext{\P}}_{\smallertext{\omega}\smalltext{\otimes}_\tinytext{s}\smalltext{\tilde\omega}\smalltext{,}\smalltext{r}})}$ is $\cB(\widehat{\Omega}^C_s)\otimes\cP$-measurable on $\Omega^C_s \times \Omega \times [0,\infty)$, 
		 and $\|\cU^{\omega,\P}_r(\tilde\omega;\cdot)\|_{\hat\L^\smalltext{2}_{\smalltext{\omega}\smalltext{\otimes}_\tinytext{s}\smalltext{\tilde\omega}\smalltext{,}\smalltext{s}\smalltext{+}\smalltext{r}}(\mathsf{K}^{\smalltext{s}\smalltext{,}\smalltext{\omega}\smalltext{,}\smalltext{\P}}_{\smallertext{\omega}\smalltext{\otimes}_\tinytext{s}\smalltext{\tilde\omega}\smalltext{,}\smalltext{r}})} < \infty$, $\P\otimes\d (C^{s,\omega}_{s\smallertext{+}\smallertext{\cdot}} - C_s(\omega))$--a.e., we can redefine the $\cF\otimes\cB(\fP_\textnormal{sem})\otimes\cP$-measurable map
				\begin{equation*}
					\Omega \times \fP_\textnormal{sem} \times \Omega \times [0,\infty) \times \R^d \ni (\omega,\P,\tilde\omega,t,x) \longmapsto \cU^{\omega,\P}_t(\tilde\omega;x) \in \R,
				\end{equation*}
				to be zero on the $\cF\otimes\cB(\fP_\textnormal{sem})\otimes\cP$-measurable subset
				\[
					\big\{ (\omega,\P,\tilde\omega,t) \in \Omega^C_s \times \Omega \times [0,\infty) : \|\cU^{\omega,\P}_r(\tilde\omega;\cdot)\|_{\hat\L^\smalltext{2}_{\smalltext{\omega}\smalltext{\otimes}_\tinytext{s}\smalltext{\tilde\omega}\smalltext{,}\smalltext{s}\smalltext{+}\smalltext{r}}(\mathsf{K}^{\smalltext{s}\smalltext{,}\smalltext{\omega}\smalltext{,}\smalltext{\P}}_{\smallertext{\omega}\smalltext{\otimes}_\tinytext{s}\smalltext{\tilde\omega}\smalltext{,}\smalltext{r}})} = \infty \big\} \subseteq \Omega \times \fP_\textnormal{sem} \times \Omega \times [0,\infty).
				\]
				This yields the desired property stated in $(ii)$, which completes the proof.
	\end{proof}

\begin{proof}[Proof of \Cref{lem::conditioning_bsde2}]
For simplicity, we assume $s = 0$, so $\P \in \fP_0$; the general case follows by analogous arguments. Throughout the proof, we denote by $\sN$ a $\P$--null set that may be enlarged from line to line by at most countably many $\P$--null sets. We also refrain from explicitly writing the terminal time $T$ and terminal condition $\xi$ when referring to the BSDE solutions. Finally, we observe that whenever $\P[A] = 1$ for $A \in \cF$, we can find a $\P$--null set $\sN$ (in $\cF_t$) such that $\E^{\P^{t,\omega}}[\1_{A}(\omega\otimes_t\cdot)] = 1$ for each $\omega \in \Omega \setminus \sN$.

\medskip
Since 
\begin{equation*}
	\E^\P[\cY^\P_t|\cF_t](\omega) = \E^{\P^{\smalltext{t}\smalltext{,}\smalltext{\omega}}}[\cY^\P_t(\omega\otimes_t\cdot)], \; \omega \in \Omega\setminus\sN,
\end{equation*}
it is enough to show that $\cY^\P_t(\omega\otimes_t\cdot) = \cY^{t,\omega,\P^{\smalltext{t}\smalltext{,}\smalltext{\omega}}}_0((T-t\land T)^{t,\omega},\xi^{t,\omega})$, $\P^{t,\omega}$--a.s., for $\P$--a.e. $\omega \in \Omega$. We will do this by first arguing that the $\P^{t,\omega}$--BSDE with terminal time $(T - t \land T)^{t,\omega}$, terminal condition $\xi^{t,\omega}$ and generator $f^{t,\omega,\P^{\smalltext{t}\smalltext{,}\smalltext{\omega}}}$ is well-posed for $\P$--a.s. $\omega \in \Omega$, and then that $$(\cY^{t,\omega}_{t\smallertext{+}\smallertext{\cdot}},\cZ^{t,\omega}_{t\smallertext{+}\smallertext{\cdot}},\cU^{t,\omega}_{t\smallertext{+}\smallertext{\cdot}},\cN^{t,\omega}_{t\smallertext{+}\smallertext{\cdot}}-\cN^{t,\omega}_t) \coloneqq (\cY^\P_{t\smallertext{+}\smallertext{\cdot}}(\omega\otimes_t\cdot), \cZ^{\P}_{t\smallertext{+}\smallertext{\cdot}}(\omega\otimes_t\cdot), \cU^{\P}_{t\smallertext{+}\smallertext{\cdot}}(\omega\otimes_t\cdot),\cN^\P_{t\smallertext{+}\smallertext{\cdot}}(\omega\otimes_t\cdot) -\cN^\P_{t}(\omega\otimes_t\cdot))$$ corresponds to its unique solution for $\P$--a.e. $\omega \in \Omega$. For simplicity, we drop the dependence on $\P$ in the notation of the $\P$--BSDE solution, and simply write $(\cY,\Z,\cU,\cN)$ instead of $(\cY^\P,\cZ^\P,\cU^\P,\cN^\P)$.

\medskip
The required measurability of the data $((T - t \land T)^{t,\omega},\xi^{t,\omega},f^{t,\omega,\P^{\smalltext{t}\smalltext{,}\smalltext{\omega}}})$ is immediate for $\omega \in \Omega\setminus\sN$. Integrability follows from
\begin{equation*}
	\E^{\P}[\cE(\hat\beta A)_T\xi^2] + \E^{\P}\bigg[\int_{0}^T \cE(\hat\beta A)_r \frac{|f^{\P}_r(0,0,0,\mathbf{0})|^2}{\alpha^2_r}\d C_r\bigg] < \infty, 
\end{equation*}
which implies
\begin{align*}
	\E^{\P^{\smalltext{t}\smalltext{,}\smalltext{\omega}}}\big[\cE(\hat\beta A^{t,\omega}_{t\smallertext{+}\smallertext{\cdot}})_{(T - t\land T)^{\smalltext{t}\smalltext{,}\smalltext{\omega}}} |\xi^{t,\omega}|^2\big]  + \E^{\P^{\smalltext{t}\smalltext{,}\smalltext{\omega}}}\bigg[\int_0^{(T-t\land T)^{t,\omega}} \cE(\hat\beta A^{t,\omega}_{t\smallertext{+}\smallertext{\cdot}})_r \frac{|f^{t,\omega,\P^{\smalltext{t}\smalltext{,}\smalltext{\omega}}}_r(0,0,0,\mathbf{0})|^2}{|\alpha^{t,\omega}_{t\smallertext{+}r}|^2}\d (C^{t,\omega}_{t\smallertext{+}\smallertext{\cdot}})_r\bigg] < \infty, \; \omega \in \Omega \setminus \sN.
\end{align*}
Consequently, the conditioned BSDE admits a unique solution in the sense of \cite[Section 3]{possamai2024reflections} for every $\omega \in \Omega \setminus\sN$.

\medskip
Next,
\begin{align*}
	\cY_{t\smallertext{+}v} 
	&= \cY_{t} + (\cY_{t\smallertext{+}v} - \cY_{t}) \\
	&= \cY_{t}-\int_{t \land T}^{(t\smallertext{+}v)\land T} f^{\P}_r\big(\cY_r,\cY_{r\smallertext{-}}, \cZ_r,\cU_r(\cdot)\big)\d C_r + \bigg(\int_{t\land T}^{(t\smallertext{+}v)\land T}\cZ_r\d X^{c,\P}_r\bigg)^{(\P)} + \bigg(\int_{t\land T}^{(t\smallertext{+}v)\land T}\d(\cU\ast\tilde{\mu}^{X,\P})_r\bigg)^{(\P)} \\
	&\quad + \int_{t\land T}^{(t\smallertext{+}v)\land T}\d\cN_{r} \\
	&= \cY_{t}-\int_0^{v \land (T-t\land T)} f^{\P}_{t\smallertext{+}r}\big(\cY_{t\smallertext{+}r},\cY_{(t\smallertext{+}r)\smallertext{-}}, \cZ_{t\smallertext{+}r},\cU_{t\smallertext{+}r}(\cdot)\big)\d (C_{t\smallertext{+}\smallertext{\cdot}})_r + \bigg(\int_{t\land T}^{(t\smallertext{+}v)\land T}\cZ_r\d X^{c,\P}_r\bigg)^{(\P)} \\
	&\quad + \bigg(\int_{t\land T}^{(t\smallertext{+}v)\land T}\d(\cU\ast\tilde{\mu}^{X,\P})_r\bigg)^{(\P)} + \int_{t\land T}^{(t\smallertext{+}v)\land T}\d\cN_{r}, \; v \in [0,\infty), \; \text{$\P$--a.s.,}
\end{align*}
yields
\[
\begin{cases}
\begin{aligned}
	&\cY^{t,\omega}_{t\smallertext{+}(T-t\land T)^{t,\omega}} 
	= \xi^{t,\omega}, \; \textnormal{$\P^{t,\omega}$--a.s.}, \; \omega \in \Omega \setminus \sN, \\
	&\cY^{t,\omega}_{v} 
	= \cY^{t,\omega}_{t}-\int_0^{v \land (T-t\land T)^{t,\omega}} f^{t,\omega,\P^{\smalltext{t}\smalltext{,}\smalltext{\omega}}}_{r}\big(\cY^{t,\omega}_{t\smallertext{+}r},\cY^{t,\omega}_{(t\smallertext{+}r)\smallertext{-}}, \cZ^{t,\omega}_{t\smallertext{+}r},\cU^{t,\omega}_{t\smallertext{+}r}(\cdot)\big)\d (C^{t,\omega}_{t\smallertext{+}\smallertext{\cdot}})_r + \bigg(\int_{t\land T}^{(t\smallertext{+}v)\land T}\cZ_r\d X^{c,\P}_r\bigg)^{(\P)}(\omega\otimes_t\cdot) \\
	&\quad + \bigg(\int_{t\land T}^{(t\smallertext{+}v)\land T}\d(\cU\ast\tilde{\mu}^{X,\P})_r\bigg)^{(\P)}(\omega\otimes_t\cdot) + \int_0^{v \land (T-t\land T)^{t,\omega}}\d(\cN^{t,\omega}_{t\smallertext{+}\smallertext{\cdot}})_{r}, \; v \in [0,\infty), \; \text{$\P^{t,\omega}$--a.s.,} \; \omega \in \Omega\setminus\sN.
\end{aligned}
\end{cases}
\]
We now show that $(\cY^{t,\omega}_{t\smallertext{+}\smallertext{\cdot}},\alpha^{t,\omega}_{t\smallertext{+}\smallertext{\cdot}}\cY^{t,\omega}_{t\smallertext{+}\smallertext{\cdot}},\alpha^{t,\omega}_{t\smallertext{+}\smallertext{\cdot}}\cY^{t,\omega}_{(t\smallertext{+}\smallertext{\cdot})\smallertext{-}}) \in \cS^{2,t,\omega}_{T}(\F_\smallertext{+},\P^{t,\omega}) \times \big(\H^{2,t,\omega}_{T,\hat{\beta}}(\F_\smallertext{+},\P^{t,\omega})\big)^2$,  $\cZ^{t,\omega}_{t\smallertext{+}\smallertext{\cdot}} \in \H^{2,t,\omega}_{T,\hat{\beta}}(X^{c,\P^{\smalltext{t}\smalltext{,}\smalltext{\omega}}};\F,\P^{t,\omega})$, $\cU^{t,\omega}_{t\smallertext{+}\smallertext{\cdot}} \in \H^{2,t,\omega}_{T,\hat{\beta}}(\mu^X;\F,\P^{t,\omega})$ and $\cN^{t,\omega}_{t\smallertext{+}\smallertext{\cdot}}-\cN^{t,\omega}_t\in \cH^{2,t,\omega,\perp}_{0,T,\hat{\beta}}(X^{c,\P^{\smalltext{t}\smalltext{,}\smalltext{\omega}}},\mu^X;\F_\smallertext{+},\P^{t,\omega})$, and that
\begin{gather*}
	\bigg(\int_{t\land T}^{(t\smallertext{+}v)\land T}\cZ^\P_r\d X^{c,\P}_r\bigg)^{(\P)}(\omega\otimes_t\cdot) = \bigg(\int_0^{v \land (T-t\land T)^{t,\omega}}\cZ^{\omega}_r\d X^{c,\P^{\smalltext{t}\smalltext{,}\smalltext{\omega}}}_r\bigg)^{(\P^{\smalltext{t}\smalltext{,}\smalltext{\omega}})}, \, v \in [0,\infty),  \; \textnormal{$\P^{t,\omega}$--a.s.},\\
	\bigg(\int_{t\land T}^{(t\smallertext{+}v)\land T}\d(\cU^\P\ast\tilde{\mu}^{X,\P})_r\bigg)^{(\P)}(\omega\otimes_t\cdot) = \cU^{\omega}\ast\tilde\mu^{X,\P^{\smalltext{t}\smalltext{,}\smalltext{\omega}}}_{v \land (T-t\land T)^{t,\omega}}, \; v \in [0,\infty), \; \textnormal{$\P^{t,\omega}$--a.s.},
\end{gather*}
for $\P$--a.e. $\omega \in \Omega$.

\medskip
Since $\cN$ is $\F_\smallertext{+}$-adapted, $\cN^{t,\omega}_{t\smallertext{+}\smallertext{\cdot}}-\cN^{t,\omega}_t$ is $\F_\smallertext{+}$-adapted by Galmarino's test; see \cite[Theorem IV.101.(b)]{dellacherie1978probabilities}. 
Note that $\cY^{t,\omega}$ and $\cN^{t,\omega}_{t\smallertext{+}\smallertext{\cdot}}-\cN^{t,\omega}_t$ are right-continuous, and the $\P$--a.s. c\`adl\`ag property of $\cY$ is then transferred to a $\P^{t,\omega}$--a.s. c\`adl\`ag property of $\cY^{t,\omega}_{t\smallertext{+}\smallertext{\cdot}}$ for $\omega \in \Omega\setminus \sN$. Moreover, a functional monotone class argument together with \cite[Theorem IV.101.(b), page 150]{dellacherie1978probabilities} implies that shifting the $\F_\smallertext{+}$-optional process $\cY$ to $\cY^{t,\omega}_{t\smallertext{+}\smallertext{\cdot}}$ preserves $\F_\smallertext{+}$-optionality.
 The $\F$-predictability of $\cZ^{t,\omega}_{t\smallertext{+}\smallertext{\cdot}}$ follows from \cite[Theorem IV.97.(b) and Theorem IV.99.(b), page 147]{dellacherie1978probabilities}, and the $\widetilde{\cP}$-measurability of $\cU^{t,\omega}_{t\smallertext{+}\smallertext{\cdot}}$ follows from the same result together with a functional monotone class argument.

\medskip
We turn to the required integrability of $\cY^{t,\omega}$. Since $(\cY,\alpha\cY,\alpha\cY_\smallertext{-}) \in \cS^2_T(\F_\smallertext{+},\P) \times \big(\H^{2}_{T,\hat\beta}(\F_\smallertext{+},\P)\big)^2$, we obtain
\begin{align*}
	&\|\cY^{t,\omega}_{t\smallertext{+}\smallertext{\cdot}}\|^2_{\cS^{\smalltext{2}\smalltext{,}\smalltext{t}\smalltext{,}\smalltext{\omega}}_{\smalltext{T}}(\F_\tinytext{+},\P^{\smalltext{t}\smalltext{,}\smalltext{\omega}})} 
	+ \|\alpha^{t,\omega}_{t\smallertext{+}\smallertext{\cdot}}\cY^{t,\omega}_{t\smallertext{+}\smallertext{\cdot}}\|^2_{\H^{\smalltext{2}\smalltext{,}\smalltext{t}\smalltext{,}\smalltext{\omega}}_{\smalltext{T}\smalltext{,}\smalltext{\hat\beta}}(\F_\tinytext{+},\P^{\smalltext{t}\smalltext{,}\smalltext{\omega}})}
	+ \|\alpha^{t,\omega}_{t\smallertext{+}\smallertext{\cdot}}\cY^{t,\omega}_{(t\smallertext{+}\smallertext{\cdot})\smallertext{-}}\|^2_{\H^{\smalltext{2}\smalltext{,}\smalltext{t}\smalltext{,}\smalltext{\omega}}_{\smalltext{T}\smalltext{,}\smalltext{\hat\beta}}(\F_\tinytext{+},\P^{\smalltext{t}\smalltext{,}\smalltext{\omega}})} \\
	&=\E^{\P^{\smalltext{t}\smalltext{,}\smalltext{\omega}}}\bigg[\sup_{r \in [0,(T-t\land T)^{\smalltext{t}\smalltext{,}\smalltext{\omega}}]}|\cY^{\P}_{t\smallertext{+}r}(\omega\otimes_t\cdot)|^2+\int_0^{(T-t\land T)^{t,\omega}}\cE(\hat\beta A^{t,\omega}_{t\smallertext{+}\smallertext{\cdot}})_r |\cY_{t\smallertext{+}r}(\omega\otimes_t\cdot)|^2\d (A^{t,\omega}_{t\smallertext{+}\smallertext{\cdot}})_r\bigg] \\
	&\quad + \E^{\P^{\smalltext{t}\smalltext{,}\smalltext{\omega}}}\bigg[\int_0^{(T-t\land T)^{t,\omega}}\cE(\hat\beta A^{t,\omega}_{t\smallertext{+}\smallertext{\cdot}})_r |\cY_{(t\smallertext{+}r)\smallertext{-}}(\omega\otimes_t\cdot)|^2\d (A^{t,\omega}_{t\smallertext{+}\smallertext{\cdot}})_r\bigg] \\
	&= \E^{\P}\bigg[\sup_{r \in [t\land T,T]}|\cY_r|^2 + \int_{t}^{T}\frac{\cE(\hat\beta A)_{r}}{\cE(\hat{\beta}A)_t} |\cY_{r}|^2\d A_r + \int_{t}^{T}\frac{\cE(\hat\beta A)_{r}}{\cE(\hat{\beta}A)_t} |\cY_{r-}|^2\d A_r \bigg| \cF_t \bigg](\omega) < \infty, \; \text{for $\omega \in \Omega\setminus\sN$.}
\end{align*}

\medskip
The following are the remaining conditions that need to be verified:
\begin{enumerate}
	\item[$(i)$] $\cZ^{t,\omega}_{t\smallertext{+}\smallertext{\cdot}} \in \H^{2,t,\omega}_{T,\hat\beta}(X^{c,\P^{\smalltext{t}\smalltext{,}\smalltext{\omega}}};\F,\P^{t,\omega})$ and $\displaystyle \bigg(\int_{t\land T}^{(t\smallertext{+}\smallertext{\cdot})\land T}\cZ_r\d X^{c,\P}_r\bigg)^{(\P)}(\omega\otimes_t\cdot) = \bigg(\int_0^\cdot\cZ^{t,\omega}_{t\smallertext{+}r}\d X^{c,\P^{\smalltext{t}\smalltext{,}\smalltext{\omega}}}_r\bigg)^{(\P^{\smalltext{t}\smalltext{,}\smalltext{\omega}})}$, $\P^{t,\omega}$--a.s., 
	\item[$(ii)$] $\cU^{t,\omega}_{t\smallertext{+}\smallertext{\cdot}} \in \H^{2,t,\omega}_{T,\hat\beta}(\mu^{X};\F,\P^{t,\omega})$ and $\displaystyle \bigg(\int_{t\land T}^{(t\smallertext{+}\smallertext{\cdot})\land T}\d(\cU\ast\tilde{\mu}^{X,\P})_r\bigg)^{(\P)}(\omega\otimes_t\cdot) = \big(\cU^{t,\omega}_{t\smallertext{+}\smallertext{\cdot}}\ast\tilde\mu^{X,\P^{\smalltext{t}\smalltext{,}\smalltext{\omega}}}\big)^{(\F,\P^{\smalltext{t}\smalltext{,}\smalltext{\omega}})}$, $\P^{t,\omega}$--a.s., 
	\item[$(iii)$] $\cN^{t,\omega}_{t\smallertext{+}\smallertext{\cdot}}-\cN^{t,\omega}_t \in \cH^{2,t,\omega,\perp}_{0,T,\hat\beta}(X^{c,\P^{\smalltext{t}\smalltext{,}\smalltext{\omega}}},\mu^{X};\F_\smallertext{+},\P^{t,\omega})$,
\end{enumerate}
for $\omega \in \Omega\setminus\sN$.

\medskip
We start with $(iii)$. By \Cref{lem::conditioning_martingale2}, the process $\cN^{t,\omega}_{t\smallertext{+}\smallertext{\cdot}}-\cN^{t,\omega}_t = \cN_{t\smallertext{+}\smallertext{\cdot}}(\omega\otimes_t\cdot) -\cN_{t}(\omega\otimes_t\cdot)$ is a square-integrable $(\F_\smallertext{+},\P^{t,\omega})$-martingale for $\omega \in \Omega\setminus\sN$. We also have
\begin{equation*}
	[\cN^{t,\omega}_{t\smallertext{+}\smallertext{\cdot}}-\cN^{t,\omega}_t]^{(\F_\smalltext{+},\P^{\smalltext{t}\smalltext{,}\smalltext{\omega}})} = [\cN - \cN_{\cdot\land t}]^{(\F_\smalltext{+},\P)}_{t\smallertext{+}\smallertext{\cdot}}(\omega\otimes_t\cdot), \; \textnormal{for $\omega \in \Omega\setminus\sN$,}
\end{equation*}
by \cite[Theorem I.4.47.a)]{jacod2003limit}. Then
\begin{align*}
		\int\|\cN^{t,\omega}_{t\smallertext{+}\smallertext{\cdot}}-\cN^{t,\omega}_t\|^2_{\cH^{\smalltext{2}\smalltext{,}\smalltext{t}\smalltext{,}\smalltext{\omega}}_{\smalltext{T}\smalltext{,}\smalltext{\hat\beta}}(\F_\smalltext{+},\P^{\smalltext{t}\smalltext{,}\smalltext{\omega}})}\P(\d\omega) 
		&= \int\E^{\P^{\smalltext{t}\smalltext{,}\smalltext{\omega}}}\bigg[\int_0^{(T-t\land T)^{\smalltext{t}\smalltext{,}\smalltext{\omega}}}\cE(\hat\beta A^{t,\omega}_{t\smallertext{+}\smallertext{\cdot}})_{r}\d[\cN^{t,\omega}_{t\smallertext{+}\smallertext{\cdot}}-\cN^{t,\omega}_t]^{(\F_\smalltext{+},\P^{\smalltext{t}\smalltext{,}\smalltext{\omega}})}_r\bigg]\P(\d\omega) \\
	&= \int\E^{\P^{\smalltext{t}\smalltext{,}\smalltext{\omega}}}\bigg[\int_0^{(T-t\land T)^{\smalltext{t}\smalltext{,}\smalltext{\omega}}}\cE(\hat\beta A^{t,\omega}_{t\smallertext{+}\smallertext{\cdot}})_{r}\d[\cN - \cN_{\cdot\land t}]^{(\F_\smalltext{+},\P)}_{t\smallertext{+}r}(\omega\otimes_t\cdot)\bigg]\P(\d\omega) \\
	&= \int\E^{\P^{\smalltext{t}\smalltext{,}\smalltext{\omega}}}\Bigg[\bigg(\int_0^{(T-t\land T)}\cE(\hat\beta A_{t\smallertext{+}\smallertext{\cdot}})_{r}\d[\cN - \cN_{\cdot\land t}]^{(\F_\smalltext{+},\P)}_{t\smallertext{+}r}\bigg)^{t,\omega}\Bigg]\P(\d\omega) \\
	&= \int\E^{\P}\bigg[\int_{t}^{T}\frac{\cE(\hat\beta A)_{r}}{\cE(\hat\beta A)_t}\d[\cN]^{(\F_\smalltext{+},\P)}_{r}\bigg|\cF_t\bigg](\omega)\P(\d\omega) 
	\leq \|\cN\|^2_{\cH^{\smalltext{2}}_{\smalltext{T}\smalltext{,}\smalltext{\hat\beta}}} < \infty,
\end{align*}
where the first equality follows by dual predictable projection (see \cite[Proposition 6.6.5]{weizsaecker1990stochastic}). This yields
\begin{equation*}
	\|\cN^{t,\omega}_{t\smallertext{+}\smallertext{\cdot}}-\cN^{t,\omega}_t\|^2_{\cH^{\smalltext{2}\smalltext{,}\smalltext{t}\smalltext{,}\smalltext{\omega}}_{\smalltext{T}\smalltext{,}\smalltext{\hat\beta}}(\F_\smalltext{+},\P^{\smalltext{t}\smalltext{,}\smalltext{\omega}})} < \infty, \; \textnormal{$\omega \in \Omega\setminus\sN$.}
\end{equation*}
From
\begin{equation*}
	[\cN - \cN_{\cdot\land t}, X^{c,\P} - X^{c,\P}_{\cdot\land t} ]^{(\F_\tinytext{+},\P)} = \langle \cN - \cN_{\cdot\land t}, X^{c,\P} - X^{c,\P}_{\cdot\land t} \rangle^{(\F_\tinytext{+},\P)} = 0, \; \text{$\P$--a.s.},
\end{equation*}
and then
\begin{equation*}
	[\cN - \cN_{\cdot\land t}, X^{c,\P} - X^{c,\P}_{\cdot\land t} ]^{(\F_\tinytext{+},\P)}(\omega\otimes_t\cdot) = \langle \cN - \cN_{\cdot\land t}, X^{c,\P} - X^{c,\P}_{\cdot\land t} \rangle^{(\F_\tinytext{+},\P)}(\omega\otimes_t\cdot) = 0, \; \text{$\P^{t,\omega}$--a.s.}, \; \omega \in \Omega\setminus\sN,
\end{equation*}
we deduce together with \cite[Theorem I.4.47.a)]{jacod2003limit} and since $X^{c,\P}_{t\smallertext{+}\smallertext{\cdot}}(\omega\otimes_t\cdot) - X^{c,\P}_{t}(\omega) = X^{c,\P^{\smalltext{t}\smalltext{,}\smalltext{\omega}}}$, $\P^{t,\omega}$--a.s., for $\omega \in \Omega\setminus\sN$ (see \Cref{cor::shif_quadratic_variation_continuous_martingale_part}) that
\begin{align*}
	\langle\cN^{t,\omega}_{t\smallertext{+}\smallertext{\cdot}}-\cN^{t,\omega}_t, X^{c,\P^{\smalltext{t}\smalltext{,}\smalltext{\omega}}} \rangle^{(\F_\smalltext{+},\P^{\smalltext{t}\smalltext{,}\smalltext{\omega}})} 
	&= \big[\cN_{t\smallertext{+}\smallertext{\cdot}}(\omega\otimes_t\cdot)-\cN_{t}(\omega\otimes_t\cdot), X^{c,\P}_{t\smallertext{+}\smallertext{\cdot}}(\omega\otimes_t\cdot) - X^{c,\P}_{t}(\omega) \big]^{(\F_\smalltext{+},\P^{\smalltext{t}\smalltext{,}\smalltext{\omega}})} \\
	&= \big[\cN - \cN_{\cdot\land t}, X^{c,\P} - X^{c,\P}_{\cdot \land t} \big]^{(\F_\smalltext{+},\P)}_{t\smallertext{+}\smallertext{\cdot}}(\omega\otimes_t\cdot) = 0, \; \text{$\P^{t,\omega}$--a.s., $\omega \in \Omega\setminus\sN$.}
\end{align*}
This proves the orthogonality with respect to the continuous local martingale part of $X$ for $\omega \in \Omega\setminus\sN$.

\medskip
We turn to the $\P^{t,\omega}$-orthogonality of $\cN^{t,\omega}_{t\smallertext{+}\smallertext{\cdot}}-\cN^{t,\omega}_t$ and the jump measure $\mu^X$ for $\omega \in \Omega\setminus\sN$. Let $0 < V = V \land 1 \leq 1$ be a $\widetilde\cP$-measurable function satisfying $0 \leq V \ast\mu^X \leq 1$; see \cite[Lemma 6.5]{neufeld2014measurability}. We write $\mu(\omega;\d t, \d x) \coloneqq V_t(\omega;x)\mu^X(\omega;\d t, \d x)$. The property $M^\P_{\mu^\smalltext{X}}[\Delta\cN|\widetilde\cP] = 0$, $M^\P_{\mu^\smalltext{X}}$--a.e., is equivalent to $M^\P_{\mu}[\Delta\cN|\widetilde\cP] = 0$, $M^\P_\mu$--a.e., which in turn is equivalent to
\begin{align*}
	\E^\P\big[(W\Delta\cN) \ast\mu_\infty\big] = 0,
\end{align*}
for each $\widetilde\cP$-measurable and bounded function $W$. This follows from $A \longmapsto M^\P_{\mu}[\1_A\Delta\cN]$ being a real-valued, signed measure on $(\widetilde\Omega,\widetilde{\cP})$; we have
\begin{align*}
	M^\P_{\mu}\big[|\Delta\cN|\big]^2 
	\leq M^\P_{\mu^\smalltext{X}}[V^2]M^\P_{\mu^\smalltext{X}}\big[|\Delta\cN|^2\big]
	\leq  \E^\P[( \Delta\cN)^2\ast\mu^X_\infty] \leq \E^\P\Bigg[\sum_{s \in (0,\infty)} (\Delta\cN_s)^2\Bigg] 
	\leq \E^\P\big[[\cN]^{(\F_\tinytext{+},\P)}_\infty\big] < \infty.
\end{align*}
Fix a $\widetilde\cP$-measurable and bounded function $W$ and an $\cF_t$-measurable and bounded random variable $\eta$. Then
\begin{align*}
	\E^{\P}\bigg[\eta \E^{\P} \Big[ (W\1_{\llparenthesis t,\infty\rrparenthesis}\Delta\cN)\ast\mu_\infty  \Big|\cF_t\Big] \bigg] 
	= \E^{\P}\bigg[\E^{\P} \Big[ (W \eta\1_{\llparenthesis t,\infty\rrparenthesis}\Delta\cN)\ast\mu_\infty  \Big|\cF_t\Big] \bigg] 
	= \E^{\P}\Big[ (W \eta\1_{\llparenthesis t,\infty\rrparenthesis}\Delta\cN)\ast\mu_\infty\Big] = 0,
\end{align*}
since $W\eta\1_{\llparenthesis t, \infty \rrparenthesis}$ is $\widetilde\cP$-measurable and bounded. Therefore, 
\begin{equation*}
	\E^{\P^{\smalltext{t}\smalltext{,}\smalltext{\omega}}} \Big[ \big((W\1_{\llparenthesis t,\infty\rrparenthesis}\Delta\cN)\ast\mu_\infty\big)(\omega\otimes_t\cdot)  \Big]
	=\E^{\P} \Big[ (W\1_{\llparenthesis t,\infty\rrparenthesis}\Delta\cN)\ast\mu_\infty  \Big|\cF_t\Big] = 0, \textnormal{for $\P$--a.e. $\omega \in \Omega$.}
\end{equation*}
Since $\widetilde{\cP}$ is separable (see \cite[Lemma 6.3]{neufeld2014measurability}), a monotone class argument then implies that
\begin{equation}\label{eq::rcpd_stochastic_integral_compensated_zero}
	\E^{\P^{\smalltext{t}\smalltext{,}\smalltext{\omega}}} \Big[ (W\1_{\llparenthesis t,\infty\rrparenthesis}\Delta\cN^\P)\ast\mu_\infty(\omega\otimes_t\cdot)  \Big] = 0,
\end{equation}
holds for all $\widetilde{\cP}$-measurable and bounded functions $W$ outside some $\P$--null set, and thus for every $\omega \in \Omega\setminus\sN$. We now fix $\omega \in \Omega\setminus\sN$, a $\widetilde{\cP}$-measurable and bounded function $W$, and another $\widetilde{\cP}$-measurable and bounded function $\widetilde{W}$ such that $\widetilde{W}_{t+r}(\omega\otimes_t\tilde\omega,x) = W(\tilde\omega,r,x)$ for all $(\tilde\omega,r,x) \in \Omega \times (0,\infty)\times \R^d$, see \cite[Lemma 3.6]{neufeld2016nonlinear}. Since $X^{t,\omega}_{t\smallertext{+}\smallertext{\cdot}}-X_t(\omega) = X$ and thus $\Delta(X^{t,\omega}_{t\smallertext{+}\smallertext{\cdot}}) = \Delta X$, it follows that
\begin{align*}
	\int_{(0,\infty)\times\R^\smalltext{d}} V_{t\smallertext{+}r}(\omega\otimes_t\tilde\omega;x)\mu^X(\tilde\omega;\d r, \d x) 
	&= \int_{(0,\infty)\times\R^\smalltext{d}} V_{t\smallertext{+}r}(\omega\otimes_t\tilde\omega;x)\mu^{X^{\smalltext{t}\smalltext{,}\smalltext{\omega}}_{\smalltext{t}\smalltext{+}\smalltext{\cdot}}}(\tilde\omega;\d r, \d x) \\
	&= \sum_{r \in (0,\infty)} V_{t\smallertext{+}r}(\omega\otimes_t\tilde\omega; \Delta(X^{t,\omega}_{t\smallertext{+}\smallertext{\cdot}})_r(\tilde\omega))\1_{\{\Delta(X^{t,\omega}_{t\smallertext{+}\smallertext{\cdot}})_r(\tilde\omega) \neq 0\}} \\
	&= \sum_{r \in (0,\infty)} V_{t\smallertext{+}r}(\omega\otimes_t\tilde\omega; \Delta X_{t+r}(\omega\otimes_t\tilde\omega))\1_{\{(\Delta X)_{t+r}(\omega\otimes_t\tilde\omega) \neq 0\}} \\
	&= \sum_{r \in (t,\infty)} V_{r}(\omega\otimes_t\tilde\omega; \Delta X_{r}(\omega\otimes_t\tilde\omega))\1_{\{\Delta X_{r}(\omega\otimes_t\tilde\omega) \neq 0\}} \\
	&= \bigg(\int_{(t,\infty)\times\R^\smalltext{d}} V_{r}(x) \mu^X(\d r, \d x)\bigg)(\omega\otimes_t\tilde\omega) \leq 1,
\end{align*}
and therefore $M^{\P^{\smalltext{t}\smalltext{,}\smalltext{\omega}}}_{\mu^\smalltext{X}}[\Delta\cN^{t,\omega}_{t\smallertext{+}\smallertext{\cdot}}|\widetilde{\cP}] = 0$, $M^{\P^{\smalltext{t}\smalltext{,}\smalltext{\omega}}}_{\mu^\smalltext{X}}$--a.e., is equivalent to $M^{\P^{\smalltext{t}\smalltext{,}\smalltext{\omega}}}_{\mu^\smalltext{X}}[V^{t,\omega}_{t\smallertext{+}\smallertext{\cdot}}\Delta\cN^{t,\omega}_{t\smallertext{+}\smallertext{\cdot}}|\widetilde{\cP}] = 0$, $M^{\P^{\smalltext{t}\smalltext{,}\smalltext{\omega}}}_{\mu^\smalltext{X}}$--a.e., similar to before. The latter follows immediately from
\begin{align*}
	& \E^{\P^{\smalltext{t}\smalltext{,}\smalltext{\omega}}}\bigg[\int_{(0,\infty)\times\R^\smalltext{d}} W_r(x)V_{t\smallertext{+}r}(\omega\otimes_t\cdot\,; x) \Delta\cN^{t,\omega}_{t\smallertext{+}r}\mu^X(\d r, \d x)\bigg] \\
	&\quad= \E^{\P^{\smalltext{t}\smalltext{,}\smalltext{\omega}}}\bigg[\int_{(0,\infty)\times\R^\smalltext{d}} \widetilde{W}_{t\smallertext{+}r}(\omega\otimes_t\cdot\,; x)V_{t\smallertext{+}r}(\omega\otimes_t\cdot\,; x) \Delta\cN^{t,\omega}_{t\smallertext{+}r}\mu^X(\d r, \d x)\bigg] \\
	&\quad = \E^{\P^{\smalltext{t}\smalltext{,}\smalltext{\omega}}}\bigg[\int_{(0,\infty)\times\R^\smalltext{d}} \widetilde{W}_{t\smallertext{+}r}(\omega\otimes_t\cdot\,; x)V_{t\smallertext{+}r}(\omega\otimes_t\cdot\,; x) \Delta\cN^{t,\omega}_{t\smallertext{+}r}\mu^{X^{\smalltext{t}\smalltext{,}\smalltext{\omega}}_{\smalltext{t}\smalltext{+}\smalltext{\cdot}}}(\d r, \d x)\bigg] \\
	&\quad = \E^{\P^{\smalltext{t}\smalltext{,}\smalltext{\omega}}}\bigg[\int_{(0,\infty)\times\R^\smalltext{d}} \widetilde{W}_{t\smallertext{+}r}(\omega\otimes_t\cdot\,; x)V_{t\smallertext{+}r}(\omega\otimes_t\cdot\,; x) \Delta\cN_{t\smallertext{+}r}(\omega\otimes_t\cdot)\mu^{X_{\smalltext{t}\smalltext{+}\smalltext{\cdot}}}(\omega\otimes_t\cdot\,; \d r, \d x)\bigg] \\
	&\quad = \E^{\P^{\smalltext{t}\smalltext{,}\smalltext{\omega}}}\bigg[\int_{(t,\infty)\times\R^\smalltext{d}} \widetilde{W}_{r}(\omega\otimes_t\cdot\,; x)V_{r}(\omega\otimes_t\cdot\,;x) \Delta\cN_{r}(\omega\otimes_t\cdot)\mu^{X}(\omega\otimes_t\cdot\,; \d r, \d x)\bigg] \\
	&\quad = \E^{\P^{\smalltext{t}\smalltext{,}\smalltext{\omega}}}\bigg[\bigg(\int_{(t,\infty)\times\R^\smalltext{d}} \widetilde{W}_{r}(x)V_{r}(x) \Delta\cN_{r}\mu^{X}(\d r, \d x)\bigg)(\omega\otimes_t\cdot)\bigg] 
	= \E^{\P^{\smalltext{t}\smalltext{,}\smalltext{\omega}}} \Big[ (\widetilde{W}\1_{\llparenthesis t,\infty\rrparenthesis}\Delta\cN)\ast\mu_\infty(\omega\otimes_t\cdot)  \Big]
	= 0,
\end{align*}
where in the second equality we used $X = X^{t,\omega}_{t\smallertext{+}\smallertext{\cdot}} - X_{t}(\omega)$ again and in the last equality \eqref{eq::rcpd_stochastic_integral_compensated_zero}. This proves the required orthogonality for every $\omega \in \Omega\setminus\sN$.

\medskip
We turn to $(ii)$ and start with the first assertion. Since $\mathsf{K}^{t,\omega,\P^{\smalltext{t}\smalltext{,}\smalltext{\omega}}}_{r} = \mathsf{K}^{\P}_{\omega\otimes_t\cdot,t+r}$, $\d(C^{t,\omega}_{t\smallertext{+}\smallertext{\cdot}}-C_t(\omega))_r$--a.e., for $\omega \in \Omega\setminus\sN$ (see \Cref{prop::conditioning_characteristics2} and \Cref{lem::uniqueness_of_kernel}), we find
\begin{align*}
		\int\|\cU^{t,\omega}_{t\smallertext{+}\smallertext{\cdot}}\|^2_{\H^{\smalltext{2}\smalltext{,}\smalltext{t}\smalltext{,}\smalltext{\omega}}_{\smalltext{T}\smalltext{,}\smalltext{\hat\beta}}(\mu^X;\F,\P^{\smalltext{t}\smalltext{,}\smalltext{\omega}})}\P(\d\omega) 
	&= \int\E^{\P^{\smalltext{t}\smalltext{,}\smalltext{\omega}}}\bigg[\int_0^{(T-t\land T)^{\smalltext{t}\smalltext{,}\smalltext{\omega}}}\cE(\hat\beta A^{t,\omega}_{t\smallertext{+}\smallertext{\cdot}})_{r}\|\cU^{t,\omega}_{t\smallertext{+}r}(\cdot)\|^2_{\hat\L^\smalltext{2}_{\smalltext{\omega}\smalltext{\otimes}_\tinytext{t}\smalltext{\cdot}\smalltext{,}\smalltext{t}\smalltext{+}\smalltext{r}}(\mathsf{K}^{\P}_{\smalltext{\omega}\smalltext{\otimes}_\tinytext{t}\smalltext{\cdot}\smalltext{,}\smalltext{t}\smalltext{+}\smalltext{r}})}\d(C^{t,\omega}_{t\smallertext{+}\smallertext{\cdot}})_r\bigg]\P(\d\omega) \\
	&= \int\E^{\P^{\smalltext{t}\smalltext{,}\smalltext{\omega}}}\bigg[\int_0^{(T-t\land T)^{\smalltext{t}\smalltext{,}\smalltext{\omega}}}\cE(\hat\beta A_{t\smallertext{+}\smallertext{\cdot}})^{t,\omega}_{r}\|\cU_{t\smallertext{+}r}(\omega\otimes_t\cdot)\|^2_{\hat\L^\smalltext{2}_{\smalltext{\omega}\smalltext{\otimes}_\tinytext{t}\smalltext{\cdot}\smalltext{,}\smalltext{t}\smalltext{+}\smalltext{r}}(\mathsf{K}^{\P}_{\smalltext{\omega}\smalltext{\otimes}_\tinytext{t}\smalltext{\cdot}\smalltext{,}\smalltext{t}\smalltext{+}\smalltext{r}})}\d(C_{t\smallertext{+}\smallertext{\cdot}})^{t,\omega}_r\bigg]\P(\d\omega) \\
	&= \int\E^{\P^{\smalltext{t}\smalltext{,}\smalltext{\omega}}}\Bigg[\bigg(\int_0^{(T-t\land T)}\cE(\hat\beta A_{t\smallertext{+}\smallertext{\cdot}})_{r}\|\cU_{t\smallertext{+}r}(\cdot)\|^2_{\hat\L^\smalltext{2}_{\smalltext{t}\smalltext{+}\smalltext{r}}(\mathsf{K}^{\P}_{\smalltext{t}\smalltext{+}\smalltext{r}})}\d(C_{t\smallertext{+}\smallertext{\cdot}})_r\bigg)(\omega\otimes_t\cdot)\Bigg]\P(\d\omega) \\
	&= \int\E^{\P}\bigg[\int_{t}^{T}\frac{\cE(\hat\beta A)_{r}}{\cE(\hat\beta A)_t}\|\cU^{\P}_{r}(\cdot)\|^2_{\hat\L^\smalltext{2}_{\smalltext{r}}(\mathsf{K}^{\P}_{\smalltext{r}})}\d C_r \bigg| \cF_t \bigg](\omega)\P(\d\omega) \\
	&\leq \E^{\P}\bigg[\int_{t}^{T}\cE(\hat\beta A)_{r}\|\cU_{r}(\cdot)\|^2_{\hat\L^\smalltext{2}_{\smalltext{r}}(\mathsf{K}^{\P}_{\smalltext{r}})}\d C_r \bigg] \leq \|\cU\|^2_{\H^\smalltext{2}_{\smalltext{T}\smalltext{,}\smalltext{\hat\beta}}(\mu^X;\F,\P)}  < \infty,
\end{align*}
which implies
\begin{equation*}
	\|\cU^{t,\omega}_{t\smallertext{+}\smallertext{\cdot}}\|^2_{\H^{\smalltext{2}\smalltext{,}\smalltext{t}\smalltext{,}\smalltext{\omega}}_{\smalltext{T}\smalltext{,}\smalltext{\hat\beta}}(\mu^X;\F,\P^{\smalltext{t}\smalltext{,}\smalltext{\omega}})} < \infty, \; \omega \in \Omega\setminus\sN.
\end{equation*}
Therefore $\cU^{t,\omega}_{t\smallertext{+}\smallertext{\cdot}} \in \H^{2,t,\omega}_{T,\hat\beta}\big(\mu^X; \F,\P^{t,\omega}\big)$ for $\omega \in \Omega\setminus\sN$, since $\cU^{t,\omega}_{t\smallertext{+}\smallertext{\cdot}} = \cU^{t,\omega}_{t\smallertext{+}\smallertext{\cdot}}\1_{\llbracket 0,(T-t\land T)^{\smalltext{t}\smalltext{,}\smalltext{\omega}}\rrbracket}$.
We turn to the second assertion. By \Cref{lem::conditioning_martingale2}, the process 
\begin{equation*}
	\bigg(\int_{t\land T}^{(t\smallertext{+}\smallertext{\cdot})\land T}\d(\cU\ast\tilde{\mu}^{X,\P})_r\bigg)^{(\P)}(\omega\otimes_t\cdot)  =(\cU\ast\tilde{\mu}^{X,\P})^{t,\omega}_{t\smallertext{+}\smallertext{\cdot}} - (\cU\ast\tilde{\mu}^{X,\P})^{t,\omega}_{t},
\end{equation*}
is a square-integrable martingale relative to $(\F,\P^{t,\omega})$ for $\omega \in \Omega\setminus\sN$.
Since
\begin{equation*}
	\Delta(\cU\ast\tilde{\mu}^{X,\P})_r 
	= \cU_r(\Delta X_r)\1_{\{\Delta X_{\smalltext{r}}\neq 0\}} - \int_{\R^\smalltext{d}}\cU_r(x)\nu^{\P}(\{r\}\times \d x), \; r \in [0,\infty), \; \textnormal{$\P$--a.s.},
\end{equation*}
we have
\begin{align*}
	&\Delta(\cU\ast\tilde{\mu}^{X,\P})_r(\omega\otimes_t\cdot) \\
	&= \cU_r(\omega\otimes_t\cdot\,;\Delta X_r(\omega\otimes_t\cdot))\1_{\{\Delta X_{\smalltext{r}}(\omega\otimes_\smalltext{t}\cdot)\neq 0\}} - \int_{\R^\smalltext{d}}\cU_r(\omega\otimes_t\cdot\,;x)\nu^{\P}(\omega\otimes_t\cdot\,;\{r\}\times \d x) \\
	&= \cU_r(\omega\otimes_t\cdot;\Delta X_r(\omega\otimes_t\cdot))\1_{\{\Delta X_{\smalltext{r}}(\omega\otimes_\smalltext{t}\cdot)\neq 0\}} - \int_{\R^\smalltext{d}}\cU_r(\omega\otimes_t\cdot\,;x)\mathsf{K}^{\P}_{\omega\otimes_\smalltext{t}\cdot, r}(\d x) \Delta C_r(\omega\otimes_t\cdot), \; r \in [0,\infty), \; \textnormal{$\P^{t,\omega}$--a.s.},
\end{align*}
for $\omega \in \Omega\setminus\sN$. This then implies
\begin{align*}
	&\Delta((\cU\ast\tilde{\mu}^{X,\P})^{t,\omega}_{t\smallertext{+}\smallertext{\cdot}} - (\cU\ast\tilde{\mu}^{X,\P})^{t,\omega}_{t})_r = \Delta(\cU^\P\ast\tilde{\mu}^{X,\P})_{t\smallertext{+}r}(\omega\otimes_t\cdot)  \\
	&= \cU_{t\smallertext{+}r}(\omega\otimes_t\cdot;\Delta X_{t\smallertext{+}r}(\omega\otimes_t\cdot))\1_{\{\Delta X_{\smalltext{t}\smalltext{+}\smalltext{r}}(\omega\otimes_\smalltext{t}\cdot)\neq 0\}} - \int_{\R^\smalltext{d}}\cU_{t\smallertext{+}r}(\omega\otimes_t\cdot; x)\mathsf{K}^{\P}_{\omega\otimes_\smalltext{t}\cdot,{t\smallertext{+}r}}(\d x) \Delta C_{t\smallertext{+}r}(\omega\otimes_t\cdot) \\
	&= \cU^{t,\omega}_{t\smallertext{+}r}(\Delta X_r)\1_{\{\Delta X_{\smalltext{r}}\neq 0\}} - \int_{\R^\smalltext{d}}\cU^{t,\omega}_{t\smallertext{+}r}(x)\mathsf{K}^{\P}_{\omega\otimes_\smalltext{t}\cdot,{t\smallertext{+}r}}(\d x) \Delta (C^{t,\omega}_{t\smallertext{+}\smallertext{\cdot}}-C_{t}(\omega))_{r} \\
	&= \cU^{t,\omega}_{t\smallertext{+}r}(\Delta X_r)\1_{\{\Delta X_{\smalltext{r}}\neq 0\}} - \int_{\R^\smalltext{d}}\cU^{t,\omega}_{t\smallertext{+}r}(x)\nu^{\P^{\smalltext{t}\smalltext{,}\smalltext{\omega}}}(\{r\}\times \d x) \\
	&= \Delta(\cU^{t,\omega}_{t\smallertext{+}r}\ast\tilde\mu^{X,\P^{\smalltext{t}\smalltext{,}\smalltext{\omega}}})_r, \; r \in (0,\infty), \; \textnormal{$\P^{t,\omega}$--a.s.}, \; \omega \in \Omega \setminus\sN.
\end{align*}
Moreover
\begin{equation*}
	[\cU\ast\tilde{\mu}^{X,\P}]^{(\F,\P)} = \sum_{r \in (0,\cdot]} (\Delta (\cU\ast\tilde{\mu}^{X,\P})_r)^2, \; \textnormal{$\P$--a.s.},
\end{equation*}
implies
\begin{align*}
	[(\cU\ast\tilde{\mu}^{X,\P})^{t,\omega}_{t\smallertext{+}\smallertext{\cdot}} - (\cU\ast\tilde{\mu}^{X,\P})^{t,\omega}_{t}]^{(\F,\P^{\smalltext{t}\smalltext{,}\smalltext{\omega}})}
	&= [\cU\ast\tilde{\mu}^{X,\P} - \cU\ast\tilde{\mu}^{X,\P}_{\cdot\land t}]^{(\F,\P)}_{t\smallertext{+}\smallertext{\cdot}}(\omega\otimes_t\cdot) \\
	&= \sum_{r \in (0,\cdot]} (\Delta((\cU\ast\tilde{\mu}^{X,\P})^{t,\omega}_{t\smallertext{+}\smallertext{\cdot}} - (\cU\ast\tilde{\mu}^{X,\P})^{t,\omega}_{t})_r)^2, \; \textnormal{$\P^{t,\omega}$--a.s.}, \; \omega \in \Omega\setminus\sN.
\end{align*}
This implies that $(\cU\ast\tilde{\mu}^{X,\P})^{t,\omega}_{t\smallertext{+}\smallertext{\cdot}} - (\cU\ast\tilde{\mu}^{X,\P})^{t,\omega}_{t}$ is a purely discontinuous, $(\F,\P^{t,\omega})$--square-integrable martingale, see \cite[Theorem I.4.52]{jacod2003limit}. Since its jumps agree with the jumps of $\cU^{t,\omega}_{t\smallertext{+}\smallertext{\cdot}}\ast\tilde\mu^{X,\P^{\smalltext{t}\smalltext{,}\smalltext{\omega}}}$, up to $\P^{t,\omega}$-evanescence, for each $\omega\in\Omega\setminus\sN$, this therefore implies
\begin{equation*}
	(\cU\ast\tilde{\mu}^{X,\P})^{t,\omega}_{t\smallertext{+}\smallertext{\cdot}} - (\cU\ast\tilde{\mu}^{X,\P})^{t,\omega}_{t} = \cU^{t,\omega}_{t\smallertext{+}\smallertext{\cdot}}\ast\tilde\mu^{X,\P^{\smalltext{t}\smalltext{,}\smalltext{\omega}}}, \; \textnormal{$\P^{t,\omega}$--a.s.,} \; \omega \in \Omega\setminus\sN
\end{equation*}
by \cite[Corollary I.4.19]{jacod2003limit}.

\medskip
We turn to $(i)$ and start with the first assertion again. As before, we have 
\begin{align*}
	\int\|\cZ^{t,\omega}_{t\smallertext{+}\smallertext{\cdot}}\|^2_{\H^{\smalltext{2}\smalltext{,}\smalltext{t}\smalltext{,}\smalltext{\omega}}_{\smalltext{T}\smalltext{,}\smalltext{\hat\beta}}(X^{\smalltext{c}\smalltext{,}\smalltext{\P}^{\tinytext{t}\tinytext{,}\tinytext{\omega}}};\F,\P^{\smalltext{t}\smalltext{,}\smalltext{\omega}})}\P(\d\omega)
	&= \int\E^{\P^{\smalltext{t}\smalltext{,}\smalltext{\omega}}}\bigg[\int_0^{(T-t\land T)^{\smalltext{t}\smalltext{,}\smalltext{\omega}}}\cE(\hat\beta A^{t,\omega}_{t\smallertext{+}\smallertext{\cdot}})_{r}\big(\cZ^{t,\omega}_{t\smallertext{+}r}\big)^\top \hat{\mathsf{a}}^{t,\omega}_r \cZ^{t,\omega}_{t\smallertext{+}r}\d(C^{t,\omega}_{t\smallertext{+}\smallertext{\cdot}})_r\bigg]\P(\d\omega) \\
	&= \int\E^{\P^{\smalltext{t}\smalltext{,}\smalltext{\omega}}}\Bigg[\bigg(\int_0^{(T-t\land T)}\cE(\hat\beta A_{t\smallertext{+}\smallertext{\cdot}})_{r}\big(\cZ_{t\smallertext{+}r}\big)^\top \hat{\mathsf{a}}_{t\smallertext{+}r} \cZ_{t\smallertext{+}r}\d (C_{t\smallertext{+}\cdot})_r\bigg)^{t,\omega}\Bigg]\P(\d\omega) \\
	&= \int\E^{\P}\bigg[\int_t^{T}\frac{\cE(\hat\beta A)_{r}}{\cE(\hat\beta A)_t}\big(\cZ_{r}\big)^\top \hat{\mathsf{a}}_{r} \cZ_{r}\d C_r \bigg| \cF_t\bigg](\omega)\P(\d\omega) \\
	&\leq \E^{\P}\bigg[\int_t^{T}\cE(\hat\beta A)_{r}\big(\cZ_{r}\big)^\top \hat{\mathsf{a}}_{r} \cZ_{r}\d C_r \bigg] \leq \|\cZ\|^2_{\H^\smalltext{2}_{\smalltext{T}\smalltext{,}\smalltext{\hat\beta}}(X^{\smalltext{c}\smalltext{,}\smalltext{\P}};\F,\P)} < \infty,
\end{align*}
which implies 
\begin{equation*}
	\|\cZ^{t,\omega}_{t\smallertext{+}\smallertext{\cdot}}\|^2_{\H^{\smalltext{2}\smalltext{,}\smalltext{t}\smalltext{,}\smalltext{\omega}}_{\smalltext{T}\smalltext{,}\smalltext{\hat\beta}}(X^{\smalltext{c}\smalltext{,}\smalltext{\P}^{\tinytext{t}\tinytext{,}\tinytext{\omega}}};\F,\P^{\smalltext{t}\smalltext{,}\smalltext{\omega}})} < \infty, \; \omega \in \Omega\setminus\sN.
\end{equation*}
Therefore $\cZ^{t,\omega}_{t\smallertext{+}\smallertext{\cdot}} \in \H^{2,t,\omega}_{T,\hat\beta}\big(X^{c,\P^{\smalltext{t}\smalltext{,}\smalltext{\omega}}}; \F,\P^{t,\omega}\big)$ for $\omega \in \Omega\setminus\sN$ since $\cZ^{t,\omega}_{t\smallertext{+}\smallertext{\cdot}} = \cZ^{t,\omega}_{t\smallertext{+}\smallertext{\cdot}}\1_{\llbracket 0,(T-t\land T)^{\smalltext{t}\smalltext{,}\smalltext{\omega}}\rrbracket}$.

\medskip
We turn to the equality of the stochastic integrals. By \cite[Theorem III.6.4.a)]{jacod2003limit}, the integral $(\cZ\bcdot X^{c,\P})^{(\P)}$ is the limit, as $n$ tends to infinity, in the sense of uniform convergence on compacts in $\P$-measure, of the component-wise stochastic integrals
\begin{equation*}
	(\cZ(n) \bcdot X^{c,\P})^{(\P)} 
	\coloneqq \sum_{i = 1}^d (\cZ^{i}(n) \bcdot X^{c,\P,i})^{(\P)}, \; n \in \N,
\end{equation*}
where $\cZ(n) \coloneqq \cZ\1_{\{|\cZ|\leq n\}}$, and $\cZ^{i}(n)$ and $X^{c,\P,i}$ denote the $i$-th components of $\cZ(n)$ and $X^{c,\P}$, respectively. Note that $\cZ^i(n) \in \H^2_\textnormal{loc}(X^{c,\P,i};\F,\P)$. Up to extracting a subsequence if necessary, we have that $(\cZ(n) \bcdot X^{c,\P})^{(\P)}(\omega\otimes_t\cdot)$ converges uniformly on compacts in $\P^{t,\omega}$-measure to $(\cZ\bcdot X^{c,\P})^{(\P)}(\omega\otimes_t\cdot)$ for every $\omega\in\Omega\setminus\sN$, and since $(\cZ(n))^{t,\omega}_{t\smallertext{+}\smallertext{\cdot}} = \cZ(n)_{t\smallertext{+}\smallertext{\cdot}}(\omega\otimes_t\cdot) = \cZ^{t,\omega}_{t\smallertext{+}\smallertext{\cdot}} \1_{\{|\cZ^{\smalltext{t}\smalltext{,}\smalltext{\omega}}_{\smalltext{t}\smalltext{+}\smalltext{\cdot}}| \leq n\}}$ for $\omega \in \Omega$, we obtain
\[
	\lim_{n \rightarrow \infty}\sum_{i = 1}^d\big((\cZ^i(n))^{t,\omega}_{t\smallertext{+}\smallertext{\cdot}} \bcdot X^{c,\P^{\smalltext{t}\smalltext{,}\smalltext{\omega}},i}\big)^{(\P^{\smalltext{t}\smalltext{,}\smalltext{\omega}})}
	= (\cZ^{t,\omega}_{t\smallertext{+}\smallertext{\cdot}} \bcdot X^{c,\P^{\smalltext{t}\smalltext{,}\smalltext{\omega}}})^{(\P^{\smalltext{t}\smalltext{,}\smalltext{\omega}})},
\]
and
\begin{equation*}
	\lim_{n \rightarrow \infty}\sum_{i = 1}^d \Big((\cZ^i(n)\bcdot X^{c,\P,i})^{(\P)}_{t\smallertext{+}\smallertext{\cdot}}(\omega\otimes_t\cdot) - (\cZ^i(n)\bcdot X^{c,\P,i})^{(\P)}_{t}(\omega) \Big)
	= \bigg(\int_{t}^{t\smallertext{+}\smallertext{\cdot}}\cZ_r \d X^{c,\P}_r\bigg)^{(\P)}(\omega\otimes_t\cdot),
\end{equation*}
uniformly on compacts in $\P^{t,\omega}$-measure for every $\omega \in \Omega\setminus\sN$.
It thus suffices to show, for $i\in\{1,\dots,d\}$, that
\begin{equation*}
	(\cZ^i(n)\bcdot X^{c,\P,i})^{(\P)}_{t\smallertext{+}\smallertext{\cdot}}(\omega\otimes_t\cdot) - (\cZ^i(n)\bcdot X^{c,\P,i})^{(\P)}_{t}(\omega) 
	= \big((\cZ^i(n))^{t,\omega}_{t\smallertext{+}\smallertext{\cdot}} \bcdot X^{c,\P^{\smalltext{t}\smalltext{,}\smalltext{\omega}},i}\big)^{(\P^{\smalltext{t}\smalltext{,}\smalltext{\omega}})}, \; \text{$\P^{t,\omega}$--a.s.}, \; \omega \in \Omega\setminus\sN.
\end{equation*}
We can therefore suppose, without loss of generality, that $X$ is one-dimensional and that $\cZ \in \H^2_\textnormal{loc}(X^{c,\P};\F,\P)$. By \cite[Theorem 6.2.2, Lemma 6.2.9]{weizsaecker1990stochastic} and \Cref{lem::predictable_localisation},
 there exists an $(\F,\P)$--localising sequence $(\tau_k)_{k \in \N}$ of $\F$-predictable stopping times and a sequence $(\cZ^\ell)_{\ell \in \N}$ of elementary predictable processes in the sense of \cite[Definition 4.4.1, Proposition 4.4.2.(b)]{weizsaecker1990stochastic} such that, for each $k \in \N$,
\begin{enumerate}
	\item[$(a)$] $\displaystyle \E^\P\bigg[\int_0^{\tau_\smalltext{k}}| \cZ_r|^2 \d\langle X^{c,\P}\rangle^{(\P)}_r\bigg] < \infty$ and $\displaystyle \lim_{\ell \rightarrow \infty}\E^\P\bigg[\int_0^{\tau_\smalltext{k}}\big|\cZ^\ell_r - \cZ_r\big|^2 \d\langle X^{c,\P}\rangle^{(\P)}_r\bigg] = 0$;
	\item[$(b)$] $X^{c,\P}_{\cdot \land \tau_\smalltext{k}}$ is an $(\F,\P)$--square-integrable martingale;
	\item[$(c)$] $\displaystyle \lim_{\ell \rightarrow \infty} \cZ^\ell \bcdot X^{c,\P} = (\cZ \bcdot X^{c,\P})^{(\P)}$ uniformly on compacts in $\P$-measure.
\end{enumerate}
The third property $(c)$ implies that
\begin{equation}\label{eq::convergence_ucp}
	\lim_{\ell \rightarrow \infty} (\cZ^\ell \bcdot X^{c,\P})(\omega\otimes_t\cdot) = (\cZ \bcdot X^{c,\P})^{(\P)}(\omega\otimes_t\cdot),
\end{equation}
uniformly on compacts in $\P^{t,\omega}$-measure for $\omega \in \Omega\setminus\sN$. The first property described in $(a)$ together with \Cref{cor::shif_quadratic_variation_continuous_martingale_part} implies, using previous arguments, that for $\omega \in \Omega\setminus\sN$ and every $k \in \N$
\begin{align*}
	\E^{\P^{\smalltext{t}\smalltext{,}\smalltext{\omega}}}\bigg[\int_0^{(\tau_\smalltext{k}-t\land\tau_\smalltext{k})^{\tinytext{t}\tinytext{,}\tinytext{\omega}}}| \cZ^{t,\omega}_{t\smallertext{+}r}|^2 \d\langle X^{c}\rangle_{r}\bigg]
	&= \E^{\P^{\smalltext{t}\smalltext{,}\smalltext{\omega}}}\bigg[\int_0^{(\tau_\smalltext{k}-t\land\tau_\smalltext{k})^{t,\omega}}| \cZ_{t\smallertext{+}r}(\omega\otimes_t\cdot)|^2 \d\langle X^{c}\rangle^{t,\omega}_{t\smallertext{+}r}\bigg] \\
	&= \E^{\P}\bigg[\int_{t}^{\tau_\smalltext{k}}| \cZ_{r}|^2 \d\langle X^{c}\rangle_{r}\bigg|\cF_t\bigg](\omega) < \infty,
\end{align*}
and the second one property in $(a)$, for $k = 0$, yields
\begin{align*}
	& \lim_{\ell\rightarrow\infty}\int \E^{\P^{\smalltext{t}\smalltext{,}\smalltext{\omega}}}\bigg[\int_0^{(\tau_\smalltext{0}-t\land\tau_\smalltext{0})^{\smalltext{t}\smalltext{,}\smalltext{\omega}}}\big((\cZ^\ell)^{t,\omega}_{t\smallertext{+}r} - \cZ^{t,\omega}_{t\smallertext{+}r}\big)^2 \d\langle X^{c}\rangle_r\bigg]\P(\d\omega) \\
	&= \lim_{\ell\rightarrow\infty}\int \E^{\P^{\smalltext{t}\smalltext{,}\smalltext{\omega}}}\bigg[\int_0^{(\tau_\smalltext{0}-t\land\tau_\smalltext{0})^{\smalltext{t}\smalltext{,}\smalltext{\omega}}}\big((\cZ^\ell)^{t,\omega}_{t\smallertext{+}r} - \cZ^{t,\omega}_{t\smallertext{+}r}\big)^2 \d\langle X^{c}\rangle^{t,\omega}_{t\smallertext{+}r}\bigg] \P(\d\omega) \\
	&= \lim_{\ell\rightarrow\infty}\int \E^{\P}\bigg[\int_{t}^{\tau_\smalltext{0}}\big(\cZ^\ell_{r} - \cZ_{r}\big)^2 \d\langle X^{c}\rangle_r\bigg|\cF_t\bigg](\omega) \P(\d\omega) \leq \lim_{\ell\rightarrow\infty} \E^{\P}\bigg[\int_{0}^{\tau_\smalltext{0}}\big(\cZ^\ell_{r} - \cZ_{r}\big)^2 \d\langle X^{c}\rangle_r\bigg] = 0.
\end{align*}
Thus there exists a subsequence $(\cZ^{\ell^\smalltext{0}_\smalltext{m}})_{m \in \N}$ for which
\begin{equation*}
	\lim_{m \rightarrow \infty}\E^{\P^{\smalltext{t}\smalltext{,}\smalltext{\omega}}}\bigg[\int_0^{(\tau_\smalltext{0}-t\land\tau_\smalltext{0})^{\smalltext{t}\smalltext{,}\smalltext{\omega}}}\big((\cZ^{\ell^\smalltext{0}_\smalltext{m}})^{t,\omega}_{t\smallertext{+}r} - \cZ^{t,\omega}_{t\smallertext{+}r}\big)^2 \d\langle X^{c}\rangle_r\bigg] = 0, \; \text{for $\P$--a.e. $\omega \in \Omega$},
\end{equation*}
holds. However, $(a)$ is still satisfied for each $k \in \N$ when replacing $(\cZ^\ell)_{\ell \in \N}$ by $(\cZ^{\ell^\smalltext{0}_\smalltext{m}})_{m \in \N}$ since the latter is a subsequence of the former. We can therefore find a further subsequence of $(\cZ^{\ell^{\smalltext{1}}_m})_{m \in \N}$ of $(\cZ^{\ell^\smalltext{0}_\smalltext{m}})_{m \in \N}$ for which 
\begin{equation*}
	\lim_{m \rightarrow \infty}\E^{\P^{\smalltext{t}\smalltext{,}\smalltext{\omega}}}\bigg[\int_0^{(\tau_\smalltext{1}-t\land\tau_\smalltext{1})^{\smalltext{t}\smalltext{,}\smalltext{\omega}}}\big((\cZ^{\ell^\smalltext{1}_\smalltext{m}})^{t,\omega}_{t\smallertext{+}r} - \cZ^{t,\omega}_{t\smallertext{+}r}\big)^2 \d\langle X^{c}\rangle_r\bigg] = 0, \; \text{for $\P$--a.e. $\omega \in \Omega$,}
\end{equation*}
holds. Since $(a)$ again holds, we can inductively build subsequences $(\cZ^{\ell^\smalltext{j}_m})_{m \in \N}$ such that $(\cZ^{\ell^{\smalltext{j}\smalltext{+}\smalltext{1}}_m})_{m \in \N}$ is a subsequence of $(\cZ^{\ell^\smalltext{j}_m})_{m \in \N}$ for each $j \in \N$ and
\begin{equation*}
	\lim_{m \rightarrow \infty}\E^{\P^{\smalltext{t}\smalltext{,}\smalltext{\omega}}}\bigg[\int_0^{(\tau_\smalltext{j}-t\land\tau_\smalltext{j})^{\smalltext{t}\smalltext{,}\smalltext{\omega}}}\big((\cZ^{\ell^\smalltext{j}_\smalltext{m}})^{t,\omega}_{t\smallertext{+}r} - \cZ^{t,\omega}_{t\smallertext{+}r}\big)^2 \d\langle X^{c}\rangle_r\bigg] = 0, \; \text{for $\P$--a.e. $\omega \in \Omega$,}
\end{equation*}
Using a diagonal argument, we therefore have that, for every $k \in \N$, the tail $(\cZ^{\ell^\smalltext{m}_\smalltext{m}})_{k \leq m \in \N}$ is a subsequence of $(\cZ^{\ell^\smalltext{k}_\smalltext{m}})_{m \in \N}$, and therefore
\begin{equation*}
	\lim_{m \rightarrow \infty}\E^{\P^{\smalltext{t}\smalltext{,}\smalltext{\omega}}}\bigg[\int_0^{(\tau_\smalltext{k}-t\land\tau_\smalltext{k})^{\smalltext{t}\smalltext{,}\smalltext{\omega}}}\big((\cZ^{\ell^\smalltext{m}_\smalltext{m}})^{t,\omega}_{t\smallertext{+}r} - \cZ^{t,\omega}_{t\smallertext{+}r}\big)^2 \d\langle X^{c}\rangle_r\bigg] = 0, \; k \in \N, \; \text{for $\omega \in \Omega\setminus\sN$.}
\end{equation*}
Note that $((\tau_k-t\land\tau_k)^{t,\omega})_{k \in \N}$ is  an $(\F,\P^{t,\omega})$--localising sequence for $\omega \in \Omega\setminus\sN$. To summarise, we found that
\begin{enumerate}
	\item[$(i^\prime)$] $\displaystyle \E^{\P^{\smalltext{t}\smalltext{,}\smalltext{\omega}}}\bigg[\int_0^{(\tau_\smalltext{k}-t\land\tau_\smalltext{k})^{\smalltext{t}\smalltext{,}\smalltext{\omega}}}| \cZ^{t,\omega}_{t\smallertext{+}r}|^2 \d\langle X^{c}\rangle_r\bigg] < \infty$, $k \in \N$, and $\displaystyle \lim_{m \rightarrow \infty}\E^{\P^{\smalltext{t}\smalltext{,}\smalltext{\omega}}}\bigg[\int_0^{(\tau_\smalltext{k}-t\land\tau_\smalltext{k})^{\smalltext{t}\smalltext{,}\smalltext{\omega}}}\big((\cZ^{\ell_\smalltext{m}})^{t,\omega}_{t\smallertext{+}r} - \cZ^{t,\omega}_{t\smallertext{+}r}\big)^2 \d\langle X^{c}\rangle_r\bigg] = 0$, $\omega \in \Omega\setminus\sN$;
	\item[$(ii^\prime)$] $X^{c,\P^{\smalltext{t}\smalltext{,}\smalltext{\omega}}}_{\cdot\land(\tau_\smalltext{k}-t\land\tau_\smalltext{k})^{\smalltext{t}\smalltext{,}\smalltext{\omega}}}$ is an $(\F,\P^{t,\omega})$--square-integrable martingale for $k \in \N$ and $\omega \in \Omega\setminus\sN$.
\end{enumerate}
Since $X^{c,\P}_{t\smallertext{+}\smallertext{\cdot}}(\omega\otimes_t\cdot) - X^{c,\P}_{t}(\omega) = X^{c,\P^{\smalltext{t}\smalltext{,}\smalltext{\omega}}}$, $\P^{t,\omega}$--a.s., for $\omega \in \Omega\setminus\sN$ (see \Cref{cor::shif_quadratic_variation_continuous_martingale_part}) implies 
\begin{equation*}
	\big(\cZ^{\ell^\smalltext{m}_\smalltext{m}}_{t\smallertext{+}\smallertext{\cdot}} \bcdot (X^{c,\P}_{t\smallertext{+}\smallertext{\cdot}} - X^{c,\P}_{t}(\omega))\big)(\omega\otimes_t\cdot) = (\cZ^{\ell^\smalltext{m}_\smalltext{m}})^{t,\omega}_{t\smallertext{+}\smallertext{\cdot}} \bcdot X^{c,\P^{\smalltext{t}\smalltext{,}\smalltext{\omega}}}, \; \textnormal{$\P^{t,\omega}$--a.s.}, \; \omega \in \Omega\setminus\sN,
\end{equation*}
it follows from \eqref{eq::convergence_ucp} and \cite[Theorem 6.2.2.(b)]{weizsaecker1990stochastic} that
\begin{equation*}
	\bigg(\int_{t}^{t\smallertext{+}\smallertext{\cdot}}\cZ_r \d X^{c,\P}_r\bigg)^{(\P)}(\omega\otimes_t\cdot) 
	\underset{\infty \leftarrow m}{\longleftarrow}
	\big(\cZ^{\ell^\smalltext{m}_\smalltext{m}}_{t\smallertext{+}\smallertext{\cdot}} \bcdot (X^{c,\P}_{t\smallertext{+}\smallertext{\cdot}} - X^{c,\P}_{t}(\omega))\big)(\omega\otimes_t\cdot) 
	=(\cZ^{\ell^\smalltext{m}_\smalltext{m}})^{t,\omega}_{t\smallertext{+}\smallertext{\cdot}} \bcdot X^{c,\P^{\smalltext{t}\smalltext{,}\smalltext{\omega}}} 
	\underset{m\to\infty}{\longrightarrow} \bigg(\int_0^\cdot\cZ^{t,\omega}_{t\smallertext{+}r} \d X^{c,\P^{\smalltext{t}\smalltext{,}\smalltext{\omega}}}_r\bigg)^{(\P^{\smalltext{t}\smalltext{,}\smalltext{\omega}})},
\end{equation*}
uniformly on compacts in $\P^{t,\omega}$-measure for each $\omega\in\Omega\setminus\sN$. We therefore find the desired equality
\begin{equation*}
	\bigg(\int_{t}^{t\smallertext{+}\smallertext{\cdot}}\cZ_r \d X^{c,\P}_r\bigg)^{(\P)}(\omega\otimes_t\cdot) 
	= \bigg(\int_0^\cdot\cZ^{t,\omega}_{t\smallertext{+}r} \d X^{c,\P^{\smalltext{t}\smalltext{,}\smalltext{\omega}}}_r\bigg)^{(\P^{\smalltext{t}\smalltext{,}\smalltext{\omega}})}, \; \text{$\P^{t,\omega}$--a.s.}, \; \omega \in \Omega\setminus\sN.
\end{equation*}

\medskip
Lastly, we show that the exceptional set outside of which \eqref{eq::conditioning_solution_bsde} holds can actually be chosen to belong to $\cF_{t-s}$. Denote by $\mathsf{g} : \Omega \times \fP(\Omega) \longrightarrow \R$ an arbitrary Borel-measurable extension of the Borel-measurable function
		\[
		\{(\omega^\prime,\P^\prime) : \omega^\prime \in \Omega, \; \P^\prime\in\fP(t,\omega^\prime)\} \ni (\omega^\prime,\P^\prime) \longmapsto \E^{\P^\smalltext{\prime}}[\cY^{t,\omega^\smalltext{\prime},\P^\smalltext{\prime}}_0((T-t\land T)^{t,\omega^\smalltext{\prime}},\xi^{t,\omega^\smalltext{\prime}})] \in \R.
		\]
		Here, Borel-measurability means with respect to the trace $\sigma$-algebra
		\[
			\cB(\Omega\times\fP(\Omega)) \cap \{(\omega^\prime,\P^\prime) : \omega^\prime \in \Omega, \; \P^\prime\in\fP(t,\omega^\prime)\};
		\]
		see the argument surrounding  \eqref{eq::picard_limit_value_map}. The map
		\[
		\mathsf{h} : \Omega \ni \omega \longmapsto (\bar{\omega}\otimes_s\omega_{\cdot\land(t-s)},\P^{t-s,\omega_{\smalltext{\cdot}\smalltext{\land}\smalltext{(}\smalltext{t}\smalltext{-}\smalltext{s}\smalltext{)}}}) \in \Omega \times \fP(\Omega),
		\]
		is Borel-measurable and, by Galmarino's test, even $\cF_{t-s}$-measurable. Note that we have $\P^{t-s,\omega_{\cdot \land (t-s)}} = \P^{t-s,\omega}$ identically, again by Galmarino's test. By {\rm\Cref{ass::probabilities2}}.$(ii)$, the map $\mathsf{h}$ takes values $\P$--a.s. in the analytic set $\{(\omega^\prime,\P^\prime) : \omega^\prime \in \Omega, \; \P^\prime \in \fP(t,\omega^\prime)\}$. Since analytic sets are universally measurable, the pre-image of $\{(\omega^\prime,\P^\prime) : \omega^\prime \in \Omega, \; \P^\prime \in \fP(t,\omega^\prime)\}$ under $\mathsf{h}$ is therefore $\cF_{t-s}$-universally measurable. Consequently, there exists an $\P$--null set $\sN_1 \in \cF_{t-s}$ outside of which $\mathsf{h}$ takes values in $\{(\omega^\prime,\P^\prime) : \omega^\prime \in \Omega, \; \P^\prime \in \fP(t,\omega)\}$. In particular, this yields
		\[
		(\mathsf{g} \circ \mathsf{h})(\omega)
		= \E^{\P^{\smalltext{t}\smalltext{-}\smalltext{s}\smalltext{,}\smalltext{\omega}}}\big[\cY^{t,\bar\omega\otimes_\smalltext{s}\omega,\P^{\smalltext{t}\smalltext{-}\smalltext{s}\smalltext{,}\smalltext{\omega}}}_0 ((T -t\land T)^{t,\bar\omega\otimes_\smalltext{s}\omega},\xi^{t,\bar\omega\otimes_\smalltext{s}\omega})\big], \omega \in \Omega\setminus \sN_1,
		\]
		and then
		\[
		\E^\P\big[\cY^{s,\bar\omega,\P}_{t-s}((T-s \land T)^{s,\bar\omega},\xi^{s,\bar\omega}) \big| \cF_{t\smallertext{-}s}\big]
		= (\mathsf{g} \circ \mathsf{h}), \; \textnormal{$\P$--a.s.}
		\]
		Since both sides are $\cF_{t-s}$-measurable, there exists a further $\P$--null set $\sN_2 \in \cF_{t-s}$ outside of which equality holds above. For $\omega \in \Omega\setminus(\sN_1 \cup \sN_2)$, we then have
		\[
		\E^\P\big[\cY^{s,\bar\omega,\P}_{t-s}((T-s \land T)^{s,\bar\omega},\xi^{s,\bar\omega}) \big| \cF_{t\smallertext{-}s}\big](\omega)
		= (\mathsf{g}\circ\mathsf{h})(\omega) 
		= \E^{\P^{\smalltext{t}\smalltext{-}\smalltext{s}\smalltext{,}\smalltext{\omega}}}\big[\cY^{t,\bar\omega\otimes_\smalltext{s}\omega,\P^{\smalltext{t}\smalltext{-}\smalltext{s}\smalltext{,}\smalltext{\omega}}}_0 ((T -t\land T)^{t,\bar\omega\otimes_\smalltext{s}\omega},\xi^{t,\bar\omega\otimes_\smalltext{s}\omega})\big].		
		\]
		Thus, the set $\sN \coloneqq \sN_1 \cup \sN_2 \in \cF_{t-s}$ satisfies the desired properties, which completes the proof.
\end{proof}

	\subsection{Auxiliary lemmata from Section \ref{sec::proof_regularisation}}\label{sec::proofs_lemmata_regularisation}

\begin{proof}[Proof of \Cref{lem::solv_bsde_cond}]
The first equality is simply the well-known flow property of solutions to (classical) BSDEs. A proof in our generality is straightforward; see, for example, the arguments in the proof of \cite[Proposition 2.5]{elkaroui1997backward}. We therefore immediately turn to the second equality. The $\F_\smallertext{+}$-optionality of the finite-variation process appearing in the BSDE prevents us from following the arguments in the proof of \cite[Lemma 2.7]{possamai2018stochastic}. 
More precisely, we cannot simply take the $\F$-optional projection of the BSDE and identify the orthogonal martingale parts, since in our case the finite variation process is not necessarily $\F$-adapted. 
We thus argue differently. But before that, let us note that for two arbitrary $\F_\smallertext{+}$--stopping times $S$ and $T$ and an integrable random variable $\zeta$, we have
\begin{equation*}
	\E^\P[\zeta| \cF_{S\smallertext{+}} \cap \cF_{T\smallertext{-}}] \1_{\{S < T\}} = \E^\P[\zeta|\cF_{S\smallertext{+}}] \1_{\{S < T\}}, \; \textnormal{$\P$--a.s.}
\end{equation*}
This follows from the $\cF_{S\smallertext{+}}\cap\cF_{T\smallertext{-}}$--measurability of $\{S<T\}$ and $\E^\P[\zeta|\cF_{S\smallertext{+}}] \1_{\{S < T\}}$ (see \cite[Theorem IV.56]{dellacherie1978probabilities}). Indeed, by the tower property
\begin{align*}
	\E^\P[\zeta| \cF_{S\smallertext{+}} \cap \cF_{T\smallertext{-}}] \1_{\{S < T\}} 
	&= \E^\P\big[\E^\P[\zeta | \cF_{S\smallertext{+}}] \big| \cF_{S\smallertext{+}} \cap \cF_{T\smallertext{-}}] \1_{\{S < T\}} \\
	&= \E^\P\big[\E^\P[\zeta | \cF_{S\smallertext{+}}] \1_{\{S < T\}} \big| \cF_{S\smallertext{+}} \cap \cF_{T\smallertext{-}}] 
	= \E^\P[\zeta | \cF_{S\smallertext{+}}] \1_{\{S < T\}}, \; \textnormal{$\P$--a.s.}
\end{align*}

We turn to the proof of the claim. For notational simplicity, we prove the result for $s = 0$, and thus $\P \in \fP_0$; however, analogous arguments apply to the general case. Let $\theta \coloneqq \tau \land T$. We write $(\cY, \cZ, \cU, \cN)$ and $(\cY^{\prime},\cZ^{\prime},\cU^{\prime},\cN^{\prime})$ for the solution of the BSDE with terminal condition $\xi_{\theta} \coloneqq \cY^\P_{\theta}(T,\xi)$ and $\xi^{\prime}_{\theta} \coloneqq \E^\P[\cY^\P_{\theta}(T,\xi)|\cF_{\theta}]$, respectively, at terminal time $\theta$. Let $\delta \cY \coloneqq \cY - \cY^{\prime}$ and $\delta\xi_{\theta} = \xi_{\theta} - \xi^{\prime}_{\theta}$, $\delta f^{\P} \coloneqq f^{\P}(\cY,\cY_{-},\cZ,\cU(\cdot)) - f^{\P}(\cY^\prime,\cY_{-}^\prime,\cZ^\prime,\cU^\prime(\cdot))$, and then
\begin{equation*}
	w \coloneqq \int_0^{\cdot \land \theta} \lambda_s \d C_s \; 
	\text{and} \; v \coloneqq \int_0^{\cdot \land \theta} \frac{\widehat\lambda^{\cY_{\smalltext{s}\tinytext{-}}, \cY^{\smalltext{\prime}}_{\smalltext{s}\tinytext{-}}}_s}{1-\widehat\lambda^{\cY_{\smalltext{s}\tinytext{-}}, \cY^{\smalltext{\prime}}_{\smalltext{s}\tinytext{-}}}_s \Delta C_s} \d C_s.
\end{equation*}
Here $\lambda$ is the process from \Cref{ass::crossing}.$(i)$ and $\widehat{\lambda}$ is defined through \eqref{eq::definition_lambda_hat_lambda}. 
Redefine $\cY^\prime$ beyond $\theta$ by
\[
	\cY^\prime_t \coloneqq \cY^\prime_t \1_{\{t \leq \theta\}} + \cY^\prime_\theta \bigg( \frac{\cE(\hat\beta A)_\theta} {\cE(\hat\beta A)_t} \bigg)^{1/2} \1_{\{t > \theta\}}, \; t \in [0,\infty].
\]
Then, recalling \eqref{eq::lipschitz_linearisation}, we apply \Cref{ass::crossing}.$(iii)$ to $(\cY^\prime,\cZ^\prime,\cU,\cU^\prime)$ and obtain
\begin{align*}
    \delta f^\P_s 
	\geq \lambda_s\delta\cY_s + \widehat\lambda^{\cY_{s\smallertext{-}},\cY^\prime_{s\smallertext{-}}}_s \delta\cY_{s\smallertext{-}} + \big(\eta^{\cZ_s,\cZ^\prime_s}_s\big)^\top \mathsf a_s\delta\cZ_s + \frac{\d\langle\rho\ast\tilde{\mu}^{X,\P},(\cU-\cU^\prime)\ast\tilde{\mu}^{X,\P}\rangle^{(\P)}_s}{\d C_s}, \; \textnormal{$\P\otimes\d C_s$--a.e. on $\llparenthesis0,\theta\rrbracket$.}
\end{align*}
Let
\[
	L \coloneqq \int_0^{\cdot \land \theta}\eta^{\cZ_\smalltext{s},\cZ^{\smalltext{\prime}}_\smalltext{s}}_s\d X^{c,\P}_s + \rho \ast\tilde\mu^{X,\P}_{\cdot \land \theta}.
\]
By \Cref{ass::crossing}.$(ii)$--$(iii)$, $\langle L\rangle^{(\F,\P)}_\theta$ is $\P$--essentially bounded and $\Delta L>-1$ outside some $\P$--null set. Hence 
\begin{equation}\label{eq::stoch_exp_measure_change}
	\frac{\d\overline{\P}}{\d\P} \coloneqq \cE(L)_{\theta} \coloneqq \cE\bigg(\int_0^{\cdot \land \theta}\eta^{\cZ_\smalltext{s},\cZ^{\smalltext{\prime}}_\smalltext{s}}_s\d X^{c,\P}_s + \rho \ast\tilde\mu^{X,\P}_{\cdot \land \theta}\bigg)_{\theta},
\end{equation}
defines a strictly positive and $\P$--square-integrable probability density (see \cite[Lemma 7.4]{possamai2024reflections}).
Given that
\begin{equation*}
	\cE(L) = 1 + \int_0^\cdot \cE(L)_{s\smallertext{-}} \d L_s, \; \text{$\P$--a.s.}, 
\end{equation*}
we choose with \Cref{prop::good_version_stochastic_integral} an $\F$-adapted version of $\cE(L)$. 
Note that $\cE(v)$ is nonnegative and $\P$--essentially bounded by \Cref{ass::crossing}.$(ii)$ since $\cE(v) \leq e^{v}$ (see \cite[Lemma 4.1]{cohen2012existence}) and $(1-\widehat\lambda^{\cY_{\smalltext{s}\tinytext{-}}, \cY^{\smalltext{\prime}}_{\smalltext{s}\tinytext{-}}}_s \Delta C_s)^{-1} \leq (1-\Phi)^{-1}$. 
Since $|\cE(w)|^2 \leq \cE(\hat\beta A)$ (see the proof of \Cref{prop::comparison}), we find
\[
	\E^{\P}\bigg[\sup_{t \in [0,\infty)}|\cE(w)_{t\land\theta}\delta\cY_{t\land\theta}|^2\bigg] 
	\leq \E^{\P}\bigg[\sup_{t \in [0,\theta]}|\cE(\hat\beta A)^{1/2}_{t}\delta\cY_{t}|^2\bigg] < \infty,
\]
and, by Doob's martingale inequality,
\[
	\E^{\P}\bigg[\sup_{t \in \D_\smalltext{+}}\big|\cE(w)_{t\land\theta}\E^\P[\delta\xi_{\theta}|\cF_{t\smallertext{+}}]\big|^2\bigg] 
	\leq \E^{\P}\bigg[\sup_{t \in \D_\smalltext{+}}\Big|\E^\P\big[|\cE(\hat\beta A)^{1/2}_{\theta}\delta\xi_{\theta}|\big|\cF_{t\smallertext{+}}\big]\Big|^2\bigg] 
	\leq 4 \E^\P\Big[ \big|\cE(\hat\beta A)^{1/2}_{\theta}\delta\xi_{\theta}\big|^2 \Big] < \infty.
\]
Following the arguments that lead to \cite[Equation 7.7]{possamai2024reflections}, we obtain, for $t\in[0,\infty)$ and $t^\prime \in \D_\smallertext{+}$ with $t\leq t^\prime$, $\P$--a.s.,
\begin{align}\label{eq::flow_property_conditional_expectation}
	&\cE(w)_{t \land \theta}\cE(v)_{t \land \theta}\delta \cY_{t \land \theta} \nonumber\\
	&\geq \E^{\bar{\P}}\big[ \cE(w)_{t^\smalltext{\prime} \land \theta}\cE(v)_{t^\smalltext{\prime} \land \theta}\delta\cY_{t^\smalltext{\prime} \land \theta} \big| \cF_{t\smallertext{+}}\big] \nonumber\\
	&=  \underbrace{\E^{\bar{\P}}\big[ \cE(w)_{t^\smalltext{\prime} \land \theta}\cE(v)_{t^\smalltext{\prime} \land \theta}\big(\delta \cY_{t^\smalltext{\prime} \land \theta} - \E^\P\big[\delta\xi_{\theta} \big| \cF_{t^\smalltext{\prime}\smallertext{+}}\big]\big) \big| \cF_{t\smallertext{+}} \big]}_{I^\smalltext{1}_{t,t^\smalltext{\prime}}\coloneqq} + \underbrace{\E^{\bar{\P}}\big[ \cE(w)_{t^\smalltext{\prime} \land \theta}\cE(v)_{t^\smalltext{\prime} \land \theta} \E^\P\big[\delta\xi_{\theta}\big|\cF_{t^\smalltext{\prime}\smallertext{+}}\big] \big| \cF_{t\smallertext{+}} \big]}_{I^\smalltext{2}_{t,t^\smalltext{^\prime}}\coloneqq }. \nonumber
\end{align}
We will now show that $\lim_{\D_\smallertext{+}\ni t^\prime\rightarrow\infty}I^1_{t,t^\prime} = 0$ in $\L^1(\overline{\P})$, and that $I^2_{t,t^\prime}\1_{\{t<\theta\}} = 0$ outside some $\P$--null set.
By following the arguments that lead to \cite[Equation 5.22]{possamai2024reflections}, we obtain
\begin{align*}
	\Big| \cE(w)_{t^\smalltext{\prime} \land \theta}\big(\delta \cY_{t^\smalltext{\prime} \land \theta} - \E^\P\big[\delta\xi_{\theta} \big| \cF_{t^\smalltext{\prime}\smallertext{+}}\big]\big) \Big| 
	&= \bigg|\cE(w)_{t^\smalltext{\prime} \land \theta} \E^\P\bigg[\int_{t^\smalltext{\prime}}^{\theta} \delta f^\P_r \d C_r \bigg| \cF_{t^\smalltext{\prime}\smallertext{+}} \bigg]\bigg| \\
	&\leq \bigg|\cE(\hat\beta A)^{1/2}_{t^\smalltext{\prime} \land \theta} \E^\P\bigg[\int_{t^\smalltext{\prime}}^{\theta} \delta f^\P_r \d C_r \bigg| \cF_{t^\smalltext{\prime}\smallertext{+}} \bigg]\bigg| \\
	&\leq \frac{1}{\hat\beta^{1/2}}\E^\P\bigg[\bigg(\int_{t^\smalltext{\prime}}^{\theta} \cE(\hat\beta A)_{r} \frac{|\delta f^\P_r|^2}{\alpha^2_r} \d C_r \bigg)^{1/2} \bigg| \cF_{t^\smalltext{\prime}\smallertext{+}} \bigg], \; \textnormal{$\P$--a.s.}, \; t^\prime \in [0,\infty).
\end{align*}
which then yields
\begin{align*}
	\E^{\bar{\P}}\bigg[ \Big| \cE(w)_{t^\smalltext{\prime} \land \theta}\big(\delta \cY_{t^\smalltext{\prime} \land \theta} - \E^\P\big[\delta\xi_{\theta} \big| \cF_{t^\smalltext{\prime}\smallertext{+}}\big]\big) \Big| \bigg] 
	\leq \frac{1}{\hat\beta^{1/2}} \E^\P\bigg[\bigg(\frac{\d\overline{\P}}{\d\P}\bigg)^2\bigg]^{1/2} \E^\P\bigg[\int_{t^\smalltext{\prime}}^{\theta} \cE(\hat\beta A)_{r} \frac{|\delta f^\P_r|^2}{\alpha^2_r} \d C_r \bigg]^{1/2} \xrightarrow{t^\prime \rightarrow \infty} 0,
\end{align*}
by the Lipschitz-continuity property of the generator and the integrability of the solutions, from which we deduce, together with the $\P$--essential boundedness of $\cE(v)$, that 
\[
	\lim_{\D_\smallertext{+}\ni t^\prime \rightarrow\infty}\E^{\overline{\P}}\big[|I^{1}_{t,t^\prime}|\big] = 0, \; t \in [0,\infty).
\]
For the other term $I^2_{t,t^\smalltext{\prime}}$, we have, $\P$--a.s.,
\begin{align*}
	I^2_{t,t^\smalltext{\prime}}\1_{\{t<\theta\}}
	&= \E^{\bar{\P}}\big[ \cE(w)_{t^\smalltext{\prime} \land \theta}\cE(v)_{t^\smalltext{\prime} \land \theta} \E^\P\big[\delta\xi_{\theta}\big|\cF_{t^\smalltext{\prime}\smallertext{+}}\big] \big| \cF_{t\smallertext{+}} \big] \1_{\{t<\theta\}}\\
	&= \frac{1}{\cE(L)_{t\land\theta}}\E^\P\big[ \cE(L)_{t^\smalltext{\prime}\land\theta} \cE(w)_{t^\smalltext{\prime}\land\theta}\cE(v)_{t^\smalltext{\prime}\land \theta}\E^\P\big[\delta\xi_{\theta}\big|\cF_{t^\smalltext{\prime}\smallertext{+}}\big] \big| \cF_{t\smallertext{+}}\big] \1_{\{t<\theta\}} \\
	&= \frac{1}{\cE(L)_{t\land\theta}}\E^\P\big[ \cE(L)_{t^\smalltext{\prime}\land\theta} \cE(w)_{t^\smalltext{\prime}\land\theta}\cE(v)_{t^\smalltext{\prime}\land \theta}\delta\xi_{\theta} \big| \cF_{t\smallertext{+}}\big] \1_{\{t<\theta\}} \\
	&= \frac{1}{\cE(L)_{t\land\theta}}\E^\P\big[ \cE(L)_{t^\smalltext{\prime}\land\theta} \cE(w)_{t^\smalltext{\prime} \land \theta}\cE(v)_{t^\smalltext{\prime} \land \theta}\delta\xi_{\theta} \big| \cF_{t\smallertext{+}} \cap \cF_{\theta\smallertext{-}}\big] \1_{\{t<\theta\}} \\
	&= \frac{1}{\cE(L)_{t\land\theta}}\E^\P\Big[ \E^\P\big[ \cE(L)_{t^\smalltext{\prime}\land\theta} \cE(w)_{t^\smalltext{\prime} \land \theta}\cE(v)_{t^\smalltext{\prime} \land \theta}\delta\xi_{\theta} \big|\cF_{\theta}\big] \Big| \cF_{t\smallertext{+}} \cap \cF_{\theta\smallertext{-}}\Big] \1_{\{t<\theta\}} \\
	&= \frac{1}{\cE(L)_{t\land\theta}}\E^\P\Big[ \cE(L)_{t^\smalltext{\prime}\land\theta} \cE(w)_{t^\smalltext{\prime} \land \theta}\cE(v)_{t^\smalltext{\prime} \land \theta} \E^\P\big[\delta\xi_{\theta} \big|\cF_{\theta}\big] \Big| \cF_{t\smallertext{+}} \cap \cF_{\theta\smallertext{-}}\Big] \1_{\{t<\theta\}} = 0.
\end{align*}
Here we used the fact $\theta$ is an $\F$--stopping time, that $\cF_{\theta\smallertext{-}} \subseteq \cF_{\theta}$, that $\cE(w)$ and $\cE(v)$ are $\F$-adapted by \eqref{eq::definition_lambda_hat_lambda} and \Cref{ass::crossing}.$(i)$, and that $\E^\P[\delta\xi_{\theta}|\cF_{\theta}] = 0$. Since both $\cE(w)$ and $\cE(v)$ are strictly positive by the fact that $\Phi < 1$ and \Cref{ass::crossing}.$(i)$, we conclude that $\delta\cY_{t \land \theta}\1_{\{t < \theta\}} \geq 0$, $\P$--a.s., for $t \in [0,\infty)$. A symmetric argument then also yields $\delta\cY_{t \land \theta}\1_{\{t < \theta\}} \leq 0$, $\P$--a.s., for $t \in [0,\infty)$. Right-continuity of the solution component then yields the desired result. This completes the proof.
\end{proof}

	\begin{proof}[Proof of \Cref{lem::stopping_value_function}]
		We follow the proofs of \cite[Theorem 2.4, pages 570--571]{possamai2018stochastic} and \cite[Theorem 2.3]{nutz2013constructing}. We will also repeatedly use the identity $(t-s) \land (T - s \land T)^{s,\bar{\omega}} = (t \land T - s \land T)^{s,\bar{\omega}}$ throughout this proof. From \Cref{thm::measurability2}, we immediately obtain
		\[
			\widehat\cY_{t}(T,\xi)^{s,\bar\omega}(\omega) = \widehat\cY_{t}(T,\xi)(\bar{\omega}\otimes_s\omega) = \widehat\cY_{t \land T(\bar{\omega}\otimes_\smalltext{s}\omega)}(T,\xi)(\bar{\omega}\otimes_s\omega) = \widehat\cY_{t \land T^{\smalltext{s}\smalltext{,}\smalltext{\bar\omega}}(\omega)}(T,\xi)^{s,\bar\omega}(\omega), \; \omega \in \Omega.
		\]
		The stated measurability can be argued as follows. Let $\theta \coloneqq (t\land T - s\land T)^{s,\bar\omega} = (t-s)\land(T-s\land T)^{s,\bar\omega}$, and let $\omega \in \Omega$ be arbitrary. We have
		\[
			\tilde{\omega} \coloneqq \bar{\omega}\otimes_s\omega_{\cdot\land\theta(\omega)} = \bar{\omega}\otimes_s\omega \; \textnormal{on $[0,t\land T(\bar{\omega}\otimes_s\omega)]$}.
		\]
		Thus $t \land T(\tilde{\omega}) = t \land T(\bar{\omega}\otimes_s\omega)$ by Galmarino's test (see \cite[Theorem IV.100(a)]{dellacherie1978probabilities}), and we then obtain
		\begin{align*}
			\widehat{\cY}_t(T,\xi)^{s,\bar\omega}(\omega_{\cdot\land\theta(\omega)}) 
			&= \widehat{\cY}_t(T,\xi)(\tilde{\omega})
			= \widehat{\cY}_{t \land T(\tilde{\omega})}(T,\xi)(\tilde{\omega}) 
			= \widehat{\cY}_{t \land T(\tilde{\omega})}(T,\xi)(\tilde{\omega}_{\cdot \land t\land T(\tilde{\omega})}) 
			= \widehat{\cY}_{t \land T(\tilde{\omega})}(T,\xi)((\bar{\omega}\otimes_s\omega)_{\cdot \land t\land T(\tilde{\omega})}) \\
			&= \widehat{\cY}_{t \land T(\tilde{\omega})}(T,\xi)(\bar{\omega}\otimes_s\omega)
			= \widehat{\cY}_{t \land T(\bar{\omega}\otimes_\smalltext{s}\omega)}(T,\xi)(\bar{\omega}\otimes_s\omega)
			= \widehat{\cY}_t(T,\xi)(\bar{\omega}\otimes_s\omega) = \widehat{\cY}_t(T,\xi)^{s,\bar{\omega}}(\omega).
		\end{align*}
		The stated measurability then follows from the arguments in the proof of \cite[Lemma 2.5]{nutz2013constructing}.
		
		\medskip
		We turn to the integrability \eqref{eq::integrability_y_hat_tau}. Fix some $\varepsilon \in (0,\infty)$. By \cite[Proposition 7.50]{bertsekas1978stochastic}, there exists a $\cF^\ast$-measurable map $\Q^\prime : \Omega \longrightarrow \fP(\Omega)$ such that 
		\begin{equation*}
			\Q^\prime(\omega) \in \fP(t,\omega)
			\;
			\text{and}
			\;
			\E^{\Q^\smalltext{\prime}(\omega)} \big[\cY^{t,\omega,{\Q^\smalltext{\prime}(\omega)}}_0((T-t\land T)^{t,\omega},\xi^{t,\omega})\big] 
			\geq \big(\widehat\cY_t(T,\xi)(\omega) - \varepsilon\big) \1_{\{\hat\cY_\smalltext{t}(T,\xi) < \infty\}}(\omega) + \frac{1}{\varepsilon}\1_{\{\hat\cY_\smalltext{t}(T,\xi) = \infty\}}(\omega),
		\end{equation*}
		for every $\omega \in \Omega$ with $\fP(t,\omega) \neq \varnothing$. We now fix $\bar\omega \in \Omega$ and $\P \in \fP(s,\bar\omega)$. The map $\widetilde\Q : \Omega \longrightarrow \fP(\Omega)$ given by $\widetilde\Q(\omega) \coloneqq \Q^\prime(\bar\omega\otimes_{s}{\omega_{\cdot \land (t-s)}})$ is $\cF^\ast_{t-s}$-measurable; see \cite[Proposition 7.44]{bertsekas1978stochastic} and the proof of \cite[Lemma 2.5]{nutz2013constructing} for the arguments. We now choose an $\cF_{t-s}$-measurable map $\Q : \Omega \longrightarrow \fP(\Omega)$ satisfying $\Q(\omega) = \widetilde\Q(\omega)$ for $\P$--a.e. $\omega \in \Omega$, see \cite[Lemma 1.27]{kallenberg2021foundations}. Since $\P^{t-s,\omega} \in \fP(t,\bar\omega\otimes_s\omega) \neq \varnothing$ for $\P$--a.e. $\omega \in \Omega$, this implies that $\Q(\omega) \in \fP(t,\bar\omega\otimes_s\omega)$ and
		\begin{equation*}
			\E^{\Q(\omega)} \big[\cY^{t,\bar\omega\otimes_s\omega,{\Q(\omega)}}_0((T-t\land T)^{t,\bar\omega\otimes_s\omega},\xi^{t,\bar\omega\otimes_s\omega})\big] 
			\geq \big(\widehat\cY_t(T,\xi)(\bar\omega\otimes_s\omega) - \varepsilon\big) \1_{\{\hat\cY_\smalltext{t}(T,\xi) < \infty\}}(\bar\omega\otimes_s\omega) + \frac{1}{\varepsilon}\1_{\{\hat\cY_\smalltext{t}(T,\xi) = \infty\}}(\bar\omega\otimes_s\omega),
		\end{equation*}
		for $\P$--a.e. $\omega \in \Omega$. Here we used $(\bar\omega\otimes_s\omega)\otimes_t\tilde\omega = (\bar\omega\otimes_s\omega_{\cdot\land(t-s)})\otimes_t\tilde\omega$ for every $\tilde\omega\in\Omega$. We now define
		\begin{equation*}
			\overline{\P}[A] \coloneqq  \iint \big(\1_A\big)^{t-s,\omega}(\omega^\prime)\Q(\omega;\d\omega^\prime)\P(\d\omega), \; A \in \cF.
		\end{equation*}
		Then $\overline\P \in \fP(s,\bar\omega)$ by \Cref{ass::probabilities2}.$(iii)$, $\overline\P = \P$ on $\cF_{t-s}$, and $\Q(\omega) = \bar\P^{t-s,\omega}$ for $\P$--a.e. $\omega \in \Omega$, which, together with \Cref{lem::conditioning_bsde2}, yields
		\begin{align*}
			\E^{\bar\P}\big[\cY^{s,\bar\omega,\bar\P}_{t-s}((T-s\land T)^{s,\bar\omega},\xi^{s,\bar\omega})\big|\cF_{t-s}\big](\omega)
			&= \E^{\Q(\omega)}\big[\cY^{t,\bar\omega\otimes_s\omega,\Q(\omega)}_0((T-t\land T)^{t,\bar\omega\otimes_s\omega},\xi^{t,\bar\omega\otimes_s\omega})\big] \\
			&\geq 
			\big(\widehat\cY_t(T,\xi)(\bar\omega\otimes_s\omega) - \varepsilon\big) \1_{\{\hat\cY_\smalltext{t}(T,\xi) < \infty\}}(\bar\omega\otimes_s\omega) + \frac{1}{\varepsilon}\1_{\{\hat\cY_\smalltext{t}(T,\xi) = \infty\}}(\bar\omega\otimes_s\omega),
		\end{align*}
		for $\P$--a.e. $\omega \in \Omega$. This implies
		\begin{equation*}
			\E^{\bar\P}\big[\cY^{s,\bar\omega,\bar\P}_{t-s}((T-s\land T)^{s,\bar\omega},\xi^{s,\bar\omega})\big|\cF_{t-s}\big] \land \varepsilon^{-1} \leq \widehat\cY_t(T,\xi)(\bar\omega\otimes_s\cdot) \land \varepsilon^{-1} \leq \E^{\bar\P}\big[\cY^{s,\bar\omega,\bar\P}_{t-s}((T-s\land T)^{s,\bar\omega},\xi^{s,\bar\omega})\big|\cF_{t-s}\big] + \varepsilon, \; \textnormal{$\P$--a.s.}
		\end{equation*}
		and then
		\begin{equation*}
			\big|\widehat\cY_t(T,\xi)(\bar\omega\otimes_s\cdot) \land \varepsilon^{-1}\big| \leq \E^{\bar\P}\Big[\big|\cY^{s,\bar\omega,\bar\P}_{t-s}((T-s\land T)^{s,\bar\omega},\xi^{s,\bar\omega})\big|\Big|\cF_{t-s}\Big] + \varepsilon, \; \textnormal{$\P$--a.s.}
		\end{equation*}
		Since $\cE(\hat\beta A^{s,\bar\omega}_{s\smallertext{+}\smallertext{\cdot}})_{(t-s) \land (T-s\land T)^{s,\bar\omega}} < \infty$, we have, using $(a+b)^2 \leq 2a^2 + 2b^2$ for real numbers $a$ and $b$, that
		\begin{align*}
			&\frac{1}{2}\big(\cE(\hat\beta A^{s,\bar\omega}_{s\smallertext{+}\smallertext{\cdot}})_{(t-s) \land (T-s\land T)^{s,\bar\omega}} \land \varepsilon^{-1}\big)\big|\widehat\cY_t(T,\xi)(\bar\omega\otimes_s\cdot) \land \varepsilon^{-1}\big|^2 \\
			& \leq \cE(\hat\beta A^{s,\bar\omega}_{s\smallertext{+}\smallertext{\cdot}})_{(t-s) \land (T-s\land T)^{s,\bar\omega}}\E^{\bar\P}\Big[\big|\cY^{s,\bar\omega,\bar\P}_{t-s}((T-s\land T)^{s,\bar\omega},\xi^{s,\bar\omega})\big|^2\Big|\cF_{t-s}\Big] + \varepsilon, \; \textnormal{$\P$--a.s.}
		\end{align*}
		and then
		\begin{align*}
			& \frac{1}{2}\E^\P\Big[ \big(\cE(\hat\beta A^{s,\bar\omega}_{s\smallertext{+}\smallertext{\cdot}})_{(t-s) \land (T-s\land T)^{s,\bar\omega}} \land \varepsilon^{-1}\big)\big|\widehat\cY_t(T,\xi)(\bar\omega\otimes_s\cdot) \land \varepsilon^{-1}\big|^2 \Big] \\
			& \leq \E^{\bar\P}\Big[ \cE(\hat\beta A^{s,\bar\omega}_{s\smallertext{+}\smallertext{\cdot}})_{(t-s) \land (T-s\land T)^{s,\bar\omega}}\big|\cY^{s,\bar\omega,\bar\P}_{t-s}((T-s\land T)^{s,\bar\omega},\xi^{s,\bar\omega})\big|^2\Big] + \varepsilon \\
			& \leq \sup_{\P^\smalltext{\prime}\in \fP(s,\bar\omega)}\E^{\P^\smalltext{\prime}}\Big[ \cE(\hat\beta A^{s,\bar\omega}_{s\smallertext{+}\smallertext{\cdot}})_{(t-s) \land (T-s\land T)^{s,\bar\omega}}\big|\cY^{s,\bar\omega,\P^\smalltext{\prime}}_{t-s}((T-s\land T)^{s,\bar\omega},\xi^{s,\bar\omega})\big|^2\Big] + \varepsilon \\
			& \leq \mathfrak{C} \sup_{\P^\smalltext{\prime} \in \fP(s,\bar{\omega})} \E^{\P^\smalltext{\prime}} \bigg[\cE(\hat\beta  A^{s,\bar{\omega}}_{s\smallertext{+}\smallertext{\cdot}})_{(T-s\land T)^{\smalltext{s}\smalltext{,}\smalltext{\bar{\omega}}}}|\xi^{s,\bar{\omega}}|^2 + \int_{t-s}^{(T-s\land T)^{\smalltext{s}\smalltext{,}\smalltext{\bar{\omega}}}} \cE(\hat\beta  A^{s,\bar{\omega}}_{s\smallertext{+}\smallertext{\cdot}} )_r \frac{|f^{s,\bar{\omega},\P^\smalltext{\prime}}_r(0,0,0,\mathbf{0})|^2}{|\alpha^{s,\bar{\omega}}_r|^2} \d (C^{s,\bar{\omega}}_{s\smallertext{+}\smallertext{\cdot}})_r\bigg] + \varepsilon,
		\end{align*}
		where $\mathfrak{C}$ is the constant depending only on $\hat{\beta}$ and $\Phi$ from \Cref{prop::stability}. Since the right-hand side is finite by \Cref{ass::generator2}, it follows from $\widehat{\cY}_{t\land T(\cdot)}(T,\xi)(\cdot) = \widehat{\cY}_{t}(T,\xi)(\cdot)$ and an application of Fatou's lemma for $\varepsilon = 1/n$ where $n \in \N^\star$ that
		\begin{align*}
			& \frac{1}{2}\E^\P\Big[ \cE(\hat\beta A^{s,\bar\omega}_{s\smallertext{+}\smallertext{\cdot}})_{(t-s) \land (T-s\land T)^{s,\bar\omega}}\big|\widehat\cY_{t\land T(\bar\omega\otimes_\smalltext{s}\cdot)}(T,\xi)(\bar\omega\otimes_s\cdot)\big|^2 \Big] \\
			& \leq \liminf_{n \rightarrow \infty} \frac{1}{2} \E^\P\Big[ \big(\cE(\hat\beta A^{s,\bar\omega}_{s\smallertext{+}\smallertext{\cdot}})_{(t-s) \land (T-s\land T)^{s,\bar\omega}} \land n\big)\big|\widehat\cY_t(T,\xi)(\bar\omega\otimes_s\cdot) \land n\big|^2 \Big] \\
			& \leq \mathfrak{C} \sup_{\P^\smalltext{\prime} \in \fP(s,\bar{\omega})} \E^{\P^\smalltext{\prime}} \bigg[\cE(\hat\beta  A^{s,\bar{\omega}}_{s\smallertext{+}\smallertext{\cdot}})_{(T-s\land T)^{\smalltext{s}\smalltext{,}\smalltext{\bar{\omega}}}}|\xi^{s,\bar{\omega}}|^2 + \int_{t-s}^{(T-s\land T)^{\smalltext{s}\smalltext{,}\smalltext{\bar{\omega}}}} \cE(\hat\beta  A^{s,\bar{\omega}}_{s\smallertext{+}\smallertext{\cdot}})_r \frac{|f^{s,\bar{\omega},\P^\smalltext{\prime}}_r(0,0,0,\mathbf{0})|^2}{|\alpha^{s,\bar{\omega}}_r|^2} \d (C^{s,\bar{\omega}}_{s\smallertext{+}\smallertext{\cdot}})_r\bigg] < \infty,
		\end{align*}
		which implies \eqref{eq::integrability_y_hat_tau} by \Cref{ass::generator2}.\ref{eq::shifted_integrability}.
		
		\medskip
		We turn to \eqref{eq::dpp_y_hat_sup}. The assertion is immediate if $s=t$ or $T(\bar{\omega})\leq s$. We therefore assume that $s<t$ and $T(\bar{\omega})>s$. We again fix $\varepsilon \in (0,\infty)$. Since
		\begin{equation*}
			\cE(\hat\beta A_{s\smallertext{+}\smallertext{\cdot}})^{1/2}_{t\land T - s\land T}(\omega) \widehat{\cY}_t(T,\xi)(\omega) 
			= \sup_{\P \in \fP(t,\omega)} 
			\cE(\hat\beta A_{s\smallertext{+}\smallertext{\cdot}})^{1/2}_{t\land T - s\land T}(\omega) \E^\P\big[\cY^{t,\omega,\P}_0((T-t\land T)^{t,\omega},\xi^{t,\omega})\big], \;  \omega \in \Omega,
		\end{equation*}
		we employ \cite[Proposition 7.50]{bertsekas1978stochastic} similar to before to find an $\cF^\ast$-measurable map $\Q^\prime:\Omega \longrightarrow \fP(\Omega)$ such that for every $\omega \in \Omega$ with $\fP(t,\omega) \neq 0$, we have $\Q^\prime(\omega) \in \fP(t,\omega)$ and
		\begin{align*}
			&\cE(\hat\beta A_{s\smallertext{+}\smallertext{\cdot}})^{1/2}_{t\land T - s\land T}(\omega) \E^{\Q^\smalltext{\prime}(\omega)}\big[\cY^{t,\omega,\Q^\smalltext{\prime}(\omega)}_0((T-t\land T)^{t,\omega},\xi^{t,\omega})\big] 
			\geq 
			\begin{cases}
				\cE(\hat\beta A_{s\smallertext{+}\smallertext{\cdot}})^{1/2}_{t\land T - s\land T}(\omega) \widehat{\cY}_t(T,\xi)(\omega) - \varepsilon,\; \textnormal{if this is finite,} \\[0.5em]
				\displaystyle \frac{1}{\varepsilon}, \; \textnormal{otherwise}.
			\end{cases}
		\end{align*}
		Fix $\bar\omega \in \Omega$ and $\P \in \fP(s,\bar\omega)$. As before, we choose an $\cF_{t-s}$-measurable map $\Q : \Omega \longrightarrow \fP(\Omega)$ satisfying $\Q(\omega) = \Q^\prime(\bar\omega\otimes_s\omega_{\cdot\land(t-s)})$ for $\P$--a.e. $\omega \in \Omega$. Since $\P^{t-s,\omega} \in \fP(t,\bar\omega\otimes_s\omega) = \fP(t,\bar\omega\otimes_s\omega_{\cdot\land(t-s)}) \neq \varnothing$ for $\P$--a.e. $\omega \in \Omega$, this implies that $\Q(\omega) \in \fP(t,\bar\omega\otimes_s\omega)$ and 
		\begin{align*}
			\cE(\hat\beta A^{s,\bar\omega}_{s\smallertext{+}\smallertext{\cdot}})^{1/2}_{(t\land T - s\land T)^{\smalltext{s}\smalltext{,}\smalltext{\bar\omega}}}(\omega)  \E^{\Q(\omega)} & \big[\cY^{t,\bar\omega\otimes_s\omega,{\Q(\omega)}}_0((T-t\land T)^{t,\bar\omega\otimes_s\omega},\xi^{t,\bar\omega\otimes_s\omega})\big] \\
			&\geq 
			\begin{cases}
				\cE(\hat\beta A^{s,\bar\omega}_{s\smallertext{+}\smallertext{\cdot}})^{1/2}_{(t\land T - s\land T)^{\smalltext{s}\smalltext{,}\smalltext{\bar\omega}}}(\omega) \widehat{\cY}_t(T,\xi)(\bar\omega\otimes_s\omega) - \varepsilon, \; \textnormal{if this is finite,} \\[0.5em]
				\displaystyle \frac{1}{\varepsilon}, \; \textnormal{otherwise},
			\end{cases}
		\end{align*}
		for $\P$--a.e. $\omega \in \Omega$. Here we used the fact that $\cE(\hat\beta A^{s,\bar\omega}_{s\smallertext{+}\smallertext{\cdot}})_{(t\land T - s\land T)^{\smalltext{s}\smalltext{,}\smalltext{\bar\omega}}}$ is $\cF_{t-s}$-measurable 
		and that $(\bar\omega\otimes_s\omega_{\cdot\land (t-s)})\otimes_t\tilde\omega = (\bar\omega\otimes_s\omega)\otimes_t\tilde\omega$ for every $\tilde\omega\in\Omega$.  
		By \eqref{eq::integrability_y_hat_tau} we know that we must be $\P$--a.s. in the first case. We now define
		\begin{equation*}
			\overline{\P}[A] \coloneqq  \iint \big(\1_A\big)^{t-s,\omega}(\omega^\prime)\Q(\omega;\d\omega^\prime)\P(\d\omega), \; A \in \cF.
		\end{equation*}
		As before, $\overline\P \in \fP(s,\bar\omega)$ by \Cref{ass::probabilities2}.$(iii)$, $\overline\P = \P$ on $\cF_{t-s}$, and $\bar\P^{t-s,\omega} = \Q(\omega)$ for $\P$--a.e. $\omega \in \Omega$, which, together with \Cref{lem::conditioning_bsde2}, yields
		\begin{align*}
			\E^{\bar\P}\big[\cY^{s,\bar\omega,\bar\P}_{t-s}((T-s\land T)^{s,\bar\omega},\xi^{s,\bar\omega})\big|\cF_{t-s}\big]
			&\geq 
			\widehat{\cY}_t(T,\xi)(\bar\omega\otimes_s\cdot) - \frac{\varepsilon}{\cE(\hat\beta A^{s,\bar\omega}_{s\smallertext{+}\smallertext{\cdot}})^{1/2}_{(t\land T - s\land T)^{\smalltext{s}\smalltext{,}\smalltext{\bar\omega}}}} \\
			&\geq 
			\E^{\bar\P}\big[\cY^{s,\bar\omega,\bar\P}_{t-s}((T-s\land T)^{s,\bar\omega},\xi^{s,\bar\omega})\big|\cF_{t-s}\big] - \frac{\varepsilon}{\cE(\hat\beta A^{s,\bar\omega}_{s\smallertext{+}\smallertext{\cdot}})^{1/2}_{(t\land T - s\land T)^{\smalltext{s}\smalltext{,}\smalltext{\bar\omega}}}}, \; \textnormal{$\P$--a.s.},
		\end{align*}
		and also $\overline{\P}$--almost surely. Therefore
		\begin{align*}
			\big|\widehat{\cY}_t(T,\xi)(\bar\omega\otimes_s\cdot) - \E^{\bar\P}\big[\cY^{s,\bar\omega,\bar\P}_{t-s}((T-s\land T)^{s,\bar\omega},\xi^{s,\bar\omega})\big|\cF_{t-s}\big]\big| 
			&=  \widehat{\cY}_t(T,\xi)(\bar\omega\otimes_s\cdot) - \E^{\bar\P}\big[\cY^{s,\bar\omega,\bar\P}_{t-s}((T-s\land T)^{s,\bar\omega},\xi^{s,\bar\omega})\big|\cF_{t-s}\big] \\
			&\leq \frac{\varepsilon}{\cE(\hat\beta A^{s,\bar\omega}_{s\smallertext{+}\smallertext{\cdot}})^{1/2}_{(t\land T - s\land T)^{\smalltext{s}\smalltext{,}\smalltext{\bar\omega}}}}, \; \textnormal{$\P$--a.s.},
		\end{align*}
		and also $\overline{\P}$--almost surely. Note that
		\begin{align*}
			\E^{\bar\P}\big[\cY^{s,\bar\omega,\bar\P}_{t-s}((T-s\land T)^{s,\bar\omega},\xi^{s,\bar\omega})\big|\cF_{t-s}\big]
			&= \E^{\bar\P}\big[\cY^{s,\bar\omega,\bar\P}_{(t\land T-s\land T)^{\smalltext{s}\smalltext{,}\smalltext{\bar\omega}}}((T-s\land T)^{s,\bar\omega},\xi^{s,\bar\omega})\big|\cF_{t-s}\big] \\
			&= \E^{\bar\P}\Big[ \E^{\bar\P}\big[ \cY^{s,\bar\omega,\bar\P}_{(t\land T-s\land T)^{\smalltext{s}\smalltext{,}\smalltext{\bar\omega}}}((T-s\land T)^{s,\bar\omega},\xi^{s,\bar\omega})\big| \cF_{(T-s\land T)^{\smalltext{s}\smalltext{,}\smalltext{\bar\omega}}}\big]\Big|\cF_{t-s}\Big] \\
			&= \E^{\bar\P}\big[ \cY^{s,\bar\omega,\bar\P}_{(t\land T-s\land T)^{\smalltext{s}\smalltext{,}\smalltext{\bar\omega}}}((T-s\land T)^{s,\bar\omega},\xi^{s,\bar\omega})\big| \cF_{(t \land T-s\land T)^{\smalltext{s}\smalltext{,}\smalltext{\bar\omega}}}\big], \; \textnormal{$\overline{\P}$--a.s.},
		\end{align*}
		holds by the $\cF_{(T-s\land T)^{\smalltext{s}\smalltext{,}\smalltext{\bar\omega}}}$-measurability of the solution of the BSDE, that is, $ \cY^{s,\bar\omega,\bar\P}_\sigma((T-s\land T)^{s,\bar\omega},\xi^{s,\bar\omega})$ is $\cF_{(T-s\land T)^{s,\bar\omega}}$-measurable for any $\F$--stopping time $\sigma$ (using \cite[Theorem~IV.56]{dellacherie1978probabilities}), and since, as before, $(t-s) \land (T-s\land T)^{s,\bar\omega} = (t\land T - s\land T)^{s,\bar\omega}$. The stability of solutions to BSDEs from \Cref{cor::stability} then implies that there exists a constant $\mathfrak{C}\in (0,\infty)$ which only depends on $\Phi$ and $\hat\beta$ such that
		\begin{align*}
			\E^{\bar\P}\Big[\big|\cY^{s,\bar\omega,\bar\P}_0 & ((t\land T - s\land T)^{s,\bar\omega},\widehat{\cY}_{t}(T,\xi)(\bar\omega\otimes_s\cdot)) \\
			&- \cY_0^{s,\bar\omega,\bar\P}\big((t\land T-s\land T)^{s,\bar\omega},\E^{\bar\P}\big[ \cY^{s,\bar\omega,\bar\P}_{(t\land T-s\land T)^{\smalltext{s}\smalltext{,}\smalltext{\bar\omega}}}((T-s\land T)^{s,\bar\omega},\xi^{s,\bar\omega})\big| \cF_{(t \land T-s\land T)^{\smalltext{s}\smalltext{,}\smalltext{\bar\omega}}}\big]\big)\big|^2\Big] \leq \varepsilon^2\mathfrak{C},
		\end{align*}
		which then yields		
		\begin{align*}
			\E^\P\big[&\cY^{s,\bar\omega,\P}_0((t\land T - s\land T)^{s,\bar\omega},\widehat{\cY}_{t}(T,\xi)(\bar\omega\otimes_s\cdot))\big] = \E^{\bar\P}\big[\cY^{s,\bar\omega,\bar\P}_0((t\land T - s\land T)^{s,\bar\omega},\widehat{\cY}_{t}(T,\xi)(\bar\omega\otimes_s\cdot))\big]\\ 
			&\leq \varepsilon \mathfrak{C}^{1/2} + \E^{\bar\P}\Big[ \cY_0^{s,\bar\omega,\bar\P}\big((t\land T-s\land T)^{s,\bar\omega},\E^{\bar\P}\big[ \cY^{s,\bar\omega,\bar\P}_{(t\land T-s\land T)^{\smalltext{s}\smalltext{,}\smalltext{\bar\omega}}}((T-s\land T)^{s,\bar\omega},\xi^{s,\bar\omega})\big| \cF_{(t \land T-s\land T)^{\smalltext{s}\smalltext{,}\smalltext{\bar\omega}}}\big]\big)\Big] \\
			&= \varepsilon \mathfrak{C}^{1/2} + \E^{\bar\P}\Big[ \cY_0^{s,\bar\omega,\bar\P}\big((t\land T-s\land T)^{s,\bar\omega}, \cY^{s,\bar\omega,\bar\P}_{(t\land T-s\land T)^{\smalltext{s}\smalltext{,}\smalltext{\bar\omega}}}((T-s\land T)^{s,\bar\omega},\xi^{s,\bar\omega})\big)\Big] \\
			&= \varepsilon \mathfrak{C}^{1/2} + \E^{\bar\P}\Big[\cY_0^{s,\bar\omega,\bar\P}\big((T-s\land T)^{s,\bar\omega},\xi^{s,\bar\omega})\Big] 
			\leq \varepsilon \mathfrak{C}^{1/2} + \widehat{\cY}_s(T,\xi)(\bar\omega).
		\end{align*}
		Here, we used \Cref{lem::solv_bsde_cond} in the second-to-last equality and the fact that $\P = \overline{\P}$ on $\cF_{t-s}$ in the first equality, ensuring that the solutions of the BSDEs with horizon $\leq t-s$ relative to $\P$ and $\overline{\P}$ must agree. Since $\P \in \fP(s,\bar\omega)$ and $\varepsilon \in (0,\infty)$ were arbitrary, we find
		\begin{equation}\label{eq::dpp_geq}
			\sup_{\P^\smalltext{\prime}\in\fP(s,\bar\omega)}\E^{\P^\smalltext{\prime}}\big[\cY^{s,\bar\omega,\P^\smalltext{\prime}}_0((t\land T - s\land T)^{s,\bar\omega},\widehat{\cY}_{t}(T,\xi)(\bar\omega\otimes_s\cdot))\big] 
			\leq \widehat{\cY}_s(T,\xi)(\bar\omega).
		\end{equation}
		Conversely, \Cref{lem::conditioning_bsde2} implies for any $\P\in\fP(s,\bar{\omega})$ that
		\begin{align*}
			\E^{\P}\big[ \cY^{s,\bar\omega,\P}_{(t\land T-s\land T)^{\smalltext{s}\smalltext{,}\smalltext{\bar\omega}}}((T-s\land T)^{s,\bar\omega},\xi^{s,\bar\omega})\big| \cF_{(t \land T-s\land T)^{\smalltext{s}\smalltext{,}\smalltext{\bar\omega}}}\big] 
			&= \E^\P\big[\cY^{s,\bar\omega,\P}_{t-s}((T-s \land T)^{s,\bar\omega},\xi^{s,\bar\omega}) \big| \cF_{t-s}\big] \\
			&= \E^{\P^{\smalltext{t}\smalltext{-}\smalltext{s}\smalltext{,}\smalltext{\cdot}}}\big[\cY^{t,\bar\omega\otimes_\smalltext{s}\cdot,\P^{\smalltext{t}\smalltext{-}\smalltext{s}\smalltext{,}\smalltext{\cdot}}}_0 ((T -t\land T)^{t,\bar\omega\otimes_\smalltext{s}\cdot},\xi^{t,\bar\omega\otimes_\smalltext{s}\cdot})\big] \\
			&\leq \widehat{\cY}_t(T,\xi)(\bar\omega\otimes_s\cdot), 
			\; \textnormal{$\P$--a.s.,}
		\end{align*}
		we find using \Cref{lem::solv_bsde_cond} again and the comparison principle for BSDEs from \Cref{prop::comparison} that
		\begin{align*}
			&\E^{\P}\Big[\cY_0^{s,\bar\omega,\P}\big((T-s\land T)^{s,\bar\omega},\xi^{s,\bar\omega})\Big] \\
			&= \E^\P\Big[\cY_0^{s,\bar\omega,\P}\big((t\land T-s\land T)^{s,\bar\omega},\E^{\P}\big[ \cY^{s,\bar\omega,\P}_{(t\land T\smallertext{-}s\land T)^{\smalltext{s}\smalltext{,}\smalltext{\bar\omega}}}((T-s\land T)^{s,\bar\omega},\xi^{s,\bar\omega})\big| \cF_{(t \land T\smallertext{-}s\land T)^{\smalltext{s}\smalltext{,}\smalltext{\bar\omega}}}\big]\big)\Big] \\
			&\leq \E^\P\big[\cY^{s,\bar\omega,\P}_0((t\land T - s\land T)^{s,\bar\omega},\,\widehat{\cY}_{t}(T,\xi)(\bar\omega\otimes_s\cdot))\big] \leq \sup_{\P^\smalltext{\prime} \in \fP(s,\bar\omega)} \E^{\P^\smalltext{\prime}}\big[\cY^{s,\bar\omega,\P^\smalltext{\prime}}_0((t\land T - s\land T)^{s,\bar\omega},\,\widehat{\cY}_{t}(T,\xi)(\bar\omega\otimes_s\cdot))\big].
		\end{align*}
		Since $\P \in \fP(s,\bar\omega)$ was arbitrary, this yields
		\begin{equation*}
			\widehat{\cY}_s(T,\xi)(\bar\omega) 
			\leq \sup_{\P^\smalltext{\prime}\in\fP(s,\bar\omega)}\E^{\P^\smalltext{\prime}}\big[\cY^{s,\bar\omega,\P^\smalltext{\prime}}_0((t\land T - s\land T)^{s,\bar\omega},\widehat{\cY}_{t}(T,\xi)(\bar\omega\otimes_s\cdot))\big],
		\end{equation*}
		which, together with \eqref{eq::dpp_geq}, yields the desired equality. This completes the proof.
	\end{proof}
	
\begin{proof}[Proof of \Cref{lem::linearising_bsde}]
	We note first that \Cref{ass::crossing}.$(ii)$ implies that $\langle\eta\bcdot X^{c,\P}\rangle^{(\P)}$, $\langle\eta^\prime\bcdot X^{c,\P}\rangle^{(\P)}$ and $\langle \rho\ast\tilde\mu^{X,\P}\rangle^{(\P)}$ are $\P$--essentially bounded. Moreover, since $\Delta(\rho\ast\tilde\mu^{X,\P}) > -1$, $\P$--a.s., the random variables $\d\mathscr{Q}/\d\P$ and $\d\mathcal{Q}/\d\P$ are well-defined and in $\L^2(\cF_T,\P)$ by \cite[Lemma~7.4]{possamai2024reflections}. In particular, they are strictly positive probability densities with respect to $\P$.
	By repeating the Girsanov and localisation arguments from the proof of \Cref{prop::comparison} with $\eta$ and then with $\eta^\prime$, we obtain
	\[
		\cY_t = \E^{\mathcal{Q}}[ \zeta | \cF_{t\smallertext{+}}]
		\; \textnormal{and} \; 
		\sY_t = \E^{\mathscr{Q}}[ \zeta | \cF_{t\smallertext{+}}], \; \textnormal{$\P$--a.s.}, \; t \in [0,\infty].
	\]
	Moreover, since $\eta^\top\mathsf{a} z\geq -\sqrt{\theta^X}\|\mathsf{a}^{1/2}z\|$ the same comparison argument gives $\sY\leq\cY$ outside some $\P$--null set, which completes the proof.
\end{proof}

	\section{Stability and comparison of solutions to BSDEs}\label{sec_stability}
	
	In this section, we discuss the stability and comparison of solutions to our BSDEs. To ensure that this section is self-contained, we depart from the assumptions made in the rest of the manuscript and adopt the framework in \cite{possamai2024reflections}. Consider an arbitrary probability space $(\Omega, \cG, \P)$. Throughout this section, we fix once and for all the data $(X,\mu,\G,T,\xi,f,C)$, where
	\begin{enumerate}[{\bf(D1)}, leftmargin=1cm]
		\item \label{data::filtration} $\G = (\cG_t)_{t \in [0,\infty)}$ is a filtration on $(\Omega,\cG,\P)$ with the convention $\cG_{0\smallertext{-}} \coloneqq \{\varnothing,\Omega\}$;
		\item \label{data::martingale} $X = (X_t)_{t \in [0,\infty)}$ is an $\R^d$-valued process whose components are right-continuous, $\G_\smallertext{+}$-adapted, $\G_\smallertext{+}$--locally square-integrable martingales starting at zero, $\mu$ is a $\G_\smallertext{+}$--integer-valued random measure on $[0,\infty) \times E$, where $(E,\cE)$ is some Blackwell space,\footnote{See \cite[Definition III.24]{dellacherie1978probabilities}. However, for simplicity, one can think of $E$ as a Polish space with its Borel $\sigma$-algebra $\cE = \cB(E)$.} and $M_\mu[\Delta X^i | \widetilde\cP(\G)] = 0$ for each $i \in \{1,\ldots,d\}$;
		\item $C = (C_t)_{t \in [0,\infty)}$ is a real-valued, right-continuous and $\P$--a.s. non-decreasing,\footnote{See also \cite[Lemma 6.5.10]{weizsaecker1990stochastic}.}  $\G$-predictable process starting at zero satisfying 
		\begin{equation*}
			\d\langle X\rangle_s = \pi_s\d C_s, \; \text{and} \; \nu(\,\cdot\,;\d s, \d x) = K_{\cdot,s}(\d x)\d C_s, \; \text{$\P$--a.s.},
		\end{equation*}
		where $\pi=(\pi_t)_{t \in [0,\infty)}$ is a $\G$-predictable process with values in $\S^d_\smallertext{+}$, $\nu$ is the $(\G,\P)$--predictable compensator of $\mu$, and $K$ is a transition kernel on $(E,\cE)$ given $(\Omega \times [0,\infty),\cP(\G))$;
		\item $T$ is a $\G_\smallertext{+}$--stopping time;
		\item \label{data::expectation_sup}$\xi$ is a real-valued, $\cG_{T\smallertext{+}}$-measurable random variable satisfying $\E[|\xi|^2] < \infty;$
		\item \label{data::generator}$f : \bigsqcup_{(\omega,t) \in \Omega \times [0,\infty)} \big(\R \times \R \times \R^d \times \widehat{\L}^2(K_{\omega,t})\big) \longrightarrow \R$ is such that for each $(y,\mathrm{y},z) \in \R \times \R \times \R^d$ and $u \in \H^2_T(\mu)$, the map		
		\begin{equation*}
			\Omega \times [0,\infty) \ni (\omega,t) \longmapsto f_t(\omega, y,\mathrm{y},z, u_t(\omega; \cdot) ) \in \R,
		\end{equation*}
		is $\G_\smallertext{+}$-progressive\footnote{Although we assumed optional measurability in \cite[Assumption \textbf{(D6)}]{possamai2024reflections}, we could have instead assumed progressive measurability.} and $f$ is $(r,\mathrm{r},\theta^X,\theta^\mu)$--Lipschitz-continuous on $\llparenthesis 0, T \rrbracket \coloneqq \{(\omega,t) \in \Omega \times (0,\infty) : t \leq T(\omega) \}$ in the sense that
		\begin{align*}
			& |f_t(\omega, y,\mathrm{y},z, u_t(\omega; \cdot)) - f_t(\omega, y^\prime,\mathrm{y}^\prime,z^\prime,u_t^\prime(\omega; \cdot))\big|^2 \\
			&\leq r_t(\omega) |y-y^\prime|^2 + \mathrm{r}_t(\omega) |\mathrm{y}-\mathrm{y}^\prime|^2 + \theta^X_t(\omega) (z-z^\prime)^\top \pi_t(\omega)(z-z^\prime) 
			+ \theta^\mu_t(\omega) \|u_t(\omega;\cdot) - u^\prime_t(\omega;\cdot)\|^2_{\hat{\L}^\smalltext{2}_{\smalltext{\omega}\smalltext{,}\smalltext{t}}(K_{\smalltext{\omega}\smalltext{,}\smalltext{t}})} ,
		\end{align*}
		for $\P \otimes \mathrm{d}C$--a.e. $(\omega,t) \in \llparenthesis 0, T \rrbracket$, where $r = (r_t)_{t \in [0,\infty)}$, $\mathrm{r} = (\mathrm{r}_t)_{t \in [0,\infty)}$, $\theta^X = (\theta^X_t)_{t \in [0,\infty)}$ and $\theta^\mu =(\theta^\mu_t)_{t \in [0,\infty)}$ are $[0,\infty)$-valued, $\G$-predictable processes;
		\item \label{data::integ_f0} the process $f_\cdot(0,0,0,\mathbf{0})$ satisfies
		\begin{equation*}
			\E^\P \Bigg[\bigg(\int_0^T |f_s(0,0,0,\mathbf{0})| \d C_s \bigg)^2 \Bigg] < \infty;
		\end{equation*}
		\item \label{data::process_A} the nonnegative, $\G$-predictable process $\alpha = (\alpha_t)_{t \in [0,\infty)}$ defined through $\alpha^2_t = \max\{\sqrt{r_t},\sqrt{\mathrm{r}_t},\theta^X_t,\theta^\mu_t\}$ satisfies $\alpha_t(\omega) > 0$ for $\P \otimes \mathrm{d}C$--a.e. $(\omega,t) \in \llparenthesis 0,T\rrbracket$, and the $\G$-predictable process $A = (A_t)_{t \in [0,\infty)}$ defined by $A_t \coloneqq \int_0^{t \land T} \alpha^2_s \d C_s$ is real-valued and satisfies $\Delta A \leq \Phi$, up to $\P$-evanescence, for some $\Phi \in [0,\infty)$.
	\end{enumerate}
	
	A solution $(\cY,\cZ,\cU,\cN)$ to the BSDE with generator $f$ (even without the Lipschitz-continuity assumption) and terminal condition $\xi$ is always understood to satisfy the conditions:
	\begin{enumerate}[{\bf(B1)}, leftmargin=1cm]
		\item $(\cZ,\cU,\cN) \in \H^2_T(X;\G,\P) \times \H^2_T(\mu;\G,\P) \times \cH^{2,\perp}_T(X,\mu;\G_\smallertext{+},\P)$$;$
		\item $\cY = (\cY_t)_{t \in [0,\infty]}$ is a real-valued, right-continuous and $\P$--a.s. c\`adl\`ag, $\G_\smallertext{+}$-optional process satisfying
		\begin{equation*}
			\E^\P\bigg[\int_0^T |f_s\big(\cY_s,\cY_{s\smallertext{-}},\cZ_s,\cU_s(\cdot)\big)\big|\d C_s\bigg] < \infty,
		\end{equation*}
		and
		\begin{equation*}
			\cY_t = \xi + \int_t^T f_s\big(\cY_s,\cY_{s\smallertext{-}},\cZ_s,\cU_s(\cdot)\big)\d C_s - \int_t^T \cZ_s \d X_s - \int_t^T \cU_s(x)\tilde\mu(\d s, \d x) - \int_t^T \d \cN_s, \; t \in [0,\infty], \; \text{$\P$--a.s.}
		\end{equation*}
	\end{enumerate}
	
	\medskip
	To state the stability and comparison result, we define, for $\beta \in (0,\infty)$, the constants
	\begin{align*}\label{eq::tilde_pi_psi_1} 
		\widetilde M_1^\Phi(\beta) \coloneqq \ff^\Phi(\beta) + \frac{1}{\beta} + \max\bigg\{1,\frac{(1+\beta\Phi)}\beta\bigg\}\bigg(\frac{1}{\beta} + \beta \fg^\Phi(\beta)\bigg),\;	\widetilde M_2^\Phi(\beta) \coloneqq \ff^\Phi(\beta) + \bigg(\frac{1}{\beta} + \beta \fg^\Phi(\beta)\bigg),
	\end{align*}
	\begin{equation*}
		\widetilde M_3^\Phi(\beta) \coloneqq \frac{1}{\beta} + \max\bigg\{1,\frac{(1+\beta\Phi)}\beta\bigg\}\bigg(\frac{1}{\beta} + \beta \fg^\Phi(\beta)\bigg),
	\end{equation*}
	where $\ff^\Phi$ and $\fg^\Phi$ are defined in \eqref{eq_frak_f_g}.
	
	\begin{definition}
		The pair $(\xi, f)$ is standard data for $\hat\beta \in (0,\infty)$ if
		\begin{equation*}
			\E^\P\big[|\cE(\hat\beta A)^{1/2}_T\xi|^2\big] + \E^\P\bigg[\int_0^T \cE(\hat\beta A)_s \frac{|f_s(0,0,0,\mathbf{0})|^2}{\alpha^2_s}\d C_s\bigg] < \infty.
		\end{equation*}
	\end{definition}
	
	By \cite[Theorem 3.7 and Remark 3.8]{possamai2024reflections}, the assumption that $\widetilde{M}_1^\Phi(\hat\beta) < 1$ for $\hat\beta \in (0,\infty)$ ensures that the BSDE with standard data $(\xi,f)$ for $\hat\beta$ is well-posed. Similarly, if $f$ does not depend on the $\mathrm{y}$-variable, then $\widetilde{M}_2^\Phi(\hat\beta) < 1$ ensures well-posedness, and if $f$ does not depend on the $y$-variable, then $\widetilde{M}_3^\Phi(\hat\beta) < 1$ ensures well-posedness.
		
	\subsection{Stability}
	
	Suppose now that we are given another terminal condition $\xi^\prime$ satisfying \ref{data::expectation_sup} and another generator $f^\prime$ satisfying \ref{data::generator}--\ref{data::integ_f0}, with the same Lipschitz-continuity coefficients $(r,\mathrm{r},\theta^X,\theta^\mu)$ as $f$. We denote by $(\cY,\cZ,\cU,\cN)$ the solution to the BSDE with generator $f$ and terminal condition $\xi$, and by $(\cY^\prime,\cZ^\prime,\cU^\prime,\cN^\prime)$ the solution to the BSDE with generator $f^\prime$ and terminal condition $\xi^\prime$, in case the BSDEs are well-posed in the sense of the results established in \cite[Section 3.2]{possamai2024reflections}. Let
	\begin{equation*}
		\delta\cY \coloneqq \cY - \cY^\prime, \; \delta\cZ \coloneqq \cZ - \cZ^\prime, \; \delta\cU \coloneqq \cU - \cU^\prime, \; \delta\cN \coloneqq \cN - \cN^\prime,
	\end{equation*}
	\begin{equation*}
		\delta\xi \coloneqq \xi - \xi^\prime, \; \delta_1 f \coloneqq f\big(\cY,\cY_\smallertext{-},\cZ,\cU(\cdot)\big) - f\big(\cY^\prime,\cY^\prime_\smallertext{-},\cZ^\prime,\cU^\prime(\cdot)\big), \; \delta_2 f \coloneqq f\big(\cY^\prime,\cY^\prime_\smallertext{-},\cZ^\prime,\cU^\prime(\cdot)\big) - f^\prime\big(\cY^\prime,\cY^\prime_\smallertext{-},\cZ^\prime,\cU^\prime(\cdot)\big).
	\end{equation*}
	
	We obtain the following stability result.
	
	\begin{proposition}\label{prop::stability}
		Suppose that both $(\xi, f)$ and $(\xi^\prime, f^\prime)$ are standard data for $\beta \in (0,\infty)$, and let $\sigma$ be a $\G_\smallertext{+}$--stopping time.

\medskip
$(i)$ If $\widetilde{M}^\Phi_1(\beta) < 1$, then there exists a constant $\mathfrak{C} \in (0,\infty)$ that depends only on $\beta$ and on $\Phi$ such that
			\begin{align*}
				&\cE(\beta A)_{\sigma\land T}|\delta\cY_\sigma|^2 + \E^\P\bigg[\int_\sigma^T\cE(\beta A)_s|\delta\cY_s|^2\d A_s +\int_\sigma^T \cE(\beta A)_s|\delta \cY_{s\smallertext{-}}|^2\d A_s \bigg|\cG_{\sigma\smallertext{+}}\bigg]  \\
				&\quad + \E^\P\bigg[ \int_\sigma^T \cE(\beta A)_s (\delta\cZ_s)^\top\pi_s\delta\cZ_s \d C_s
				+ \int_\sigma^T \cE(\beta A)_s \|\delta\cU_s(\cdot)\|^2_{\hat{\L}^\smalltext{2}_{\smalltext{\cdot}\smalltext{,}\smalltext{s}}(K_{\smalltext{\cdot}\smalltext{,}\smalltext{s}})} \d C_s 
				+ \int_\sigma^T \cE(\beta A)_s \d \langle\delta\cN\rangle_s\bigg|\cG_{\sigma\smallertext{+}}\bigg]\\
				& \leq \mathfrak{C}\Bigg(\E^\P\bigg[\cE(\beta A)_T|\delta\xi|^2 + \int_\sigma^T \cE(\beta A)_s \frac{|\delta_2 f_s|^2}{\alpha^2_s}\d C_s \bigg| \cG_{\sigma\smallertext{+}}\bigg] \Bigg), \; \textnormal{$\P$--a.s.}
			\end{align*}
			
$(ii)$ If $\widetilde{M}^\Phi_2(\beta) < 1$ and both $f$ and $f^\prime$ do not depend on the $\mathrm{y}$-variable, then there exists a constant $\mathfrak{C} \in (0,\infty)$ that depends only on $\beta$ and on $\Phi$ such that
			\begin{align*}
				&\E^\P\bigg[\int_\sigma^T\cE(\beta A)_s|\delta\cY_s|^2\d A_s+\int_\sigma^T \cE(\beta A)_s (\delta\cZ_s)^\top\pi_s\delta\cZ_s \d C_s +\int_\sigma^T \cE(\beta A)_s \|\delta\cU_s(\cdot)\|^2_{\hat{\L}^\smalltext{2}_{\smalltext{\cdot}\smalltext{,}\smalltext{s}}(K_{\smalltext{\cdot}\smalltext{,}\smalltext{s}})} \d C_s +\int_\sigma^T \cE(\beta A)_s \d \langle\delta\cN\rangle_s\bigg|\cG_{\sigma\smallertext{+}}\bigg]\\
				&\quad \leq \mathfrak{C}\Bigg(\E^\P\bigg[\cE(\beta A)_T|\delta\xi|^2 +\int_\sigma^T \cE(\beta A)_s \frac{|\delta_2 f_s|^2}{\alpha^2_s}\d C_s \bigg| \cG_{\sigma\smallertext{+}}\bigg] \Bigg), \; \textnormal{$\P$--a.s.}
			\end{align*}
			
$(iii)$ If $\widetilde{M}^\Phi_3(\beta) < 1$ and both $f$ and $f^\prime$ do not depend on the $y$-variable, then there exists a constant $\mathfrak{C} \in (0,\infty)$ that depends only on $\beta$ and on $\Phi$ such that
			\begin{align*}
				&\cE(\beta A)_{\sigma\land T}|\delta\cY_\sigma|^2 + \E^\P\bigg[\int_\sigma^T\cE(\beta A)_s|\delta\cY_{s\smallertext{-}}|^2\d A_s + \int_\sigma^T \cE(\beta A)_s (\delta\cZ_s)^\top\pi_s\delta\cZ_s \d C_s \bigg|\cG_{\sigma\smallertext{+}}\bigg] \\
				&\quad+ \E^\P\bigg[\int_\sigma^T \cE(\beta A)_s \|\delta\cU_s(\cdot)\|^2_{\hat{\L}^\smalltext{2}_{\smalltext{\cdot}\smalltext{,}\smalltext{s}}(K_{\smalltext{\cdot}\smalltext{,}\smalltext{s}})} \d C_s + \int_\sigma^T \cE(\beta A)_s \d \langle\delta\cN\rangle_s\bigg|\cG_{\sigma\smallertext{+}}\bigg]\\
				& \leq \mathfrak{C}\Bigg(\E^\P\bigg[\cE(\beta A)_T|\delta\xi|^2 +\int_\sigma^T \cE(\beta A)_s \frac{|\delta_2 f_s|^2}{\alpha^2_s}\d C_s \bigg| \cG_{\sigma\smallertext{+}}\bigg] \Bigg), \; \textnormal{$\P$--a.s.}
			\end{align*}
	\end{proposition}
	
	\begin{proof}
		We only prove $(i)$, as an analogous argument yields the bounds in $(ii)$ and $(iii)$. Moreover, we suppose, without loss of generality, that $\sigma = \sigma \land T$. Since all random variables considered below are $\mathcal{G}_{T\smallertext{+}}$-measurable, this additional assumption is harmless. Let
		\begin{equation*}
			\delta\eta \coloneqq (\cZ-\cZ^\prime) \bcdot X +  \big(\cU-\cU^\prime\big)\ast\tilde{\mu} + (\cN-\cN^\prime).
		\end{equation*}
		We fix a $\G_\smallertext{+}$--stopping time $S$ and define $\delta f \coloneqq f\big(\cY,\cY_\smallertext{-},\cZ,\cU(\cdot)\big) - f^\prime\big(\cY^\prime,\cY^\prime_\smallertext{-},\cZ^\prime,\cU^\prime(\cdot)\big) = \delta_1 f + \delta_2 f$. Note that
		\begin{equation}\label{eq::representation_delta_y}
			\delta\cY_S = \delta\xi + \int_S^T \delta f_s \d C_s - \int_S^T \d\delta\eta_s \; \text{and} \;
			\delta\cY_S = \E^\P\bigg[\delta\xi + \int_S^T \delta f_s \d C_s \bigg| \cG_{S\smallertext{+}}\bigg], \; \text{$\P$--a.s.},
		\end{equation}
		yields
		\begin{equation*}
			(\delta\cY_S)^2 + 2(\delta\cY_S)\int_S^T\d\delta\eta_s + \bigg(\int_S^T\d\delta\eta_s\bigg)^2 = \bigg(\delta\cY_S + \int_S^T \d \delta\eta_s\bigg)^2 = \bigg(\delta\xi + \int_S^T \delta f_s \d C_s\bigg)^2, \; \text{$\P$--a.s.}
		\end{equation*}
		By taking conditional expectation with respect to $\cG_{S\smallertext{+}}$, we obtain
		\begin{equation*}
			(\delta\cY_S)^2 + \E^\P\bigg[\int_S^T\d\langle\delta\eta\rangle_s\bigg|\cG_{S\smallertext{+}}\bigg] = \E^\P\bigg[\bigg(\delta\xi + \int_S^T \delta f_s \d C_s\bigg)^2\bigg|\cG_{S\smallertext{+}}\bigg], \; \text{$\P$--a.s.}
		\end{equation*}
		Let $\varpi \in (0,\infty)$ be arbitrary. Using $(a+b)^2 = a^2 + 2ab + b^2 \leq a^2 + \varpi a^2 + b^2/\varpi + b^2 = (1+\varpi)a^2 + (1+1/\varpi)b^2$ for real numbers $a$ and $b$, we obtain
		\begin{equation}\label{eq::stability_conditional_expectation}
			(\delta\cY_S)^2 + \E^\P\bigg[\int_S^T\d\langle\delta\eta\rangle_s\bigg|\cG_{S\smallertext{+}}\bigg] 
			\leq (1+\varpi)\E^\P\big[|\delta\xi|^2 \big| \cG_{S\smallertext{+}}\big] + \bigg(1+\frac{1}{\varpi}\bigg)\E^\P\bigg[\bigg(\int_S^T \delta f_s \d C_s\bigg)^2\bigg|\cG_{S\smallertext{+}}\bigg], \; \text{$\P$--a.s.},
		\end{equation}
		and in case $S$ is a $\G$-predictable stopping time, we obtain
		\begin{equation}\label{eq::stability_conditional_expectation2}
			(\delta\cY_{S\smallertext{-}})^2 + \E^\P\bigg[\int_{S\smallertext{-}}^T\d\langle\delta\eta\rangle_s\bigg|\cG_{S\smallertext{-}}\bigg] \leq (1+\varpi)\E^\P\big[|\delta\xi|^2 \big|\cG_{S\smallertext{-}}\big] + \bigg(1+\frac{1}{\varpi}\bigg)\E^\P\bigg[\bigg(\int_{S\smallertext{-}}^T \delta f_s \d C_s\bigg)^2\bigg|\cG_{S\smallertext{-}}\bigg], \; \text{$\P$--a.s.}
		\end{equation}
		We can now proceed as in the proof of \cite[Proposition 5.4]{possamai2024reflections}. First, using
		\[
			\cE(\beta A) = 1+ \int_0^\cdot \beta\cE(\beta A)_{s-}\d A_s
		\]
		and Tonelli's theorem, we obtain
		\begin{align*}
			\E^\P\bigg[\int_\sigma^T \cE(\beta A)_s \d \langle\delta\eta\rangle_s\bigg|\cG_{\sigma\smallertext{+}}\bigg]
			&= \cE(\beta A)_{\sigma \land T} \E^\P[\langle\delta\eta\rangle_T - \langle\delta\eta\rangle_{\sigma}|\cG_{\sigma\smallertext{+}}] 
			+ \beta \E^\P\bigg[\int_\sigma^T \cE(\beta A)_{t\smallertext{-}}\int_{t\smallertext{-}}^T \d\langle\delta\eta\rangle_s\d A_t\bigg| \cG_{\sigma\smallertext{+}}\bigg], \; \textnormal{$\P$--a.s.}
		\end{align*}
		We then find for an arbitrary $\gamma \in (0,\beta)$
		\begin{align*}
			&\E^\P\bigg[\int_\sigma^T \cE(\beta A)_{t\smallertext{-}}\int_{t\smallertext{-}}^T \d\langle\delta\eta\rangle_s\d A_t\bigg| \cG_{\sigma\smallertext{+}}\bigg] \\
			&\leq (1+\varpi)\E^\P\bigg[\int_\sigma^T \cE(\beta A)_{t\smallertext{-}}|\delta\xi|^2\d A_t \bigg| \cG_{\sigma\smallertext{+}}\bigg] 
			+ \bigg(1+\frac{1}{\varpi}\bigg) \E^\P\bigg[\int_\sigma^T \cE(\beta A)_{t\smallertext{-}}\bigg(\int_{t\smallertext{-}}^T |\delta f_s|\d C_s\bigg)^2\d A_t \bigg| \cG_{\sigma\smallertext{+}} \bigg] \\
			&\quad - \E^\P\bigg[\int_\sigma^T \cE(\beta A)_{t\smallertext{-}}|\delta \cY_{t\smallertext{-}}|^2\d A_t \bigg| \cG_{\sigma\smallertext{+}} \bigg] \\
			&\leq \frac{(1+\varpi)}{\beta} \E^\P\big[ \cE(\beta A)_{T}|\delta\xi|^2 \big| \cG_{\sigma\smallertext{+}}\big] 
			+ \bigg(1+\frac{1}{\varpi}\bigg) \frac{(1+\gamma\Phi)}{\gamma(\beta-\gamma)} \E^\P\bigg[ \int_\sigma^T \cE(\beta A)_s\frac{|\delta f_s|^2}{\alpha^2_s}\d C_s\bigg| \cG_{\sigma\smallertext{+}}\bigg] \\
			&\quad - \E^\P\bigg[\int_\sigma^T \cE(\beta A)_{t\smallertext{-}}|\delta \cY_{t\smallertext{-}}|^2\d A_t \bigg| \cG_{\sigma\smallertext{+}} \bigg], \; \textnormal{$\P$--a.s.}
		\end{align*}
		The first inequality follows from \eqref{eq::stability_conditional_expectation2} and from the predictable projection theorem \cite[Remark VI.58.(b), page 123]{dellacherie1982probabilities}, and the second inequality follows from the arguments used to deduce \cite[Equation (5.23)]{possamai2024reflections}. Similarly, using \cite[Equation (5.21)]{possamai2024reflections},
		\begin{align*}
			\E^\P[\langle\delta\eta\rangle_T - \langle\delta\eta\rangle_{\sigma}|\cG_{\sigma\smallertext{+}}] 
			&\leq (1+\varpi)\E^\P\big[|\delta\xi|^2\big|\cG_{\sigma\smallertext{+}}\big] + \bigg(1+\frac{1}{\varpi}\bigg)\E^\P\bigg[\bigg(\int_\sigma^T \delta f_s \d C_s\bigg)^2\bigg| \cG_{\sigma\smallertext{+}}\bigg] \\
			&\leq (1+\varpi)\E^\P\big[|\delta\xi|^2\big|\cG_{\sigma\smallertext{+}}\big] + \bigg(1+\frac{1}{\varpi}\bigg) \frac{1}{\beta} \frac{1}{\cE(\beta A)_{\sigma}} \E^\P\bigg[ \int_\sigma^T \cE(\beta A)_s\frac{|\delta f_s|^2}{\alpha^2_s}\d C_s\bigg| \cG_{\sigma\smallertext{+}}\bigg], \; \textnormal{$\P$--a.s.}
		\end{align*}
		This then yields, after a rearrangement of the terms, that
		\begin{align}\label{eq::stability_delta_eta_norm}
			&\beta \E^\P\bigg[\int_\sigma^T \cE(\beta A)_{t\smallertext{-}}|\delta \cY_{t\smallertext{-}}|^2\d A_t \bigg| \cG_{\sigma\smallertext{+}} \bigg] 
			+ \E^\P\bigg[\int_\sigma^T \cE(\beta A)_s \d \langle\delta\eta\rangle_s\bigg|\cG_{\sigma\smallertext{+}}\bigg] \nonumber\\
			& \leq (1+\varpi)\E^\P\big[\cE(\beta A)_{\sigma}|\delta\xi|^2\big|\cG_{\sigma\smallertext{+}}\big] + (1+\varpi)\E^\P\big[\cE(\beta A)_T|\delta\xi|^2\big| \cG_{\sigma\smallertext{+}}\big] \nonumber\\
			& \quad + \bigg(1+\frac{1}{\varpi}\bigg)\bigg(\frac{1}{\beta} + \beta\fg^\Phi(\beta)\bigg) \E^\P\bigg[ \int_\sigma^T \cE(\beta A)_s\frac{|\delta f_s|^2}{\alpha^2_s}\d C_s\bigg| \cG_{\sigma\smallertext{+}}\bigg], \; \textnormal{$\P$--a.s.}
		\end{align}
		Since
		\begin{align*}
			\frac{\beta}{(1+\beta\Phi)}\E^\P\bigg[\int_\sigma^T \cE(\beta A)_{t}|\delta \cY_{t\smallertext{-}}|^2\d A_t \bigg| \cG_{\sigma\smallertext{+}} \bigg] 
			&= \frac{\beta}{(1+\beta\Phi)} \E^\P\bigg[\int_\sigma^T \cE(\beta A)_{t\smallertext{-}}(1+\beta\Delta A_t)|\delta\cY_{t\smallertext{-}}|^2\d A_t \bigg| \cG_{\sigma\smallertext{+}} \bigg] \\
			&\leq \beta \E^\P\bigg[\int_\sigma^T \cE(\beta A)_{t\smallertext{-}}|\delta\cY_{t\smallertext{-}}|^2\d A_t \bigg| \cG_{\sigma\smallertext{+}} \bigg], \; \textnormal{$\P$--a.s.},
		\end{align*}
		we have
		\begin{align*}
			&\min\bigg\{1,\frac{\beta}{(1+\beta\Phi)}\bigg\} \bigg( \E^\P\bigg[\int_\sigma^T \cE(\beta A)_{t}|\delta \cY_{t\smallertext{-}}|^2\d A_t \bigg| \cG_{\sigma\smallertext{+}} \bigg] 
			+ \E^\P\bigg[\int_\sigma^T \cE(\beta A)_s \d \langle\delta\eta\rangle_s\bigg|\cG_{\sigma\smallertext{+}}\bigg] \bigg) \\
			&\leq \frac{\beta}{(1+\beta\Phi)}\E^\P\bigg[\int_\sigma^T \cE(\beta A)_{t}|\delta \cY_{t\smallertext{-}}|^2\d A_t \bigg| \cG_{\sigma\smallertext{+}} \bigg] 
			+ \E^\P\bigg[\int_\sigma^T \cE(\beta A)_s \d \langle\delta\eta\rangle_s\bigg|\cG_{\sigma\smallertext{+}}\bigg] \\
			&\leq 2(1+\varpi)\E^\P\big[\cE(\beta A)_T|\delta\xi|^2\big|\cG_{\sigma\smallertext{+}}\big] + \bigg(1+\frac{1}{\varpi}\bigg)\bigg(\frac{1}{\beta} + \beta\fg^\Phi(\beta)\bigg) \E^\P\bigg[ \int_\sigma^T \cE(\beta A)_s\frac{|\delta f_s|^2}{\alpha^2_s}\d C_s\bigg| \cG_{\sigma\smallertext{+}}\bigg], \; \textnormal{$\P$--a.s.},
		\end{align*}
		which we can rewrite to
		\begin{align}\label{eq::stability_conditional_y_minus}
			&\E^\P\bigg[\int_\sigma^T \cE(\beta A)_{t}|\delta \cY_{t\smallertext{-}}|^2\d A_t \bigg| \cG_{\sigma\smallertext{+}} \bigg] 
			+ \E^\P\bigg[\int_\sigma^T \cE(\beta A)_s \d \langle\delta\eta\rangle_s\bigg|\cG_{\sigma\smallertext{+}}\bigg] \nonumber\\
			& \leq 2(1+\varpi)\max\bigg\{1,\frac{(1+\beta\Phi)}{\beta}\bigg\}\E^\P\big[\cE(\beta A)_T|\delta\xi|^2\big|\cG_{\sigma\smallertext{+}}\big] 
			\nonumber\\
			&\quad + \bigg(1+\frac{1}{\varpi}\bigg)\max\bigg\{1,\frac{(1+\beta\Phi)}{\beta}\bigg\} \bigg(\frac{1}{\beta} + \beta\fg^\Phi(\beta)\bigg) \E^\P\bigg[ \int_\sigma^T \cE(\beta A)_s\frac{|\delta f_s|^2}{\alpha^2_s}\d C_s\bigg| \cG_{\sigma\smallertext{+}}\bigg], \; \textnormal{$\P$--a.s.}
		\end{align}
		Next, it follows from \eqref{eq::stability_conditional_expectation} and from the arguments used to derive \cite[Equation~(5.23)]{possamai2024reflections} that
		\begin{align}\label{eq::stability_delta_y}
			&\E^\P\bigg[\int_\sigma^T\cE(\beta A)_s|\delta\cY_s|^2\d A_s\bigg|\cG_{\sigma\smallertext{+}}\bigg] \nonumber\\
			&\leq (1+\varpi)\E^\P\bigg[\int_\sigma^T\cE(\beta A)_s|\delta\xi|^2\d A_s \bigg|\cG_{\sigma\smallertext{+}}\bigg] + \bigg(1+\frac{1}{\varpi}\bigg)\E^\P\bigg[\int_\sigma^T\cE(\beta A)_s\bigg(\int_s^T |\delta f_u| \d C_u\bigg)^2\d A_s \bigg| \cG_{\sigma\smallertext{+}} \bigg] \nonumber\\
			&\leq (1+\varpi)\frac{(1+\beta\Phi)}{\beta}\E^\P\bigg[|\delta\xi|^2 \int_\sigma^T \cE(\beta A)_{s\smallertext{-}}\d (\beta A)_s \bigg|\cG_{\sigma\smallertext{+}} \bigg] + \bigg(1+\frac{1}{\varpi}\bigg)\ff^\Phi(\beta) \E^\P\bigg[\int_\sigma^T\cE(\beta A)_s\frac{|\delta f_s|^2}{\alpha^2_s} \d C_s \bigg| \cG_{\sigma\smallertext{+}} \bigg]  \nonumber\\
			&\leq (1+\varpi)\frac{(1+\beta\Phi)}{\beta}\E^\P\big[\cE(\beta A)_T|\delta\xi|^2\big|\cG_{\sigma\smallertext{+}} \big] + \bigg(1+\frac{1}{\varpi}\bigg)\ff^\Phi(\beta) \E^\P\bigg[\int_\sigma^T\cE(\beta A)_s\frac{|\delta f_s|^2}{\alpha^2_s} \d C_s \bigg| \cG_{\sigma\smallertext{+}} \bigg], \; \textnormal{$\P$--a.s.}
		\end{align}
		Lastly, \eqref{eq::stability_conditional_expectation}, together with the arguments that lead to \cite[Equation~(5.21) and (5.25)]{possamai2024reflections}, yields
		\begin{align*}
			\cE(\beta A)_{\sigma}|\delta\cY_\sigma|^2 
			&\leq (1+\varpi)\cE(\beta A)_{\sigma}\E^\P\big[|\delta\xi|^2 \big| \cG_{\sigma\smallertext{+}}\big] + \bigg(1+\frac{1}{\varpi}\bigg)\cE(\beta A)_{\sigma}\E^\P\bigg[\bigg(\int_\sigma^T \delta f_s \d C_s\bigg)^2\bigg|\cG_{\sigma\smallertext{+}}\bigg] \\
			&\leq (1+\varpi)\E^\P\big[\cE(\beta A)_{T}|\delta\xi|^2 \big| \cG_{\sigma\smallertext{+}}\big] + \bigg(1+\frac{1}{\varpi}\bigg)\frac{1}{\beta}\E^\P\bigg[\int_\sigma^T \cE(\beta A)_s\frac{|\delta f_s|^2}{\alpha^2_s}\d C_s\bigg|\cG_{\sigma\smallertext{+}}\bigg], \; \textnormal{$\P$--a.s.}
		\end{align*}
		Combining this with \eqref{eq::stability_delta_y} and \eqref{eq::stability_conditional_y_minus}, we obtain, $\P$--a.s.,
		\begin{align}\label{eq::stability_M_1_delta_f}
			&\cE(\beta A)_{\sigma}|\delta\cY_\sigma|^2 + \E^\P\bigg[\int_\sigma^T\cE(\beta A)_s|\delta\cY_s|^2\d A_s + \int_\sigma^T \cE(\beta A)_{s}|\delta \cY_{s\smallertext{-}}|^2\d A_s + \int_\sigma^T \cE(\beta A)_s \d \langle\delta\eta\rangle_s\bigg|\cG_{\sigma\smallertext{+}}\bigg] \nonumber\\
			&\leq (1+\varpi)\bigg( 1 + \frac{(1+\beta\Phi)}{\beta} + 2\max\bigg\{1,\frac{(1+\beta\Phi)}{\beta}\bigg\}\bigg)\E^\P\big[\cE(\beta A)_T|\delta\xi|^2\big| \cG_{\sigma\smallertext{+}}\big] \nonumber\\
			&\quad + \bigg(1+\frac{1}{\varpi}\bigg)\widetilde M^\Phi_1(\beta)\E^\P\bigg[\int_\sigma^T \cE(\beta A)_s\frac{|\delta f_s|^2}{\alpha^2_s}\d C_s\bigg|\cG_{\sigma\smallertext{+}}\bigg].
		\end{align}
		We now bound the last term in the above inequality. Letting $\kappa \in (0,\infty)$ be arbitrary, and recalling $\delta f = \delta_1 f + \delta_2 f$, we find using $2ab \leq a^2/\kappa + \kappa b^2$ and the Lipschitz-continuity property of $f$ that
		\begin{align}\label{eq::stability_decomposition_delta_f}
			&\E^\P\bigg[\int_\sigma^T \cE(\beta A)_s\frac{|\delta f_s|^2}{\alpha^2_s}\d C_s\bigg|\cG_{\sigma\smallertext{+}}\bigg] \nonumber\\
			&= \E^\P\bigg[\int_\sigma^T \cE(\beta A)_s \frac{|\delta_1 f_s + \delta_2 f_s|^2}{\alpha^2_s}\d C_s \bigg| \cG_{\sigma\smallertext{+}}\bigg] \nonumber\\
			&=\E^\P\bigg[\int_\sigma^T \cE(\beta A)_s \frac{|\delta_1 f_s|^2  + 2 (\delta_1 f_s)(\delta_2 f_s) + |\delta_2 f_s|^2}{\alpha^2_s}\d C_s \bigg| \cG_{\sigma\smallertext{+}}\bigg] \nonumber\\
			&\leq \bigg(1+\frac{1}{\kappa}\bigg)\E^\P\bigg[\int_\sigma^T \cE(\beta A)_s \frac{|\delta_1 f_s|^2}{\alpha^2_s}\d C_s \bigg| \cG_{\sigma\smallertext{+}}\bigg] + (1+\kappa)\E^\P\bigg[\int_\sigma^T \cE(\beta A)_s \frac{|\delta_2 f_s|^2}{\alpha^2_s}\d C_s \bigg| \cG_{\sigma\smallertext{+}}\bigg] \nonumber\\
			&\leq
			\bigg(1+\frac{1}{\kappa}\bigg)\Bigg(\E^\P\bigg[\int_\sigma^T\cE(\beta A)_s|\delta\cY_s|^2\d A_s + \int_\sigma^T \cE(\beta A)_{s}|\delta \cY_{s\smallertext{-}}|^2\d A_s + \int_\sigma^T \cE(\beta A)_s \d \langle\delta\eta\rangle_s\bigg|\cG_{\sigma\smallertext{+}}\bigg]\Bigg) \nonumber\\
			&\quad + (1+\kappa)\E^\P\bigg[\int_\sigma^T \cE(\beta A)_s \frac{|\delta_2 f_s|^2}{\alpha^2_s}\d C_s \bigg| \cG_{\sigma\smallertext{+}}\bigg], \; \textnormal{$\P$--a.s.}
		\end{align}
		Since $\widetilde M^\Phi_1(\beta) < 1$, we choose $(\varpi,\kappa) \in (0,\infty)^2$ large enough so that $(1+1/\varpi)(1+1/\kappa)\widetilde M^\Phi_1(\beta) < 1$. With \eqref{eq::stability_decomposition_delta_f}, we rearrange the terms in \eqref{eq::stability_M_1_delta_f} and obtain
		\begin{align*}
			&\cE(\beta A)_{\sigma}|\delta\cY_\sigma|^2 + \E^\P\bigg[\int_\sigma^T\cE(\beta A)_s|\delta\cY_s|^2\d A_s + \int_\sigma^T \cE(\beta A)_s|\delta \cY_{s\smallertext{-}}|^2\d A_s \bigg| \cG_{\sigma\smallertext{+}} \bigg]  \\
			&\quad+ \E^\P\bigg[\int_\sigma^T \cE(\beta A)_s (\delta\cZ_s)^\top\pi_s\delta\cZ_s \d C_s  + \int_\sigma^T \cE(\beta A)_s \|\delta\cU_s(\cdot)\|^2_{\hat{\L}^\smalltext{2}_{\smalltext{\cdot}\smalltext{,}\smalltext{s}}(K_{\smalltext{\cdot}\smalltext{,}\smalltext{s}})} \d C_s + \int_\sigma^T \cE(\beta A)_s \d \langle\delta\cN\rangle_s\bigg|\cG_{\sigma\smallertext{+}}\bigg]\\
			& \leq \frac{1}{1- \big(1+\frac{1}{\varpi}\big)\big(1+\frac{1}{\kappa}\big)\widetilde M^\Phi_1(\beta)} \Bigg( (1+\varpi)\bigg( 1 + \frac{(1+\beta\Phi)}{\beta} + 2\max\bigg\{1,\frac{(1+\beta\Phi)}{\beta}\bigg\}\bigg)\E^\P\big[\cE(\beta A)_T|\delta\xi|^2\big| \cG_{\sigma\smallertext{+}}\big] \\
			&\quad + \bigg(1+\frac{1}{\varpi}\bigg)(1+\kappa) \widetilde M^\Phi_1(\beta)\E^\P\bigg[\int_\sigma^T \cE(\beta A)_s \frac{|\delta_2 f_s|^2}{\alpha^2_s}\d C_s \bigg| \cG_{\sigma\smallertext{+}}\bigg] \Bigg), \; \textnormal{$\P$--a.s.}
		\end{align*}
		This immediately implies $(i)$, and thus concludes the proof.
	\end{proof}
	
	\begin{corollary}\label{cor::stability}
		Suppose that both $(\xi, f)$ and $(\xi^\prime, f^\prime)$ are standard data for $\beta \in (0,\infty)$. If
		\begin{enumerate}
			\item[$(i)$] $\widetilde M^\Phi_1(\beta) < 1$, or
			\item[$(ii)$] $\widetilde M^\Phi_2(\beta) < 1$ and both $f$ and $f^\prime$ do not depend on the $\mathrm{y}$-variable, or
			\item[$(iii)$] $\widetilde M^\Phi_3(\beta) < 1$ and both $f$ and $f^\prime$ do not depend on the $y$-variable,
		\end{enumerate}
		then there exists a constant $\mathfrak{C}^\prime \in (0,\infty)$ that depends only on $\beta$ and on $\Phi$ such that
		\begin{equation*}
			\E^\P\bigg[\sup_{s \in [0,T]} \big|\cE(\beta A)^{1/2}_{s}\delta\cY_{s}\big|^2\bigg] 
			\leq \mathfrak{C}^\prime \bigg( \E^\P\bigg[ \cE(\beta A)_{T} |\delta\xi|^2 + \int_0^T \cE(\beta A)_s\frac{|\delta_2 f_s|^2}{\alpha^2_s} \d C_s\bigg]\bigg).
		\end{equation*}
	\end{corollary}
	
	\begin{proof}
		We only prove the statement for $(i)$, as an analogous argument yields the bounds in the other cases. Let $S$ be a $\G_\smallertext{+}$--stopping time. The representation \eqref{eq::representation_delta_y} and \cite[Equation 5.21]{possamai2024reflections}, together with $|a+b| \leq \sqrt{2}\sqrt{a^2 + b^2}$, yield
		\begin{align*}\label{eq::inequality_stability_y}
			\cE(\beta A)^{1/2}_{S\land T} |\delta\cY_S| 
			&\leq \E^\P\Bigg[\bigg|\cE(\beta A)^{1/2}_{S\land T} \delta\xi + \cE(\beta A)^{1/2}_{S\land T} \bigg(\int_S^T \delta f_s \d C_s\bigg) \bigg| \Bigg|\cG_{S\smallertext{+}}\Bigg] \nonumber\\
			&\leq \sqrt{2} \E^\P\Bigg[ \bigg(\cE(\beta A)_{S\land T} \big|\delta\xi\big|^2 + \cE(\beta A)_{S\land T} \bigg(\int_S^T \delta f_s \d C_s\bigg)^2 \bigg)^{1/2} \Bigg|\cG_{S\smallertext{+}}\Bigg] \nonumber\\
			&\leq \sqrt{2} \E^\P\Bigg[ \bigg(\cE(\beta A)_{S\land T} \big|\delta\xi\big|^2 + \frac{1}{\beta}\int_S^T \cE(\beta A)_s \frac{|\delta f_s|^2}{\alpha^2_s} \d C_s  \bigg)^{1/2} \Bigg|\cG_{S\smallertext{+}}\Bigg] \nonumber \\
			&\leq \sqrt{2} \E^\P\Bigg[ \bigg(\cE(\beta A)_{T} \big|\delta\xi\big|^2 + \frac{1}{\beta}\int_0^T \cE(\beta A)_s \frac{|\delta f_s|^2}{\alpha^2_s} \d C_s  \bigg)^{1/2} \Bigg|\cG_{S\smallertext{+}}\Bigg], \; \textnormal{$\P$--a.s.}
		\end{align*}

		The random variable inside the last conditional expectation is square-integrable, and thus, by Doob's $\L^2$-inequality for martingales (see \cite[page 444]{doob1984classical}) and by \cite[Proposition C.3]{possamai2024reflections}, we obtain
		\begin{equation}\label{eq::bound_sup_delta_y_sigma}
			\E^\P\bigg[\sup_{s \in [0,T]}\big|\cE(\beta A)^{1/2}_{s}\delta\cY_{s}\big|^2\bigg] \leq 8 \E^\P\Bigg[\cE(\beta A)_T \big|\delta\xi\big|^2 + \frac{1}{\beta}\int_0^T \cE(\beta A)_s \frac{|\delta f_s|^2}{\alpha^2_s} \d C_s\Bigg].
		\end{equation}
		We now use the Lipschitz-continuity property of $f$ and then apply \Cref{prop::stability} to deduce 
		\begin{align*}
			&\E^\P\Bigg[\int_0^T \cE(\beta A)_s \frac{|\delta f_s|^2}{\alpha^2_s} \d C_s\Bigg] \\
			& = \E^\P\Bigg[\int_0^T \cE(\beta A)_s \frac{|\delta_1 f_s + \delta_2 f_s|^2}{\alpha^2_s} \d C_s\Bigg] \\
			&\leq 2\Bigg(\E^\P\bigg[\int_0^T\cE(\beta A)_s|\delta\cY_s|^2\d A_s +\int_0^T \cE(\beta A)_s|\delta \cY_{s\smallertext{-}}|^2\d A_s + \int_0^T \cE(\beta A)_s (\delta\cZ_s)^\top\pi_s\delta\cZ_s \d C_s   \\
			&\quad + \int_0^T \cE(\beta A)_s \|\delta\cU_s(\cdot)\|^2_{\hat{\L}^\smalltext{2}_{\smalltext{\cdot}\smalltext{,}\smalltext{s}}(K_{\smalltext{\cdot}\smalltext{,}\smalltext{s}})} \d C_s  + \int_0^T \cE(\beta A)_s \frac{|\delta_2 f_s|^2}{\alpha^2_s}\d C_s \bigg]\Bigg) \\
			&\leq \widetilde{\mathfrak{C}} \bigg( \E^\P\bigg[\cE(\beta A)_T|\delta\xi|^2 + \int_0^T \cE(\beta A)_s \frac{|\delta_2 f_s|^2}{\alpha^2_s}\d C_s \bigg] \bigg),
		\end{align*}
		for some $\widetilde{\mathfrak{C}} \in (0,\infty)$ depending only on $\beta$ and on $\Phi$. This, together with \eqref{eq::bound_sup_delta_y_sigma}, then yields the stated result, which concludes the proof.
	\end{proof}
	
	\subsection{Comparison}
	
	A comparison principle for our BSDEs already appeared in \cite[Proposition 7.3]{possamai2024reflections}. However, \cite[Assumption 7.1]{possamai2024reflections} can be weakened. Since the comparison principle is crucial in this work, we include a weaker assumption and a generalised statement for completeness.
	
	\medskip
	Suppose we are given another terminal condition $\xi^\prime$ satisfying \ref{data::expectation_sup} and another generator $f^\prime$ as in \ref{data::generator}. The Lipschitz-continuity property of $f$ then allows us to write
	\begin{align*}
		f_s\big(\omega,y,\mathrm{y},z,u_s(\omega;\cdot)\big) - f^\prime_s\big(\omega,y^\prime,\mathrm{y}^\prime,z^\prime,u^\prime_s(\omega;\cdot)\big) 
		&\geq \lambda^{y,y^\smalltext{\prime}}_s(\omega)(y-y^\prime) + \widehat\lambda^{\mathrm{y},\mathrm{y}^\smalltext{\prime}}_s(\omega)(\mathrm{y}-\mathrm{y}^\prime)  + (\eta^{z,z^\smalltext{\prime}}_s(\omega))^\top \pi_s(\omega)(z-z^\prime) \\
		&\quad+ f_s\big(\omega,y^\prime,\mathrm{y}^\prime,z^\prime,u_s(\omega;\cdot)\big) - f_s\big(\omega,y^\prime,\mathrm{y}^\prime,z^\prime,u^\prime_s(\omega;\cdot)\big) \\
		&\quad + f_s\big(\omega,y^\prime,\mathrm{y}^\prime,z^\prime,u^\prime_s(\omega;\cdot)\big) - f^\prime_s\big(\omega,y^\prime,\mathrm{y}^\prime,z^\prime,u^\prime_s(\omega;\cdot)\big),
	\end{align*}
	where
	\begin{equation*}
		\lambda^{y,y^\smalltext{\prime}}_s(\omega) \coloneqq -\sqrt{r_s(\omega)}\sgn(y-y^\prime), \; \widehat\lambda^{\mathrm{y},\mathrm{y}^\smalltext{\prime}}_s(\omega) \coloneqq - \sqrt{\mathrm{r}_s(\omega)}\sgn(\mathrm{y}-\mathrm{y}^\prime),
	\end{equation*}
	\begin{equation*}
		\eta^{z,z^\smalltext{\prime}}_s(\omega) \coloneqq -\sqrt{\theta^X_s(\omega)} \frac{(z-z^\prime)}{|(z-z^\prime)^\top\pi_s(\omega)(z-z^\prime)|^{1/2}}\mathbf{1}_{\{(z-z^\smalltext{\prime})^\smalltext{\top}\pi_\smalltext{s}(z-z^\smalltext{\prime})\neq 0\}}(\omega).
	\end{equation*}
	
	\medskip
	The structural assumptions on the generator $f$ and its Lipschitz-continuity coefficients that ensure the comparison of solutions are as follows.
	
	\begin{assumption}\label{ass::comparison}
		Given $\hat{\beta} \in (0,\infty)$, the following conditions hold
		\begin{enumerate}[leftmargin=0.8cm]
			\item[$(i)$] $\int_0^T \sqrt{\mathrm{r}_t}\d C_t$ and $\int_0^T\theta^{X}_t \d C_t$ are both $\P$--essentially bounded$;$
			
			\item[$(ii)$]  For each $\P$--{\rm a.s.} c\`adl\`ag process $\cY^\prime \in \cS^{2}_{T,\hat\beta}(\G_\smallertext{+},\P)$,  and $(\cZ,\cZ^\prime,\cU,\cU^{\prime}) \in \big(\H^{2}_{T,\hat\beta}(X;\G,\P)\big)^2 \times \big(\H^{2}_{T,\hat\beta}(\mu;\G,\P)\big)^2$, there exists $\rho  = \rho^{\cY^\smalltext{^\prime},\cZ,\cZ^\smalltext{\prime},\cU,\cU^{\smalltext\prime}} \in \H^2_T(\mu;\G_\smallertext{+},\P)$ such that $\eta^{\cZ,\cZ^\smalltext{\prime}}\Delta X +\Delta (\rho \ast \tilde\mu) > -1$, \textnormal{$\P$--a.s.}, $\langle \rho \ast \tilde\mu\rangle_T$ is $\P$--essentially bounded, and
			\begin{equation*}
				f\big(\cY^\prime,\cY^\prime_{\smallertext{-}},\cZ^\prime,\cU(\cdot)\big) - f\big(\cY^\prime,\cY^\prime_{\smallertext{-}},\cZ^\prime,\cU^{\prime}(\cdot)\big) 
				\geq \frac{\d\langle \rho \ast\tilde\mu,(\cU-\cU^{\prime})\ast\tilde\mu\rangle}{\d C}, \; \text{$\P \otimes \mathrm{d}C$--{\rm a.e.} on $\llparenthesis 0, T \rrbracket$}.
			\end{equation*}
		\end{enumerate}
	\end{assumption}
	
\begin{proposition}\label{prop::comparison}
Suppose that $\xi^\prime$ satisfies \ref{data::expectation_sup} and that $f^\prime$ satisfies \ref{data::generator}, with the same Lipschitz coefficients $(r,\mathrm r,\theta^X,\theta^\mu)$ as $f$. Suppose further that both $(\xi,f)$ and $(\xi^\prime,f^\prime)$ are standard data for some $\hat\beta\in[2+\Phi,\infty)$, where $\Phi<1$, and that \textnormal{\Cref{ass::comparison}} holds for the same $\hat\beta$. The restrictions $\hat\beta\in[2+\Phi,\infty)$ and $\Phi < 1$ are immediately satisfied if
\[
	\min_{i\in\{1,2,3\}}\widetilde M_i^\Phi(\hat\beta)<1.
\]
Let $(\cY,\cZ,\cU,\cN)$ and $(\cY^\prime,\cZ^\prime,\cU^\prime,\cN^\prime)$ be solutions in
\[
	\cS^2_{T,\hat\beta}(\G_\smallertext{+},\P) \times\H^2_{T,\hat\beta}(X;\G,\P) \times\H^2_{T,\hat\beta}(\mu;\G,\P) \times\cH^{2,\perp}_{T,\hat\beta}(X,\mu;\G_\smallertext{+},\P)
\]
to the \textnormal{BSDEs} with data $(\xi,f)$ and $(\xi^\prime,f^\prime)$, respectively, and terminal time $T$.
If $\xi^\prime \leq \xi$, \textnormal{$\P$--a.s.}, and
\begin{equation*}
	f^\prime\big(\cY^\prime,\cY^\prime_{\smallertext{-}},\cZ^\prime,\cU^\prime(\cdot)\big) 
	\leq f\big(\cY^\prime,\cY^\prime_{\smallertext{-}},\cZ^\prime,\cU^\prime(\cdot)\big), \; \textnormal{$\P\otimes \mathrm{d}C$--a.e.},
\end{equation*}
then $\cY^\prime \leq \cY$ up to $\P$-indistinguishability.
\end{proposition}
	
	\begin{proof}
		The arguments are analogous to those in the proof of \cite[Proposition 7.3]{possamai2024reflections}. Thus, we mention only the necessary changes and omit the details. We write 
		\begin{equation*}
			\delta \cY \coloneqq \cY-\cY^\prime, \; \delta \cZ \coloneqq \cZ-\cZ^\prime, \; \delta \cU \coloneqq \cU - \cU^\prime, \; \delta \cN \coloneqq \cN - \cN^\prime, \; \delta \xi \coloneqq \xi - \xi^\prime,
		\end{equation*}
		\begin{equation*}
			\delta f \coloneqq f\big(\cY,\cY_\smallertext{-},\cZ,\cU(\cdot)) - f^\prime\big(\cY^\prime,\cY^\prime_\smallertext{-},\cZ^\prime,\cU^\prime(\cdot)\big),
		\end{equation*}
		and
		\begin{equation*}
			\lambda_s(\omega) \coloneqq \lambda^{\cY_\smalltext{s}(\omega),\cY^\smalltext{\prime}_\smalltext{s}(\omega)}_s(\omega), 
			\; \widehat\lambda_s(\omega) \coloneqq \widehat\lambda^{\cY_{\smalltext{s}\tinytext{-}}(\omega),\cY^\smalltext{\prime}_{\smalltext{s}\tinytext{-}}(\omega)}_s(\omega), 
			\; \eta_s(\omega) \coloneqq \eta^{\cZ_\smalltext{s}(\omega),\cZ^\smalltext{\prime}_\smalltext{s}(\omega)}_s(\omega),
			\; \rho_s(\omega;x) = \rho^{\cY^\smalltext{\prime},\cZ,\cZ^\smalltext{\prime},\cU,\cU^\smalltext{\prime}}_s(\omega;x),
		\end{equation*}
		for simplicity. Let $v \coloneqq \int_0^{\cdot \land T} \gamma_s \d C_s,$ where  $\gamma \coloneqq \frac{\widehat\lambda}{1-\widehat\lambda\Delta C}$, let $w \coloneqq \int_0^{\cdot \land T} \lambda_s \d C_s$, and lastly, let
		\begin{equation}\label{eq::stoch_exp_measure_change_2}
			\frac{\d\Q}{\d\P} \coloneqq \cE(L)_T \coloneqq \cE\big(\eta \bcdot X + \rho \ast\tilde\mu\big)_T.
		\end{equation}
		By following the arguments in the proof of \cite[Proposition 7.3]{possamai2024reflections}, we will eventually arrive at the inequality
		\begin{equation}\label{eq::comp_limit_t2}
			\cE(w)_{t \land T}\cE(v)_{t \land T}\delta \cY_{t \land T} 
			\geq \E^\Q\bigg[ \cE(w)_{t^\smalltext{\prime} \land T}\cE(v)_{t^\smalltext{\prime} \land T}\delta \cY_{t^\smalltext{\prime} \land T} \bigg| \cG_{t\smallertext{+}}\bigg], \; 0 \leq t \leq t^\prime < \infty, \; \text{$\P$--a.s.}
		\end{equation}
		This is \cite[Equation (7.7)]{possamai2024reflections}. Since $\int_0^T \sqrt{\mathrm{r}_u} \d C_u$ is $\P$--essentially bounded, the process $\cE(v)$ is so too, say by $\mathfrak{C} \in (0,\infty)$. Note also that
		\begin{equation*}
			|\cE(w)|^2 = \cE(2w + [w]) = \cE\big(2 w + (\Delta w)^2\big) \leq \cE(2 A + \Phi A) = \cE\big((2+\Phi)A\big) \leq \cE(\hat\beta A),
		\end{equation*}
		where the last inequality holds by assumption.
		 This then implies that
		\begin{align*}
			& \bigg| \E^\Q\bigg[ \cE(w)_{t^\smalltext{\prime} \land T}\cE(v)_{t^\smalltext{\prime} \land T}\delta \cY_{t^\smalltext{\prime} \land T} \bigg| \cG_{t\smallertext{+}}\bigg] - \E^\Q\bigg[ \cE(w)_{t^\smalltext{\prime} \land T}\cE(v)_{t^\smalltext{\prime} \land T}\delta \xi \bigg| \cG_{t\smallertext{+}}\bigg] \bigg| \\
			&\quad \leq \frac{1}{\cE(L)_t} \bigg| \E^\P\bigg[ \cE(L)_{t^\smalltext{\prime}} \cE(w)_{t^\smalltext{\prime} \land T}\cE(v)_{t^\smalltext{\prime} \land T}\delta \cY_{t^\smalltext{\prime} \land T} \bigg| \cG_{t\smallertext{+}}\bigg] - \E^\P\bigg[ \cE(L)_{t^\smalltext{\prime}} \cE(w)_{t^\smalltext{\prime} \land T}\cE(v)_{t^\smalltext{\prime} \land T}\delta \xi \bigg| \cG_{t\smallertext{+}}\bigg] \bigg| \\
			&\quad \leq \frac{1}{\cE(L)_t} \bigg| \E^\P\bigg[ \cE(L)_{t^\smalltext{\prime}} \cE(w)_{t^\smalltext{\prime} \land T}\cE(v)_{t^\smalltext{\prime} \land T}\big(\delta \cY_{t^\smalltext{\prime} \land T} -  \E^\P\big[\delta\xi|\cG_{t^\smalltext{\prime}\smallertext{+}}\big] \big) \bigg| \cG_{t\smallertext{+}}\bigg] \bigg| \\
			&\quad \leq \mathfrak{C}^2 \frac{1}{\cE(L)_t} \E^\P\big[\cE(L)^2_{t^\smalltext{\prime}}\big|\cG_{t\smallertext{+}}\big]^{1/2} \E^\P\bigg[ \cE(\hat\beta A)_{t^\smalltext{\prime}\land T}\big(\delta \cY_{t^\smalltext{\prime} \land T} -  \E^\P\big[\delta\xi|\cG_{t^\smalltext{\prime}\smallertext{+}}\big] \big)^2 \bigg| \cG_{t\smallertext{+}}\bigg]^{1/2} \\
			&\quad \leq \mathfrak{C}^2 \frac{1}{\cE(L)_t} \E^\P\big[\cE(L)^2_{t^\smalltext{\prime}}\big|\cG_{t\smallertext{+}}\big]^{1/2} \E^\P\bigg[ \cE(\hat\beta A)_{t^\smalltext{\prime}\land T}\E^\P\bigg[\int_{t^\smalltext{\prime}}^T \delta f_s \d C_s\bigg|\cG_{t^\smalltext{\prime}\smallertext{+}}\bigg]^2 \bigg| \cG_{t\smallertext{+}}\bigg]^{1/2} \\
			&\quad \leq \mathfrak{C}^2 \frac{1}{\cE(L)_t} \E^\P\big[\cE(L)^2_{t^\smalltext{\prime}}\big|\cG_{t\smallertext{+}}\big]^{1/2} \E^\P\bigg[ \cE(\hat\beta A)_{t^\smalltext{\prime}\land T}\bigg(\int_{t^\smalltext{\prime}}^T \delta f_s \d C_s\bigg)^2 \bigg| \cG_{t\smallertext{+}}\bigg]^{1/2} \\
			&\quad \leq \mathfrak{C}^2 \frac{1}{\cE(L)_t} \E^\P\big[\cE(L)^2_{t^\smalltext{\prime}}\big|\cG_{t\smallertext{+}}\big]^{1/2} \E^\P\bigg[ \int_{t^\smalltext{\prime}}^T \cE(\hat\beta A)_s \frac{|\delta f_s|^2}{\alpha^2_s} \d C_s \bigg| \cG_{t\smallertext{+}}\bigg]^{1/2}, \; \textnormal{$\P$--a.s.}
		\end{align*}
		By \cite[Theorem 2, page 259]{shiryaev2016probability}, we deduce that, $\P$--a.s. along a sequence $(t^\prime_n)_{n \in \N}$ tending to infinity, the first conditional expectation term in the last line converges to $\E^\P[\cE(L)^2_{T}|\cG_{t\smallertext{+}}]$ and the second term converges to zero. Since $\delta \xi \geq 0$, $\P$--a.s., this and \eqref{eq::comp_limit_t2} yield
		\begin{equation*}
			\cE(w)_{t \land T}\cE(v)_{t \land T}\delta \cY_{t \land T} 
			\geq 0, \; \text{$\P$--a.s.}, \; t \in [0,\infty).
		\end{equation*}
		Since $\cE(w)$ and $\cE(v)$ are strictly positive because $\Phi < 1$, we obtain $\delta\cY_t \geq 0$, $\P$--a.s., for every $t\in[0,\infty)$. Right-continuity yields the desired result. This concludes the proof.
	\end{proof}

	{\footnotesize
		\bibliography{bibliographyDylan}}

\end{document}